\documentclass[11pt,reqno]{amsart}  
\usepackage{amsmath,amssymb,amsthm, comment,graphicx,color, cite}
\usepackage{tikz}
\usepackage{fancyhdr}
\usepackage{epsfig}
\usepackage{mathrsfs}

\usepackage[margin=1.1in]{geometry}

\allowdisplaybreaks

\newtheorem{theorem}{Theorem}[section]
\newtheorem{lemma}{Lemma}[section]
\newtheorem{proposition}{Proposition}[section]

\theoremstyle{definition}

\theoremstyle{remark}
\newtheorem{remark}{Remark}[section]

\numberwithin{equation}{section}

\allowdisplaybreaks

\begin{document}
\title[
Transonic contact discontinuity in a nozzle]
{Stability of transonic contact discontinuity for two-dimensional steady compressible Euler flows in a finitely long nozzle}

\author[F. Huang]{Feimin Huang}
\address{Institute of Applied Mathematics,
Academy of Mathematics and Systems Science,
Chinese Academy of Sciences, Beijing 100190,
China}
\email{fhuang@amt.ac.cn}

\author[J. Kuang]{Jie Kuang}
\address{ Innovation Academy for Precision Measurement Science and
Technology, Chinese Academy of Sciences, Wuhan 430071, China;  Wuhan Institute of Physics and Mathematics,
Chinese Academy of Sciences, Wuhan 430071, China;
Institute of Applied Mathematics,
Academy of Mathematics and Systems Science,
Chinese Academy of Sciences, Beijing 100190,
China}
\email{jkuang@wipm.ac.cn, \ jkuang@apm.ac.cn}

\author[D. Wang]{Dehua Wang}
\address{Department of Mathematics,
University of Pittsburgh,
Pittsburgh, PA 15260, USA.}
\email{dwang@math.pitt.edu}

\author[W. Xiang]{Wei Xiang}
\address{Department of Mathematics, City University of Hong Kong, Kowloon, Hong Kong, China}
\email{weixiang@cityu.edu.hk}


\date{}

\begin{abstract}

We consider the stability of transonic contact discontinuity for the two-dimensional steady compressible Euler flows in a finitely long nozzle. This is the first work on the  mixed-type problem of transonic flows  across a contact discontinuity as a free boundary  in nozzles. We start with the Euler-Lagrangian transformation to straighten the contact discontinuity in the new coordinates. However, the upper nozzle wall in the subsonic region depending on the mass flux becomes a free boundary after the transformation. Then we develop new ideas and techniques to solve the free-boundary problem in three steps: (1) we fix the free boundary and generate a new iteration scheme to solve the corresponding fixed boundary value problem of the hyperbolic-elliptic mixed type by building  some powerful  estimates for both the first-order hyperbolic equation  and a second-order nonlinear elliptic equation in a Lipschitz domain; (2) we update the new free boundary by constructing a mapping that has a fixed point; (3) we establish via the inverse Lagrangian coordinate transformation that the original free interface problem admits a unique piecewise smooth transonic solution near the background state, which consists of a smooth subsonic flow and a smooth supersonic flow with a contact discontinuity.
\end{abstract}

\maketitle
\section{Introduction }\setcounter{equation}{0}
We are concerned with the stability of steady transonic contact discontinuity for the compressible flows in a {two-dimensional} (2D) finitely long nozzle.
The underlying equations are the 2D steady full compressible Euler equations of the following form:
\begin{eqnarray}\label{eq:1.1}
\left\{
\begin{array}{llll}
     \partial_x(\rho u)+\partial_y(\rho v)=0, \\[5pt]
     \partial_{x}(\rho u^{2}+p)+\partial_{y}(\rho uv)=0, \\[5pt]
   \partial_{x}(\rho uv)+\partial_{y}(\rho v^{2}+p)=0, \\[5pt]
     \partial_{x}\big((\rho E+p)u\big)+\partial_{y}\big((\rho E+p)v\big)=0,
     \end{array}
     \right.
\end{eqnarray}
where $(u,v)$, $p$ and $\rho$ stand for the velocity, pressure and density, respectively;  and the total energy $E$ is given by
\begin{eqnarray}\label{eq:1.2}
E=\frac{1}{2}(u^2+v^2)+e(\rho, p).
\end{eqnarray}
Here $e$ is the {\color{black} internal} energy that is a function of $(\rho,p)$ through the thermodynamics relations.
For the ideal gas

\begin{eqnarray}\label{eq:1.6}
p ={\color{black}A(S)\rho^{\gamma}},
\qquad\mbox{and}\qquad e=\frac{\kappa}{\gamma-1}
\rho^{\gamma-1} e^{\frac{S}{c_{\nu}}}, \qquad{\color{black}\mbox{for}\quad A(S)=\kappa e^{\frac{S}{c_{\nu}}},}
\end{eqnarray}
where $S$ is the entropy, and $\gamma>1$, $\kappa$ and $c_{\nu}$ are all positive constants. The sonic speed of the flow for the ideal gas is 
\begin{equation}\label{eq:1.7}
c=\sqrt{\frac{\gamma p}{\rho}}.
\end{equation}
The Mach number {\color{black} and the flow slope (i.e., the tangent function of flow angle) are} defined by
\begin{eqnarray}\label{eq:1.7b}
M=\frac{\sqrt{u^{2}+v^{2}}}{c} {\color{black}\qquad\mbox{and}\qquad \omega=\frac{v}{u}}.
\end{eqnarray}
The flow is called supersonic if $M>1$, subsonic if $M<1$ and sonic if $M=1$.
For  supersonic flows  the system \eqref{eq:1.1} is a hyperbolic system of conservation laws, 
 for subsonic flows it is of {hyperbolic-elliptic composite type}, and for transonic flows
it is of hyperbolic-elliptic composite and mixed type.
%
For the smooth solutions to the system \eqref{eq:1.1}, the Bernoulli function
\begin{eqnarray}\label{eq:1.9}
 B= \frac{1}{2}(u^2+v^2)+\frac{\gamma p}{(\gamma-1)\rho},
\end{eqnarray}
 and the entropy $S$ satisfy the following:
\begin{eqnarray}\label{eq:1.8}
u\partial_{x}B+v\partial_{y}B=0\qquad\mbox{and}\qquad  u\partial_{x}S+v\partial_{y}S=0,
\end{eqnarray}
and consequently the Bernoulli function and entropy are preserved along the streamlines.

As discussed in Courant-Friedrichs \cite{cf}, compressible fluids in nozzle exhibit abundant nonlinear phenomena, for example, supersonic bubble, Mach shock configuration, jet flow, and their interactions. All these phenomena are formulated by elementary waves such as contact discontinuities. Due to the importance in   applications (e.g.,  aerodynamics)  and complex nonlinear phenomena, rigorous mathematical analysis of flows in a nozzle is of great interest but a formidable task.

Many works have been done on   the   compressible Euler flow in nozzles.
The problem of the transonic shock in a nozzle, governed by an equation of the hyperbolic-elliptic mixed-type (see \cite{csx3, fk}), has been extensively studied, see  \cite{ccf,cf1,cf2,csx1,csx2,cheny,fly, fx,lxy1,lxy2,y} and their references for the recent progress.
The  problem of  the subsonic flow in 
nozzles was first studied in \cite{bl}.  A   well-posedness result for the subsonic irrotational flow in a {2D} 
infinitely long nozzle with a given appropriate incoming mass flux was obtained in \cite{xx1}, and
  was  extended to the isentropic Euler flow 
with small non-zero vorticity in 
\cite{xx2} and  then  in  \cite{dxx} for large non-zero vorticity with sign condition on the second-order derivative of the incoming horizontal velocity. For the full Euler flow, the well-posedness   was obtained in \cite{cdx} for the smallness non-zero vorticity  and 
  in \cite{chwx} for large non-zero vorticity without any additional conditions on the vorticity. For other related problems, one can see \cite{cdsw,chw,chwx,wx1} for the sonic-subsonic limit, \cite{chwx1} for the incompressible limit, \cite{chengdx,chengdx1,chengdx2,qx} for the jet flow, \cite{dengwx} for the axisymmetric flow with nontrivial swirl, and \cite{dwx} for the subsonic flow in a finite nozzle. 
For the supersonic flow in nozzles, {one  basic feature is that a shock will be generated 
if the nozzles are taking sufficiently long (see \cite[Appendix]{hkwx}).} 
We see that most of these results are on the nozzles with special structures, for instance, the supersonic flow through an expanding nozzle, which was proposed by Courant-Friedrichs in \cite{cf} and was studied recently in \cite{cq,wx2,xy}.

\par The study of 
the vortex sheets in steady compressible fluids is also an interesting topic, which has drawn a lot of attention recently.
For the subsonic flow, 
the stability of an almost flat 
 contact discontinuity in two dimensional nozzles of infinite  length was established in  \cite{bm}; and see 
 \cite{bp1, bp2} for  further related results. Recently, the uniqueness and existence of the contact discontinuity, which is large, i.e., not a perturbation of the straight one, was obtained in \cite{chwx}.
For the supersonic flow,
{the  stability 
  over a 2D Lipschitz wall  in the BV space was proved in \cite{czz}}.
Recently, 
the well-posedness theory for the steady supersonic compressible Euler flows through a 2D finitely long nozzle with a contact discontinuity 
was established in \cite{hkwx}. 
For other related problems on the steady supersonic contact discontinuity over a wedge with a sharp convex corner,
we refer the reader to \cite{cky,dy,kyz,qx,wy1,wy2,wy3,xiangzz}. 
The contact discontinuity in the Mach reflections was also studied in \cite{csx4,chenf, chf}.

In this paper, we study the stability of steady transonic contact discontinuity in a {2D} finitely long nozzle,
which is the first work on this topic. The transonic contact discontinuity is a free boundary that separates the subsonic flow on the upper layer and  supersonic flow on the lower layer in the nozzle. {The length of the nozzle is taken to be finite since  singularities will generally develop from the smooth supersonic flow when the nozzle is infinitely long (see the Appendix of \cite {hkwx}).
Moreover, the loss of the regularity for the supersonic flow
will lead to the low regularities of the boundary as well as the data along the boundary for the subsonic region,  which makes the elliptic boundary value problem for the subsonic flow difficult to deal with.}
The problem of transonic contact discontinuity {in a nozzle} is different from the problems studied in
\cite{ccf, cf1,cf2,cf3,csx1,csx2, cf} for the transonic shock in a nozzle where
the supersonic state can be solved {priorly}, and also different from the ones in \cite{csx4,chenf,chf} where the domain concerned is a sector type. 

Mathematically, our problem can be formulated as a nonlinear free boundary value problem governed by the hyperbolic-elliptic composite and mixed-type equations in a nozzle with the contact discontinuity as a free interface. 
The main difficulties 
stem from the fact that the transonic contact discontinuity is a free interface, the states on both sides are unknown, and the contact discontinuity is characteristic from the supersonic side, which requires
that {the mixed-type equations should be solved  first in the subsonic region and then in the supersonic region together}. 
To fix the free boundary, thanks to the fundamental feature of the contact discontinuity, namely, the flow velocity on its both sides is parallel to the tangent of the interface, we can straighten the contact discontinuity, as well as the up and lower nozzle walls,  by employing the Euler-Lagrangian coordinate transformation such that the free interface is fixed in the new coordinates.
However, the upper nozzle wall becomes a free boundary in the new coordinates, since it is a function of the mass flux which can not be determined by the incoming flux completely. Thus, the original free boundary value problem after transformation
 actually becomes a new free boundary value problem, denoted by $(\mathbf{NP})$ in the Lagrangian coordinates. 
We develop a boundary iteration scheme to tackle the new free boundary value problem. More precisely, for a given incoming mass flux
$m^{(\rm e)}$, we solve the nonlinear fixed boundary value problem $(\mathbf{FP})$, then use the solution to update the new mass flux through the conservation of the mass (see \eqref{eq:3.64}). Hence we define a map $\mathcal{T}$ and then show that it is well-defined and contractive, so that it admits a unique fixed point (see Section 7 below). Therefore, the remaining task is to establish the existence and uniqueness of the fixed boundary value problem $(\mathbf{FP})$.

When the flow is $C^{1}$-smooth, the Bernoulli  function $B$ and the entropy $S$ can be solved by \eqref{eq:1.8}
 depending completely on the incoming flow. With this property in hand, for the flow in the subsonic region $\tilde{\Omega}^{(\rm e)}$,
the Euler equations in the Lagrangian coordinates can be further reduced into a nonlinear second-order elliptic equations (see Section 3.2 below)
by employing the { stream} function $\varphi$. While in the supersonic region $\tilde{\Omega}^{(\rm h)}$, due to the genuine nonlinearity of the characteristic fields for the Euler system, a pair of 
Riemann invariants $z_{\pm}$ (see Section 3.3 below) can be found such that
  the Euler equations in the Lagrangian coordinates can be written as a $2\times 2$ diagonal form for $z_{\pm}$.
Therefore,  solving the Euler equations in the region $\tilde{\Omega}^{(\rm e)}\cup\tilde{\Omega}^{(\rm h)}$ is equivalent to
solving the boundary value problem $(\mathbf{FP})$ for the { stream} function $\varphi$ and Riemann invariants $z_{\pm}$ with a contact discontinuity separating the regions. More precisely, we will construct the approximate solutions for the problem $(\mathbf{FP})$ by linearizing it near the background state,    denoted by $(\mathbf{FP})_{\rm n}$,  and then developing a  well-defined iteration scheme
to show that 
the sequence of approximate solutions 
constructed by $(\mathbf{FP})_{\rm n}$ is convergent.
The main difficulty for the linearized problem $(\mathbf{FP})_{\rm n}$ is that the solutions for $z_{\pm}$ in the supersonic region $\tilde{\Omega}^{(\rm h)}$ cannot be determined completely because of 
the loss of normal components on the contact discontinuity.
To overcome these difficulties, we first replace the boundary condition $\partial_{\xi}\varphi=\tan\frac{z_{-}+z_{+}}{2}$
on $\tilde{\Gamma}_{\rm cd}$ by $\partial_{\xi}\varphi=\omega_{\rm cd}$ for any given flow slope $\omega_{\rm cd}$ in $(\mathbf{FP})_{\rm n}$ and formulate the modified linearized boundary value problem $(\widetilde{\mathbf{FP}})_{\rm n}$.
The advantage of this replacement  is that we can avoid the singularities of the solution $\varphi$ near the corner point $\mathcal{O}=(0,0)$  being  transformed into the supersonic region $\tilde{\Omega}^{(\rm h)}$, otherwise it would lead  to the loss of regularities for $z_{\pm}$.
Then we use the equation \eqref{eq:3.79} to update a new $\tilde{\omega}_{\rm cd}$, which creates a map $\mathbf{T}_{\omega}$.
Due to 
lack of the contractivity  for $\mathbf{T}_{\omega}$, we employ the implicit function theorem  to show that $\mathbf{T}_{\omega}$ admits a unique fixed point. This step will be achieved in Section 5.

\par The existence and uniqueness of the solutions for the modified linearized boundary value problem $(\widetilde{\mathbf{FP}})_{\rm n}$ can be proved as follows. We first solve the elliptic boundary value problem in $\tilde{\Omega}^{(\rm e)}$ by applying the second-order linear elliptic theory, and obtain the solutions for $\varphi$ and the regularity near the corner $\mathcal{O}=(0,0)$.  
Then, with the solution $\varphi$, we employ the continuity  of the pressure $\tilde{p}$ on $\tilde{\Gamma}_{\rm cd}$ to derive a boundary condition for $z$ on $\tilde{\Gamma}_{\rm cd}$. Next, we need to solve the initial-boundary value problem for $z$ in the supersonic region $\tilde{\Omega}^{(\rm h)}$ through the characteristic method within different subregions generated by the reflections on the contact discontinuity $\tilde{\Gamma}_{\rm cd}$ as well as the reflections on the lower nozzle wall. Finally, from the above procedures, several $C^{2,\alpha}$ and $C^{1,\alpha}$  estimates for the solutions to 
the modified linearized boundary value problem $(\widetilde{\mathbf{FP}})_{\rm n}$ are derived. With these estimates as well as the estimates for $\omega_{\rm cd}$ obtained in Section 5, one can deduce 
in Section 6 that the map generated by the iteration scheme is contractive 
in {\color{black}$C^{1,\alpha}(\tilde{\Omega}^{(\rm e)})\times C^{0,\alpha}(\tilde{\Omega}^{(\rm h)})$}-norm.
Using this property together with the compactness in the function space $C^{2,\alpha}(\tilde{\Omega}^{(\rm e)})\times C^{1,\alpha}(\tilde{\Omega}^{(\rm h)})$,
we can further obtain the convergence of the approximate solutions with  a unique limit
that  is the solution of the problem $(\mathbf{FP})$.

\par The rest of the paper is organized as follows.
In Section 2, we 
first formulate the transonic flow problem in nozzles with a contact discontinuity mathematically, i.e., {Problem A}, and then state the main theorem of the paper.
In Section 3,  we use the Euler-Lagrangian coordinate transformation   to straighten the free boundary
 and then   further formulate this problem in the new coordinates, i.e., {Problem B}.
In order to solve {Problem B}, the {stream} function and the generalized Riemann invariants are introduced to reduce the Euler equations in the Lagrangian coordinates  to the nonlinear second-order elliptic equation in the subsonic region $\tilde{\Omega}^{(\rm e)}$ and the first-order diagonalized system of hyperbolic equations in the supersonic region $\tilde{\Omega}^{(\rm h)}$, and state the main theorem in the Lagrangian coordinates.
We then develop some iteration scheme to prove the main theorem in the Lagrangian coordinates by first fixing the free boundary
and then using the flow slope as the boundary condition to solve the fixed boundary value problem.
In Section 4, we  consider the linearized fixed boundary value problem with  the flow slope as the boundary condition introduced in Section 3 and some 
estimates for the approximate solutions are established case by case.
In Section 5, we will show that the mapping for the flow slope has a unique fixed point by employing the implicit function theorem.
In Section 6, we show the convergence of the approximate solutions obtained by the iteration scheme in Section 3
and then complete the proof of existence and uniqueness for the fixed boundary value problem by using the contraction mapping arguments.
Finally, in Section 7, we prove the main result in the Lagrangian coordinates by constructing a boundary iteration map and show
that it is one to one and contractive and thus has a unique fixed point.


\section{Mathematical Problems and Main Results}\setcounter{equation}{0}
In this section, we first formulate the stability problem of transonic flows in a {2D} finitely long nozzle (see Fig. \ref{fig1.1}) for the Euler equations \eqref{eq:1.1} with a contact discontinuity as a free interface, and then present the main theorem of the paper.

For two given functions $g_{+}(x)$ and $g_{-}(x)$ {\color{black}satisfying $g_{-}(0)<0<g_{+}(0)$}, set
\begin{equation}\label{eq:1.10}
\Omega:=\big\{(x,y)\in \mathbb{R}^{2}: 0<x<L,\ g_{-}(x)<y<g_{+}(x)\big\},
\end{equation}
which describes the domain in the nozzle.
Denoted by $\Gamma_{-}$ and $\Gamma_{+}$  the lower and upper walls of nozzle   defined in the following: 
\begin{eqnarray}\label{eq:1.13}
\begin{split}
\Gamma_{-}:=\big\{(x,y): 0<x<L,\ y=g_{-}(x) \big\}, \ \
\Gamma_{+}:=\big\{(x,y):0<x<L,\ y=g_{+}(x) \big\}.
\end{split}
\end{eqnarray}
Let the location of the contact discontinuity be $\Gamma_{\rm cd}=\big\{y=g_{\rm cd}(x), \ 0<x<L \big\}$ with $g_{\rm cd}(0)=0$, which divides the domain
$\Omega$ into the subsonic and supersonic  regions:
\begin{eqnarray}\label{eq:2.9}
\Omega^{(\rm e)}:=\Omega\cap\big\{g_{\rm cd}(x)<y<g_{+}(x)\big\},\ \
\Omega^{(\rm h)}:=\Omega\cap\big\{g_{-}(x)<y<g_{\rm cd}(x) \big\}.
\end{eqnarray}
The entrance of the nozzle in the subsonic and supersonic regions is described as
\begin{eqnarray}\label{eq:1.11}
\begin{split}
\Gamma^{(\rm e)}_{\rm in}:=\big\{(x,y):x=0,\ 0<y<g_{+}(0) \big\},\ \
\Gamma^{(\rm h)}_{\rm in}:=\big\{(x,y):x=0,\ g_{-}(0)<y<0 \big\},
\end{split}
\end{eqnarray}
and the exit of the nozzle in the subsonic region is denoted by
\begin{eqnarray}\label{eq:1.12x}
\begin{split}
\Gamma^{(\rm e)}_{\rm ex}:=\big\{(x,y):x=L,\ g_{\rm cd}(L)<y<g_{+}(L)\big\}.
\end{split}
\end{eqnarray}

\vspace{10pt}
\begin{figure}[ht]
	\begin{center}
		\begin{tikzpicture}[scale=1.1]
		\draw [thick][->] (-5.7,1.9)to[out=30, in=-160](-4.8, 1.9);
		\draw [thick][->] (-5.7,1.4)to[out=30, in=-160](-4.8, 1.4);
		\draw [thick][->] (-5.7,0.5)to[out=30, in=-160](-4.8, 0.5);
		\draw [thick][->] (-5.7,0.0)to[out=30, in=-160](-4.8, 0.0);
		\draw [thick][->] (1.2,1.73)to[out=30, in=-160](2.0, 1.73);
		\draw [thick][->] (1.2,1.4)to[out=30, in=-160](2.0, 1.4);
		\draw [line width=0.05cm](-4.5, 2.3)to
		[out=10,in=-160](-1,2.1)to [out=20,in=-170](2.2,2.3);
		\draw [line width=0.05cm](-4.5, -0.3)to
		[out=-10,in=160](-1,-0.1)to [out=-20,in=170](2.2,-0.5);
		\draw [line width=0.03cm][dashed][red] (-4.5,1) to [out=0, in=0]
		(-3.3, 0.9)to [out=-10,in=170](-1,1.1)to [out=-20,in=170](2.2,1.1);
		\draw [thick] (-4.5,2.3)--(-4.5, -0.3);
		\draw [thick] (2.2,2.3)--(2.2, -0.5);
		\node at (2.8,2.3) {$\Gamma_{+}$};
		\node at (2.8,-0.4) {$\Gamma_{-}$};
		\node at (2.8, 1.1) {$\Gamma_{\rm cd}$};
		\node at (-1.0, 1.6) {$\Omega^{(\rm e)}$};
		\node at (-1.0, 0.4) {$\Omega^{(\rm h)}$};
		\node at (-5.3, 2.2) {$V^{(\rm e)}_{0}(y)$};
		\node at (-5.3, 0.9) {$U^{(\rm h)}_{0}(y)$};
		\node at (1.6, 2.0) {$\omega_{\rm e}(y)$};
		\node at (-4.4, -0.8) {$x=0$};
		\node at (2.3, -0.8) {$x=L$};
		\end{tikzpicture}
	\end{center}
	\caption{Transonic flow in a finitely long nozzles with contact discontinuity}\label{fig1.1}
\end{figure}

 \vspace{2pt}
\begin{figure}[ht]
\begin{center}
\begin{tikzpicture}[scale=1.3]
\draw [thin][->] (-5,0.4) --(-4,0.4);
\draw [thin][->] (-5,0) --(-4,0);
\draw [thin][->] (-5,-0.4) --(-4,-0.4);
\draw [thin][->] (-5,-1.0) --(-4,-1.0);
\draw [thin][->] (-5.0,-1.4) --(-4,-1.4);
\draw [thin][->] (-5.0,-1.8) --(-4,-1.8);
\draw [thin][->] (0.5,0) --(1.6,0);
\draw [thin][->] (0.5,-0.4) --(1.6,-0.4);
\draw [line width=0.06cm] (-3.7,0.6) --(1.8,0.6);
\draw [line width=0.03cm][dashed][red] (-3.7,-0.7) --(1.8,-0.7);
\draw [line width=0.06cm] (-3.7,-2) --(1.8,-2);
\draw [thin] (-3.7,-2) --(-3.7,0.6);
\draw [thin] (1.8,-2) --(1.8,0.6);
\node at (2.4, -0.7) {$y=0$};
\node at (-4.5, 0.6) {$\underline{V}^{(\rm e)}$};
\node at (-4.5, -0.8) {$\underline{U}^{(\rm h)}$};
\node at (1.2, 0.2) {$\underline{\omega}_{\rm e}=0$};
\node at (2.4, 0.6) {$y=1$};
\node at (2.4, -2) {$y=-1$};
\node at (-3.8, -2.4) {$x=0$};
\node at (1.7, -2.4) {$x=L$};
\node at (-1, -1) {$\underline{\Omega}$};
\end{tikzpicture}
\caption{Transonic flow in a finitely long flat nozzle with contact discontinuity}\label{fig2.1}
\end{center}
\end{figure}
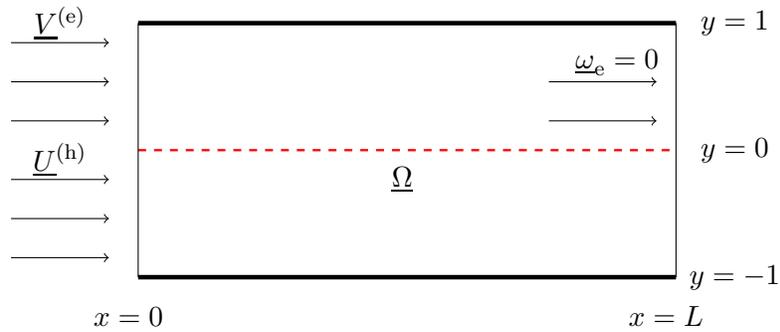

A special case is that a uniform transonic incoming flow goes through a {2D} finitely long flat nozzle,
with a flat transonic contact discontinuity    dividing the flat nozzle into two layers with subsonic and supersonic constant states on the corresponding regions;
see Fig. \ref{fig2.1}.
Let the nozzle with flat boundaries be described as:
\begin{equation}\label{eq:2.1}
\underline{\Omega}:=\big\{(x,y)\in \mathbb{R}^{2}: 0<x<L,\ -1<y<1 \big\},
\end{equation}
%
  the solution in the subsonic region be
\begin{equation}\label{eq:2.4}
\underline{U}^{(\rm e)}:=(\underline{u}^{(\rm e)}, 0,\underline{p}^{(\rm e)}, \underline{\rho}^{(\rm e)})^{\top},
\end{equation}
and  the solution in the supersonic region be
\begin{equation}\label{eq:2.3}
\underline{U}^{(\rm h)}:= \big(\underline{u}^{(\rm h)}, 0,\underline{p}^{(\rm h)},
\underline{\rho}^{(\rm h)}\big)^{\top}.
\end{equation}
These two layers (top layer is the subsonic region and lower layer is the supersonic region) are separated by the straight line $y = 0$. The line $y = 0$ is the contact discontinuity, which divides the domain $\underline{\Omega}$ into two regions, 
\begin{eqnarray*}
	&&\underline{\Omega}^{(\rm e)}:=\underline{\Omega}\cap\big\{0<y<1 \big\},\quad
	\underline{\Omega}^{(\rm h)}:=\underline{\Omega}\cap\big\{ -1<y<0 \big\}.
\end{eqnarray*}
%
Let
\begin{equation}\label{eq:2.2}
\underline{V}^{(\rm e)}:=\big(\underline{p}^{(\rm e)}, \underline{B}^{(\rm e)},\underline{S}^{(\rm e)}\big)^{\top},
\end{equation}
and let $\underline{\omega}^{(\rm e)}=0$ be the angle of the velocity in $\underline{\Omega}^{(\rm e)}$. 
The horizontal velocities, densities and pressures in the top and bottom layers
are all positive, and the pressures in this two 
layers are equal,  
 {i.e.}, $\underline{p}^{(\rm e)}=\underline{p}^{(\rm h)}=\underline{p}$. Moreover, the supersonic and subsonic conditions hold for some $\delta_{0}>0$,
\begin{equation}\label{eq:2.5}
\underline{u}^{(\rm h)}-\underline{c}^{(\rm h)}>\delta_{0}\qquad\mbox{and}\qquad \underline{c}^{(\rm e)}-\underline{u}^{(\rm e)}>\delta_{0},
\end{equation}
where $\underline{c}^{(i)}=\sqrt\frac{\gamma \underline{p}^{(i)}}{\underline{\rho}^{(i)}}$, for $i= \rm e$ or $\rm h$.
Let
\begin{eqnarray}\label{eq:2.6}
\underline{U}(x,y)=\left\{
\begin{array}{llll}
\underline{U}^{(\rm e)},\quad\quad (x, y)\in \underline{\Omega}^{(\rm e)},\\[5pt]
\underline{U}^{(\rm h)},\quad\quad (x, y)\in \underline{\Omega}^{(\rm h)}.
\end{array}
\right.
\end{eqnarray}
Obviously, $\underline{U}(x,y)$ is a weak solution of the 
boundary value problem governed by the Euler system \eqref{eq:1.1} in $\underline{\Omega}$ with 
the boundary $\{(x,y):0<x<L,\ y=0\}$ as the transonic contact discontinuity. The flow is supersonic in $\underline{\Omega}^{(\rm h)}$ and subsonic in $\underline{\Omega}^{(\rm e)}$.
We call the solution $\underline{U}(x,y)$ the \emph{background solution}.


Based on the background solution, we now introduce the problem considered in this paper. Assume that the initial incoming flow $U_{0}(y)$ at $x=0$ is of the following form:
\begin{eqnarray}\label{eq:2.7}
U_{0}(y)=\left\{
\begin{array}{llll}
V^{(\rm e)}_{0}(y),\quad\quad y\in \Gamma^{(\rm e)}_{\rm in},\\[5pt]
U^{(\rm h)}_{0}(y),\quad\quad y\in \Gamma^{(\rm h)}_{\rm in},
\end{array}
\right.
\end{eqnarray}
where $U^{(\rm h)}_{0}(y)=\big(u^{(\rm h)}_{0},v^{(\rm h)}_{0},p^{(\rm h)}_{0}, \rho^{(\rm h)}_{0}\big)^{\top}(y)$
and 
$V^{(\rm e)}_{0}(y)=\big(p^{(\rm e)}_{0}, B^{(\rm e)}_{0}, S^{(\rm e)}_{0}\big)^{\top}(y)$.
Moreover, $V^{(\rm e)}_{0}(y)$ and $\ U^{(\rm h)}_{0}(y)$ satisfy the following compatible conditions:
\begin{eqnarray}\label{eq:2.8}
\begin{split}
{\color{black}v^{(\rm h)}_{0}(0)=0,\ \ \ p^{(\rm e)}_{0}(0)=p^{(\rm h)}_{0}(0),\ \ \ (p^{(\rm e)}_{0})'(0)=(B^{(\rm e)}_{0})'(0)=(S^{(\rm e)}_{0})'(0)=0,}
\end{split}
\end{eqnarray}
and $U_{0}(y)$ satisfies the  compatible conditions at the corner point $(0,g_{-}(0))$.

Denote the flow field in
$\Omega^{(i)},\ (i=\rm e, \rm h)$, by $U^{(i)}=\big(u^{(i)}, v^{(i)}, p^{(i)}, \rho^{(i)}\big)$. Then, on the nozzle walls $\Gamma_{-}$ and $\Gamma_{+}$,
the flows satisfy that 
\begin{eqnarray}\label{eq:2.10}
\big(u^{(\rm h)}, v^{(\rm h)}\big)\cdot \mathbf{n}_{-}=0,
\quad\mbox{on}\quad \Gamma_{-},\qquad  \big(u^{(\rm e)}, v^{(\rm e)}\big)\cdot \mathbf{n}_{+}=0,\quad\mbox{on}\quad \Gamma_{+},
\end{eqnarray}
where $\mathbf{n}_{-}=(g'_{-},-1)$ and $\mathbf{n}_{+}=(-g'_{+},1)$ represent the outer normal vectors of the lower and upper walls $\Gamma_{-},\ \Gamma_{+}$, respectively.

Along the contact discontinuity $y=g_{\rm cd}(x)$,
  the following relations hold 
\begin{eqnarray}\label{eq:2.11}
(u, v)\cdot \mathbf{n}_{\rm cd}=0,
\quad [\frac{v}{u}]=[p]=0, \qquad\mbox{on}\quad \Gamma_{\rm cd},
\end{eqnarray}
where $\mathbf{n}_{\rm cd}=(g'_{\rm cd},-1)$ is the normal vector on $\Gamma_{\rm cd}$ and
the symbol $[\ ]$ stands for the jump of the states on both sides of the contact discontinuity.
Finally, in the subsonic region $\Omega^{(\rm e)}$, the flow slope at the exit
$\Gamma^{(\rm e)}_{\rm ex}$ is given by
\begin{eqnarray}\label{eq:2.12}
\omega^{(\rm e)}(L,y)=\omega_{\rm e}(y),
\end{eqnarray}
{
satisfying
\begin{eqnarray}\label{eq:2.12b}
\omega_{\rm e}(g_{\rm cd}(L))= \frac{v}{u} (L, g_{\rm cd}(L)).
\end{eqnarray}
}

In summary, the problem we will study in this paper is described as follows. \smallskip

\par $\mathbf{Problem }$ $\mathbf{A.}$ For a given transonic incoming flow $U_{0}(y)$ in  \eqref{eq:2.7} at the entrance
satisfying  \eqref{eq:2.8} {and the  compatible conditions at the corner point $(0,g_{-}(0))$},  and a given flow slope
$\omega_{\rm e}(y)$ in \eqref{eq:2.12} at the exit $\Gamma^{(\rm e)}_{\rm ex}$  {satisfying \eqref{eq:2.12b}}, find a unique piecewise smooth transonic solution
$\big(U(x,y),g_{\rm cd}(x)\big)$ 
that  is separated by the contact discontinuity $\Gamma_{\rm cd}$
satisfying {the Euler system \eqref{eq:1.1} in the weak sense and the boundary conditions} \eqref{eq:2.10}-\eqref{eq:2.11}.
Moreover, the solution $\big(U(x,y),g_{\rm cd}(x)\big)$ in $\Omega$ is a small perturbation
of the background solution $(\underline{U},0)$ in $\underline{\Omega}$.

\smallskip
A function $U(x,y)=(u,v,p,\rho)^{\top}$ of {Problem A} is called a weak solution
to the Euler equations \eqref{eq:1.1}
provided that the following
\begin{eqnarray}\label{eq:2.13}
\begin{split}
&\iint_{\Omega} \Big(\rho u\partial_x\zeta+\rho v\partial_y\zeta\Big)dxdy=0, \\[5pt]
&\iint_{\Omega} \Big({\color{black}(}\rho u^{2}+p) \partial_{x}\zeta+\rho uv\partial_{y}\zeta\Big)dxdy=0, \\[5pt]
&\iint_{\Omega}\Big(\rho uv\partial_{x}\zeta+(\rho v^{2}+p)\partial_{y}\zeta\Big)dxdy=0, \\[5pt]
&\iint_{\Omega} \Big(\big((\rho E+p)u\big)\partial_{x}\zeta+\big((\rho E+p)v\big)\partial_{y}\zeta\Big)dxdy=0,
\end{split}
\end{eqnarray}
holds for any $\zeta \in C_0^{\infty}(\Omega)$. 

\par We will give a positive answer to {Problem A}. Before  providing the main theorem of this paper, let us introduce the weighted H\"{o}lder spaces  (c.f. {\cite{bm}, \cite{ccf},} \cite{gt}).
%
Let $\Sigma$ be an open subset of $\partial \Omega$. Denote by  $X=(x,y)$ a point in $\Omega$. Set
\begin{eqnarray}\label{eq:2.15}
 \delta_{X}={\rm dist}(X,\Sigma),\quad \delta_{X,X'}=\min (\delta_{X}, \delta_{X'}),
 \quad X, X'\in \Omega.
\end{eqnarray}
Then, for any integer $m\ge 0$, $k\in \mathbb{R}$ and $\alpha \in(0,1)$, {and a function $u$ defined on $\Omega$}, we define
\begin{eqnarray}\label{eq:2.16}
\begin{split}
&\| u\|_{m,0,\Omega}
 =\sum_{0\leq|\beta|\leq m}\sup_{X\in \Omega}|D^{\beta}u(X)|,\\[5pt]
&[u]_{m,\alpha,\Omega}
=\sum_{|\beta|=m}\sup_{X,\ X'\in \Omega, X\neq X'}\frac{|D^{\beta}u(X)-D^{\beta}u(X')|}
 {|X-X'|^{\alpha}},
\end{split}
\end{eqnarray}
and
\begin{eqnarray}\label{eq:2.17}
\begin{split}
&\| u\|^{(k,\Sigma)}_{m,0,\Omega}
 =\sum_{0\leq|\beta|\leq m}\sup_{X\in \Omega}
 \Big(\delta^{\max(|\beta|+k,0)}_{X}|D^{\beta}u(X)|\Big),\\[5pt]
   &[u]^{(k,\Sigma)}_{m,\alpha,\Omega}
 =\sum_{|\beta|=m}\sup_{X,\ X'\in \Omega, X\neq X'}
 \Big(\delta^{\max(m+\alpha+k,0)}_{X,X'}\frac{|D^{\beta}u(X)-D^{\beta}u(X')|}
 {|X-X'|^{\alpha}}\Big),\\[5pt]
 &\| u\|_{m,\alpha,\Omega}=\| u\|_{m,0,\Omega}+[u]_{m,\alpha,\Omega},\quad
 \| u\|^{(k,\Sigma)}_{m,\alpha,\Omega}=\| u\|^{(k,\Sigma)}_{m,0,\Omega}
 +[u]^{(k,\Sigma)}_{m,\alpha,\Omega},
 \end{split}
\end{eqnarray}
where $\beta=(\beta_{1}, \beta_{2})$  is a multi-index with
$ \beta_{j}\ge 0 \, (j=1,2)$,  $|\beta|=\beta_{1}+\beta_{2}$,  and $D^{\beta}=\partial^{\beta_{1}}_{x}\partial^{\beta_{2}}_{y}$.
We denote by $C^{m, \alpha}(\Omega)$ and $C^{m, \alpha}_{(k, \Sigma)}(\Omega)$ the  function spaces defined below:
\begin{eqnarray}\label{eq:2.18}
\begin{split}
C^{m, \alpha}(\Omega)=\big\{u: \|u\|_{m,\alpha,\Omega}<\infty\big\}, \quad
C^{m, \alpha}_{(k, \Sigma)}(\Omega)=\big\{u: \|u\|^{(k,\Sigma)}_{m,\alpha,\Omega}<\infty\big\}.
\end{split}
\end{eqnarray}
For a vector-valued function $\mathbf{u}=(u_{1}, u_{2},\cdots, u_{n})$, define
\begin{eqnarray}\label{eq:2.19}
\begin{split}
&\|\mathbf{u}\|_{m,\alpha,\Omega}=\sum^{n}_{i=1}\|u_{i}\|_{m,\alpha,\Omega},\quad
\|\mathbf{u}\|^{(k,\Sigma)}_{m,\alpha,\Omega}=\sum^{n}_{i=1}\|u_{i}\|^{(k,\Sigma)}_{m,\alpha,\Omega}.
\end{split}
\end{eqnarray}

{\color{black} If $u$ is a function  
defined on an open subset $\Sigma$ of $\partial\Omega$, let $\Upsilon$ be a subset of the set $\overline{\Sigma}$,   the weighted H\"{o}lder
norm $\|u\|^{(k, \Upsilon)}_{m,\alpha; \Sigma}$ can be defined similarly to \eqref{eq:2.16}-\eqref{eq:2.18} by changing $\Omega$ and $\Sigma$ to $\Sigma$ and $\Upsilon$, respectively.
The standard H\"{o}lder norm $\|u\|_{m,\alpha; \Sigma}$ can also be defined similarly to that in \eqref{eq:2.19}.}

Finally, we define $\Sigma^{(\rm e)}=\partial \Omega^{(\rm e)}\backslash\{\Gamma^{(\rm e)}_{\rm in}\cup\Gamma_{\rm cd}\}$, {\color{black}and
$O=(0,0)$, $P_{\rm e}=(L,g_{\rm cd}(L))$, $Q_{\rm e}=(L,g_{+}(L))$}.

Then the main theorem of this paper can be stated as follows.
\begin{theorem}[Main Theorem]\label{thm:2.1}
There exist constants $\alpha_0\in (0,1)$ and $\epsilon_{0}>0$ 
depending only on $\underline{U}$ and $L$,  such that for any given $\alpha\in(0,\alpha_0)$ and $\epsilon\in(0,\epsilon_0)$, if 
\begin{eqnarray}\label{eq:2.20}
\begin{split}
&
{ \big\|V^{(\rm e)}_{0}-\underline{V}^{(\rm e)}\big\|_{1, \alpha; \Gamma^{(\rm e)}_{\rm in}}}+\big\|U^{(\rm h)}_{0}-\underline{U}^{(\rm h)}\big\|_{1, \alpha; \Gamma^{(\rm h)}_{\rm in}}
 +{\color{black} \big\|\omega_{\rm e}\big\|^{(-1-\alpha,\{P_{\rm e}, Q_{\rm e}\})}_{2,\alpha; \Gamma^{(\rm e)}_{\rm ex}}}\\[5pt]
&\qquad \qquad\qquad\qquad\qquad\qquad \ \ \ +\big\|g_{-}+1\big\|_{2, \alpha; \Gamma_{-}}
+\big\|g_{+}-1\big\|_{2, \alpha; \Gamma_{+}}\leq \epsilon,
\end{split}
\end{eqnarray}
and
\begin{eqnarray}\label{eq:2.21}
\begin{split}
\underline{M}^{(\rm h)}=\frac{\underline{u}^{(\rm h)}}{\underline{c}^{(\rm h)}}>\sqrt{1+\frac14L^{2}},
\end{split}
\end{eqnarray}
 there exists a unique solution $(U(x,y), g_{\rm cd})\in H^{1}_{\rm loc}(\Omega)\times C^{2,\alpha}([0,L))$ to
\emph{Problem A} with the following properties:

{\rm (i)}\ The solution $U$ consists of the supersonic flow $U^{(\rm h)} \in C^{1,\alpha}(\Omega^{(\rm h)})$ and subsonic flow
{ $U^{(\rm e)} \in C^{1,\alpha}_{(-\alpha, \Sigma^{(\rm e)}\backslash\{O\})}(\Omega^{(\rm e)})$}
separated by $y=g_{\rm cd}(x)$,  and the following estimate holds:
\begin{eqnarray}\label{eq:2.22}
\begin{split}
{\big\|U^{(\rm e)}-\underline{U}^{(\rm e)}\big\|^{(-\alpha, \Sigma^{(\rm e)}\backslash\{O\})}
_{1, \alpha; \Omega^{(\rm e)}}}
+\big\|U^{(\rm h)}-\underline{U}^{(\rm h)}\big\|_{1, \alpha; \Omega^{(\rm h)}}\leq C_{0}\epsilon;
\end{split}
\end{eqnarray}

{\rm (ii)}\ The contact discontinuity $y=g_{\rm cd}(x)$ is a stream line with $g_{\rm cd}(0)=0$
and satisfies
\begin{eqnarray}\label{eq:2.23}
\big\|g_{\rm cd}\big\|_{2, \alpha; \Gamma_{\rm cd}\cup\{O\}}\leq C_{0}\epsilon,
\end{eqnarray}
 where $C_{0}>0$ is a constant depending
only on $\underline{U}$ and $L$.
\end{theorem}

\bigskip

\section{Mathematical Reformulation of  {Problem A}}\setcounter{equation}{0}

In this section, we first reformulate     {Problem A} into a new free boundary value problem,  {i.e.,}  {Problem B} in the Lagrangian coordinates through the Euler-Lagrangian coordinate transformation, and then state the main theorem for    {Problem B}
in the new coordinates. Later  we further introduce the {stream} function $\varphi$ and Riemann invariant $z$ to reduce  {Problem B}  to a new free boundary value problem for $(\varphi, z)$, {i.e.,}  {Problem C}. Finally, some strategies and iteration schemes are developed to solve  {Problem C}.

\subsection{Mathematical problem in the Lagrangian coordinates}\setcounter{equation}{0}
Note that the contact discontinuity $\Gamma_{\rm cd}$ is a free streamline   by \eqref{eq:2.11}. We shall
  straighten the free interface $\Gamma_{\rm cd}$ 
in term of the Euler-Lagrangian coordinate  transformation and reformulate {Problem A}
into  a new free boundary value problem in the Lagrangian coordinates.

\par Let $\big(U(x,y), g_{\rm cd}(x)\big)$ be a solution to  {Problem A}. Define
\begin{eqnarray}\label{eq:3.1}
\begin{split}
m^{(\rm e)}=\int^{g_{+}(0)}_{g_{\rm cd}(0)}\rho^{(\rm e)} u^{(\rm e)}(0,\tau)d\tau\qquad\mbox{and}\qquad m^{(\rm h)}=\int^{g_{\rm cd}(0)}_{g_{-}(0)}\rho^{(\rm h)}_{0} u^{(\rm h)}_{0}d\tau.
\end{split}
\end{eqnarray}
Then, 
it follows from $\eqref{eq:1.1}_1$ that
\begin{eqnarray}\label{eq:3.2}
\begin{split}
\int^{g_{+}(x)}_{g_{\rm cd}(x)}\rho^{(\rm e)} u^{(\rm e)}(x, \tau)d\tau=m^{(\rm e)}, 
\qquad \int^{g_{\rm cd}(x)}_{g_{-}(x)}\rho^{(\rm h)} u^{(\rm h)}(x,\tau)d\tau=m^{(\rm h)},
\end{split}
\end{eqnarray}
and
\begin{eqnarray}\label{eq:3.3}
\begin{split}
\int^{g_{+}(x)}_{g_{-}(x)}\rho u(x,\tau)d\tau=m^{(\rm e)}+m^{(\rm h)},
\end{split}
\end{eqnarray}
for any $0<x<L$.

{\color{black}We remark that since $\rho^{(\rm e)}(0,y)$ and $u^{(\rm e)}(0,y)$ are not given in the boundary conditions, $m^{(\rm e)}$ defined by \eqref{eq:3.1} cannot be determined in $\Omega^{(\rm e)}$ priorly.
Thus  $m^{(\rm e)}$ is an unknown quantity.}


Let
\begin{eqnarray}\label{eq:3.4}
\begin{split}
\eta(x,y)=\int^{y}_{g_{-}(x)}\rho u(x,\tau)d\tau-m^{(\rm h)}.
\end{split}
\end{eqnarray}
By $\eqref{eq:1.1}_1$ and the boundary condition \eqref{eq:2.10}, it is easy to see that
\begin{eqnarray}\label{eq:3.5}
\begin{split}
\frac{\partial \eta(x,y)}{\partial x}=-\rho v,\quad  \frac{\partial \eta(x,y)}{\partial y}=\rho u.
 \end{split}
\end{eqnarray}
Thus we can introduce the Lagrangian coordinate transformation $\mathcal{L}$ as
\begin{eqnarray}\label{eq:3.6}
\mathcal{L} : \ \left\{
\begin{array}{llll}
\xi &= x, \\[5pt]
\eta &= \eta(x,y).
\end{array}
\right.
\end{eqnarray}
By a direct computation, we have
\begin{equation}\label{eq:3.7}
\frac{\partial (\xi,\eta)}{\partial (x,y)}=
\Big(
\begin{array}{ccc}
  1 & 0 \\
 -\rho v & \rho u\\
\end{array}
\Big).
\end{equation}
Then, it is easy to see that
\begin{eqnarray}\label{eq:3.16}
\begin{split}
\left\{
\begin{array}{llll}
\partial_{x} &=\partial_{\xi}- \rho v \partial_{\eta},\\[5pt]
\partial_{y} &=\rho u \partial_{\eta},
\end{array}
\right.
\end{split}
\end{eqnarray}
thus  the system \eqref{eq:1.1} in the Lagrangian coordinates becomes
\begin{eqnarray}\label{eq:3.17}
\begin{split}
\left\{
\begin{array}{llll}
\partial_{\xi}\Big(\frac{1}{\tilde{\rho} \tilde{u}}\Big)
- \partial_{\eta}\Big(\frac{\tilde{v}}{\tilde{u}}\Big)=0,\\[5pt]
\partial_{\xi}\Big(\tilde{u}+\frac{\tilde{p}}{\tilde{\rho} \tilde{u}}\Big)
-\partial_{\eta}\Big(\frac{\tilde{p} \tilde{v}}{\tilde{u}}\Big)=0,\\[5pt]
\partial_{\xi}\tilde{v}+\partial_{\eta}\tilde{p}=0,
\end{array}
\right.
\end{split}
\end{eqnarray}
with the Bernoulli laws:
\begin{eqnarray}\label{eq:3.18}
\frac{1}{2}(\tilde{u}^2+\tilde{v}^2)+\frac{\gamma \tilde{p}}{(\gamma-1)\tilde{\rho}}
=\left\{
\begin{array}{llll}
\tilde{B}^{(\rm e)}_{0}(\eta) ,    &\ \ \ \ (\xi,\eta)\in \tilde{\Omega}^{(\rm e)}, \\[5pt]
\tilde{B}^{(\rm h)}_{0}(\eta),  &\ \ \ \ (\xi,\eta)\in \tilde{\Omega}^{(\rm h)}.
\end{array}
\right.
\end{eqnarray}
Here we use the fact that the Bernoulli function is only a function of $\eta$, because by $\eqref{eq:1.1}_4$,  $\tilde{B}^{(i)}_{0}(\eta)$ for $i= \rm e, h$ are conserved along the streamlines. 

The boundary conditions {\color{black}\eqref{eq:2.10}} in the new coordinates become
\begin{eqnarray}\label{eq:3.19}
\frac{\tilde{v}^{(\rm e)}}{\tilde{u}^{(\rm e)}}\bigg|_{\tilde{\Gamma}_{+}}=g'_{+}(\xi),
\quad\quad \frac{\tilde{v}^{(\rm h)}}{\tilde{u}^{(\rm h)}}\bigg|_{\tilde{\Gamma}_{-}}=g'_{-}(\xi),
\end{eqnarray}
and the Rankine-Hugoniot conditions on  $\tilde{\Gamma}_{\rm cd}$ are
\begin{eqnarray}\label{eq:3.20}
\frac{\tilde{v}^{(\rm h)}}{\tilde{u}^{(\rm h)}}\bigg|_{\tilde{\Gamma}_{\rm cd}}
=\frac{\tilde{v}^{(\rm e)}}{\tilde{u}^{(\rm e)}}\bigg|_{\tilde{\Gamma}_{\rm cd}}=g'_{\rm cd}(\xi),
\quad\quad \tilde{p}^{(\rm h)}\big|_{\tilde{\Gamma}_{\rm cd}}=\tilde{p}^{(\rm e)}\big|_{\tilde{\Gamma}_{\rm cd}}.
\end{eqnarray}
We remark that \eqref{eq:3.20} indicates that $\frac{\tilde{v}}{\tilde{u}}$ and $\tilde{p}$ are continuous across $\tilde{\Gamma}_{\rm cd}$.

\par As shown in Fig. \ref{fig2.3}, in the Lagrangian coordinates the domain $\Omega$ becomes
\begin{equation*}
\tilde{\Omega}=\big\{(\xi,\eta)\in \mathbb{R}^{2}:
0<\xi<L,\ -m^{(\rm h)}<\eta<m^{(\rm e)} \big\},
\end{equation*}
the entrance of the nozzle in the subsonic and supersonic regions is
\begin{eqnarray*}
\begin{split}
\tilde{\Gamma}^{(\rm e)}_{\rm in}:=\big\{(\xi,\eta):0<\eta<m^{(\rm e)},\ \xi=0 \big\},\ \
\tilde{\Gamma}^{(\rm h)}_{\rm in}:=\big\{(\xi,\eta):-m^{(\rm h)}<\eta<0,\ \xi=0  \big\},
\end{split}
\end{eqnarray*}
the exit of the nozzle in the subsonic region is
\begin{eqnarray*}\label{eq:1.12}
\begin{split}
\tilde{\Gamma}^{(\rm e)}_{\rm ex}:=\big\{(\xi,\eta):0<\eta<m^{(\rm e)},\ \xi=L \big\}{\color{black},}
\end{split}
\end{eqnarray*}
and the lower and upper nozzle walls are
\begin{equation}\label{3.13}
\tilde{\Gamma}_{-}:=\big\{(\xi,\eta):\eta=-m^{(\rm h)},\ 0<\xi<L\big\},\ \
\tilde{\Gamma}_{+}:=\big\{(\xi,\eta):\eta=m^{(\rm e)},\  0<\xi<L\big\}.
\end{equation}
Moreover, on $\Gamma_{\rm cd}$, we have
\begin{eqnarray*}
\eta(x, g_{\rm cd}(x))=\int^{g_{\rm cd}(x)}_{g_{-}(x)}\rho u(x,\tau)d\tau-m^{(\rm h)}=0.
\end{eqnarray*}
Hence, the free interface $\Gamma_{\rm cd}$ becomes the following fixed straight line
\begin{eqnarray}\label{eq:3.9}
\tilde{\Gamma}_{\rm cd}:=\big\{(\xi,\eta):\eta=0,\ 0<\xi<L\big\}.
\end{eqnarray}

\vspace{5pt}
\begin{figure}[ht]
\begin{center}
\begin{tikzpicture}[scale=1.4]
\draw [thin][->] (-5,0)to[out=30, in=-150](-4,0.1);
\draw [thin][->] (-5,-0.4)to[out=30, in=-150](-4,-0.3);
\draw [thin][->] (-5.0,-1.4)to[out=30, in=-150](-4,-1.3);
\draw [thin][->] (-5.0,-1.8)to[out=30, in=-150](-4,-1.7);
\draw [thin][->] (0.6,-0.25)to[out=30, in=-150](1.6,-0.15);
\draw [thin][->] (0.6,-0.55)to[out=30, in=-150](1.6,-0.45);
\draw [line width=0.06cm] (-3.7,0.6) --(1.8,0.6);
\draw [line width=0.03cm][dashed][red] (-3.7,-0.9) --(1.8,-0.9);
\draw [line width=0.06cm] (-3.7,-2) --(1.8,-2);
\draw [thin] (-3.7,-2) --(-3.7,0.6);
\draw [thin](1.8,-2) --(1.8,0.6);
\node at (-4.5, 0.4) {$\tilde{V}^{(\rm e)}_{0}(\eta)$};
\node at (-4.5, -1.0) {$\tilde{U}^{(\rm h)}_{0}(\eta)$};
\node at (1.2, 0.15) {$\tilde{\omega}_{\rm e}(\eta)$};
\node at (2.4, 0.6) {$\tilde{\Gamma}_{+}$};
\node at (2.4, -0.9) {$\tilde{\Gamma}_{\rm cd}$};
\node at (2.4, -2) {$\tilde{\Gamma}_{-}$};
\node at (-1, -0.4) {$\tilde{\Omega}$};
\node at (1.8, -2.4) {$\xi=L$};
\node at (-3.8, -2.4) {$\xi=0$};
\end{tikzpicture}
\end{center}
\caption{ Transonic contact discontinuity flows in the Lagrangian coordinates}\label{fig2.3}
\end{figure}
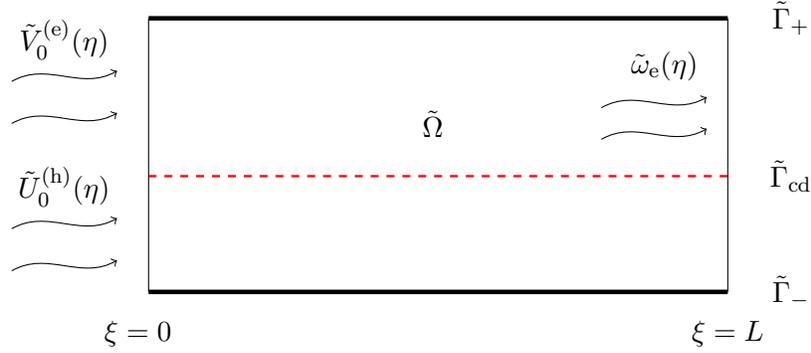
\vspace{5pt}

Denote
\begin{eqnarray}\label{eq:3.10}
\begin{split}
&\tilde{\Omega}^{(\rm e)}:=\tilde{\Omega}\cap\{0<\eta<m^{(\rm e)}\big\},
&\tilde{\Omega}^{(\rm h)}:=\tilde{\Omega}\cap\big\{-m^{(\rm h)}<\eta<0\big\},
\end{split}
\end{eqnarray}
and set
\begin{eqnarray}\label{eq:3.11}
\tilde{U}^{(i)}(\xi,\eta)=(\tilde{u}^{(i)}(\xi,\eta),\tilde{v}^{(i)}(\xi,\eta), \tilde{p}^{(i)}(\xi,\eta), \tilde{\rho}^{(i)}(\xi,\eta))^{\top},
\quad \forall(\xi,\eta)\in\tilde{\Omega}^{(i)}, \ \  i=\rm e, \rm h,
\end{eqnarray}
as the corresponding solutions in $\tilde{\Omega}^{(\rm e)}$ and $\tilde{\Omega}^{(\rm h)}$, respectively.
Then the background solution \eqref{eq:2.6} in the new coordinates is
\begin{eqnarray}\label{eq:3.12}
\underline{\tilde{U}}(\xi,\eta)=\left\{
\begin{array}{llll}
\underline{U}^{(\rm e)},\quad\quad (\xi, \eta)\in \tilde{\underline{\Omega}}^{(\rm e)},\\[5pt]
\underline{U}^{(\rm h)},\quad\quad (\xi, \eta)\in \tilde{\underline{\Omega}}^{(\rm h)},
\end{array}
\right.
\end{eqnarray}
where
\begin{eqnarray*}
&&\tilde{\underline{\Omega}}^{(\rm e)}:=\big\{(\xi, \eta): 0<\xi<L, 0<\eta<\underline{m}^{(\rm e)} \big\},
\ \
\tilde{\underline{\Omega}}^{(\rm h)}:=\big\{(\xi, \eta): 0<\xi<L, -\underline{m}^{(\rm h)} <\eta<0 \big\},
\end{eqnarray*}
and $\underline{m}^{(i)}=\underline{\rho}^{(i)}\underline{u}^{(i)},\ i=\rm e, \rm h$.

At the inlet $\xi=0$, the flow is given by
\begin{eqnarray}\label{eq:3.13}
\tilde{U}_{0}(\eta)=\left\{
\begin{array}{llll}
\tilde{V}^{(\rm e)}_{0}(\eta),\quad\quad  \eta \in \tilde{\Gamma}^{(\rm e)}_{\rm in},\\[5pt]
\tilde{U}^{(\rm h)}_{0}(\eta),\quad\quad  \eta \in \tilde{\Gamma}^{(\rm h)}_{\rm in},
\end{array}
\right.
\end{eqnarray}
where
$\tilde{V}^{(\rm e)}_{0}(\eta)=\big(\tilde{p}^{(\rm e)}_{0}, \tilde{B}^{(\rm e)}_{0}, \tilde{S}^{(\rm e)}_{0}\big)(\eta)$ and
$\tilde{U}^{(\rm h)}_{0}(\eta)=(\tilde{u}^{(\rm h)}_{0},\tilde{v}^{(\rm h)}_{0}, \tilde{p}^{(\rm h)}_{0},
\tilde{\rho}^{(\rm h)}_{0})(\eta)$
satisfy  the following compatible conditions:
\begin{eqnarray}\label{eq:3.14}
\begin{split}
{\color{black}\tilde{v}^{(\rm h)}_{0}(0)=0,\ \ \  \tilde{p}^{(\rm e)}_{0}(0)=\tilde{p}^{(\rm h)}_{0}(0), \ \ \ (\tilde{p}^{(\rm e)}_{0})'(0)=(\tilde{B}^{(\rm e)}_{0})'(0)=(\tilde{S}^{(\rm e)}_{0})'(0)=0.}
\end{split}
\end{eqnarray}
At the corner point {$(0, -m^{(\rm h)})$}, the compatible conditions also hold for $\tilde{U}^{(\rm h)}_{0}$.

Finally, the boundary condition at the exit of the nozzle in the subsonic region is
\begin{eqnarray}\label{eq:3.15}
\tilde{\omega}^{(\rm e)}(L,\eta)=\tilde{\omega}_{\rm e}(\eta),
\end{eqnarray}
{satisfying
\begin{eqnarray}\label{eq:3.15b}
\tilde{\omega}_{\rm e}(0)=\Big(\frac{v}{u}\Big)(L,0),
\end{eqnarray}
}
where $\tilde{\omega}_{\rm e}(\eta)=\omega_{\rm e}\big(y(L, \eta)\big)$.

With the above preparation, we can reformulate  {Problem A} again in the Euler coordinates 
as a new problem in the Lagrangian coordinates with  the upper nozzle wall as a free boundary, {i.e.},  {Problem B}.

\smallskip
\par $\mathbf{Problem }$ $\mathbf{B.}$ For the transonic incoming flow given by \eqref{eq:3.13} at the entrance satisfying \eqref{eq:3.14}
and the compatible conditions at the corner {$(0, -m^{(\rm h)})$}
and a given flow slope \eqref{eq:3.15} at the outlet {\color{black}satisfying \eqref{eq:3.15b}},
find a transonic piecewise  smooth solution $\tilde{U}(\xi,\eta)$ with the straight line $\tilde{\Gamma}_{\rm cd}$ as a contact discontinuity
and a constant $m^{(\rm e)}>0$, satisfying the Euler equations \eqref{eq:3.17} with \eqref{eq:3.18} in $\tilde{\Omega}^{(\rm e)}\cup\tilde{\Omega}^{(h)}$,
the slip boundary condition \eqref{eq:3.19} on $\tilde{\Gamma}_{\pm}$, the Rankine-Hugoniot condition \eqref{eq:3.20} on $\tilde{\Gamma}_{\rm cd}$,
and the boundary condition \eqref{eq:3.15} at  the exit $\tilde{\Gamma}_{\rm ex}^{\rm (e)}$. Moreover, the flow is subsonic in $\tilde{\Omega}^{(\rm e)}$, supersonic in $\tilde{\Omega}^{(\rm h)}$,
and is a small perturbation
of background solution $\underline{\tilde{U}}$ by \eqref{eq:3.12} in $\tilde{\underline{\Omega}}$.

Set $\tilde{\Sigma}^{(\rm e)}=\partial \tilde{\Omega}^{(\rm e)}\backslash \{\tilde{\Gamma}_{\rm in}\cup \tilde{\Gamma}_{\rm cd}\}$, and
$\mathcal{O}=(0,0)$ and $\mathcal{P}_{\rm e}=(L,0)$, $\mathcal{Q}_{\rm e}=(L,m^{(\rm e)})$.

Our main result in the Lagrangian coordinates is the following.

\begin{theorem}\label{thm:3.1}
For a given transonic incoming flow  \eqref{eq:3.13} at the entrance with \eqref{eq:3.14}
and a given flow slope \eqref{eq:3.15} at the outlet { satisfying the compatible conditions at the corner {$(0,-m^{(h)})$} and \eqref{eq:3.15b}}, there exist constants $\alpha_0\in (0,1)$ and $\tilde{\epsilon}_{0}>0$ depending only on $\tilde{\underline{U}}$ and $L$,
such that for any $\alpha\in(0,\alpha_0)$ and $\tilde{\epsilon} \in (0, \tilde{\epsilon}_{0})$, 
if the given data satisfy
\begin{eqnarray}\label{eq:3.21}
\begin{split}
&
{\big\|\tilde{V}^{(\rm e)}_{0}-\underline{V}^{(\rm e)}\big\|_{1, \alpha; \Gamma^{(\rm e)}_{\rm in}}} +\big\|\tilde{U}^{(\rm h)}_{0}-\underline{U}^{(\rm h)}\big\|_{1, \alpha; \tilde{\Gamma}^{(\rm h)}_{\rm in}}
+\big\|\tilde{\omega}_{\rm e}\big\|^{(-1-\alpha, \{\mathcal{P}_{\rm e}, \mathcal{Q}_{\rm e}\})}_{2,\alpha; \tilde{\Gamma}^{(\rm e)}_{\rm ex}}\\[5pt]
&\qquad\qquad\qquad\qquad\qquad\qquad\ \ \  +\big\|g_{-}+1\big\|_{2, \alpha; \tilde{\Gamma}_{-}}+\big\|g_{+}-1\big\|_{2, \alpha; \tilde{\Gamma}_{+}}\leq  \tilde{\epsilon},
\end{split}
\end{eqnarray}
and
\begin{eqnarray}
\underline{M}^{(\rm h)}=\frac{\underline{u}^{(\rm h)}}{\underline{c}^{(\rm h)}}>\sqrt{1+\frac{1}{4}L^{2}},
\end{eqnarray}
{where $\underline{V}^{(\rm e)}$ is given by \eqref{eq:2.2}},
then \emph{Problem B} admits a positive constant $m^{(\rm e)}$ and a unique solution $\tilde{U}(\xi,\eta)\in H^{1}_{\rm loc}(\tilde{\Omega})$, which consists of the subsonic flow $\tilde{U}^{(\rm e)}(\xi, \eta) \in { C^{1,\alpha}_{(-\alpha,\tilde{\Sigma}^{(\rm e)}\backslash \{\mathcal{O}\})}}(\tilde{\Omega}^{(\rm e)})$
and supersonic flow $\tilde{U}^{(\rm h)}(\xi, \eta) \in C^{1,\alpha}(\tilde{\Omega}^{(\rm h)})$
with $\eta=0$ as the contact discontinuity.
Moreover, it holds that
\begin{eqnarray}\label{eq:3.23}
\begin{split}
{\big\|\tilde{U}^{(\rm e)}-\underline{U}^{(\rm e)}\big\|^{(-\alpha, \tilde{\Sigma}^{(\rm e)}\backslash \{\mathcal{O}\})}_{1, \alpha; \tilde{\Omega}^{(\rm e)}}}
+\big\|\tilde{U}^{(\rm h)}-\underline{U}^{(\rm h)}\big\|_{1, \alpha; \tilde{\Omega}^{(\rm h)}}
+|m^{(\rm e)}-\underline {m}^{(\rm e)}|\leq \tilde{C}_{0}\tilde{\epsilon},
\end{split}
\end{eqnarray}
where $m^{(\rm e)}$ satisfies
\begin{eqnarray}\label{eq:3.23m}
\begin{split}
\int^{m^{(\rm e)}}_{0}\Big(\frac{1}{\tilde{\rho}^{(\rm e)}\tilde{u}^{(\rm e)}}\Big)(0,\tau)d\tau=g_{+}(0),
\end{split}
\end{eqnarray}
and the constant $\tilde{C}_{0}>0$ depends only on $\tilde{\underline{U}}$ and $L$.
\end{theorem}

\begin{remark}\label{rem:3.1}
Note that
\begin{equation*}
\det\bigg(\frac{\partial (\xi,\eta)}{\partial (x,y)}\bigg)=\rho u>0,
\end{equation*}
if $\tilde{\epsilon}>0$ is taking sufficiently small. 
Then the inverse Lagrangian transformation $\mathcal{L}^{-1}$ exists and is
\begin{eqnarray*}
\mathcal{L}^{-1} : \ \left\{
\begin{array}{llll}
x &= \xi, \\[5pt]
y &= \int^{\eta}_{-m^{(h)}}\Big(\frac{1}{\rho u}\Big)(x,\tau)d\tau+g_{-}(x).
\end{array}
\right.
\end{eqnarray*}
Therefore, if Theorem \ref{thm:3.1} holds, then from the inverse Lagrangian transformation $\mathcal{L}^{-1}$ 
we can show Theorem \ref{thm:2.1} by setting 
$U(x,y)=\tilde{U}(\xi(x,y), \eta(x,y))$. 
Furthermore, the contact discontinuity $g_{\rm cd}$ is given by
\begin{eqnarray*}
g_{\rm cd}(x)=\int^{0}_{-m^{(\rm h)}}\Big(\frac{1}{\rho u}\Big)(x,\tau)d\tau+g_{-}(x), \  x\in [0, L).
\end{eqnarray*}
Obviously,
\begin{eqnarray*}
g'_{\rm cd}(x)=\frac{v}{u}(x,g_{\rm cd}(x)), \  x\in [0, L),
\end{eqnarray*}
which belongs to {\color{black}$C^{1,\alpha}$ and implies that $g_{\rm cd}(x)\in C^{2,\alpha}([0,L))$}. 
Hence, 
we only need to 
prove Theorem \ref{thm:3.1}.
\end{remark}

\subsection{Equations \eqref{eq:3.17} in the subsonic region
$\tilde{\Omega}^{(\rm e)}$}
In the subsonic region $\tilde{\Omega}^{(\rm e)}$, the Euler equations \eqref{eq:3.17} are a hyperbolic-elliptic coupled system. If the flow is $C^{1}$ in $\tilde{\Omega}^{(\rm e)}$, then by the direct computation, one has
\begin{eqnarray}\label{eq:3.24}
\partial_{\xi}\Big(\frac{\tilde{p}^{(\rm e)}}{(\tilde{\rho}^{(\rm e)})^{\gamma}}\Big)=0,
\end{eqnarray}
which implies that
\begin{eqnarray}\label{eq:3.25}
\tilde{S}^{(\rm e)}(\xi, \eta)=A^{-1}\Big(\frac{\tilde{p}^{(\rm e)}}{(\tilde{\rho}^{(\rm e)})^{\gamma}}\Big)(\xi, \eta)
=A^{-1}\Big(\frac{\tilde{p}^{(\rm e)}_{0}}{(\tilde{\rho}^{(\rm e)}_{0})^{\gamma}}\Big)(\eta)=\tilde{S}^{(\rm e)}_{0}(\eta),
\end{eqnarray}
holds along each streamline.

\par Next in $\tilde{\Omega}^{(\rm e)}$, we reformulate the Euler equations \eqref{eq:3.17} as a second-order elliptic equation by using the { stream} function. By $\eqref{eq:3.17}_{1}$, define the { stream} function $\varphi$ by
\begin{eqnarray}\label{eq:3.27}
\partial_{\xi}\varphi=\frac{\tilde{v}^{(\rm e)}}{\tilde{u}^{(\rm e)}},\quad
\partial_{\eta}\varphi=\frac{1}{\tilde{\rho}^{(\rm e)}\tilde{u}^{(\rm e)}}.
\end{eqnarray}
Then
\begin{eqnarray}\label{eq:3.28}
\tilde{u}^{(\rm e)}=\frac{1}{\tilde{\rho}^{(\rm e)} \partial_{\eta}\varphi},\quad
\tilde{v}^{(\rm e)}=\frac{\partial_{\xi}\varphi}{\tilde{\rho}^{(\rm e)} \partial_{\eta}\varphi},
\end{eqnarray}
and  Bernoulli's law \eqref{eq:3.18} can be reduced to
\begin{eqnarray}\label{eq:3.29}
\frac{\gamma A(\tilde{S}^{(\rm e)}_{0})}{\gamma-1}(\tilde{\rho}^{(\rm e)})^{\gamma+1}
-\tilde{B}^{(\rm e)}_{0}(\tilde{\rho}^{(\rm e)})^{2}
+\frac{(\partial_{\xi}\varphi)^{2}+1}{2(\partial_{\eta}\varphi)^{2}}=0.
\end{eqnarray}

\par We have the following property for $\tilde{\rho}^{(\rm e)}$.
\begin{proposition}\label{prop:3.1}
In the subsonic region, the equation \eqref{eq:3.29} admits a unique solution
$\tilde{\rho}^{(\rm e)}
=\tilde{\rho}^{(\rm e)}\big(D\varphi; \tilde{B}^{(\rm e)}_{0}, \tilde{S}^{(\rm e)}_{0}\big)$. Moreover,
\begin{eqnarray}\label{eq:3.30}
\begin{split}
&\frac{\partial \tilde{\rho}^{(\rm e)}}{\partial (\partial_{\xi}\varphi)}
=- \frac{\partial_{\xi}\varphi}{\tilde{\rho}^{(\rm e)}(\partial_{\eta}\varphi)^{2}
\Big((\tilde{c}^{(\rm e)})^{2}-\frac{1+(\partial_{\xi}\varphi)^{2}}{(\tilde{\rho}^{(\rm e)})^{2}(\partial_{\eta}\varphi)^{2}}\Big)},\\[5pt]
&\frac{\partial \tilde{\rho}^{(\rm e)}}{\partial(\partial_{\eta}\varphi)}
=\frac{(\partial_{\xi}\varphi)^{2}+1}{\tilde{\rho}^{(\rm e)} (\partial_{\eta}\varphi)^{3}
\Big((\tilde{c}^{(\rm e)})^{2}-\frac{1+(\partial_{\xi}\varphi)^{2}}{(\tilde{\rho}^{(\rm e)})^{2}(\partial_{\eta}\varphi)^{2}}\Big)},
\end{split}
\end{eqnarray}
where $D\varphi=\big(\partial_{\xi}\varphi, \partial_{\eta}\varphi\big)$ and
$\tilde{c}^{(\rm e)}=\sqrt{\gamma A(\tilde{S}^{(\rm e)}_{0})(\tilde{\rho}^{(\rm e)})^{\gamma-1}}$.
\end{proposition}

\begin{proof}
Define
\begin{eqnarray*}
{\mathcal{B}(\tilde{\rho}^{(\rm e)},\chi, \tilde{B}^{(\rm e)}_{0},\tilde{S}^{(\rm e)}_{0})}:=\frac{\gamma A(\tilde{S}^{(\rm e)}_{0})}{\gamma-1}(\tilde{\rho}^{(\rm e)})^{\gamma+1}
 -\tilde{B}^{(\rm e)}_{0}(\tilde{\rho}^{(\rm e)})^{2}+\chi,
\end{eqnarray*}
where $\chi=\frac{(\partial_{\xi}\varphi)^{2}+1}{2(\partial_{\eta}\varphi)^{2}}$.
For any point in $\tilde{\Omega}^{(\rm e)} $, we have
\begin{eqnarray*}
\begin{split}
\frac{\partial{\mathcal{B}(\tilde{\rho}^{(\rm e)},\chi, \tilde{B}^{(\rm e)}_{0},\tilde{S}^{(\rm e)}_{0})}}{\partial \tilde{\rho}^{(\rm e)}}
= \tilde{\rho}^{(\rm e)}\big((\tilde{c}^{(\rm e)})^{2}-(\tilde{u}^{(\rm e)})^{2}-(\tilde{v}^{(\rm e)})^{2}\big)>0,
\end{split}
\end{eqnarray*}
and
\begin{eqnarray*}
\begin{split}
\frac{\partial{\mathcal{B}(\tilde{\rho}^{(\rm e)},\chi, \tilde{B}^{(\rm e)}_{0},\tilde{S}^{(\rm e)}_{0})}}{\partial \chi}=1.
\end{split}
\end{eqnarray*}
Then,  for any $(\xi,\eta)\in\tilde{\Omega}^{(\rm e)}$,
\begin{eqnarray}\label{eq:3.31x}
\begin{split}
\frac{\partial\tilde{\rho}^{(\rm e)}}{\partial \chi}=-\frac{1}{\tilde{\rho}^{(\rm e)}
\big((\tilde{c}^{(\rm e)})^{2}-(\tilde{u}^{(\rm e)})^{2}-(\tilde{v}^{(\rm e)})^{2}\big)}<0.
\end{split}
\end{eqnarray}
Therefore, from equation \eqref{eq:3.29}, we can get a continuously differentiable function $\tilde{\rho}^{(\rm e)}$
with respect to $\chi$, i.e., $\tilde{\rho}^{(\rm e)}=\tilde{\rho}^{(\rm e)}(D\varphi; \tilde{B}^{(\rm e)}_{0},\tilde{S}^{(\rm e)}_{0})$.
Finally, because
\begin{equation*}
\frac{\partial \chi}{\partial(\partial_{\xi}\varphi)}=\frac{\partial_{\xi}\varphi}{(\partial_{\eta}\varphi)^{2}},\quad \ \
\frac{\partial \chi}{\partial (\partial_{\eta}\varphi)}=-\frac{(\partial_{\xi}\varphi)^{2}+1}{(\partial_{\eta}\varphi)^{3}},
\end{equation*}
it follows from \eqref{eq:3.31x} that \eqref{eq:3.30} holds.
\end{proof}

\begin{remark}\label{rem:3.2}
It follows from \eqref{eq:3.25} and Proposition \ref{prop:3.1} that the pressure $\tilde{p}^{(\rm e)}$ is also a function of $D\varphi$, $\tilde{B}^{(\rm e)}_{0}$ and $\tilde{S}^{(\rm e)}_{0}$, that is,
$\tilde{p}^{(\rm e)}=\tilde{p}^{(\rm e)}\big(D\varphi; \tilde{B}^{(\rm e)}_{0}, \tilde{S}^{(\rm e)}_{0}\big)$.
\end{remark}

Now, by \eqref{eq:3.28}, Proposition \ref{prop:3.1} and Remark \ref{rem:3.2},
the equations {\color{black}\eqref{eq:3.17}} can be reduced to the following nonlinear elliptic equation of second-order:
\begin{eqnarray}\label{eq:3.31}
\begin{split}
\partial_{\xi}\Big(\frac{\partial_{\xi}\varphi}{\tilde{\rho}^{(\rm e)} \partial_{\eta}\varphi}\Big)
+\partial_{\eta}\tilde{p}^{(\rm e)}(D\varphi;\tilde{B}^{(\rm e)}_{0},\tilde{S}^{(\rm e)}_{0})=0.
\end{split}
\end{eqnarray}
Next, for the boundary conditions of $\varphi$, from \eqref{eq:3.19}, \eqref{eq:3.13}, {\color{black}\eqref{eq:3.15}} and \eqref{eq:3.27}, we know that
\begin{eqnarray}\label{eq:3.32}
\begin{split}
\partial_{\xi}\varphi=\tilde{\omega}_{\rm e}(\eta),\qquad\mbox{on }\  \tilde{\Gamma}^{(\rm e)}_{\rm ex},
\end{split}
\end{eqnarray}
\begin{eqnarray}\label{eq:3.33}
\begin{split}
\varphi=g_{+}(\xi),\qquad\mbox{on } \ \tilde{\Gamma}^{(\rm e)}_{+},
\end{split}
\end{eqnarray}
and
\begin{eqnarray}\label{eq:3.34}
\begin{split}
\tilde{p}^{(\rm e)}(D\varphi; \tilde{B}^{(\rm e)}_{0}, \tilde{S}^{(\rm e)}_{0})=\tilde{p}_{0}(\eta),
\qquad\mbox{on }\ \tilde{\Gamma}^{(\rm e)}_{\rm in}.
\end{split}
\end{eqnarray}

\subsection{Equations \eqref{eq:3.17} in the supersonic region $\tilde{\Omega}^{(\rm h)}$}
For smooth solutions in the supersonic region $\tilde{\Omega}^{(\rm h)}$,
we can rewrite the system \eqref{eq:3.17} as the first-order nonlinear symmetric form:
{
\begin{eqnarray}\label{eq:3.35}
\mathcal{A}^{(\rm h)}_{1}(\mathcal{V}^{(\rm h)})\partial_{\xi}\mathcal{V}^{(\rm h)}+\mathcal{A}^{(\rm h)}_2(\mathcal{V}^{(\rm h)})\partial_{\eta}\mathcal{V}^{(\rm h)}=0,
\end{eqnarray}
}
where {$\mathcal{V}^{(\rm h)}(\xi, \eta)=\big(\tilde{u}^{(\rm h)}, \tilde{v}^{(\rm h)}, \tilde{p}^{(\rm h)}\big)^{\top}(\xi, \eta)$}
and
\begin{equation*}
{\mathcal{A}^{(\rm h)}_{1}(\mathcal{V}^{(\rm h)})}=\left(
\begin{array}{ccc}
\tilde{u}^{(\rm h)} &0 & \frac{1}{\tilde{\rho}^{(\rm h)}} \\[5pt]
0& \tilde{u}^{(\rm h)} & 0 \\[5pt]
\frac{1}{\tilde{\rho}^{(\rm h)}} & 0 & \frac{\tilde{u}^{(\rm h)}}{(\tilde{c}^{(\rm h)})^{2}(\tilde{\rho}^{(\rm h)})^{2}}
\end{array}
\right), \quad
{\mathcal{A}^{(\rm h)}_{2}(\mathcal{V}^{(\rm h)})}=\left(
\begin{array}{ccc}
0 &0 & -\tilde{v}^{(\rm h)} \\[5pt]
0& 0 & \tilde{u}^{(\rm h)} \\[5pt]
-\tilde{v}^{(\rm h)} & \tilde{u}^{(\rm h)} & 0
\end{array}
\right).
\end{equation*}

System \eqref{eq:3.35} is hyperbolic in $\tilde{\Omega}^{(\rm h)}$. By direct computation, the eigenvalue and the corresponding right eigenvectors are
\begin{eqnarray}\label{eq:3.36}
\begin{split}
&\lambda_{-}=\frac{\tilde{\rho}^{(\rm h)} \tilde{u}^{(\rm h)} (\tilde{c}^{(\rm h)})^{2}}
{(\tilde{u}^{(\rm h)})^{2}-(\tilde{c}^{(\rm h)})^{2}}\Big(\frac{\tilde{v}^{(\rm h)}}{\tilde{u}^{(\rm h)}}
-\frac{\sqrt{(\tilde{q}^{(\rm h)})^{2}-(\tilde{c}^{(\rm h)})^{2}}}{\tilde{c}^{(\rm h)}}\Big),\\[5pt]
&\lambda_{0}=0,  \\[5pt]
&\lambda_{+}=\frac{\tilde{\rho}^{(\rm h)} \tilde{u}^{(\rm h)} (\tilde{c}^{(\rm h)})^{2}}
{(\tilde{u}^{(\rm h)})^{2}-(\tilde{c}^{(\rm h)})^{2}}\Big(\frac{\tilde{v}^{(\rm h)}}{\tilde{u}^{(\rm h)}}
+\frac{\sqrt{(\tilde{q}^{(\rm h)})^{2}-(\tilde{c}^{(\rm h)})^{2}}}{\tilde{c}^{(\rm h)}}\Big),
\end{split}
\end{eqnarray}
where $\tilde{q}^{(\rm h)}=\sqrt{(\tilde{u}^{(\rm h)})^{2}+(\tilde{v}^{(\rm h)})^{2}}$ and
\begin{eqnarray}\label{eq:3.37}
\begin{split}
&r_{-}=\Big(\frac{\lambda_{-}}{\tilde{\rho}^{(\rm h)}}+\tilde{v}^{(\rm h)}, -\tilde{u}^{(\rm h)}, -\lambda_{-}\tilde{u}^{(\rm h)}\Big)^{\top},\\[5pt]
&r_{0}=\big(\tilde{u}^{(\rm h)},\tilde{v}^{(\rm h)}, 0\big)^{\top},   \\[5pt]
&r_{+}=\Big(\frac{\lambda_{+}}{\tilde{\rho}^{(\rm h)}}+\tilde{v}^{(\rm h)}, -\tilde{u}^{(\rm h)},
 -\lambda_{+}\tilde{u}^{(\rm h)}\Big)^{\top},
\end{split}
\end{eqnarray}
respectively. The corresponding left eigenvectors are
\begin{eqnarray}\label{eq:3.38}
l_{\pm}=(r_{\pm})^{\top}, \qquad l_{0}=(r_{0})^{\top}.
\end{eqnarray}

Define
\begin{eqnarray}\label{eq:3.43}
\tilde{\omega}^{(\rm h)}:=\frac{\tilde{v}^{(\rm h)}}{\tilde{u}^{(\rm h)}},
\quad \ \ \Lambda^{(\rm h)}:=\frac{\sqrt{(\tilde{q}^{(\rm h)})^{2}-(\tilde{c}^{(\rm h)})^{2}}}{\tilde{\rho}^{(\rm h)} \tilde{c}^{(\rm h)} (\tilde{u}^{(\rm h)})^{2}}.
\end{eqnarray}
Multiplying the system \eqref{eq:3.35} by $l_{\pm}$ and $l_{0}$,
and by Bernoulli's law \eqref{eq:3.18}, we have
{\begin{eqnarray}\label{eq:3.45}
\begin{split}
&&\partial_{\xi}\tilde{\omega}^{(\rm h)}+\lambda_{-}\partial_{\eta}\tilde{\omega}^{(\rm h)}-\Lambda^{(\rm h)}\big(\partial_{\xi}\tilde{p}^{(\rm h)}+\lambda_{-}\partial_{\eta}\tilde{p}^{(\rm h)}\big)=0, \\[5pt]
&&\partial_{\xi}\tilde{\omega}^{(\rm h)}+\lambda_{+}\partial_{\eta}\tilde{\omega}^{(\rm h)}+\Lambda^{(\rm h)}\big(\partial_{\xi}\tilde{p}^{(\rm h)}+\lambda_{+}\partial_{\eta}\tilde{p}^{(\rm h)}\big)=0,
\end{split}
\end{eqnarray}
}
and
\begin{eqnarray}\label{eq:3.41}
\partial_{\xi}\Big(\frac{\tilde{p}^{(\rm h)}}{(\tilde{\rho}^{(\rm h)})^{\gamma}}\Big)=0,
\end{eqnarray}
which implies that
\begin{eqnarray}\label{eq:3.42}
\tilde{S}^{(\rm h)}(\xi, \eta)=A^{-1}\Big(\frac{\tilde{p}^{(\rm h)}}{(\tilde{\rho}^{(\rm h)})^{\gamma}}\Big)(\xi, \eta)
=A^{-1}\Big(\frac{\tilde{p}^{(\rm h)}_{0}}{(\tilde{\rho}^{(\rm h)}_{0})^{\gamma}}\Big)(\eta)=\tilde{S}^{(\rm h)}_{0}(\eta),
\end{eqnarray}
holds along each streamline.

\begin{remark}\label{rem:3.3}
When  the solution $(\tilde{\omega}^{(\rm h)},\tilde{p}^{(\rm h)})$ of the system \eqref{eq:3.45} is known,
 we can get the solution  $(\tilde{u}^{(\rm h)}, \tilde{v}^{(\rm h)})$ by \eqref{eq:3.43} and
\eqref{eq:3.18} as the following: 
\begin{eqnarray}\label{eq:3.46}
\begin{split}
&\tilde{u}^{(\rm h)}=\sqrt{\frac{2\Big((\gamma-1)\tilde{B}^{(\rm h)}_{0}-\gamma \big(A(\tilde{S}^{(\rm h)}_{0}))(\tilde{p}^{(\rm h)})^{\gamma-1}\big)^{\frac{1}{\gamma}}\Big)}
{(\gamma-1)\big(1+(\tilde{\omega}^{(\rm h)})^{2}\big)}},\\[5pt]
&\tilde{v}^{(\rm h)}=\tilde{\omega}^{(\rm h)}\sqrt{\frac{2\Big((\gamma-1)\tilde{B}^{(\rm h)}_{0}
-\gamma\big(A(\tilde{S}^{(\rm h)}_{0}))(\tilde{p}^{(\rm h)})^{\gamma-1}\big)^{\frac{1}{\gamma}}\Big)}{(\gamma-1)\big(1+(\tilde{\omega}^{(\rm h)})^{2}\big)}},
\end{split}
\end{eqnarray}
where $\tilde{S}^{(\rm h)}_{0}$ is given by \eqref{eq:3.42}.
\end{remark}

Therefore, in the supersonic region, it remains to solve equations \eqref{eq:3.45}. To this end,
we will further 
introduce the Riemann invariants.

Let $\mathscr{U}=(\tilde{\omega}^{(\rm h)},\tilde{p}^{(\rm h)})^{\top}$,
then the system \eqref{eq:3.45} can be rewritten as
 \begin{eqnarray}\label{eq:3.47}
\partial_{\xi}\mathscr{U}+\mathscr{F}(\mathscr{U})\partial_{\eta}\mathscr{U}=0,
\end{eqnarray}
 where
\begin{equation}\label{eq:3.48}
\mathscr{F}(\mathscr{U})=\left(
\begin{array}{ccc}
\frac{\tilde{\rho}^{(\rm h)} (\tilde{c}^{(\rm h)})^{2}\tilde{v}^{(\rm h)}}{(\tilde{u}^{(\rm h)})^{2}-(\tilde{c}^{(\rm h)})^{2}} &
\frac{(\tilde{u}^{(\rm h)})^{2}+(\tilde{v}^{(\rm h)})^{2}-(\tilde{c}^{(\rm h)})^{2}}{\tilde{u}^{(\rm h)}
\big((\tilde{u}^{(\rm h)})^{2}-(\tilde{c}^{(\rm h)})^{2}\big)} \\[10pt]
\frac{(\tilde{\rho}^{(\rm h)})^{2} (\tilde{c}^{(\rm h)})^{2}(\tilde{u}^{(\rm h)})^{3}}
{(\tilde{u}^{(\rm h)})^{2}-(\tilde{c}^{(\rm h)})^{2}}  & \frac{\tilde{\rho}^{(\rm h)} (\tilde{c}^{(\rm h)})^{2}\tilde{v}^{(\rm h)}}
{(\tilde{u}^{(\rm h)})^{2}-(\tilde{c}^{(\rm h)})^{2}}
\end{array}
\right).
\end{equation}
 The generalized Riemann invariants $z_{\pm}$ for system \eqref{eq:3.47} are defined as 
\begin{eqnarray}\label{eq:3.50}
z_{-}=\arctan\tilde{\omega}^{(\rm h)}+\Theta\big(\tilde{p}^{(\rm h)}; \tilde{B}^{(\rm h)}_{0}, \tilde{S}^{(\rm h)}_{0}\big),
\ \ z_{+}=\arctan\tilde{\omega}^{(\rm h)}-\Theta\big(\tilde{p}^{(\rm h)}; \tilde{B}^{(\rm h)}_{0}, \tilde{S}^{(\rm h)}_{0}\big),
\end{eqnarray}
where
{\small\begin{eqnarray}\label{eq:3.51}
\begin{split}
\Theta\big(\tilde{p}^{(\rm h)}; \tilde{B}^{(\rm h)}_{0}, \tilde{S}^{(\rm h)}_{0}\big)
=&\int^{\tilde{p}^{(\rm h)}}\frac{\sqrt{(\gamma-1)\big(A(\tilde{S}^{(\rm h)}_{0})\big)^{\frac{1}{\gamma}}\Big(2(\gamma-1) \tilde{B}^{(\rm h)}_{0}-\gamma(\gamma+1)
\big(A(\tilde{S}^{(\rm h)}_{0})\big)^{\frac{1}{\gamma}}\tau^{\frac{\gamma-1}{\gamma}}}\Big)}
{2\gamma^{\frac{1}{2}}
\Big((\gamma-1)\tilde{B}^{(\rm h)}_{0}-\gamma\big(A(\tilde{S}^{(\rm h)}_{0})\big)^{\frac{1}{\gamma}}\tau^{1-\frac{1}{\gamma}}\Big)\tau^{\frac{\gamma+1}{2\gamma}}}d\tau.
\end{split}
\end{eqnarray}}
  Let $z=(z_{-},z_{+})^{\top} $. System \eqref{eq:3.47} can be reduced  to the following characteristic form:
\begin{eqnarray}\label{eq:3.53x}
\partial_{\xi}z+ \textrm{diag} \big( \lambda_{+}, \lambda_{-} \big)\partial_{\eta} z=0.
\end{eqnarray}

Finally, we remark that once $z$ is solved,   the solution in the supersonic region is solved by the following proposition.
\begin{proposition}\label{prop:3.2}
If the flow is supersonic in $\tilde{\Omega}^{(\rm h)}$, then, for any given $z=(z_{-}, z_{+})$,
$(\tilde{p}^{(\rm h)},\tilde{\omega}^{(\rm h)})$ can be uniquely solved by  \eqref{eq:3.50} as
$\tilde{p}^{(\rm h)}=\tilde{p}^{(\rm h)}\big(z; \tilde{B}^{(\rm h)}_{0},\tilde{S}^{(\rm h)}_{0}\big)$
and $\tilde{\omega}^{(\rm h)}=\tilde{\omega}^{(\rm h)}(z)$.
\end{proposition}

\begin{proof}
By \eqref{eq:3.50}, we obtain that
\begin{eqnarray*}\label{eq:3.53}
\begin{split}
\tilde{\omega}^{(\rm h)}=\tan\Big(\frac{z_{-}+z_{+}}{2}\Big), \quad
\Theta\big(\tilde{p}^{(\rm h)}; \tilde{B}^{(\rm h)}_{0},\tilde{S}^{(\rm h)}_{0}\big)=\frac{1}{2}(z_{-}-z_{+}).
\end{split}
\end{eqnarray*}
For $\tilde{q}^{(\rm h)}>\tilde{c}^{(\rm h)}$, one has
\begin{eqnarray*}
\partial_{\tilde{p}^{(\rm h)}}\Theta\big(\tilde{p}^{(\rm h)}; \tilde{B}^{(\rm h)}_{0},\tilde{S}^{(\rm h)}_{0}\big)
=\frac{\sqrt{(\tilde{q}^{(\rm h)})^{2}-(\tilde{c}^{(\rm h)})^{2}}}
{\tilde{\rho}^{(\rm h)}\tilde{c}^{(\rm h)}(\tilde{q}^{(\rm h)})^{2}}>0,
\end{eqnarray*}
then, by the straightforward computation, it follows that
\begin{eqnarray*}
\begin{split}
\frac{\partial \tilde{p}^{(\rm h)}}{\partial z_{-}}=\frac{1}{2\partial_{\tilde{p}^{(\rm h)}}
\Theta\big(\tilde{p}^{(\rm h)}; \tilde{B}^{(\rm h)}_{0},\tilde{S}^{(\rm h)}_{0}\big)}, \quad
\frac{\partial \tilde{p}^{(\rm h)}}{\partial z_{+}}=-\frac{1}{2\partial_{\tilde{p}^{(\rm h)}}
\Theta\big(\tilde{p}^{(\rm h)}; \tilde{B}^{(\rm h)}_{0},\tilde{S}^{(\rm h)}_{0}\big)}.
\end{split}
\end{eqnarray*}
Thus the proposition follows from the implicit function theorem. 
\end{proof}

\par Finally, for the boundary conditions of $z$, from \eqref{eq:3.19}, \eqref{eq:3.13}
and \eqref{eq:3.50}, we know that
\begin{equation}
\label{eq:3.54}
z(\xi, \eta)=z_{0}(\eta), \quad\  \   \ \mbox{on} \ \  \ \tilde{\Gamma}^{(\rm h)}_{\rm in},
\end{equation}
and
\begin{eqnarray}\label{eq:3.55}
\begin{split}
z_{-}+z_{+}=2\arctan g'_{-}(\xi), \quad\  \   \ \mbox{on} \ \ \  \tilde{\Gamma}_{-}.
\end{split}
\end{eqnarray}

\subsection{Free boundary value problem for $(\varphi, z)$}
From the above two subsections,
 {Problem B} can be reformulated as the following nonlinear free boundary
value problem for $(\varphi, z)$.

\smallskip
\par $\mathbf{Problem }$ $\mathbf{C.}$  For the given data $\big(\tilde{p}^{(\rm e)}_{0}(\eta),\tilde{B}^{(\rm e)}_{0}(\eta), \tilde{S}^{(\rm e)}_{0}(\eta)\big)$ at the entrance $\tilde{\Gamma}^{(\rm e)}_{\rm in}$, the given data $\big(z_{0}(\eta),\tilde{B}^{(\rm h)}_{0}(\eta), \tilde{S}^{(\rm h)}_{0}(\eta)\big)$ at the entrance $\tilde{\Gamma}^{(\rm h)}_{\rm in} $ satisfying \eqref{eq:3.14}, and the given data $\tilde{\omega}_{\rm e}$ at the exit $\tilde{\Gamma}^{(\rm e)}_{\rm ex}$, find a piecewise smooth transonic flow $\big(\varphi, z\big)(\xi,\eta)$ with a contact discontinuity 
and a positive constant $m^{(\rm e)}$ for the following nonlinear free boundary value problem:
\begin{eqnarray}\label{eq:3.56}
(\mathbf{NP})\quad   \left\{
\begin{array}{llll}
\partial_{\xi}\Big(\frac{\partial_{\xi}\varphi}{\tilde{\rho}^{(\rm e)}
(D\varphi;\tilde{B}^{(\rm e)}_{0},\tilde{S}^{(\rm e)}_{0}) \partial_{\eta}\varphi}\Big)
+\partial_{\eta}\tilde{p}^{(\rm e)}(D\varphi;\tilde{B}^{(\rm e)}_{0},\tilde{S}^{(\rm e)}_{0})=0,
&\ \ \  \mbox{in} \ \ \ \tilde{\Omega}^{(\rm e)},   \\[5pt]
\partial_{\xi}z+\textrm{diag}( \lambda_{+}, \lambda_{-})\partial_{\eta}z=0, &\ \ \ \mbox{in}\ \ \  \tilde{\Omega}^{(\rm h)},  \\[5pt]
\tilde{p}^{(\rm e)}=\tilde{p}^{(\rm e)}_{0}(\eta),     &\ \ \  \mbox{on} \ \ \ \tilde{\Gamma}^{(\rm e)}_{\rm in},\\[5pt]
\partial_{\xi}\varphi=\tilde{\omega}_{\rm e}(\eta),  &\ \ \ \mbox{on} \ \ \ \tilde{\Gamma}^{(\rm e)}_{\rm ex},\\[5pt]
z=z_{0}(\eta),   &\ \ \   \mbox{on} \ \ \ \tilde{\Gamma}^{(\rm h)}_{\rm in}, \\[5pt]
\partial_{\xi}\varphi= g'_{+}(\xi), &\ \ \  \mbox{on} \ \ \ \tilde{\Gamma}_{+},\\[5pt]
\partial_{\xi}\varphi=\tilde{\omega}^{(\rm h)}, \  \ \tilde{p}^{(\rm e)}=\tilde{p}^{(\rm h)},
&\ \ \ \mbox{on} \ \ \ \tilde{\Gamma}_{\rm cd},\\[5pt]
z_{-}+z_{+}=2\arctan g'_{-}(\xi), &\ \ \ \mbox{on} \ \ \ \tilde{\Gamma}_{-},
\end{array}
\right.
\end{eqnarray}
where $m^{(\rm e)}$ on $\tilde{\Gamma}^{(\rm e)}_{+}$ is determined by \eqref{eq:3.23m}.
To make the solution $\varphi$  unique, without loss of the generality, we prescribe the one-point condition:
\begin{equation}\label{3.80}
\varphi(0,0)=\underline{\varphi}(0,0)=0.
\end{equation}

\par We have the following theorem for the nonlinear free boundary value problem $(\mathbf{NP})$.

\begin{theorem}\label{thm:3.2}
For any given data $\big(\tilde{p}^{(\rm e)}_{0},\tilde{B}^{(\rm e)}_{0}, \tilde{S}^{(\rm e)}_{0}\big)(\eta)$ at the entrance $\tilde{\Gamma}^{(\rm e)}_{\rm in}$, $\big(z_{0},\tilde{B}^{(\rm h)}_{0}, \tilde{S}^{(\rm h)}_{0}\big)(\eta)$ at the entrance $\tilde{\Gamma}^{(\rm h)}_{\rm in} $ satisfying \eqref{eq:3.14} and the  compatible conditions at the corner point {$(0, -m^{(\rm h)})$}, and $\tilde{\omega}_{\rm e}$ at the exit $\tilde{\Gamma}^{(\rm e)}_{\rm ex}$ with \eqref{eq:3.15} and \eqref{eq:3.15b}, there exist a pair of constants $\alpha_0\in(0,1)$ and {\color{black}$\tilde{\epsilon}_{0}>0$} 
depending only on $\tilde{\underline{U}}$ and $L$, such that, if the given data satisfy
\begin{eqnarray}\label{eq:3.58}
\begin{split}
&{\big\|\tilde{p}^{(\rm e)}_{0}-\underline{p}^{(\rm e)}\big\|}_{1,\alpha;\tilde{\Gamma}^{(\rm e)}_{\rm in}}
+\sum_{k=\rm e, \rm h}\bigg(\big\|\tilde{B}^{(\rm k)}_{0}-\underline{B}^{(\rm k)}
\big\|_{1,\alpha;\tilde{\Gamma}^{(\rm k)}_{\rm in}}
+\big\|\tilde{S}^{(\rm k)}_{0}-\underline{S}^{(\rm k)}\big\|_{1,\alpha;\tilde{\Gamma}^{(\rm k)}_{\rm in}}\bigg)\\[5pt]
&\quad\ \  +{\color{black}\big\|\tilde{\omega}_{\rm e}\big\|^{(-1-\alpha, \{\mathcal{P}_{\rm e}, \mathcal{Q}_{\rm e}\})}_{2,\alpha;\tilde{\Gamma}^{(\rm e)}_{\rm ex}}}
+\big\|z_{0}-\underline{z}\big\|_{1,\alpha; \tilde{\Gamma}^{(\rm h)}_{\rm in}}
+\big\|g_{+}-1\big\|_{2,\alpha;\tilde{ \Gamma}_{+}}
+\big\|g_{-}+1\big\|_{2,\alpha;  \tilde{\Gamma}_{-}} \leq {\color{black}\tilde{\epsilon}},
\end{split}
\end{eqnarray}
and
\begin{eqnarray}
\underline{M}^{(\rm h)}=\frac{\underline{u}^{(\rm h)}}{\underline{c}^{(\rm h)}}\geq \sqrt{1+L^{2}/4},
\end{eqnarray}
for any $\alpha\in(0,\alpha_{0})$ and {\color{black}$\tilde{\epsilon} \in (0, \tilde{\epsilon}_{0})$},
then the problem $(\mathbf{NP})$ with \eqref{3.80} admits a unique solution $z\in  C^{1,\alpha}(\tilde{\Omega}^{(\rm h)})$, $\varphi\in {C^{2,\alpha}_{(-1-\alpha, \tilde{\Sigma}^{(\rm e)}\backslash \{\mathcal{O}\})}}(\tilde{\Omega}^{(\rm e)})$, and a positive
constant $m^{(\rm e)}$. Moreover, the following estimate holds:
\begin{eqnarray}\label{eq:3.60}
\begin{split}
\big\|\varphi-\underline{\varphi}\big\|^{(-1-\alpha, \tilde{\Sigma}^{(\rm e)}\backslash \{\mathcal{O}\})}_{2,\alpha;\tilde{\Omega}^{(\rm e)}}
+\big\|z-\underline{z}\big\|_{1,\alpha;\tilde{\Omega}^{(\rm h)}}+\big|m^{(\rm e)}-\underline{m}^{(\rm e)}\big|\leq \tilde{C}_{0} {\color{black}\tilde{\epsilon}},
\end{split}
\end{eqnarray}
where $m^{(\rm e)}$ satisfies
\begin{eqnarray}\label{eq:3.61}
\begin{split}
\int^{m^{(\rm e)}}_{0}\partial_{\eta}\varphi(0, \tau)d\tau=g_{+}(0).
\end{split}
\end{eqnarray}
Here, $\tilde{C}_{0}=\tilde{C}_{0}(\underline{U},L)$ is a positive constant independent of {\color{black}$\tilde{\epsilon}$}.
\end{theorem}

\begin{remark}\label{rem:3.4}
As   in Proposition \ref{prop:3.1} and Proposition \ref{prop:3.2}, Remark \ref{rem:3.2}, and Remark \ref{rem:3.3},
we know that Theorem \ref{thm:3.1} is a   corollary of Theorem \ref{thm:3.2},  which indicates
that it suffices to consider the solutions to {Problem C} instead of  {Problem B}.
Therefore, we are devoted to prove Theorem \ref{thm:3.2} in the rest of the paper. 
\end{remark}

{
We will solve $(\mathbf{NP})$ in two steps:

\emph{Step\ 1. Assign a value for $m^{(\rm e)}$ and solve the corresponding fixed boundary value problem.}
Let $\underline{m}^{(\rm e)}=\underline{\rho}^{(\rm e)}\underline{u}^{(\rm e)}$, which depends on the background solution given in \eqref{eq:3.12}. 
For any $\sigma>0$, define
\begin{eqnarray}\label{eq:3.62}
\mathcal{M}_{\sigma}=\big\{m^{(\rm e)}:\ m^{(\rm e)}>0, \ \mbox{and} \ \  \big|m^{(\rm e)}-\underline{m}^{(\rm e)}\big| \leq \sigma \big\}.
\end{eqnarray}
For any given $m^{(\rm e)} \in \mathcal{M}_{\sigma}$, define $\tilde{\Gamma}^{(\rm h)}_{+}$ by \eqref{3.13},  and then   consider the following nonlinear fixed boundary value problem for $(\mathbf{NP})$:

\begin{eqnarray}\label{eq:3.63}
(\mathbf{FP})\quad   \left\{
\begin{array}{llll}
\partial_{\xi}\Big(\frac{\partial_{\xi}\varphi}{\tilde{\rho}^{(\rm e)}
(D\varphi;\tilde{B}^{(\rm e)}_{0},\tilde{S}^{(\rm e)}_{0}) \partial_{\eta}\varphi}\Big)
+\partial_{\eta}\tilde{p}^{(\rm e)}(D\varphi;\tilde{B}^{(\rm e)}_{0},\tilde{S}^{(\rm e)}_{0})=0,
&\ \ \  \mbox{in} \ \ \ \tilde{\Omega}^{(\rm e)},   \\[5pt]
\partial_{\xi}z+\textrm{diag}( \lambda_{+}, \lambda_{-})\partial_{\eta}z=0,
&\ \ \ \mbox{in}\ \ \  \tilde{\Omega}^{(\rm h)},  \\[5pt]
\tilde{p}^{(\rm e)}=\tilde{p}^{(\rm e)}_{\rm 0}(\eta),     &\ \ \  \mbox{on}\ \ \ \tilde{\Gamma}^{(\rm e)}_{\rm in},\\[5pt]
\partial_{\xi}\varphi=\tilde{\omega}_{\rm e}(\eta),
&\ \ \  \mbox{on}\ \ \ \tilde{\Gamma}^{(\rm e)}_{\rm ex},\\[5pt]
z=z_{0}(\eta),   &\ \ \   \mbox{on}\ \ \ \tilde{\Gamma}^{(\rm h)}_{\rm in}, \\[5pt]
{\partial_{\xi}\varphi= g'_{+}(\xi)}, &\ \ \  \mbox{on}\ \ \ \tilde{\Gamma}_{+},\\[5pt]
\partial_{\xi}\varphi=\tan\Big(\frac{z_{-}+z_{+}}{2}\Big),   &\ \ \ \mbox{on}\ \ \ \tilde{\Gamma}_{\rm cd},\\[5pt]
z_{-}-z_{+}=2\Theta(\tilde{p}^{(\rm e)},\tilde{B}^{(\rm h)}_{0}, \tilde{S}^{(\rm h)}_{0}),
&\ \ \ \mbox{on}\ \ \ \tilde{\Gamma}_{\rm cd},\\[5pt]
z_{-}+z_{+}=2\arctan g'_{-}(x), &\ \ \ \mbox{on}\ \ \ \tilde{\Gamma}_{-},
\end{array}
\right.
\end{eqnarray}
where $\tilde{p}^{(\rm e)}_{0}$ and $\tilde{p}^{(\rm h)}_{0}$ satisfy \eqref{eq:3.14}.
An outline on solving the above problem $(\mathbf{FP})$ will be given in the Section \ref{sec35} below.

\emph{Step\ 2. Update the approximate boundary $\tilde{\Gamma}^{(\rm e)}_{+}$, (i.e., $\eta=m^{(\rm e)}$).}
To update the free boundary, we solve the following problem:
\begin{eqnarray}\label{eq:3.64}
\left\{
\begin{array}{llll}
\int^{m^{(\rm e)}}_{0}\partial_{\eta}\varphi(0,\tau)d\tau=g_{+}(0), \\[5pt]
 \varphi(0,0)=0.
\end{array}
\right.
\end{eqnarray}
Then we define a map $\mathcal{T}$ such that $\tilde{m}^{(\rm e)}=\mathcal{T}(m^{(\rm e)})$. 
We will show that

\smallskip
\par (1)\ $\mathcal{T}$ is well-defined, i.e., $\mathcal{T}$ exists uniquely and maps from $\mathcal{M}_{\sigma}$ to $\mathcal{M}_{\sigma}$;
\par (2)\ $\mathcal{T}$ is a contraction.

\noindent
This step will be achieved in Section 7.
}

{\subsection{Solving the \emph{Step 1}, i.e., nonlinear fixed boundary value problem $(\mathbf{FP})$}\label{sec35}

\vspace{2pt}
\begin{figure}[ht]
\begin{center}
\begin{tikzpicture}[scale=1.2]
\draw [line width=0.08cm] (-5,1.5) --(3.5,1.5);
\draw [line width=0.08cm] (-5,-2) --(3.5,-2);
\draw [thick](-5,-2) --(-5,1.5);
\draw [thick](3.5,-2) --(3.5,1.5);
\draw [line width=0.08cm][dashed][red](-5,0) --(3.5,0);
\draw [thick][blue](-5,-2) --(-2.5,0);
\draw [thick][black](-5,0)--(-2.5,-2) ;
\draw [thick][black](-2.5,0)--(0,-2) ;
\draw [thick][blue](-2.5,-2) --(0,0);
\node at (-5.5, -1) {$\tilde{\Gamma}^{(\rm h)}_{\rm in}$};
\node at (-5.5, 1) {$\tilde{\Gamma}^{(\rm e)}_{\rm in}$};
\node at (3.9, 0.8) {$\tilde{\Gamma}^{(\rm e)}_{\rm ex}$};
\node at (-1, 1.8) {$\tilde{\Gamma}_{+}$};
\node at (-1, 0.3) {$\tilde{\Gamma}_{\rm cd}$};
\node at (-1, -2.3) {$\tilde{\Gamma}_{-}$};
\node at (-5.5, 0) {$\mathcal{O}$};
\node at (4.2, 1.5) {$\eta=m^{(\rm e)}$};
\node at (4.0, 0) {$\eta=0$};
\node at (4.3, -1.9) {$\eta=-m^{(\rm h)}$};
\node at (0, 0.3) {$\xi^{*}_{\rm cd, 1}$};
\node at (-2.5, 0.3) {$\xi^{*}_{\rm cd, 0}$};
\node at (-5, -2.3) {$0$};
\node at (-2.4, -2.3) {$\xi^{*}_{-, 0}$};
\node at (0.1, -2.3) {$\xi^{*}_{-, 1}$};
\node at (3.5, -2.3) {$L$};
\node at (1.7, -0.9) {$\tilde{\Omega}^{(\rm h)}$};
\node at (1.7, 1) {$\tilde{\Omega}^{(\rm e)}$};
\end{tikzpicture}
\caption{Iteration algorithm for the problem $(\mathbf{FP})$}\label{fig3.4}
	\end{center}
\end{figure}
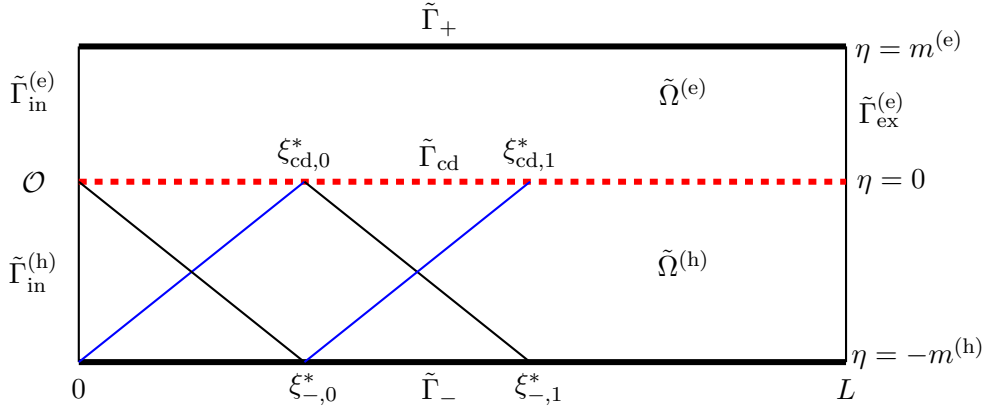

In order to solve $(\mathbf{FP})$, we will formulate an iteration scheme. 
We first construct a sequence of approximate solutions by linearizing the problem $(\mathbf{FP})$ \emph{partially} near the background state $\tilde{\underline{U}}$,
and then show its convergence to a unique limit that satisfies the problem $(\mathbf{FP})$.

\emph{Step\ 1.1.}\ The linearized problem that is a boundary value problem for elliptic-hyperbolic mixed-type equations will be solved {\color{black} first in $\tilde{\Omega}^{(\rm e)}$ and then in $\tilde{\Omega}^{(\rm h)}$
together} 
 (see Fig. \ref{fig3.4}). To do this, we define the iteration set as
\begin{eqnarray}\label{eq:3.75}
\begin{split}
\mathscr{K}_{{\color{black}\epsilon_{I}}}=\Big\{(\varphi^{(\rm n)}, z^{(\rm n)}):\ & \big\| \varphi^{(\rm n)}-\underline{\varphi}\big\|^{(-1-\alpha, \tilde{\Sigma}^{(\rm e)}\backslash \{\mathcal{O}\})}_{2, \alpha; \tilde{\Omega}^{(\rm e)}}+\big\|z^{(\rm n)}-\underline{z}\big\|_{1,\alpha; \tilde{\Omega}^{(\rm h)}}
\leq {\color{black}\epsilon_{I}} \Big\},
\end{split}
\end{eqnarray}
for some $0<{\color{black}\epsilon_{I}}<1$.

For any given $(\varphi^{(\rm n-1)}, z^{(\rm n-1)})\in\mathscr{K}_{2{\color{black}\epsilon_{I}}}$, the linearized boundary value problem of $(\mathbf{FP})$ is
\begin{eqnarray}\label{eq:3.68}
(\mathbf{FP})_{\rm n}\quad
\left\{
\begin{array}{llll}
\partial_{\xi}\mathcal{N}_{1}(D\varphi^{(\rm n)};\tilde{B}^{(\rm e)}_{0},\tilde{S}^{(\rm e)}_{0})
+\partial_{\eta}\mathcal{N}_{2}(D\varphi^{(\rm n)};\tilde{B}^{(\rm e)}_{0},\tilde{S}^{(\rm e)}_{0})=0,
&\ \  \mbox{in}\ \ \  \tilde{\Omega}^{(\rm e)}, \\[5pt]
\partial_{\xi}\delta z^{(\rm n)}+ \textrm{diag}( \lambda^{(\rm n-1)}_{+}, \lambda^{(\rm n-1)}_{-})\partial_{\eta}\delta z^{(\rm n)}=0,
&\ \  \mbox{in}\ \ \ \tilde{\Omega}^{(\rm h)},  \\[5pt]
\delta\varphi^{(\rm n)}=g^{(\rm n)}_{0}(\eta),
&\ \   \mbox{on}\ \ \ \tilde{\Gamma}^{(\rm e)}_{\rm in},\\[5pt]
\partial_{\xi}\delta\varphi^{(\rm n)}=\tilde{\omega}_{\rm e}(\eta),  &\ \   \mbox{on}\ \ \ \tilde{\Gamma}^{(\rm e)}_{\rm ex},\\[5pt]
\delta z^{(\rm n)}=\delta z_{0}(\eta),  &\ \   \mbox{on}\ \ \ \tilde{\Gamma}^{(\rm h)}_{\rm in},\\[5pt]
\delta\varphi^{(\rm n)}=g^{(\rm n)}_{+}(\xi),
&\ \  \mbox{on}\ \ \ \tilde{\Gamma}_{+},\\[5pt]
\delta\varphi^{(\rm n)}=\int^{\xi}_{0}\tan\Big(\frac{\delta z^{(\rm n)}_{-}+\delta z^{(\rm n)}_{+}}{2}\Big)(\tau,0)d\tau,
&\ \   \mbox{on}\ \ \ \tilde{\Gamma}_{\rm cd},\\[5pt]
\delta z^{(\rm n)}_{-}-\delta z^{(\rm n)}_{+}=2\beta^{(\rm n-1)}_{\rm cd, 1}\partial_{\xi}\delta\varphi^{(\rm n)}
+2\beta^{(\rm n-1)}_{\rm cd, 2}\partial_{\eta}\delta\varphi^{(\rm n)}+2c_{\rm cd}(\xi),
&\ \   \mbox{on}\ \ \ \tilde{\Gamma}_{\rm cd},\\[5pt]
\delta z^{(\rm n)}_{-}+\delta z^{(\rm n)}_{+}=2\arctan g'_{-}(\xi), &\ \  \mbox{on}\ \ \ \tilde{\Gamma}_{-},
\end{array}
\right.
\end{eqnarray}
where $\delta\varphi^{(\rm n)}(0,0)=0$, 
\begin{eqnarray}\label{eq:3.69}
\begin{split}
&\mathcal{N}_{1}(D\varphi^{(\rm n)};\tilde{B}^{(\rm e)}_{0},\tilde{S}^{(\rm e)}_{0})
=\frac{\partial_{\xi}\varphi^{(\rm n)}}{\tilde{\rho}^{(\rm e)}(D\varphi^{(\rm n)};\tilde{B}^{(\rm e)}_{0},\tilde{S}^{(\rm e)}_{0})\partial_{\eta}\varphi^{(\rm n)}}, \\[5pt]
&\mathcal{N}_{2}(D\varphi^{(\rm n)};\tilde{B}^{(\rm e)}_{0},\tilde{S}^{(\rm e)}_{0})
=\tilde{p}^{(\rm e)}(D\varphi^{(\rm n)};\tilde{B}^{(\rm e)}_{0},\tilde{S}^{(\rm e)}_{0}),
\end{split}
\end{eqnarray}
\begin{eqnarray}\label{eq:3.70}
\begin{split}
&\delta z^{(\rm n)}:= z^{(\rm n)}-\underline{z}, \ \ \
\delta z^{(\rm n)}_{0}(\eta):= z^{(\rm n)}_{0}(\eta)-\underline{z},\\[5pt]
&\delta \varphi^{(\rm n)}:=\varphi^{(\rm n)}-\underline{\varphi},\ \ \
\lambda^{(\rm n-1)}_{\pm}:=\lambda_{\pm}(z^{(\rm n-1)};\tilde{B}^{(\rm h)}_{0},\tilde{S}^{(\rm h)}_{0}),
\end{split}
\end{eqnarray}

\small\begin{equation}\label{eq:3.72}
\begin{split}
\beta^{(\rm n-1)}_{\rm cd, 1}
&:=\int^{1}_{0}\partial_{p^{(\rm e)}}\Theta\big(\underline{\tilde{p}}^{(\rm e)}_{0}+\tau(\tilde{p}^{(\rm e)}_{\rm n-1,0}
-\underline{\tilde{p}}^{(\rm e)}_{0});\tilde{B}^{(\rm h)}_{0},\tilde{S}^{(\rm h)}_{0}\big)d\tau\\[5pt]
&\qquad \ \ \ \times
\int^{1}_{0}\partial_{\partial_{\xi}\varphi}\tilde{p}^{(\rm e)}\big(D\underline{\varphi}+\tau D\delta\varphi^{(\rm n-1)};\tilde{B}^{(\rm e)}_{0},\tilde{S}^{(\rm e)}_{0}\big)d\tau,\\[5pt]
\beta^{(\rm n-1)}_{\rm cd, 2}
&:=\int^{1}_{0}\partial_{\tilde{p}^{(\rm e)}}\Theta\big(\underline{\tilde{p}}^{(\rm e)}_{0}+\tau(\tilde{p}^{(\rm e)}_{\rm n-1,0}
-\underline{\tilde{p}}^{(\rm e)}_{0});\tilde{B}^{(\rm h)}_{0},\tilde{S}^{(\rm h)}_{0}\big)d\tau\\[5pt]
&\qquad \ \ \ \times
\int^{1}_{0}\partial_{\partial_{\eta}\varphi}\tilde{p}^{(\rm e)}\big(D\underline{\varphi}+\tau D\delta\varphi^{(\rm n-1)};\tilde{B}^{(\rm e)}_{0},\tilde{S}^{(\rm e)}_{0}\big)d\tau,
\end{split}
\end{equation}
\begin{eqnarray}\label{eq:3.73}
\begin{split}
c_{\rm cd}(\xi):=\Theta(\underline{\tilde{p}}^{(\rm e)}_{0}; B^{(\rm h)}_{0}, S^{(\rm h)}_{0})
-\Theta(\underline{p}^{(\rm e)}; \underline{B}^{(\rm h)},  \underline{S}^{(\rm h)}),
\end{split}
\end{eqnarray}
and
\begin{eqnarray}\label{eq:3.73a}
\begin{split}
&g^{(\rm n)}_{0}(\eta):=\frac{1}{\underline{\beta}_{0,2}}\int^{\eta}_{0}\Big(\tilde{p}^{(\rm e)}_{0}(\mu)
-\underline{\tilde{p}}^{(\rm e)}_{0}(\mu)-\beta^{\rm (n)}_{0,1}
\partial_{1}\delta\varphi^{(\rm n)}-(\beta^{\rm (n)}_{0,2}-\underline{\beta}_{0,2})
\partial_{2}\delta\varphi^{(\rm n)}\Big)d\mu,\\[5pt]
&g^{(\rm n)}_{+}(\xi):=g_{+}(\xi)-g_{+}(0)+g^{(\rm n)}_{0}(m^{(\rm e)}).
\end{split}
\end{eqnarray}
Here
\begin{eqnarray}\label{eq:3.74}
\begin{split}
\tilde{p}^{(\rm e)}_{\rm n-1,0}(\xi, \eta):= \tilde{p}^{(\rm e)}\big(D\varphi^{(\rm n-1)};\tilde{B}^{(\rm e)}_{0},\tilde{S}^{(\rm e)}_{0}\big),
\  \  \ \
\underline{\tilde{p}}^{(\rm e)}_{0}(\eta):= \tilde{p}^{(\rm e)}\big(D\underline{\varphi};\tilde{B}^{(\rm e)}_{0},\tilde{S}^{(\rm e)}_{0}\big).
\end{split}
\end{eqnarray}
and
\begin{eqnarray}\label{eq:3.71}
\begin{split}
\beta^{(\rm n)}_{0, 1}
:=\int^{1}_{0}\partial_{\partial_{\xi}\varphi}\tilde{p}^{(\rm e)}\big(D\underline{\varphi}+\nu D\delta\varphi^{(\rm n)};\tilde{B}^{(\rm e)}_{0},\tilde{S}^{(\rm e)}_{0}\big)d\nu,\\[5pt]
\beta^{(\rm n)}_{0, 2}
:=\int^{1}_{0}\partial_{\partial_{\eta}\varphi}\tilde{p}^{(\rm e)}\big(D\underline{\varphi}+\nu D\delta\varphi^{(\rm n)};\tilde{B}^{(\rm e)}_{0},\tilde{S}^{(\rm e)}_{0}\big)d\nu.
\end{split}
\end{eqnarray}

{\color{black}We remark that the linearized boundary condition $\eqref{eq:3.68}_{8}$ on $\tilde{\Gamma}_{\rm cd}$ is derived by first taking the difference between the boundary condition $\eqref{eq:3.63}_{8}$ on $\tilde{\Gamma}_{\rm cd}$ and the corresponding background state,   and then applying the mean value theorem of the integral form. At the fixed point, the second boundary condition $\eqref{eq:3.68}_{8}$ on $\tilde{\Gamma}_{\rm cd}$ will recover the second boundary condition $\eqref{eq:3.63}_{8}$ on $\tilde{\Gamma}_{\rm cd}$.}

\par Now, we can introduce the map
\begin{eqnarray}\label{eq:3.80}
\mathcal{J}:\  \mathscr{K}_{2{\color{black}\epsilon_I}}\longrightarrow \mathscr{K}_{2{\color{black}\epsilon_I}}.
\end{eqnarray}
For the given function $(\varphi^{(\rm n-1)}, z^{(\rm n-1)})\in \mathscr{K}_{2{\color{black}\epsilon_I}}$,
we solve problem $(\widetilde{\mathbf{FP}})_{\rm n}$ to obtain the function
$(\varphi^{(\rm n)}, z^{(\rm n)})\in \mathscr{K}_{2{\color{black}\epsilon_I}}$.
Then the map $\mathcal{J}$ is defined as
\begin{eqnarray}\label{eq:3.81}
(\varphi^{(\rm n)}, z^{(\rm n)}):=\mathcal{J}(\varphi^{(\rm n-1)}, z^{(\rm n-1)}).
\end{eqnarray}
For the map $\mathcal{J}$, we are going to verify the following: 
\smallskip
\par (1)\ \ $\mathcal{J}$ is well-defined, i.e., $\mathcal{J}$ exists and maps from
$\mathscr{K}_{2{\color{black}\epsilon_I}}$ to $\mathscr{K}_{2{\color{black}\epsilon_I}}$;
\par (2)\ \ $\mathcal{J}$ is a contraction.

\noindent
These two facts will be achieved by  Lemma \ref{lem:6.1} in Section 6.
Then,   it follows from the Banach fixed point theorem that the sequence is convergent and its limit is a solution of  the problem $(\mathbf{FP})$.
{\color{black}In fact, by taking the tangential derivatives on the boundary conditions $\eqref{eq:3.68}_{3}$, $\eqref{eq:3.68}_{6}$ and $\eqref{eq:3.68}_{7}$
along the boundaries $\tilde{\Gamma}^{(\rm e)}_{\rm in}$, $\tilde{\Gamma}_{+}$ and $\tilde{\Gamma}_{\rm cd}$, it is easy to see that the solution of problem $(\mathbf{FP})_{\rm n}$ with $(\varphi^{(\rm n)}, z^{(\rm n)})=(\varphi^{(\rm n-1)}, z^{(\rm n-1)})$
is actually the solution of problem $(\mathbf{FP})$}.

\emph{Step\ 1.2. } One of the main difficulties in showing the existence and uniqueness of problem $(\mathbf{FP})_{\rm n}$ in $\mathscr{K}_{\epsilon}$ is due to the requirement of the higher order regularity of the solution $\varphi$ near the corner point. In order to overcome this difficulty,  we need to define a Banach space:
{\color{black}\begin{eqnarray}\label{eq:3.76}
\mathcal{W}=\big\{\omega_{\rm cd}:\ \omega_{\rm cd}(0)=\omega'_{\rm cd}(0)=0, \omega_{\rm cd}(L)=\tilde{\omega}_{\rm e}(0), \
 \|\omega_{\rm cd}\|^{(-\alpha, \{\mathcal{P}_{\rm e}\})}_{1,\alpha; \tilde{\Gamma}_{\rm cd}}<\infty \big\}.
\end{eqnarray}
}
Given  any $\omega_{\rm cd}\in \mathcal{W}$ satisfying
\begin{eqnarray}\label{eq:3.77}
\|\omega_{\rm cd}\|^{(-\alpha, \{\mathcal{P}_{\rm e}\})}_{1,\alpha; \tilde{\Gamma}_{\rm cd}}\leq \varrho,
\end{eqnarray}
for a 
sufficiently small constant $\varrho>0$,
let us consider the following nonlinear problem:
\begin{eqnarray}\label{eq:3.78}
(\widetilde{\mathbf{FP}})_{\rm n}\quad
\left\{
\begin{array}{llll}
\partial_{\xi}\mathcal{N}_{1}(D\varphi^{(\rm n)};\tilde{B}^{(\rm e)}_{0},\tilde{S}^{(\rm e)}_{0})
+\partial_{\eta}\mathcal{N}_{2}(D\varphi^{(\rm n)};\tilde{B}^{(\rm e)}_{0},\tilde{S}^{(\rm e)}_{0})=0,
&\ \  \mbox{in}\ \ \  \tilde{\Omega}^{(\rm e)}, \\[5pt]
\partial_{\xi}\delta z^{(\rm n)}+ \textrm{diag}( \lambda^{(\rm n-1)}_{+}, \lambda^{(\rm n-1)}_{-})\partial_{\eta}\delta z^{(\rm n)}=0,
&\ \  \mbox{in}\ \ \ \tilde{\Omega}^{(\rm h)},  \\[5pt]
\delta\varphi^{(\rm n)}=g^{(\rm n)}_{0}(\eta),
&\ \   \mbox{on}\ \ \ \tilde{\Gamma}^{(\rm e)}_{\rm in},\\[5pt]
\partial_{\xi}\delta\varphi^{(\rm n)}=\tilde{\omega}_{\rm e}(\eta),  &\ \   \mbox{on}\ \ \ \tilde{\Gamma}^{(\rm e)}_{\rm ex},\\[5pt]
\delta z^{(\rm n)}=\delta z_{0}(\eta),  &\ \   \mbox{on}\ \ \ \tilde{\Gamma}^{(\rm h)}_{\rm in},\\[5pt]
\delta\varphi^{(\rm n)}=g^{(\rm n)}_{+}(\xi),
&\ \  \mbox{on}\ \ \ \tilde{\Gamma}_{+},\\[5pt]
\delta\varphi^{(\rm n)}=\int^{\xi}_{0}\omega_{\rm cd}(\nu)d\nu,
&\ \   \mbox{on}\ \ \ \tilde{\Gamma}_{\rm cd},\\[5pt]
\delta z^{(\rm n)}_{-}-\delta z^{(\rm n)}_{+}=2\beta^{(\rm n-1)}_{\rm cd, 1}\partial_{\xi}\delta\varphi^{(\rm n)}
+2\beta^{(\rm n-1)}_{\rm cd, 2}\partial_{\eta}\delta\varphi^{(\rm n)}+2c_{\rm cd}(\xi),
&\ \   \mbox{on}\ \ \ \tilde{\Gamma}_{\rm cd},\\[5pt]
\delta z^{(\rm n)}_{-}+\delta z^{(\rm n)}_{+}=2\arctan g'_{-}(\xi), &\ \  \mbox{on}\ \ \ \tilde{\Gamma}_{-},
\end{array}
\right.
\end{eqnarray}
in $\mathscr{K}_{2\epsilon_{I}}$ for some $\epsilon_{I}>0$ sufficiently small with $g^{(\rm n)}_{0}$ and $g^{(\rm n)}_{+}$ given by \eqref{eq:3.73a}.
The  \emph{a prior} estimates of the solution
$(\delta\varphi^{(\rm n)}, \delta z^{(\rm n)})$ to the problem
$(\widetilde{\mathbf{FP}})_{\rm n}$ will be obtained in Section 4.

{\em Step\  1.3.}\  With the above solution $(\delta\varphi^{(\rm n)}, \delta z^{(\rm n)})$,
we define a iteration map $\mathbf{T}_{\omega}: \mathcal{W}\longrightarrow \mathcal{W}$, i.e.,  $\tilde{\omega}_{\rm cd}=\mathbf{T}_{\omega}(\omega_{\rm cd}, \boldsymbol{\omega}_{0})$ by the equation
\begin{eqnarray}\label{eq:3.79}
\tilde{\omega}_{\rm cd}=\tan\Big(\frac{\delta z^{(\rm n)}_{-}+\delta z^{(\rm n)}_{+}}{2}\Big),
\end{eqnarray}
{where $\boldsymbol{\omega}_{0}=( \tilde{p}^{(\rm e)}_{0}, \tilde{B}^{(\rm e)}_{0}, \tilde{S}^{(\rm e)}_{0}, z_{0}, \tilde{B}^{(\rm h)}_{0},
\tilde{S}^{(\rm h)}_{0},\tilde{\omega}_{\rm e}, g_{+}, g_{-})$.}

Then, by applying the implicit function theorem, we shall show in Section 5 that $\mathbf{T}_{\omega}$ has a fixed point $\omega_{\rm cd}$ satisfying the equation \eqref{eq:3.79}.}

\smallskip

In fact, for problem $(\mathbf{FP})$, we have

\begin{theorem}\label{thm:3.3}
	For any given $m^{(\rm e)} \in \mathcal{M}_{\sigma}$ {with $\sigma>0$ sufficiently small}, there exist some constants $\alpha_0\in(0,1)$ and {\color{black}$\tilde{\epsilon}_{1}>0$}
 depending only on $\tilde{\underline{U}}$ and $L$, such that, if the data satisfy
	\begin{eqnarray}\label{eq:3.65}
	\begin{split}
	&{\big\|\tilde{p}^{(\rm e)}_{0}-\underline{p}^{(\rm e)}\big\|_{1,\alpha;\tilde{\Gamma}^{(\rm e)}_{\rm in}}}
	+\sum_{k=\rm e, \rm h}\bigg(\big\|\tilde{B}^{(\rm k)}_{0}-\underline{B}^{(\rm k)},
	\big\|_{1,\alpha;\tilde{\Gamma}^{(\rm k)}_{\rm in}}
	+\big\|\tilde{S}^{(\rm k)}_{0}-\underline{S}^{(\rm k)}\big\|_{1,\alpha;\tilde{\Gamma}^{(\rm k)}_{\rm in}}\bigg)\\[5pt]
	&\quad\ \  +{\color{black}\big\|\tilde{\omega}_{\rm e}\big\|^{(-1-\alpha, \{\mathcal{P}_{\rm e},\mathcal{Q}_{\rm e} \})}_{2,\alpha;\tilde{\Gamma}^{(\rm e)}_{\rm ex}}}
	+\big\|z_{0}-\underline{z}\big\|_{1,\alpha; \tilde{\Gamma}^{(\rm h)}_{\rm in}}
	+\big\|g_{+}-1\big\|_{2,\alpha;\tilde{ \Gamma}_{+}}
	+\big\|g_{-}+1\big\|_{2,\alpha;  \tilde{\Gamma}_{-}} \leq {\color{black}\tilde{\epsilon}},
	\end{split}
	\end{eqnarray}
	and
	\begin{eqnarray}\label{eq:3.66}
	\underline{M}^{(\rm h)}=\frac{\underline{u}^{(\rm h)}}{\underline{c}^{(\rm h)}}\geq \sqrt{1+L^{2}/4},
	\end{eqnarray}
for any ${\color{black}\tilde{\epsilon} \in (0, \tilde{\epsilon}_{1})}$ and $\alpha\in(0,\alpha_0)$, then the fixed boundary value problem $(\mathbf{FP})$ 
admits a unique 
solution $(\varphi,z)\in { C^{2,\alpha}_{(-1-\alpha, \tilde{\Sigma}^{(\rm e)}\backslash \{\mathcal{O}\})}}(\tilde{\Omega}^{(\rm e)})\times C^{1,\alpha}(\tilde{\Omega}^{(\rm h)})$ with 
\begin{eqnarray}\label{eq:3.67}
\begin{split}
{\big\|\varphi-\underline{\varphi}\big\|^{(-1-\alpha, \tilde{\Sigma}^{(\rm e)}\backslash \{\mathcal{O}\})}_{2,\alpha;\tilde{\Omega}^{(\rm e)}}}+\big\|z-\underline{z}\big\|_{1,\alpha;\tilde{\Omega}^{(\rm h)}}\leq \tilde{C}_{1} {\color{black}\tilde{\epsilon}},
\end{split}
\end{eqnarray}
where the constant $\tilde{C}_{1}>0$ depends only on $\tilde{\underline{U}}$ and $L$. 
\end{theorem}

\bigskip

\section{
The Solutions to the Linearized Problem $(\widetilde{\mathbf{FP}})_{\rm n}$
in  $\tilde{\Omega}^{(\rm e)}\cup\tilde{\Omega}^{(\rm h)}$}\setcounter{equation}{0}


In this section we will first consider  the solutions $(\varphi^{(\rm n)}, z^{(\rm n)})$ to the problem $(\widetilde{\mathbf{FP}})_{\rm n}$ near $(\underline{\varphi},\underline{z})$ and then establish some {\it a priori} estimates for
$(\varphi^{(\rm n)}, z^{(\rm n)})$   to show that the map $\mathcal{J}$ is well defined.

\begin{theorem}\label{thm:4.1}
{\color{black}For any given $m^{(\rm e)} \in \mathcal{M}_{\sigma}$ with $\sigma>0$ sufficiently small}, {\color{black}there exist some constants $\alpha_0\in(0,1)$, $\mathcal{C}^{*}_0>0$ and $\epsilon^{*}_{I}>0$ depending only on $\tilde{\underline{U}}$ and $L$, and some constant $\varrho_{0}>0$ depending on $\tilde{\epsilon}$ and $\tilde{\underline{U}}$, such that for any $\alpha\in(0,\alpha_0)$ and $\tilde{\epsilon}<\epsilon_I\in (\mathcal{C}^{*}_0 \tilde{\epsilon},\epsilon^{*}_{I})$, $\varrho\in (0,\varrho_{0})$ with $\tilde{\epsilon}>0$ sufficiently small},
if $(\delta \varphi^{(\rm n-1)},\delta z^{(\rm n-1)})\in \mathscr{K}_{2\epsilon_I}$ {\color{black}and $\omega_{\rm cd}\in \mathcal{W}$ with $\|\omega_{\rm cd}\|^{(-\alpha, \{\mathcal{P}_{\rm e}\})}_{1,\alpha; \tilde{\Gamma}_{\rm cd}}\leq \varrho$}, then
the linearized problem $(\widetilde{\mathbf{FP}})_{\rm n}$ admits
a unique solution $(\delta \varphi^{(\rm n)},\delta z^{(\rm n)})\in\mathscr{K}_{2{\color{black}\epsilon_I}}$ satisfying
\begin{eqnarray}\label{eq:4.1}
\begin{split}
&{\big\|\delta \varphi^{(\rm n)}\big\|^{(-1-\alpha, \tilde{\Sigma}^{(\rm e)}\backslash \{\mathcal{O}\})}_{2, \alpha; \tilde{\Omega}^{(\rm e)}}}
+\|\delta z^{(\rm n)}\|_{1, \alpha; \tilde{\Omega}^{(\rm h)}}\\[5pt]
&\ \  \leq C^{*}\bigg({\big\|\delta\tilde{p}^{(\rm e)}_{0}\big\|_{1,\alpha;\tilde{\Gamma}^{(\rm e)}_{\rm in}}}
+\big\|\delta z_{0}\big\|_{1,\alpha; \tilde{\Gamma}^{(\rm h)}_{\rm in}}
+\sum_{\rm k=\rm e, h}\big\|\delta\tilde{B}^{(\rm k)}_{0}\big\|_{1,\alpha;\tilde{\Gamma}^{(\rm k)}_{\rm in}}
+\sum_{\rm k=\rm e, h}\big\|\delta\tilde{S}^{(\rm k)}_{0}\big\|_{1,\alpha;\tilde{\Gamma}^{(\rm k)}_{\rm in}}\\[5pt]
&\quad\qquad \ \ +{\color{black}\big\|\tilde{\omega}_{\rm e}\big\|^{(-1-\alpha, \{\mathcal{P}_{\rm e}, \mathcal{Q}_{\rm e}\})}_{2,\alpha;\tilde{\Gamma}^{(\rm e)}_{\rm ex}}
+\big\|\omega_{\rm cd}\big\|^{(-\alpha, \{\mathcal{P}_{\rm e}\})}_{1,\alpha;\tilde{\Gamma}_{\rm cd}}}+\big\|g_{+}-1\big\|_{2,\alpha;\tilde{ \Gamma}_{+}}
+\big\|g_{-}+1\big\|_{2,\alpha;  \tilde{\Gamma}_{-}}\bigg),
\end{split}
\end{eqnarray}
where $C^{*}>0$ {depends} only on $\tilde{\underline{U}}$ and $L$. Here, $\delta\tilde{p}^{(\rm e)}_{0}=\tilde{p}^{(\rm e)}_{0}-\underline{p}^{(\rm e)}$ and
$\delta\tilde{B}^{(\rm k)}_{0}=\tilde{B}^{(\rm k)}_{0}-\underline{B}^{(\rm k)}$, $\delta\tilde{S}^{(\rm k)}_{0}=\tilde{S}^{(\rm k)}_{0}-\underline{S}^{(\rm k)}$ for $\rm k=\rm e, h$.
\end{theorem}

\par We will show Theorem \ref{thm:4.1} by solving the problem $(\widetilde{\mathbf{FP}})_{\rm n}$ {\color{black} first in $\tilde{\Omega}^{(\rm e)}$ and then in $\tilde{\Omega}^{(\rm h)}$
together}.
To this end, we will first consider the problem $(\widetilde{\mathbf{FP}})_{\rm n}$ for
$\delta \varphi^{(\rm n)}$ in the subsonic region $\tilde{\Omega}^{(\rm e)}$, then we can deduce the boundary condition of $\delta z^{(\rm n)}$ on $\tilde{\Gamma}_{\rm cd}$ by gluing the solution $\delta \varphi^{(\rm n)}$ obtained in
$\tilde{\Omega}^{(\rm e)}$. With this boundary condition for $\delta z^{(\rm n)}$ on $\tilde{\Gamma}_{\rm cd}$, we can
solve the problem $(\widetilde{\mathbf{FP}})_{\rm n}$ for $\delta z^{(\rm n)}$ in the supersonic region $\tilde{\Omega}^{(\rm h)}$
by employing the characteristics method.

\subsection{Estimates of solutions in the subsonic region $\tilde{\Omega}^{(\rm e)}$ for problem $(\widetilde{\mathbf{FP}})_{\rm n}$}
Let us consider the problem $(\widetilde{\mathbf{FP}})_{\rm n}$ in the subsonic region $\tilde{\Omega}^{(\rm e)}$ (see Fig.  \ref{fig4.5}).
Problem $(\widetilde{\mathbf{FP}})_{\rm n}$ in $\tilde{\Omega}^{(\rm e)}$ is the following boundary value problem for $\varphi^{(\rm n)}$:
\begin{eqnarray}\label{eq:4.2}
\left\{
\begin{array}{llll}
\partial_{\xi}\mathcal{N}_{1}(D\varphi^{(\rm n)};\tilde{B}^{(\rm e)}_{0},\tilde{S}^{(\rm e)}_{0})
+\partial_{\eta}\mathcal{N}_{2}(D\varphi^{(\rm n)};\tilde{B}^{(\rm e)}_{0},\tilde{S}^{(\rm e)}_{0})=0,
&\ \ \ \mbox{in}\ \ \  \tilde{\Omega}^{(\rm e)}, \\[5pt]
\delta\varphi^{(\rm n)}=g^{(\rm n)}_{0}(\eta),
&\ \ \  \mbox{on}\ \ \ \tilde{\Gamma}^{(\rm e)}_{\rm in},\\[5pt]
\partial_{\xi}\delta\varphi^{(\rm n)}=\tilde{\omega}_{\rm e}(\eta),  &\ \ \   \mbox{on}\ \ \ \tilde{\Gamma}^{(\rm e)}_{\rm ex},\\[5pt]
\delta\varphi^{(\rm n)}=g^{(\rm n)}_{+}(\xi), &\ \ \ \mbox{on}\ \ \ \tilde{\Gamma}_{+},\\[5pt]
\delta\varphi^{(\rm n)}=\int^{\xi}_{0}\omega_{\rm cd}(\nu)d\nu, &\ \ \  \mbox{on}\ \ \ \tilde{\Gamma}_{\rm cd},
\end{array}
\right.
\end{eqnarray}
where $\delta\varphi^{(n)}(0,0)=0$, and $\mathcal{N}_{j}\ (j=1,2)$, $c_{\rm cd}(\xi)$, and $g^{(\rm n)}_{0}(\eta)$, $g^{(\rm n)}_{+}(\eta)$
are given in \eqref{eq:3.69} and \eqref{eq:3.73}-\eqref{eq:3.73a}, respectively. Our main result in this section is the following proposition.
\vspace{2pt}
\begin{figure}[ht]
	\begin{center}
		\begin{tikzpicture}[scale=1.2]
		\draw [line width=0.08cm] (-5,1.5) --(3.0,1.5);
		\draw [line width=0.08cm][dashed][blue] (-5,-1.5) --(3.0,-1.5);
		\draw [line width=0.08cm](-5,-1.5) --(-5,1.5);
		\draw [line width=0.08cm](3.0,-1.5) --(3.0,1.5);
		\node at (-5.5, 0) {$\tilde{\Gamma}^{(\rm e)}_{\rm in}$};
		\node at (3.5, 0) {$\tilde{\Gamma}^{(\rm e)}_{\rm ex}$};
		\node at (-1, 1.8) {$\tilde{\Gamma}_{+}$};
		\node at (-1, -1.8) {$\tilde{\Gamma}_{\rm cd}$};
		\node at (-5.0, -1.9) {$\xi=0$};
		\node at (3.7, 1.5) {$\eta=m^{(\rm e)}$};
		\node at (3.6, -1.4) {$\eta=0$};
		\node at (3.0, -1.9) {$\xi=L$};
		\node at (-1, 0) {$\tilde{\Omega}^{(\rm e)}$};
		\end{tikzpicture}
		\caption{Boundary value problem for $(\widetilde{\mathbf{FP}})_{\rm n}$ in $\tilde{\Omega}^{(\rm e)}$}\label{fig4.5}
	\end{center}
\end{figure}

\begin{proposition}\label{prop:4.1}
{\color{black}For any given $m^{(\rm e)} \in \mathcal{M}_{\sigma}$ with $\sigma>0$ sufficiently small}, {\color{black}there exist some constants $\alpha_0\in(0,1)$, $\mathcal{C}_{\rm e,0}>0$ and $\varepsilon_{\rm e,0}>0$ depending only on $\tilde{\underline{U}}$ and $L$, such that for $\tilde{\epsilon}<\varepsilon_{\rm e}\in (\mathcal{C}_{\rm e,0}\tilde{\epsilon}, \varepsilon_{\rm e,0})$ with $\tilde{\epsilon}>0$ sufficiently small, $\varrho\in (0, \mathcal{C}^{-1}_{\rm e,0}\varepsilon_{\rm e})$ and $\alpha\in(0,\alpha_0)$,
if $\omega_{\rm cd}\in \mathcal{W}$ with $\|\omega_{\rm cd}\|^{(-\alpha, \{\mathcal{P}_{\rm e}\})}_{1,\alpha; \tilde{\Gamma}_{\rm cd}}\leq \varrho$,
then the problem \eqref{eq:4.2} admits a unique solution $\varphi^{(\rm n)}\in { C^{2,\alpha}_{(-1-\alpha; \tilde{\Sigma}^{(\rm e)}\backslash \{\mathcal{O}\})}}\big(\tilde{\Omega}^{(\rm e)}\big)$
satisfying
$\big\|\delta \varphi^{(\rm n)}\big\|^{(-1-\alpha, \tilde{\Sigma}^{(\rm e)}\backslash \{\mathcal{O}\})}_{2, \alpha; \tilde{\Omega}^{(\rm e)}}\leq \varepsilon_{\rm e}$. Moreover, it holds that}
\begin{eqnarray}\label{eq:4.3}
\begin{split}
{\big\|\delta \varphi^{(\rm n)}\big\|^{(-1-\alpha, \tilde{\Sigma}^{(\rm e)}\backslash \{\mathcal{O}\})}_{2, \alpha; \tilde{\Omega}^{(\rm e)}}}
&\leq C^{*}_{40}\bigg({\big\|\delta\tilde{p}^{(\rm e)}_{0}\big\|_{1,\alpha;\tilde{\Gamma}^{(\rm e)}_{\rm in}}}
+\big\|\delta\tilde{B}^{(\rm e)}_{0}\big\|_{1,\alpha;\tilde{\Gamma}^{(\rm e)}_{\rm in}}
+\big\|\delta\tilde{S}^{(\rm e)}_{0}\big\|_{1,\alpha;\tilde{\Gamma}^{(\rm e)}_{\rm in}}\\[5pt]
&\qquad\quad  \ \ +{\color{black}\big\|\tilde{\omega}_{\rm e}\big\|^{(-1-\alpha, \{\mathcal{P}_{\rm e}, \mathcal{Q}_{\rm e}\})}_{2,\alpha;\tilde{\Gamma}^{(\rm e)}_{\rm ex}}
+\big\|\omega_{\rm cd}\big\|^{(-\alpha, \{\mathcal{P}_{\rm e}\})}_{1,\alpha;\tilde{\Gamma}_{\rm cd}}}
+\big\|g_{+}-1\big\|_{2,\alpha;\tilde{ \Gamma}_{+}}\bigg),
\end{split}
\end{eqnarray}
where the constant $C^{*}_{40}>0$ depends only on $\tilde{\underline{U}}$ and $L$, and $\delta\tilde{p}^{(\rm e)}_{0}$, $\delta\tilde{B}^{(\rm e)}_{0}$, $\delta\tilde{S}^{(\rm e)}_{0}$
are given by Theorem \ref{thm:4.1}.
\end{proposition}

\begin{remark}\label{rem:4.1}
From \eqref{eq:4.3}, it is easy to prove that there exists a constant $\tilde{C}^{*}_{40}>0$
depending only on $\tilde{\underline{U}}$, $L$ and $\alpha_0$ such that
\begin{eqnarray}\label{eq:4.4}
\begin{split}
{\color{black}\big\|\delta \varphi^{(\rm n)}\big\|^{(-1-\alpha, \{\mathcal{P}_{\rm e}\})}_{2,\alpha; \bar{\tilde{\Gamma}}_{\rm cd}}}
&\leq \tilde{C}^{*}_{40}\bigg({\big\|\delta\tilde{p}^{(\rm e)}_{0}\big\|_{1,\alpha;\tilde{\Gamma}^{(\rm e)}_{\rm in}}}
+\big\|\delta\tilde{B}^{(\rm e)}_{0}\big\|_{1,\alpha;\tilde{\Gamma}^{(\rm e)}_{\rm in}}
+\big\|\delta\tilde{S}^{(\rm e)}_{0}\big\|_{1,\alpha;\tilde{\Gamma}^{(\rm e)}_{\rm in}}\\[5pt]
&\qquad\quad  \ \ +{\color{black}\big\|\tilde{\omega}_{\rm e}\big\|^{(-1-\alpha, \{\mathcal{P}_{\rm e}, \mathcal{Q}_{\rm e}\})}_{2,\alpha;\tilde{\Gamma}^{(\rm e)}_{\rm ex}}
+\big\|\omega_{\rm cd}\big\|^{(-\alpha, \{\mathcal{P}_{\rm e}\})}_{1,\alpha;\tilde{\Gamma}_{\rm cd}}}
+\big\|g_{+}-1\big\|_{2,\alpha;\tilde{ \Gamma}_{+}}\bigg).
\end{split}
\end{eqnarray}
\end{remark}

\begin{proof} To prove Proposition \ref{prop:4.1}, 
we will first consider the existence and uniqueness of solutions to the problem \eqref{eq:4.2} in the weight H\"{o}lder space due to the complexity of boundary conditions and the corner points of the domain. This can be achieved by developing a nonlinear iteration scheme near the background state together with the Banach fixed point argument.
Next, for the solution of the nonlinear problem, we will further show the $C^{2, \alpha}$-regularity of the solution $\varphi^{(\rm n)}$ near the corner point $\mathcal{O}=(0,0)$ to ensure that the solution $z^{(\rm n)}$ in the supersonic region $\tilde{\Omega}^{(\rm h)}$ is also solvable.
The proof is divided into three steps.
	
\smallskip
\emph{{Step\ 1.}\ \   Linearization of problem \eqref{eq:4.2} in $\tilde{\Omega}^{(\rm e)}$}.	
{\color{black}For fixed $\rm n$, define the iteration set}
{
\begin{eqnarray*}
\begin{split}
{\bf\Xi}_{\varepsilon_{\rm e}}=\big\{\varphi^{(\rm n)}:\ \|\varphi^{(\rm n)}-\underline{\varphi}\|^{(-1-\alpha, \tilde{\Sigma}^{(\rm e)})}_{2, \alpha; \tilde{\Omega}^{(\rm e)} }\leq \varepsilon_{\rm e}\big\}.
\end{split}
\end{eqnarray*}
 }
Let
\begin{eqnarray} \label{eq:4.5}
\begin{split}
\partial_{1}:=\partial_{\xi},\quad  \partial_{2}:=\partial_{\eta}.
\end{split}
\end{eqnarray}
Notice that the background state $\underline{\varphi}$ satisfies
\begin{eqnarray} \label{eq:4.6}
\begin{split}
\partial_{1}\mathcal{N}_{1}(D\underline{\varphi};{\color{black}\underline{B}^{(\rm e)},\underline{S}^{(\rm e)}})+
\partial_{2}\mathcal{N}_{2}(D\underline{\varphi}; {\color{black}\underline{B}^{(\rm e)},\underline{S}^{(\rm e)}})=0.
\end{split}
\end{eqnarray}
Taking the difference of equations $\eqref{eq:4.2}_{1}$ and \eqref{eq:4.6}, and then linearizing the resulting equation, we have
\begin{eqnarray*}
\begin{split}
\sum_{i,j=1,2}\partial_{i}\big(a^{(\delta\varphi^{(\rm n)})}_{ij}\partial_{j}\delta\hat{\varphi}^{(\rm n)}\big)
=\sum_{i=1,2}\partial_{i}b_{i},
\end{split}
\end{eqnarray*}
where
\begin{eqnarray*}
\begin{split}
a^{(\delta\varphi^{(\rm n)})}_{ij}=\int^{1}_{0}{\color{black}\partial_{\partial\varphi^{(\rm n)}_j}}\mathcal{N}_{i}
\big(D\underline{\varphi}+\tau D\delta\varphi^{(\rm n)};\tilde{B}^{(\rm e)}_{0},\tilde{S}^{(\rm e)}_{0}\big)d\tau,\ \  i,j=1,2,
\end{split}
\end{eqnarray*}
and
\begin{eqnarray*}
\begin{split}
b_{i}=\mathcal{N}_{i}(D\underline{\varphi}; \tilde{B}^{(\rm e)}_{0},\tilde{S}^{(\rm e)}_{0})
-\mathcal{N}_{i}(D\underline{\varphi};  \underline{B}^{(\rm e)},\underline{S}^{(\rm e)}),\ \ i=1,2.
\end{split}
\end{eqnarray*}

{\color{black}Therefore, we can derive the following linearized elliptic problem for the problem \eqref{eq:4.2} in $\tilde{\Omega}^{(\rm e)}$:}
\begin{eqnarray}\label{eq:4.7}
\left\{
\begin{array}{llll}
\sum_{i,j=1,2}\partial_{i}\big(a^{(\delta\varphi^{(\rm n)})}_{ij}\partial_{j}\delta\hat{\varphi}^{(\rm n)}\big)
=\sum_{i=1,2}\partial_{i}b_{i}, &\ \ \ \mbox{in}\ \ \  \tilde{\Omega}^{(\rm e)}, \\[5pt]
\delta\hat{\varphi}^{(\rm n)}=g^{(\rm n)}_{0}(\eta),
&\ \ \  \mbox{on}\ \ \ \tilde{\Gamma}^{(\rm e)}_{\rm in},\\[5pt]
\partial_{1}\delta\hat{\varphi}^{(\rm n)}=\tilde{\omega}_{\rm e},  &\ \  \ \mbox{on}\ \ \ \tilde{\Gamma}^{(\rm e)}_{\rm ex},\\[5pt]
\delta\hat{\varphi}^{(\rm n)}=g^{(\rm n)}_{+}(\xi), &\ \ \  \mbox{on}\ \ \ \tilde{\Gamma}_{+},\\[5pt]
\delta\hat{\varphi}^{(\rm n)}=\int^{\xi}_{0}\omega_{\rm cd}(\nu)d\nu, &\ \  \  \mbox{on}\ \ \ \tilde{\Gamma}_{\rm cd},
\end{array}
\right.
\end{eqnarray}
where $g^{(\rm n)}_{0}$ and $g^{(\rm n)}_{+}$ are given by \eqref{eq:3.73a}.

Define the map $\hat{\mathcal{T}}:\ \bf\Xi_{\varepsilon_{\rm e}}\longrightarrow \bf\Xi_{\varepsilon_{\rm e}}$ as follows.
For a given function ${\varphi}^{(\rm n)}\in \bf\Xi_{\varepsilon_{\rm e}}$, we solve the linearized
fixed boundary value problem \eqref{eq:4.7}. Denote the solution by $\hat{\varphi}^{(\rm n)}$. Then
\begin{eqnarray}\label{eq:4.15}
\delta\hat{\varphi}^{(n)}:=\hat{\mathcal{T}}(\delta{\varphi}^{(\rm n)}), \qquad \mbox{for} \quad \delta\hat{\varphi}^{(\rm n)}:=\hat{\varphi}^{(\rm n)}-\underline{\varphi}.
\end{eqnarray}

\par \emph{{Step\ 2.}\ \ Solve $\hat{\varphi}^{(n)}$ by the fixed point argument.}
%
For the coefficients $a^{(\delta\varphi^{(\rm n)})}_{ij} (i,j=1,2)$, it is easy to see that
\begin{eqnarray}\label{eq:4.12}
\begin{split}
\big\|a^{(\delta\varphi^{(\rm n)})}_{ij}-a^{(\delta\underline{\varphi})}_{ij}\big\|^{(-\alpha, \tilde{\Sigma}^{(\rm e)})}
_{1,\alpha; \tilde{\Omega}^{(\rm e)}}\leq \mathcal{C} \varepsilon_{\rm e},
\end{split}
\end{eqnarray}
where $\mathcal{C}$ is a positive constant depending only on $\underline{U}^{(\rm e)}$ and $\alpha$.
Moreover, we have that
\begin{eqnarray} \label{eq:4.8a}
\begin{split}
e_{1}:= a^{(\delta\underline{\varphi})}_{11}=\underline{u}^{(e)}>0,\  \
e_{2}:= a^{(\delta\underline{\varphi})}_{22}=\frac{\underline{\rho}^{(e)}(\underline{u}^{(e)})^{3}(\underline{c}^{(e)})^{2}}
{(\underline{c}^{(e)})^{2}-(\underline{u}^{(e)})^{2}}>0,
\end{split}
\end{eqnarray}
and
$a^{(\delta\underline{\varphi})}_{12}=a^{(\delta\underline{\varphi})}_{21}=0$,
which yields
\begin{eqnarray}\label{eq:4.9}
\begin{split}
a^{(\delta\underline{\varphi})}_{11}a^{(\delta\underline{\varphi})}_{22}-a^{(\delta\underline{\varphi})}_{12}a^{(\delta\underline{\varphi})}_{21}
=\frac{\underline{\rho}^{(e)}(\underline{u}^{(e)})^{4}(\underline{c}^{(e)})^{2}}
{(\underline{c}^{(e)})^{2}-(\underline{u}^{(e)})^{2}}>0.
\end{split}
\end{eqnarray}

If $\varepsilon_{\rm e}>0$ is sufficiently small, there exist constants $\check{\lambda},\ \hat{\lambda}>0$
depending only on $\underline{U}^{(e)}$, such that for any $\varphi^{(\rm n)}\in \bf\Xi_{\varepsilon_{\rm e}}$,
\begin{eqnarray} \label{eq:4.11}
\begin{split}
\check{\lambda}|\mathbf{\vartheta}|\leq \sum_{i,j=1,2}a^{(\delta\varphi^{(\rm n)})}_{ij}\vartheta_{i}\vartheta_{j}
\leq \hat{\lambda} |\mathbf{\vartheta}|, \ \   \ \mathbf{\vartheta}=(\vartheta_{1}, \vartheta_{2}).
\end{split}
\end{eqnarray}

From \eqref{eq:3.71}, at the background solution the coefficients $\beta^{(\rm n)}_{0, 1}$ and $\beta^{(\rm n)}_{0, 2}$ are
\begin{eqnarray}\label{eq:4.13}
\begin{split}
&\left.\underline{\beta}_{0, 1}=\beta^{(\rm n)}_{0,1}\right.\Big|_{\varphi^{(\rm n)}=\underline{\varphi},
\tilde{B}^{(\rm e)}_{0}=\underline{B}^{(\rm e)}, \tilde{S}^{(\rm e)}_{0}=\underline{S}^{(\rm e)}}=0,\\[5pt]
&\left.\underline{\beta}_{0, 2}=\beta^{(\rm n)}_{0,2}\right.\Big|_{\varphi^{(\rm n)}=\underline{\varphi},
\tilde{B}^{(\rm e)}_{0}=\underline{B}^{(\rm e)}, \tilde{S}^{(\rm e)}_{0}=\underline{S}^{(\rm e)}}
=\frac{(\underline{\rho}^{(e)}\underline{c}^{(e)})^{2}(\underline{u}^{(e)})^{3}}{(\underline{c}^{(e)})^{2}-(\underline{u}^{(e)})^{2}}>0.
\end{split}
\end{eqnarray}

{\color{black}Notice that for sufficiently small constant $\sigma>0$ defined by \eqref{eq:3.62}, one has $|\eta|\leq |m^{(\rm e)}|<\mathcal{C}$ for $(0,\eta)\in \tilde{\Gamma}^{(\rm e)}_{\rm in}$ where the constant
$\mathcal{C}>0$ depends only on $\underline{U}^{(\rm e)}$}. Then, we can deduce that
\begin{eqnarray}\label{4.7e}
\begin{split}
&\big\|g^{(\rm n)}_{0}\big\|^{{\color{black}(-1-\alpha, \{\mathcal{O},\mathcal{Q}_{\rm I} \})}}_{2,\alpha;\tilde{\Gamma}^{(\rm e)}_{\rm in}}\\[5pt]
&=\frac{1}{\underline{\beta}_{0,2}}
\bigg\|\int^{\eta}_{0}\Big(\tilde{p}^{(\rm e)}_{0}(\mu)
-\underline{\tilde{p}}^{(\rm e)}_{0}(\mu)-\beta^{\rm (n)}_{0,1}
\partial_{1}\delta\varphi^{(\rm n)}-(\beta^{\rm (n)}_{0,2}-\underline{\beta}_{0,2})
\partial_{2}\delta\varphi^{(\rm n)}\Big)d\mu\bigg\|^{{\color{black}(-1-\alpha, \{\mathcal{O},\mathcal{Q}_{\rm I} \})}}_{2,\alpha;\tilde{\Gamma}^{(\rm e)}_{\rm in}}\\[5pt]
&\leq \mathcal{C}\bigg({\big\|\delta \tilde{p}^{(\rm e)}_{\rm 0}\big\|_{1,\alpha;\tilde{\Gamma}^{(\rm e)}_{\rm in}}}
+\big\|\delta\tilde{B}^{(\rm e)}_{0}\big\|_{1,\alpha;\tilde{\Gamma}^{(\rm e)}_{\rm in}}
+\big\|\delta\tilde{S}^{(\rm e)}_{0}\big\|_{1,\alpha;\tilde{\Gamma}^{(\rm e)}_{\rm in}}\\[5pt]
&\ \ \ + \Big( {\big\|\delta\varphi^{(\rm n)}\big\|^{(-1-\alpha,
\tilde{\Sigma}^{(\rm e)})}_{2,\alpha;\tilde{\Omega}^{(\rm e)}}}
+\big\|\delta\tilde{B}^{(\rm e)}_{0}\big\|_{1,\alpha;\tilde{\Gamma}^{(\rm e)}_{\rm in}}
+\big\|\delta\tilde{S}^{(\rm e)}_{0}\big\|_{1,\alpha;\tilde{\Gamma}^{(\rm e)}_{\rm in}}\Big)
{\big\|\delta\varphi^{(\rm n)}\big\|^{(-1-\alpha,
\tilde{\Sigma}^{(\rm e)})}_{2,\alpha;\tilde{\Omega}^{(\rm e)}}}\bigg),
\end{split}
\end{eqnarray}
\begin{eqnarray}\label{4.7x}
\begin{split}
&\big\|g^{(\rm n)}_{+}\big\|^{{\color{black}(-1-\alpha, \{\mathcal{Q}_{\rm I},\mathcal{Q}_{\rm e} \})}}_{2,\alpha;\tilde{\Gamma}_{+}}\\[5pt]
&=\big\|g_{+}(\xi)-g_{+}(0)+g^{(\rm n)}_{0}(m^{(\rm e)})\big\|^{{\color{black}(-1-\alpha, \{\mathcal{Q}_{\rm I},\mathcal{Q}_{\rm e} \})}}_{2,\alpha;\tilde{\Gamma}_{+}}\\[5pt]
& \leq \bigg\|\frac{1}{\underline{\beta}_{0,2}}\int^{m^{(\rm e)}}_{0}\Big(\tilde{p}^{(\rm e)}_{0}(\mu)
-\underline{\tilde{p}}^{(\rm e)}_{0}(\mu)-\beta^{\rm (n)}_{0,1}
\partial_{1}\delta\varphi^{(\rm n)}-(\beta^{\rm (n)}_{0,2}-\underline{\beta}_{0,2})
\partial_{2}\delta\varphi^{(\rm n)}\Big)d\mu\bigg\|^{{\color{black}(-1-\alpha, \{\mathcal{Q}_{\rm I},\mathcal{Q}_{\rm e} \})}}_{2,\alpha;\tilde{\Gamma}^{(\rm e)}_{\rm in}}\\[5pt]
&\ \ \ + \big\|g_{+}(\xi)-g_{+}(0)\big\|^{{\color{black}(-1-\alpha, \{\mathcal{Q}_{\rm I},\mathcal{Q}_{\rm e} \})}}_{2,\alpha;\tilde{\Gamma}_{+}}\\[5pt]
&\leq \mathcal{C}\bigg(\big\|g_{+}-1\big\|_{2,\alpha;\tilde{ \Gamma}_{+}}+{\big\|\delta \tilde{p}^{(\rm e)}_{\rm 0}\big\|_{1,\alpha;\tilde{\Gamma}^{(\rm e)}_{\rm in}}}
+\big\|\delta\tilde{B}^{(\rm e)}_{0}\big\|_{1,\alpha;\tilde{\Gamma}^{(\rm e)}_{\rm in}}
+\big\|\delta\tilde{S}^{(\rm e)}_{0}\big\|_{1,\alpha;\tilde{\Gamma}^{(\rm e)}_{\rm in}}\\[5pt]
&\ \ \ + \Big( {\big\|\delta\varphi^{(\rm n)}\big\|^{(-1-\alpha,
\tilde{\Sigma}^{(\rm e)})}_{2,\alpha;\tilde{\Omega}^{(\rm e)}}}
+\big\|\delta\tilde{B}^{(\rm e)}_{0}\big\|_{1,\alpha;\tilde{\Gamma}^{(\rm e)}_{\rm in}}
+\big\|\delta\tilde{S}^{(\rm e)}_{0}\big\|_{1,\alpha;\tilde{\Gamma}^{(\rm e)}_{\rm in}}\Big)
{\big\|\delta\varphi^{(\rm n)}\big\|^{(-1-\alpha,
\tilde{\Sigma}^{(\rm e)})}_{2,\alpha;\tilde{\Omega}^{(\rm e)}}}\bigg),
\end{split}
\end{eqnarray}
and
\begin{eqnarray}\label{4.7f}
\begin{split}
\big\|\delta\hat{\varphi}\big\|^{{\color{black}(-1-\alpha,\{\mathcal{O}, \mathcal{P}_{\rm e}\})}}_{2,\alpha;{\color{black}\tilde{\Gamma}_{\rm cd}}}
 =\bigg\|\int^{\xi}_{0}\omega_{\rm cd}(\nu)d \nu\bigg\|^{{\color{black}(-1-\alpha,\{\mathcal{O}, \mathcal{P}_{\rm e}\})}}_{2,\alpha;{\color{black}\tilde{\Gamma}_{\rm cd}}}
\leq \mathcal{C}\big\|\omega_{\rm cd}\big\|^{{\color{black}(-\alpha, \{\mathcal{P}_{\rm e}\})}}_{1,\alpha; \tilde{\Gamma}_{\rm cd}},
\end{split}
\end{eqnarray}
where $\mathcal{C}>0$ depends only on $\underline{U}^{(\rm e)}$, $L$ and $\alpha$, {\color{black}and $\mathcal{Q}_{\rm I}=(0, m^{(\rm e)})$}.

Now we can apply Theorem 2  in \cite{lg1} and  Theorem 3.2 in \cite{lg2} as well as Lemma 6.29 in \cite{gt}
to conclude that the problem \eqref{eq:4.7} admits a
unique solution $\hat{\varphi}^{(\rm n)}\in {C^{2,\alpha}_{(-1-\alpha, \tilde{\Sigma}^{(\rm e)})}}\big(\tilde{\Omega}^{(\rm e)}\big)$ which satisfies
\begin{eqnarray}\label{eq:4.14}
\begin{split}
&\quad \ \big\|\delta\hat{\varphi}^{(\rm n)}\big\|^{(-1-\alpha,\tilde{\Sigma}^{(\rm e)})}_{2,\alpha;\tilde{\Omega}^{(\rm e)}}\\[5pt]
&\leq \mathcal{C}\bigg(\big\|g^{(\rm n)}_{0}\big\|^{{\color{black}(-1-\alpha, \{\mathcal{O}, \mathcal{Q}_{\rm I}\})}}_{2,\alpha;\tilde{\Gamma}^{(\rm e)}_{\rm in}}
+\big\|g^{(\rm n)}_{+}\big\|^{{\color{black}(-1-\alpha, \{\mathcal{Q}_{\rm I},\mathcal{Q}_{\rm e} \})}}_{2,\alpha;\tilde{\Gamma}_{+}}
+\big\|\tilde{\omega}_{\rm e}\big\|^{{\color{black}(-1-\alpha, \{\mathcal{P}_{\rm e}, \mathcal{Q}_{\rm e}\})}}_{2,\alpha;\tilde{\Gamma}^{(\rm e)}_{\rm ex}}\\[5pt]
&\qquad\qquad \ +\sum_{i=1,2}\big\|b_{i}\big\|^{(-\alpha, \tilde{\Sigma}^{(\rm e)})}_{1,\alpha; \tilde{\Omega}^{(\rm e)}}
+\big\|\delta\hat{\varphi}\big\|^{{\color{black}(-1-\alpha,\{\mathcal{O}, \mathcal{P}_{\rm e}\})}}_{2,\alpha;\tilde{\Gamma}_{\rm cd}}\bigg)\\[5pt]
&\leq \mathcal{C}\Big( {\big\|\delta\varphi^{(\rm n)}\big\|^{(-1-\alpha,
\tilde{\Sigma}^{(\rm e)})}_{2,\alpha;\tilde{\Omega}^{(\rm e)}}}
+\big\|\delta\tilde{B}^{(\rm e)}_{0}\big\|_{1,\alpha;\tilde{\Gamma}^{(\rm e)}_{\rm in}}
+\big\|\delta\tilde{S}^{(\rm e)}_{0}\big\|_{1,\alpha;\tilde{\Gamma}^{(\rm e)}_{\rm in}}\Big)\big\|\delta\varphi^{(\rm n)}\big\|^{(-1-\alpha,
\tilde{\Sigma}^{(\rm e)})}_{2,\alpha;\tilde{\Omega}^{(\rm e)}}\\[5pt]
&\ \ \  \ +\mathcal{C}\bigg({\big\|\delta\tilde{p}^{(\rm e)}_{0}\big\|_{1,\alpha;\tilde{\Gamma}^{(\rm e)}_{\rm in}}}
+\big\|\delta\tilde{B}^{(\rm e)}_{0}\big\|_{1,\alpha;\tilde{\Gamma}^{(\rm e)}_{\rm in}}
+\big\|\delta\tilde{S}^{(\rm e)}_{0}\big\|_{1,\alpha;\tilde{\Gamma}^{(\rm e)}_{\rm in}}
+\big\|\tilde{\omega}_{\rm e}\big\|^{{\color{black}(-1-\alpha, \{\mathcal{P}_{\rm e}, \mathcal{Q}_{\rm e}\})}}_{2,\alpha;\tilde{\Gamma}^{(\rm e)}_{\rm ex}}\\[5pt]
&\ \ \ \ \ \ \  \  \ \ \  \ \ \ +\big\|g_{+}-1\big\|_{2,\alpha;\tilde{ \Gamma}_{+}}+{\big\|\omega_{\rm cd}\big\|^{{\color{black}(-\alpha, \{\mathcal{P}_{\rm e}\})}}_{1,\alpha; \tilde{\Gamma}_{\rm cd}}}\bigg),
\end{split}
\end{eqnarray}
where $\mathcal{C}>0$ depends only on $\underline{U}^{(\rm e)}$, $L$ and $\alpha$.

Because ${\varphi}^{(\rm n)} \in \bf\Xi_{\varepsilon_{\rm e}}$, it follows from \emph{{Step\ 1}} that $\hat{\varphi}^{(\rm n)}$ is uniquely solved, then $\hat{\mathcal{T}}$ is well-defined.
Moreover, we have
\begin{eqnarray}\label{eq:4.16b}
\begin{split}
{\color{black}\big\|\delta\hat{\varphi}^{(\rm n)}\big\|^{(-1-\alpha,\tilde{\Sigma}^{(\rm e)})}_{2,\alpha;\tilde{\Omega}^{(\rm e)}}
\leq \mathcal{C}\varepsilon^2_{\rm e}+\mathcal{C}\tilde{\epsilon}\varepsilon_{\rm e}+\mathcal{C}(\tilde{\epsilon}+\varrho).}
\end{split}
\end{eqnarray}
{\color{black}Choosing $\varepsilon'_{\rm e, 0}=\min\{\frac{1}{4\mathcal{C}},\frac{1}{2(\mathcal{C}+1)}, 1 \}$, $\mathcal{C}'_{\rm e,0}=4\mathcal{C}$ and taking $\tilde{\epsilon}>0$ sufficiently small with
$\tilde{\epsilon}<\varepsilon_{\rm e}\in (\mathcal{C}'_{\rm e,0}\tilde{\epsilon}, \varepsilon'_{\rm e, 0})$,
 we see by \eqref{eq:4.16b} that $\hat{\varphi}^{(\rm n)}\in \bf\Xi_{\varepsilon_{\rm e}}$ for $\varrho\in (0, \mathcal{C'}_{\rm e,0}^{-1}\varepsilon_{\rm e})$, 
which means that $\hat{\mathcal{T}}$ map $\bf\Xi_{\varepsilon_{\rm e}}$ to $\bf\Xi_{\varepsilon_{\rm e}}$ itself.}

Next, we will show that $\hat{\mathcal{T}}$ is a contraction map in $\bf\Xi_{\varepsilon_{\rm e}}$. 
Choose $\varphi^{(\rm n)}_{1},\varphi^{(\rm n)}_{2} \in \bf\Xi_{\varepsilon_{\rm e}}$.
Set $\delta\hat{\varphi}^{(\rm n)}_{k}:=\hat{\mathcal{T}}(\delta{\varphi}^{(\rm n)}_{k}),\ k=1,2$.
Then, 
we have
\begin{eqnarray} \label{eq:4.17x}
\begin{split}
\sum_{i,j=1,2}\partial_{i}\big(a^{(\delta\varphi^{(\rm n)}_{k})}_{ij}\partial_{j}\delta\hat{\varphi}^{(\rm n)}_{k}\big)
=\sum_{i=1,2}\partial_{i}b_{i} , \qquad \mbox{for} \quad k=1,2.
\end{split}
\end{eqnarray}
Define
\[\delta\Phi^{(\rm n)}=\delta\varphi^{(\rm n)}_{1}-\delta\varphi^{(\rm n)}_{2}, \quad
\delta \hat{\Phi}^{(\rm n)}=\delta\hat{\varphi}^{(\rm n)}_{1}-\delta\hat{\varphi}^{(\rm n)}_{2}.
\]
Then, $\delta \hat{\Phi}$ satisfies the following boundary value problem:
\begin{eqnarray}\label{eq:4.17}
\left\{
\begin{array}{llll}
&\sum_{i,j=1,2}\partial_{i}\big(a^{(\delta\varphi^{(\rm n)}_{1})}_{ij}\partial_{j}\delta \hat{\Phi}^{(\rm n)}\big)\\[5pt]
&\qquad\qquad\  =-\sum_{i,j=1,2}\partial_{i}\Big(\big(a^{(\delta\varphi^{(\rm n)}_{1})}_{ij}-a^{(\delta\varphi^{(\rm n)}_{2})}_{ij}\big)
\partial_{j}\delta\hat{\varphi}^{(\rm n )}_{2}\Big),
&\ \ \ \mbox{in} \ \ \  \tilde{\Omega}^{(\rm e)}, \\[5pt]
&\delta\hat{\Phi}^{(\rm n)}=g^{(\rm n)}_{0,1}(\eta)-g^{(\rm n)}_{0,2}(\eta),
&\ \ \  \mbox{on} \ \ \ \tilde{\Gamma}^{(\rm e)}_{\rm in},\\[5pt]
&\partial_{1}\delta\hat{\Phi}^{(\rm n)}=0,  &\ \ \  \mbox{on} \ \ \ \tilde{\Gamma}^{(\rm e)}_{\rm ex},\\[5pt]
&\delta\hat{\Phi}^{(\rm n)}=g^{(\rm n)}_{+,1}(\xi)-g^{(\rm n)}_{+,2}(\xi), &\ \ \  \mbox{on} \ \ \ \tilde{\Gamma}_{+},\\[5pt]
&\delta\hat{\Phi}^{(\rm n)}=0, &\ \ \  \mbox{on} \ \ \ \tilde{\Gamma}_{\rm cd},
\end{array}
\right.
\end{eqnarray}
where $g^{(\rm n)}_{0,1}(\eta)$, $g^{(\rm n)}_{0,2}(\eta)$ represent the right hand side of \eqref{eq:3.73a} by replacing $\varphi^{(\rm n)}$
with $\varphi^{(\rm n)}_{1}$ and $\varphi^{(\rm n)}_{2}$, respectively. Notice that
\begin{eqnarray*}
\begin{split}
&g^{(\rm n)}_{0,1}(\eta)-g^{(\rm n)}_{0,2}(\eta)\\[5pt]
&=\frac{1}{\underline{\beta}_{0,2}}\int^{\eta}_{0}
\bigg(\beta^{\rm (n)}_{0,1}|_{\varphi^{(\rm n)}=\varphi^{(\rm n)}_{1}}
\partial_{1}\delta\varphi^{(\rm n)}_{1}+\Big(\beta^{\rm (n)}_{0,2}|_{\varphi^{(\rm n)}=\varphi^{(\rm n)}_{1}}-\underline{\beta}_{0,2}\Big)
\partial_{2}\delta\varphi^{(\rm n)}_{1}
-\beta^{\rm (n)}_{0,1}|_{\varphi^{(\rm n)}=\varphi^{(\rm n)}_{2}}\partial_{1}\delta\varphi^{(\rm n)}_{2}\\[5pt]
&\ \ \ \ \ \ \ \ \ \ \   -\Big(\beta^{\rm (n)}_{0,2}|_{\varphi^{(\rm n)}=\varphi^{(\rm n)}_{2}}-\underline{\beta}_{0,2}\Big)
\partial_{2}\delta\varphi^{(\rm n)}_{2}\bigg)d\mu\\[5pt]
&=\frac{1}{\underline{\beta}_{0,2}}\Bigg(\sum_{j=1,2}\int^{\eta}_{0}\Big(\beta^{\rm (n)}_{0,j}|_{\varphi^{(\rm n)}=\varphi^{(\rm n)}_{1}}-\underline{\beta}_{0,j}\Big)\partial_{j}\delta \Phi^{\rm (n)}d\mu
+\int^{\eta}_{0}\partial_{j}\delta \varphi^{(\rm n)}_{2}\Big(\beta^{\rm (n)}_{0,j}|_{\varphi^{(\rm n)}=\varphi^{(\rm n)}_{1}}
-\beta^{\rm (n)}_{0,j}|_{\varphi^{(\rm n)}=\varphi^{(\rm n)}_{2}}\Big)d\mu\Bigg).
\end{split}
\end{eqnarray*}
Then, {\color{black}similarly to \eqref{4.7e}-\eqref{4.7f}}, 
we get that,  for $\sigma>0$  sufficiently small,
\begin{eqnarray*}
&&\big\|g^{(\rm n)}_{0,1}-g^{(\rm n)}_{0,2}\big\|^{{\color{black}(-1-\alpha, \{\mathcal{O}, \mathcal{Q}_{\rm I}\})}}_{2,\alpha;\tilde{\Gamma}^{(\rm e)}_{\rm in}}\\[5pt]
&&\ \ \ \leq \frac{1}{\underline{\beta}_{0,2}}\bigg\|\sum_{j=1,2}\int^{\eta}_{0}\Big(\beta^{\rm (n)}_{0,j}|_{\varphi^{(\rm n)}=\varphi^{(\rm n)}_{1}}-\underline{\beta}_{0,j}\Big)\partial_{j}\delta \Phi^{\rm (n)}d\mu\bigg\|^{{\color{black}(-1-\alpha, \{\mathcal{O}, \mathcal{Q}_{\rm I}\})}}_{2,\alpha;\tilde{\Gamma}^{(\rm e)}_{\rm in}}\\[5pt]
&& \ \ \ \ \ \ \ \ \ +\frac{1}{\underline{\beta}_{0,2}}\bigg\|\sum_{j=1,2}\int^{\eta}_{0}\partial_{j}\delta \varphi^{(\rm n)}_{2}\Big(\beta^{\rm (n)}_{0,j}|_{\varphi^{(\rm n)}=\varphi^{(\rm n)}_{1}}
-\beta^{\rm (n)}_{0,j}|_{\varphi^{(\rm n)}=\varphi^{(\rm n)}_{2}}\bigg)d\mu\Big\|^{{\color{black}(-1-\alpha, \{\mathcal{O}, \mathcal{Q}_{\rm I}\})}}_{2,\alpha;\tilde{\Gamma}^{(\rm e)}_{\rm in}}\\[5pt]
&&\ \ \ \leq {\color{black}\mathcal{C}\big(\varepsilon_{\rm e}+\tilde{\epsilon}\big)}
\big\|\delta\Phi^{(\rm n)}\big\|^{(-1-\alpha,
\tilde{\Sigma}^{(\rm e)})}_{2,\alpha;\tilde{\Omega}^{(\rm e)}},
\end{eqnarray*}
where the constant $\mathcal{C}>0$ depends only on $\underline{U}^{(\rm e)}$, $L$ and $\alpha$.
In the same way, for $g^{(\rm n)}_{+,1}(\xi)-g^{(\rm n)}_{+,2}(\xi)$ one also has
\begin{eqnarray*}
\big\|g^{(\rm n)}_{+,1}-g^{(\rm n)}_{+,2}\big\|^{{\color{black}(-1-\alpha, \{\mathcal{Q}_{\rm I}, \mathcal{Q}_{\rm e}\})}}_{2,\alpha;{\color{black}\tilde{\Gamma}_{+}}}
 \leq {\color{black}\mathcal{C}\big(\varepsilon_{\rm e}+\tilde{\epsilon}\big)}\big\|\delta\Phi^{(\rm n)}\big\|^{(-1-\alpha,
\tilde{\Sigma}^{(\rm e)})}_{2,\alpha;\tilde{\Omega}^{(\rm e)}},
\end{eqnarray*}
provided that $\sigma>0$ is sufficiently small, where the constant $\mathcal{C}>0$ depends only on $\underline{U}^{(\rm e)}$, $L$ and $\alpha$.

Thus, it follows from the estimate \eqref{eq:4.14} that
\begin{eqnarray*}
\begin{split}
{\big\|\delta\hat{\Phi}^{(\rm n)}\big\|^{(-1-\alpha,
\tilde{\Sigma}^{(\rm e)})}_{2,\alpha;\tilde{\Omega}^{(\rm e)}}}
&\leq \mathcal{C} \bigg(\sum_{i,j=1,2}\big\|\big(a^{(\delta\varphi^{(\rm n)}_{1})}_{ij}-a^{(\delta\varphi^{(\rm n)}_{2})}_{ij}\big)
\partial_{j}\delta\hat{\varphi}^{(\rm n)}_{2}\big\|
^{{\color{black}(-\alpha; \tilde{\Sigma}^{(\rm e)})}}
_{1,\alpha;\tilde{\Omega}^{(\rm e)}}\bigg)\\[5pt]
&\ \ \  +\mathcal{C}\bigg(\big\|g^{(\rm n)}_{0,1}-g^{(\rm n)}_{0,2}\big\|^{{\color{black}(-1-\alpha, \{\mathcal{O}, \mathcal{Q}_{\rm I}\})}}_{2,\alpha;{\color{black}\tilde{\Gamma}^{(\rm e)}_{\rm in}}}
+\big\|g^{(\rm n)}_{+,1}-g^{(\rm n)}_{+,2}\big\|^{{\color{black}(-1-\alpha, \{\mathcal{Q}_{\rm I}, \mathcal{Q}_{\rm e}\})}}_{2,\alpha;{\color{black}\tilde{\Gamma}_{+}}}\bigg)\\[5pt]
& \leq {\color{black}\mathcal{C}\big(\varepsilon_{\rm e}+\tilde{\epsilon}\big)}\big\|\delta\Phi^{(\rm n)}\big\|^{(-1-\alpha,
\tilde{\Sigma}^{(\rm e)})}_{2,\alpha;\tilde{\Omega}^{(\rm e)}},
\end{split}
\end{eqnarray*}
where 
$\mathcal{C}>0$ depends only on $\underline{U}^{(\rm e)}$, $L$ and $\alpha$.

Therefore, we show that $\mathcal{T}$ is a contraction map by letting {\color{black}$\tilde{\epsilon}>0$} sufficiently small so that
${\color{black}\tilde{\epsilon}}<\varepsilon_{\rm e}\in (\mathcal{C}''_{\rm e,0}\tilde{\epsilon}, \varepsilon''_{\rm e, 0})$ with
$\mathcal{C}''_{\rm e,0}=1$ and $\varepsilon''_{\rm e, 0}= \frac{1}{4\mathcal{C}}$.
As a result,  the existence of solutions to the problem \eqref{eq:4.2} and the estimate \eqref{eq:4.3} are established by \eqref{eq:4.14}
for $\varepsilon_{\rm e}\in (\mathcal{C}''_{\rm e,0}\tilde{\epsilon}, \varepsilon''_{\rm e, 0})$.

For the uniqueness, consider two solutions $\varphi^{(\rm n)}_{1},\ \varphi^{(\rm n)}_{2}\in \bf\Xi_{\varepsilon_{\rm e}}$
for any 
$\varepsilon_{\rm e}\in (\mathcal{C}''_{\rm e,0}\tilde{\epsilon}, \varepsilon''_{\rm e, 0})$ and fixed $\rm n$. 
Notice that $\hat{\mathcal{T}}$ is a contraction map, then  we have
\begin{eqnarray*}
\begin{split}
{\big\|\varphi^{(\rm n)}_{1}-\varphi^{(\rm n)}_{2}\big\|^{(-1-\alpha,
\tilde{\Sigma}^{(\rm e)})}_{2,\alpha;\tilde{\Omega}^{(\rm e)}}}
\leq\frac{1}{2}{\big\|\varphi^{(\rm n)}_{1}-\varphi^{(\rm n)}_{2}\big\|^{(-1-\alpha,
\tilde{\Sigma}^{(\rm e)})}_{2,\alpha;\tilde{\Omega}^{(\rm e)}}},
\end{split}
\end{eqnarray*}
which implies that $\varphi^{(\rm n)}_{1}=\varphi^{(\rm n)}_{2}$.

{\color{black} Finally, we take  $\mathcal{C}_{\rm e,0}=\max\{\mathcal{C}'_{\rm e,0}, \mathcal{C}''_{\rm e,0}\}$ and $\varepsilon_{\rm e, 0}=\min\{\varepsilon'_{\rm e, 0}, \varepsilon''_{\rm e, 0}\}$,
then, for $\tilde{\epsilon}>0$ sufficiently small, if $\tilde{\epsilon}<\varepsilon_{\rm e}\in (\mathcal{C}_{\rm e,0}\tilde{\epsilon}, \varepsilon_{\rm e, 0})$ and $\varrho\in (0, \mathcal{C}^{-1}_{\rm e,0}\varepsilon_{\rm e, 0})$,
the problem \eqref{eq:4.2} admits a unique solution $\varphi^{(\rm n)}\in { C^{2,\alpha}_{(-1-\alpha; \tilde{\Sigma}^{(\rm e)})}}\big(\tilde{\Omega}^{(\rm e)}\big)$
satisfying $\big\|\delta \varphi^{(\rm n)}\big\|^{(-1-\alpha, \tilde{\Sigma}^{(\rm e)})}_{2, \alpha; \tilde{\Omega}^{(\rm e)}}\leq \varepsilon_{\rm e}$ and the estimate \eqref{eq:4.14}. }

\smallskip
\emph{{Step 3.}\ \  Higher order regularity of the solution near the corner point $\mathcal{O}$}.
We divide this step into two sub-steps.

(Step 3.1) {\it Introduce a new function and reformulate the problem \eqref{eq:4.2} near the corner point $\mathcal{O}$.} \quad
Set
\begin{eqnarray}\label{eq:4.19}
{\color{black}\psi^{(\rm n)}=\partial_{1}\varphi^{(\rm n)}.}
\end{eqnarray}
Then it is easy to see that
\begin{eqnarray}\label{eq:4.20}
\begin{split}
{\color{black}\partial^{2}_{11}\varphi^{(\rm n)}=\partial_{1}\psi^{(\rm n)},\quad
\partial^{2}_{12}\varphi^{(\rm n)}=\partial_{2}\psi^{(\rm n)}.}
\end{split}
\end{eqnarray}
By the equation $\eqref{eq:4.2}_{1}$, we have
\begin{eqnarray}\label{eq:4.21}
\begin{split}
\sum_{i,j=1, 2}\mathcal{N}^{(\rm n)}_{ij}\partial^{2}_{ij}\varphi^{(\rm n)}=
-\partial_{\tilde{B}^{(\rm e)}_{0}}\mathcal{N}_{2}\cdot(\tilde{B}^{(e)}_{0})'
-\partial_{\tilde{S}^{(\rm e)}_{0}}\mathcal{N}_{2}\cdot(\tilde{S}^{(\rm e)}_{0})',
\end{split}
\end{eqnarray}
which gives
\begin{eqnarray}\label{eq:4.21a}
\begin{split}
\partial^{2}_{22}\varphi^{(\rm n)}
& =-\frac{\mathcal{N}^{(\rm n)}_{11}}{\mathcal{N}^{(\rm n)}_{22}}\partial^{2}_{11}\varphi^{(\rm n)}
-\frac{2\mathcal{N}^{(\rm n)}_{12}}{\mathcal{N}^{(\rm n)}_{22}}\partial^{2}_{12}\varphi^{(\rm n)}
-\frac{\partial_{\tilde{B}^{(\rm e)}_{0}}\mathcal{N}_{2}\cdot(\tilde{B}^{(e)}_{0})'
+\partial_{\tilde{S}^{(\rm e)}_{0}}\mathcal{N}_{2}\cdot(\tilde{S}^{(\rm e)}_{0})'}{\mathcal{N}^{(\rm n)}_{22}}\\[5pt]
& =-\frac{\mathcal{N}^{(\rm n)}_{11}}{\mathcal{N}^{(\rm n)}_{22}}\partial_{1}\psi^{(\rm n)}
-\frac{2\mathcal{N}^{(\rm n)}_{12}}{\mathcal{N}^{(\rm n)}_{22}}\partial_{2}\psi^{(\rm n)}
-{\color{black}\frac{\partial_{\tilde{B}^{(\rm e)}_{0}}\mathcal{N}_{2}\cdot(\tilde{B}^{(e)}_{0})'
+\partial_{\tilde{S}^{(\rm e)}_{0}}\mathcal{N}_{2}\cdot(\tilde{S}^{(\rm e)}_{0})'}{\mathcal{N}^{(\rm n)}_{22}},}
\end{split}
\end{eqnarray}
where
\begin{eqnarray}\label{eq:4.22}
\begin{split}
&\mathcal{N}^{(\rm n)}_{11}=\partial_{\partial_{1}\varphi^{(\rm n)}}\mathcal{N}_{1}\big(D\varphi^{(\rm n)}; \tilde{B}^{(\rm e)}_{0},\tilde{S}^{(\rm e)}_{0}\big)
=\frac{(\tilde{\rho}^{(\rm e)}_{\rm n})^{2}(\tilde{c}^{(\rm e)}_{\rm n})^{2}(\partial_{2}\varphi^{(\rm n)})^{2}-1}{(\tilde{\rho}^{(\rm e)}_{\rm n})^{3} (\partial_{2}\varphi^{(\rm n)})^{3}
\Big((\tilde{c}^{(\rm e)}_{\rm n})^{2}-\frac{1+(\partial_{1}\varphi^{(\rm n)})^{2}}{(\tilde{\rho}^{(\rm e)}_{\rm n})^{2}(\partial_{2}\varphi^{(\rm n)})^{2}}\Big)},\\[5pt]
&\mathcal{N}^{(\rm n)}_{22}=\partial_{\partial_{2}\varphi^{(\rm n)}}\mathcal{N}_{2}\big(D\varphi^{(\rm n)}; \tilde{B}^{(\rm e)}_{0},\tilde{S}^{(\rm e)}_{0}\big)
=\frac{(\tilde{c}^{(\rm e)}_{\rm n})^{2}\big((\partial_{1}\varphi^{(\rm n)})^{2}+1\big)}{\tilde{\rho}^{(\rm e)}_{\rm n} (\partial_{2}\varphi^{(\rm n)})^{3}
\Big((\tilde{c}^{(\rm e)}_{\rm n})^{2}-\frac{1+(\partial_{1}\varphi^{(\rm n)})^{2}}{(\tilde{\rho}^{(\rm e)}_{\rm n})^{2}(\partial_{2}\varphi^{(\rm n)})^{2}}\Big)},
\end{split}
\end{eqnarray}
and
\begin{eqnarray}\label{eq:4.23}
\begin{split}
\mathcal{N}^{(\rm n)}_{12}=\mathcal{N}^{(\rm n)}_{21}
&=\partial_{\partial_{2}\varphi^{(\rm n)}}\mathcal{N}_{1}\big(D\varphi^{(\rm n)}; \tilde{B}^{(\rm e)}_{0},\tilde{S}^{(\rm e)}_{0}\big)
=-\frac{(\tilde{c}^{(\rm e)}_{\rm n})^{2}\partial_{1}\varphi^{(\rm n)}}{\tilde{\rho}^{(\rm e)}_{\rm n} (\partial_{2}\varphi^{(\rm n)})^{2}
\Big((\tilde{c}^{(\rm e)}_{\rm n})^{2}-\frac{1+(\partial_{1}\varphi^{(\rm n)})^{2}}{(\tilde{\rho}^{(\rm e)}_{\rm n})^{2}(\partial_{2}\varphi^{(\rm n)})^{2}}\Big)}.
\end{split}
\end{eqnarray}
Obviously, $\mathcal{N}^{(\rm n)}_{ij}$ satisfies \eqref{eq:4.12}--\eqref{eq:4.11}, and
\begin{eqnarray}\label{eq:4.24}
\begin{split}
\mathcal{N}^{(\rm n)}_{11}\mathcal{N}^{(\rm n)}_{22}-\mathcal{N}^{(\rm n)}_{12}\mathcal{N}^{(\rm n)}_{21}
&=\frac{(\tilde{c}^{(\rm e)}_{\rm n})^{2}}{(\tilde{\rho}^{(\rm e)}_{\rm n})^{2} (\partial_{2}\varphi^{(\rm n)})^{4}
\Big((\tilde{c}^{(\rm e)}_{\rm n})^{2}-\frac{1+(\partial_{1}\varphi^{(\rm n)})^{2}}{(\tilde{\rho}^{(\rm e)}_{\rm n})^{2}(\partial_{2}\varphi^{(\rm n)})^{2}}\Big)}\\[5pt]
&=\frac{(\tilde{c}^{(\rm e)}_{\rm n})^{2}(\tilde{\rho}^{(\rm e)}_{\rm n})^{2}(\tilde{u}^{(\rm e)}_{\rm n})^{2}}
{(\tilde{c}^{(\rm e)}_{\rm n})^{2}-(\tilde{q}^{(\rm e)}_{\rm n})^{2}}>0,
\end{split}
\end{eqnarray}
for any $\tilde{\rho}^{(\rm e)}_{\rm n} \tilde{u}^{(\rm e)}_{\rm n}> 0$ in $\tilde{\Omega}^{(\rm e)}$,
where $\tilde{\rho}^{(\rm e)}_{\rm n}, \tilde{c}^{(\rm e)}_{\rm n}, \tilde{u}^{(\rm e)}_{\rm n}$
and $\tilde{q}^{(\rm e)}_{\rm n}$ are functions of $D\varphi^{(\rm n)}, \tilde{B}^{(\rm e)}_{0}, \tilde{S}^{(\rm e)}_{0}$.
Moreover, we know that
\begin{eqnarray}\label{eq:4.24a}
\begin{split}
\mathcal{N}^{(\rm n)}_{11}\big|_{\varphi^{(\rm n)}=\underline{\varphi},\tilde{B}^{(\rm e)}_{0}=\underline{B}^{(\rm e)},
\tilde{S}^{(\rm e)}_{0}=\underline{S}^{(\rm e)}}>0, \ \ \
\mathcal{N}^{(\rm n)}_{22}\big|_{\varphi^{(\rm n)}=\underline{\varphi},\tilde{B}^{(\rm e)}_{0}=\underline{B}^{(\rm e)},
\tilde{S}^{(\rm e)}_{0}=\underline{S}^{(\rm e)}}>0,
\end{split}
\end{eqnarray}
and
\begin{eqnarray}\label{eq:4.24b}
\begin{split}
\mathcal{N}^{(\rm n)}_{12}\big|_{\varphi^{(\rm n)}=\underline{\varphi},\tilde{B}^{(\rm e)}_{0}=\underline{B}^{(\rm e)},
\tilde{S}^{(\rm e)}_{0}=\underline{S}^{(\rm e)}}=\mathcal{N}^{(\rm n)}_{21}\big|_{\varphi^{(\rm n)}=\underline{\varphi},\tilde{B}^{(\rm e)}_{0}=\underline{B}^{(\rm e)},
\tilde{S}^{(\rm e)}_{0}=\underline{S}^{(\rm e)}}=0.
\end{split}
\end{eqnarray}

\par In order to derive the equation for $\psi^{(\rm n)}$, we take $\partial_{1}$ on \eqref{eq:4.21}
and then substitute \eqref{eq:4.20} into the resulting equation to obtain
 {\color{black}\begin{eqnarray}\label{eq:4.25}
\begin{split}
\sum_{i,j=1, 2}\mathcal{N}^{(\rm n)}_{ij}\partial^{2}_{ij}\psi^{(\rm n)}+\sum_{k,j=1,2}b^{(\rm n)}_{kj}\partial_{k}\psi^{(\rm n)}\partial_{j}\psi^{(\rm n)}
=f^{(\rm n)},
 \end{split}
\end{eqnarray}}
where $\mathcal{N}^{(\rm n)}_{ij}\ (i,j=1,2)$ are given by \eqref{eq:4.22}--\eqref{eq:4.23} with   \eqref{eq:4.24}, and
{\color{black}
\begin{align*}
\begin{split}
&b^{(\rm n)}_{k1}=\partial_{\partial_{k}\varphi^{(\rm n)}}\mathcal{N}^{(\rm n)}_{11}
-\frac{\mathcal{N}^{(\rm n)}_{11}}{\mathcal{N}^{(\rm n)}_{22}}\partial_{\partial_{k}\varphi^{(\rm n)}}\mathcal{N}^{(\rm n)}_{22},\\[5pt]
&b^{(\rm n)}_{k2}=2\Big(\partial_{\partial_{k}\varphi^{(\rm n)}}\mathcal{N}^{(\rm n)}_{12}
-\frac{\mathcal{N}^{(\rm n)}_{12}}{\mathcal{N}^{(\rm n)}_{22}}\partial_{\partial_{k}\varphi^{(\rm n)}}\mathcal{N}^{(\rm n)}_{12}\Big),
\end{split}
\end{align*}
for $ k=1,2$ }and{\color{black}
\begin{align}\label{eq:4.25b}
\begin{split}
f^{(\rm n)}&=\sum_{i,j=1,2}\Big(\frac{\partial_{\partial_{k}\varphi^{(\rm n)}}\mathcal{N}^{(\rm n)}_{22}\partial_{\tilde{B}^{(\rm e)}_{0}}\mathcal{N}_{2}}{\mathcal{N}^{(\rm n)}_{22}}
-\partial^2_{\tilde{B}^{(\rm e)}_{0}\partial_{k}\varphi^{(\rm n)}}\mathcal{N}_{2}\Big)\partial^2_{k1}\varphi^{(\rm n)}\cdot(\tilde{B}^{(e)}_{0})'\\[5pt]
&\quad +\sum_{i,j=1,2}\Big(\frac{\partial_{\partial_{k}\varphi^{(\rm n)}}\mathcal{N}^{(\rm n)}_{22}\partial_{\tilde{S}^{(\rm e)}_{0}}\mathcal{N}_{2}}{\mathcal{N}^{(\rm n)}_{22}}
-\partial^2_{\tilde{S}^{(\rm e)}_{0}\partial_{k}\varphi^{(\rm n)}}\mathcal{N}_{2}\Big)\partial^2_{k1}\varphi^{(\rm n)}\cdot(\tilde{S}^{(e)}_{0})'.
\end{split}
\end{align}
}
Next, let us consider the boundary conditions of $\psi^{(\rm n)}$. Firstly, at the entrance $\tilde{\Gamma}^{(\rm e)}_{\rm in}$, {\color{black}we rewrite $\eqref{eq:4.2}_{2}$ as
\begin{eqnarray}\label{eq:4.27}
\begin{split}
\tilde{p}^{(\rm e)}(D\varphi^{(\rm n)},\tilde{B}^{(e)}_{0}, \tilde{S}^{(e)}_{0})=\tilde{p}^{(\rm e)}_{0}(\eta).
\end{split}
\end{eqnarray}
}
{\color{black}Then, taking derivative $\partial_{2}$ on \eqref{eq:4.27} and by \eqref{eq:4.19} and \eqref{eq:4.21a} to deduce that}
\begin{eqnarray}\label{eq:4.27b}
\begin{split}
\tilde{\beta}^{(\rm n)}_{0, 1}\partial_{1}\psi^{(\rm n)}
+\tilde{\beta}^{ (\rm n)}_{0,2}\partial_{2}\psi^{(\rm n)}=g^{(\rm n)},
&\ \ \  \mbox{on}\ \ \ {\color{black}\tilde{\Gamma}^{(\rm e)}_{\rm in}},
\end{split}
\end{eqnarray}
where
\begin{eqnarray}\label{eq:4.27a}
\begin{split}
\tilde{\beta}^{(\rm n)}_{0, 1}={\color{black}-\frac{\partial_{\partial_{2}\varphi^{(\rm n)}}\tilde{p}^{(\rm e)}\mathcal{N}^{(\rm n)}_{11}}{\mathcal{N}^{(\rm n)}_{22}}},
\qquad  \tilde{\beta}^{(\rm n)}_{0, 2}={\color{black}\partial_{\partial_{1}\varphi^{(\rm n)}}\tilde{p}^{(\rm e)}-\frac{2\partial_{\partial_{2}\varphi^{(\rm n)}}\tilde{p}^{(\rm e)}\mathcal{N}^{(\rm n)}_{12}}{\mathcal{N}^{(\rm n)}_{22}}},
\end{split}
\end{eqnarray}
and{\color{black}
\begin{eqnarray}\label{eq:4.27b2}
\begin{split}
g^{(\rm n)}&=(\tilde{p}^{(\rm e)}_{0})'+\Big(\frac{\partial_{\partial_{2}\varphi^{(\rm n)}}\tilde{p}^{(\rm e)}\partial_{\tilde{B}^{(\rm e)}_{0}}\mathcal{N}_{2}}{\mathcal{N}^{(\rm n)}_{22}}
-\partial_{\tilde{B}^{(\rm e)}_{0}}\tilde{p}^{(\rm e)}\Big)(\tilde{B}^{(\rm e)}_{0})'\\[5pt]
&\qquad+\Big(\frac{\partial_{\partial_{2}\varphi^{(\rm n)}}\tilde{p}^{(\rm e)}\partial_{\tilde{S}^{(\rm e)}_{0}}\mathcal{N}_{2}}{\mathcal{N}^{(\rm n)}_{22}}
-\partial_{\tilde{S}^{(\rm e)}_{0}}\tilde{p}^{(\rm e)}\Big)(\tilde{S}^{(\rm e)}_{0})'.
\end{split}
\end{eqnarray}
}
{\color{black}Here
\begin{eqnarray}\label{eq:4.27d}
\begin{split}
&\partial_{\partial_{1}\varphi^{(\rm n)}}\tilde{p}^{(\rm e)}=-\frac{(\tilde{c}^{(\rm e)}_{\rm n})^{2}\partial_{1}\varphi^{(\rm n)}}{\tilde{\rho}^{(\rm e)}_{\rm n} (\partial_{2}\varphi^{(\rm n)})^{2}
\Big((\tilde{c}^{(\rm e)}_{\rm n})^{2}-\frac{1+(\partial_{1}\varphi^{(\rm n)})^{2}}{(\tilde{\rho}^{(\rm e)}_{\rm n})^{2}(\partial_{2}\varphi^{(\rm n)})^{2}}\Big)},\\[5pt]
&\partial_{\partial_{2}\varphi^{(\rm n)}}\tilde{p}^{(\rm e)}=\frac{(\tilde{c}^{(\rm e)}_{\rm n})^{2}\big(1+(\partial_{1}\varphi^{(\rm n)})^2\big)}{\tilde{\rho}^{(\rm e)}_{\rm n} (\partial_{2}\varphi^{(\rm n)})^{3}
\Big((\tilde{c}^{(\rm e)}_{\rm n})^{2}-\frac{1+(\partial_{1}\varphi^{(\rm n)})^{2}}{(\tilde{\rho}^{(\rm e)}_{\rm n})^{2}(\partial_{2}\varphi^{(\rm n)})^{2}}\Big)}.
\end{split}
\end{eqnarray}
}
Furthermore, from the construction of the approximate solutions, and by Proposition \ref{prop:3.1}, \eqref{eq:3.13}, \eqref{eq:4.24a}, \eqref{eq:4.24b},
 \eqref{eq:4.27a} and \eqref{eq:4.27d}, we know that
\begin{eqnarray}\label{eq:4.27c}
\begin{split}
&\tilde{\beta}^{(\rm n)}_{0, 1}\big|_{\varphi^{(\rm n)}=\underline{\varphi},\tilde{B}^{(\rm e)}_{0}=\underline{B}^{(\rm e)},
\tilde{S}^{(\rm e)}_{0}=\underline{S}^{(\rm e)}}\\[5pt]
&\ \ \ ={\color{black}-\frac{\partial_{\partial_{2}\varphi^{(\rm n)}}\tilde{p}^{(\rm e)}\big|_{\varphi^{(\rm n)}=\underline{\varphi},\tilde{B}^{(\rm e)}_{0}=\underline{B}^{(\rm e)},\tilde{S}^{(\rm e)}_{0}=\underline{S}^{(\rm e)}}\mathcal{N}^{(\rm n)}_{11}\big|_{\varphi^{(\rm n)}=\underline{\varphi},\tilde{B}^{(\rm e)}_{0}=\underline{B}^{(\rm e)},\tilde{S}^{(\rm e)}_{0}=\underline{S}^{(\rm e)}}}{\mathcal{N}^{(\rm n)}_{22}\big|_{\varphi^{(\rm n)}=\underline{\varphi},\tilde{B}^{(\rm e)}_{0}=\underline{B}^{(\rm e)},\tilde{S}^{(\rm e)}_{0}=\underline{S}^{(\rm e)}}}<0},\\[5pt]
&\tilde{\beta}^{(\rm n)}_{0, 2}\big|_{\varphi^{(\rm n)}=\underline{\varphi},\tilde{B}^{(\rm e)}_{0}=\underline{B}^{(\rm e)},
\tilde{S}^{(\rm e)}_{0}=\underline{S}^{(\rm e)}}\\[5pt]
&\ \ \ ={\color{black}\partial_{\partial_{1}\varphi^{(\rm n)}}\tilde{p}^{(\rm e)}\big|_{\varphi^{(\rm n)}=\underline{\varphi},\tilde{B}^{(\rm e)}_{0}=\underline{B}^{(\rm e)},
\tilde{S}^{(\rm e)}_{0}=\underline{S}^{(\rm e)}}}\\[5pt]
&\ \ \  \ \ \ {\color{black}-\frac{2\partial_{\partial_{2}\varphi^{(\rm n)}}\tilde{p}^{(\rm e)}\big|_{\varphi^{(\rm n)}=\underline{\varphi},\tilde{B}^{(\rm e)}_{0}=\underline{B}^{(\rm e)},\tilde{S}^{(\rm e)}_{0}=\underline{S}^{(\rm e)}}\mathcal{N}^{(\rm n)}_{12}\big|_{\varphi^{(\rm n)}=\underline{\varphi},\tilde{B}^{(\rm e)}_{0}=\underline{B}^{(\rm e)},
\tilde{S}^{(\rm e)}_{0}=\underline{S}^{(\rm e)}}}{\mathcal{N}^{(\rm n)}_{22}\big|_{\varphi^{(\rm n)}=\underline{\varphi},\tilde{B}^{(\rm e)}_{0}=\underline{B}^{(\rm e)},
\tilde{S}^{(\rm e)}_{0}=\underline{S}^{(\rm e)}}}}\\[5pt]
&\ \ \ =0,
\end{split}
\end{eqnarray}
{\color{black} where $\partial_{\partial_{j}\varphi^{(\rm n)}}\tilde{p}^{(\rm e)}$ for $j=1,2$ are given by \eqref{eq:4.27d}.}

Since $\partial_{1}\delta\varphi^{(\rm n)}=\omega_{\rm cd}$ on $\tilde{\Gamma}_{\rm cd}$, then we have, on the contact discontinuity,
\begin{eqnarray}\label{eq:4.28}
\begin{split}
{\color{black}\psi^{(\rm n)}=\omega_{\rm cd}(\xi)}, \ \ \ \
\mbox{on} \ \ \ \    \tilde{\Gamma}_{\rm cd}.
\end{split}
\end{eqnarray}

(Step 3.2)  {\it Construct a comparison function and apply the maximal principle near the corner point $\mathcal{O}$.} \quad
Let us introduce a coordinate transformation,
\begin{eqnarray*}
\hat{\pi}: \ \  \left\{
\begin{array}{llll}
\hat{\xi}=\sqrt{e_{1}}\xi,\\[5pt]
\hat{\eta}=\sqrt{e_{2}}\eta,
\end{array}
\right.
\end{eqnarray*}
where $e_{1}$, $e_{2}$ are given as in \eqref{eq:4.8a}. Then under the transformation $\hat{\pi}$, the domain $\tilde{\Omega}^{(\rm e)}$ and boundaries $\tilde{\Gamma}^{(\rm e)}_{\rm in}$ and $\tilde{\Gamma}_{\rm cd}$ become
\begin{eqnarray*}
\begin{split}
&\hat{\Omega}^{(\rm e)}=\big\{(\hat{\xi}, \hat{\eta}):\  0<\hat{\xi}<\sqrt{e_{1}}L,\  0<\hat{\eta}<\sqrt{e_{2}}m^{(\rm e)}\big\}, \\[5pt]
&\hat{\Gamma}^{(\rm e)}_{\rm in}=\big\{(\hat{\xi}, \hat{\eta}):\  \hat{\xi}=0,\  0<\hat{\eta}<\sqrt{e_{2}}m^{(\rm e)}\big\},\\[5pt]
&\hat{\Gamma}_{\rm cd}=\big\{(\hat{\xi}, \hat{\eta}):\  0<\hat{\xi}<\sqrt{e_{1}}L,\  \hat{\eta}=0\big\}.
\end{split}
\end{eqnarray*}

Let 
\begin{eqnarray}\label{eq:4.29}
\begin{split}
&\mathscr{L}^{(\rm n)}:=\sum_{i,j=1, 2}\hat{\mathcal{N}}^{(\rm n)}_{ij}\hat{\partial}^{2}_{ij}{\color{black}+\sum_{k,j=1, 2}\hat{b}^{(\rm n)}_{kj}\hat{\partial}_{k}\hat{\partial}_{j}},\ \ \ \
\mathscr{M}^{(\rm n)}:=\hat{\beta}^{(\rm n)}_{0,1}\hat{\partial}_{1}
 +\hat{\beta}^{(\rm n)}_{0,2}\hat{\partial}_{2},
 \end{split}
\end{eqnarray}
where $\hat{\partial}_{1}:=\partial_{\hat{\xi}}, \ \hat{\partial}_{2}:=\partial_{\hat{\eta}}$ and
\begin{eqnarray}\label{eq:4.30}
\begin{split}
&\hat{\mathcal{N}}^{(\rm n)}_{ij}=\frac{\mathcal{N}^{(\rm n)}_{ij}\big(\hat{\pi}^{-1}(\hat{\xi}, \hat{\eta})\big)}{\sqrt{e_{i}e_{j}}},\ \ {\color{black}\hat{b}^{(\rm n)}_{kj}=\frac{b^{(\rm n)}_{kj}\big(\hat{\pi}^{-1}(\hat{\xi}, \hat{\eta})\big)}{\sqrt{e_{k}e_{j}}}},\ \
\hat{\beta}^{(\rm n)}_{0,j}=\frac{\tilde{\beta}^{(\rm n)}_{0,j}\big(\hat{\pi}^{-1}(\hat{\xi}, \hat{\eta})\big)}{\sqrt{e_{j}}},\ \ \
i,k, j=1, 2,
\end{split}
\end{eqnarray}
  satisfying   \eqref{eq:4.24a}, \eqref{eq:4.24b} and \eqref{eq:4.27c}.
Set $\hat{\psi}^{(\rm n)}(\hat{\xi}, \hat{\eta}):=\psi^{(\rm n)}\big(\hat{\pi}^{-1}(\hat{\xi}, \hat{\eta})\big)$.
 Then $\hat{\psi}^{(n)}$ in the new coordinates $(\hat{\xi}, \hat{\eta})$ satisfies the following problem:
\begin{eqnarray}\label{eq:4.31}
\left\{
\begin{array}{llll}
\mathscr{L}^{(\rm n)}(\hat{\psi}^{(\rm n)})=\hat{f}^{(\rm n)},
&\qquad \  \  \ \mbox{in} \quad \hat{\Omega}^{(\rm e)},  \\[5pt]
\mathscr{M}^{(\rm n)}(\hat{\psi}^{(\rm n)})=\hat{g}^{(\rm n)}, &\qquad \ \ \  \mbox{on} \quad \hat{\Gamma}^{(\rm e)}_{\rm in}, \\[5pt]
\hat{\psi}^{(\rm n)}={\color{black}\omega_{\rm cd}(\hat{\xi}/\sqrt{e_{1}})},
&\qquad \ \ \  \mbox{on} \quad  \hat{\Gamma}_{\rm cd},
\end{array}
\right.
\end{eqnarray}
where $\hat{f}^{(\rm n)}$ and $\hat{g}^{(\rm n)}$ are functions {\color{black} of $f^{(\rm n)}$ and ${g}^{(\rm n)}$ under the coordinate transformation of $\hat{\pi}$}.

{\color{black} Note that due to the already obtained $C^{1,\alpha}$ regularity of $\varphi^{(\rm n)}$ up to the origin, one has the $C^{0,\alpha}$ regularity of $\hat{\psi}^{(\rm n)}$,
hence the Dirichlet condition for $\hat{\psi}^{(\rm n)}$ on $\hat{\Gamma}_{\rm cd}$ can be extended to the origin.}

We fix a constant $r_{0}>0$ and let $B_{r_{0}}(\mathcal{O})$ be a disk  with center at $\mathcal{O}$ and radius $r_{0}$. Denote $B^{+}_{r_{0}}=B_{r_{0}}(\mathcal{O})\cap\hat{\Omega}^{(\rm e)}$.
Let $r=\sqrt{\hat{\xi}^{2}+\hat{\eta}^{2}}$, and  $\theta=\arctan\frac{\hat{\eta}}{\hat{\xi}}$.
Define
\begin{eqnarray}\label{eq:4.32}
{\upsilon}(r, \theta)=Kr^{1+\alpha}\sin\big(\tau \theta+\omega\big), \quad (0\leq\theta\leq\frac{\pi}{2}),
\end{eqnarray}
where $\alpha\in(0, 1)$,\ $\tau=\frac{3+\alpha}{2}, \ \omega=\frac{1-\alpha}{8}\pi$ and $K=\hat{K}r^{-1-\alpha}_{0}\varepsilon_{\rm e}$
for $\hat{K}>0$ being a constant to be determined later.
It follows from the direct computation that
\begin{eqnarray*}
\begin{split}
&\hat{\partial}_{1}\upsilon={\color{black}Kr^{\alpha}\big[(1+\alpha)\sin\big(\tau \theta+\omega\big)\cos\theta
-\tau \cos\big(\tau \theta+\omega\big)\sin\theta\big]},\\[5pt]
&\hat{\partial}_{2}\upsilon={\color{black}Kr^{\alpha}\big[(1+\alpha)\sin\big(\tau\theta+\omega\big)\sin\theta
+\tau \cos\big(\tau\theta+\omega\big)\cos\theta\big]},
\end{split}
\end{eqnarray*}
and
\begin{eqnarray}\label{eq:4.33}
\begin{split}
\hat{\Delta} \upsilon&=Kr^{\alpha-1}\big((1+\alpha)^{2}-\tau^{2}\big){\color{black}\sin\big(\tau\theta+\omega\big)}\\[5pt]
&=-\frac{K(3\alpha+5)(1-\alpha)}{4}r^{\alpha-1}{\color{black}\sin\big(\tau\theta+\omega\big)},
\end{split}
\end{eqnarray}
where $\hat{\Delta}:=\hat{\partial}^{2}_{11}+\hat{\partial}^{2}_{22}$.

Now let us consider the boundary value problem \eqref{eq:4.31} near the corner $\mathcal{O}$.
{\color{black} By the estimates \eqref{eq:4.14}, \eqref{eq:4.25b} and the condition \eqref{eq:3.14}, for $\varepsilon_{\rm e}\in (\mathcal{C}_{\rm e,0}\tilde{\epsilon}, \varepsilon_{\rm e, 0})$, we can show that
\begin{eqnarray}\label{eq:4.26}
\begin{split}
|\hat f^{(\rm n)}|\leq \mathcal{C}\tilde{\epsilon}\varepsilon_{\rm e}r^{2\alpha-1}\leq \mathcal{C}\varepsilon^2_{\rm e}r^{2\alpha-1}.
\end{split}
\end{eqnarray}
}
By \eqref{eq:4.25}-\eqref{eq:4.25b} and the estimates \eqref{eq:4.14} and \eqref{eq:4.12}, we have
\begin{eqnarray*}
\begin{split}
&\mathscr{L}^{(\rm n)}\upsilon-\hat{f}^{(\rm n)}\\[5pt]
&\ \  =\hat{\Delta} \upsilon+(\mathscr{L}^{(\rm n)}-\hat{\Delta})\upsilon-\hat{f}^{(\rm n)}\\[5pt]
&\ \  =-\frac{K(3\alpha+5)(1-\alpha)}{4}r^{\alpha-1}\sin\big(\tau\theta +\omega\big)
+\sum_{i,j=1, 2}\Big(\hat{\mathcal{N}}^{(\rm n)}_{ij}-\delta_{ij}\Big)\hat{\partial}^{2}_{ij}\upsilon
+\sum_{k,j=1, 2}\hat{b}^{(\rm n)}_{kj}\hat{\partial}_{k}\upsilon\hat{\partial}_{j}\upsilon-\hat{f}^{(\rm n)}\\[5pt]
&\ \ \leq -\frac{\hat{K}(3\alpha+5)(1-\alpha)}{4}r^{-1-\alpha}_{0}r^{\alpha-1}\varepsilon_{\rm e}\sin\big(\tau\theta+\omega\big)\\
&\qquad\qquad +\mathcal{C}\hat{K}r^{-1-\alpha}_{0}r^{\alpha-1}\varepsilon^{2}_{\rm e}+\mathcal{C}\hat{K}^2r^{-2(1+\alpha)}_{0}r^{2\alpha}\varepsilon^{2}_{\rm e}+\mathcal{C}r^{2\alpha-1}\varepsilon^{2}_{\rm e}\\[5pt]
&\ \
=r^{-1-\alpha}_{0}r^{\alpha-1}\varepsilon_{\rm e}\bigg(-\frac{\hat{K}(3\alpha+5)(1-\alpha)}{4}\sin\big(\tau\theta+\omega\big)
+\mathcal{C}\hat{K}\varepsilon_{\rm e}+\mathcal{C}\hat{K}^2r^{-1-\alpha}_{0}r^{\alpha+1}\varepsilon_{\rm e}+\mathcal{C}r^{\alpha+1}_{0}r^{\alpha}\varepsilon_{\rm e}\bigg)\\[5pt]
&\ \
=r^{-1-\alpha}_{0}r^{\alpha-1}\varepsilon_{\rm e}\bigg(-\frac{\hat{K}(3\alpha+5)(1-\alpha)}{4}\sin\big(\tau\theta+\omega\big)
+\mathcal{C}\hat{K}\varepsilon_{\rm e}+\mathcal{C}\hat{K}^2\varepsilon_{\rm e}+\mathcal{C}r^{2\alpha+1}_{0}\varepsilon_{\rm e}\bigg),
\end{split}
\end{eqnarray*}
where $\mathcal{C}>0$ depends only on the background state $\underline{U}^{(\rm e)}$.
Due to the choice of the constants $\alpha, \theta$ and $\omega$, we know that
\begin{eqnarray*}
-\frac{\hat{K}(3\alpha+5)(1-\alpha)}{4}{\color{black}\sin\big(\tau\theta+\omega\big)}<0.
\end{eqnarray*}
Thus, for $\varepsilon_{\rm e}<\varepsilon_{\rm e,0}$ and $\hat{K}$ sufficiently large, it holds that
\begin{eqnarray}\label{eq:4.35}
\begin{split}
&\mathscr{L}^{(\rm n)}\upsilon-\hat{f}^{(\rm n)}\leq 0.
\end{split}
\end{eqnarray}

\par Next, let us consider the boundary conditions. Obviously, we have
\begin{eqnarray}\label{eq:4.36}
\begin{split}
\left.\hat{\psi}^{(\rm n)}\right.\big|_{ \partial B^{+}_{r_{0}}\cap \hat{\Omega}^{(\rm e)}}
\leq\left. \upsilon \right.\big|_{ \partial B^{+}_{r_{0}}\cap \hat{\Omega}^{(\rm e)}},
\end{split}
\end{eqnarray}
by letting $\hat{K}$ sufficiently large.

On the boundary $\partial B^{+}_{r_{0}}\cap \hat{\Gamma}_{\rm cd}$, {\color{black} due to the $C^{0,\alpha}$ regularity of $\hat{\psi}^{(\rm n)}$, the Dirichlet condition
for $\hat{\psi}^{(\rm n)}$ on $\hat{\Gamma}_{\rm cd}$ can be extended to the origin,} thus  one has
\begin{eqnarray}\label{eq:4.37}
\begin{split}
\hat{\psi}^{(\rm n)}-\upsilon&\leq \mathcal{C}r^{-1-\alpha}_{0}r^{1+\alpha}\Big(\mathcal{C}r^{1+\alpha}_{0}\|\omega_{\rm cd}\|^{(-\alpha, \{\mathcal{P}_{\rm e}\})}_{1,\alpha; \tilde{\Gamma}_{\rm cd}}-\hat{K}\varepsilon_{\rm e}\sin\omega\Big)\\[5pt]
&\leq \mathcal{C}r^{-1-\alpha}_{0}r^{1+\alpha}\Big(\mathcal{C}C_{w}r^{1+\alpha}_{0}\varrho-\hat{K}\varepsilon_{\rm e}\sin\omega\Big)\\[5pt]
&\leq \mathcal{C}r^{-1-\alpha}_{0}r^{1+\alpha}\varepsilon_{\rm e}\Big(\mathcal{C}r^{1+\alpha}_{0}-\hat{K}\sin\omega\Big)\\[5pt]
&\leq 0,
\end{split}
\end{eqnarray}
by choosing {\color{black}$\varrho\in (0, \mathcal{C}^{-1}_{\rm e}\varepsilon_{\rm e})$} and $\hat{K}$ sufficiently large.

Finally, on $\partial B^{+}_{r_{0}} \cap {\hat{\Gamma}^{(\rm e)}_{\rm in}}$,
{\color{black} by the estimate  \eqref{eq:4.14} and the assumptions \eqref{eq:3.14},
we have,  for $\tilde{\epsilon}>0$ sufficiently small and $\tilde{\epsilon}<\varepsilon_{\rm e}\in (\mathcal{C}_{\rm e}\tilde{\epsilon}, \varepsilon_{\rm e, 0})$,
\begin{eqnarray}\label{eq:4.37a}
\begin{split}
|\hat g^{(\rm n)}|\leq \mathcal{C}\tilde{\epsilon} r^{\alpha}\leq \mathcal{C}\varepsilon_{\rm e} r^{\alpha}.
\end{split}
\end{eqnarray}
}
{\color{black}Then, by \eqref{eq:4.37a}, we can deduce that
\begin{eqnarray}\label{eq:4.38}
\begin{split}
&\mathscr{M}^{(\rm n)}(\upsilon)-\hat{g}^{(\rm n)}\\[5pt]
=&\left.\hat{\beta}^{(\rm n)}_{0,1}Kr^{\alpha}\big[(1+\alpha)\sin\big(\tau \theta+\omega\big)\cos\theta
-\tau \cos\big(\tau \theta+\omega\big)\sin\theta\big]\right.\Big|_{ \theta=\frac{\pi}{2}}\\[5pt]
&+\left.\hat{\beta}^{(\rm n)}_{0,2}Kr^{\alpha}\big[(1+\alpha)\sin\big(\tau\theta+\omega\big)\sin\theta
+\tau \cos\big(\tau\theta+\omega\big)\cos\theta\big]\right.\Big|_{\theta=\frac{\pi}{2}}-\hat{g}^{(\rm n)}\\[5pt]
\leq &\hat{K}\Big(-\frac{3+\alpha}{2}\hat{\beta}^{(\rm n)}_{0,1}\cos(\frac{\tau}{2}\pi+\omega)+(1+\alpha)\hat{\beta}^{(\rm n)}_{0,2}
\sin(\frac{\tau}{2}\pi+\omega)\Big)r^{-1-\alpha}_{0}r^{\alpha}\varepsilon_{\rm e}+\mathcal{C}\varepsilon_{\rm e}r^{\alpha}\\[5pt]
=&\bigg(-\frac{\hat{K}(3+\alpha)}{2}\hat{\beta}^{(\rm n)}_{0,1}\big|_{\varphi^{(\rm n)}=\underline{\varphi},\tilde{B}^{(\rm e)}_{0}
=\underline{B}^{(\rm e)},\tilde{S}^{(\rm e)}_{0}=\underline{S}^{(\rm e)}}\cos(\frac{\tau}{2}\pi+\omega)+\mathcal{C}\varepsilon_{\rm e}
+\mathcal{C}r^{1+\alpha}_{0}\bigg)r^{-1-\alpha}_{0}r^{\alpha}\varepsilon_{\rm e}.
\end{split}
\end{eqnarray}}

By choosing 
$r_{0}=\varepsilon_{\rm e}<\varepsilon_{\rm e,0}$, and noting that $r\leq r_0$, we deduce that
\begin{eqnarray}\label{eq:4.38ax}
\begin{split}
\mathscr{M}^{(\rm n)}(\upsilon)-\hat{g}^{(\rm n)}
\leq {\color{black}-\frac{\hat{K}(3+\alpha)}{2}\hat{\beta}^{(\rm n)}_{0,1}\big|_{\varphi^{(\rm n)}=\underline{\varphi},\tilde{B}^{(\rm e)}_{0}
=\underline{B}^{(\rm e)},\tilde{S}^{(\rm e)}_{0}=\underline{S}^{(\rm e)}}\cos(\frac{\tau}{2}\pi+\omega)+\mathcal{C}\varepsilon_{\rm e}.}
\end{split}
\end{eqnarray}
We can take $0<\alpha_{0}<1$ and $\hat{K}>0$ sufficiently large depending only on $\tilde{\underline{U}}$ and $L$  so that
\begin{eqnarray*}
\begin{split}
{\color{black}-\hat{K}(1+\alpha)\hat{\beta}^{(\rm n)}_{0,1}\big|_{\varphi^{(\rm n)}=\underline{\varphi},\tilde{B}^{(\rm e)}_{0}=\underline{B}^{(\rm e)},
\tilde{S}^{(\rm e)}_{0}=\underline{S}^{(\rm e)}}\cos\left(\frac{\tau}{2}\pi+\omega\right)+\mathcal{C}\varepsilon_{\rm e}\leq 0,}
\end{split}
\end{eqnarray*}
holds for any $\alpha\in(0, \alpha_{0})$, where we have used \eqref{eq:4.27c} and the fact:
$\cos\left(\frac{\tau}{2}\pi+\omega\right)=\cos\left(\frac{7+\alpha}{8}\pi\right)<0$. It implies that
\begin{eqnarray*}\label{eq:4.38a}
\begin{split}
\mathscr{M}^{(\rm n)}(\upsilon)-\hat{g}^{(\rm n)}\leq 0.
\end{split}
\end{eqnarray*}

 Then, by the maximal principle, we have $\psi^{(\rm n)}\leq \upsilon$.
Similarly, we can also obtain $\psi^{(\rm n)}\geq -\upsilon$. Thus, we have $|\psi^{(\rm n)}|\leq K r^{1+\alpha}$.
Using the standard scaling technique, we can get $C^{1,\alpha}$-estimate for $\psi^{(n)}$, and then $C^{2,\alpha}$-estimate for $\varphi^{(\rm n)}$ up to the corner point $\mathcal{O}$.
\end{proof}

\subsection{Estimates of solutions to the problem $(\widetilde{\mathbf{FP}})_{\rm n}$ in the supersonic region
$\tilde{\Omega}^{(\rm h)}$}
In this subsection, we will consider the solution $\delta z^{(\rm n)}$ for the initial-boundary value problem
$(\widetilde{\mathbf{FP}})_{\rm n}$ in the supersonic region $\tilde{\Omega}^{(\rm h)}$:
\begin{eqnarray}\label{eq:4.40}
\left\{
\begin{array}{llll}
\partial_{\xi}\delta z^{(\rm n)}+\textrm{diag}( \lambda^{(\rm n-1)}_{+}, \lambda^{(\rm n-1)}_{-})\partial_{\eta}\delta z^{(\rm n)}=0,
&\ \ \ \mbox{in}\ \ \ \tilde{\Omega}^{(\rm h)},  \\[5pt]
\delta z^{(\rm n)}=\delta z_{0}(\eta),  &\ \ \  \mbox{on}\ \ \ \tilde{\Gamma}^{(\rm h)}_{\rm in},\\[5pt]
\delta z^{(\rm n)}_{-}-\delta z^{(\rm n)}_{+}=2\beta^{(\rm n-1)}_{\rm cd, 1}\partial_{\xi}\delta\varphi^{(\rm n)}
+2\beta^{(\rm n-1)}_{\rm cd, 2}\partial_{\eta}\delta\varphi^{(\rm n)}+2c_{\rm cd}(\xi),
&\ \ \  \mbox{on} \ \ \ \tilde{\Gamma}_{\rm cd},\\[5pt]
\delta z^{(\rm n)}_{-}+\delta z^{(\rm n)}_{+}=2\arctan g'_{-}(\xi), &\ \ \ \mbox{on}\ \ \ \tilde{\Gamma}_{-},
\end{array}
\right.
\end{eqnarray}
where the functions $\beta^{(\rm n-1)}_{\rm cd, 1}$, $\beta^{(\rm n-1)}_{\rm cd, 2}$ and $c_{\rm cd}(\xi)$
are defined in {\color{black}\eqref{eq:3.72}-\eqref{eq:3.73}}. We have the following result. 

\begin{proposition}\label{prop:4.2}
For any given $\alpha\in(0,1)$, {\color{black}there exist constants $\mathcal{C}^{*}_{\rm h,0}>0$ and $\epsilon^{*}_{\rm h,0}>0$ depending only on $\underline{\tilde{U}}$, $L$ and $\alpha$, such that for
$\tilde{\epsilon}<\epsilon_{I}\in (\mathcal{C}^{*}_{\rm h,0}\tilde{\epsilon} ,\epsilon^{*}_{\rm h, 0})$ with $\tilde{\epsilon}>0$ is sufficiently small,
if {\color{black}$(\delta z^{(\rm n-1)}, \delta \varphi^{(\rm n-1)})\in \mathscr{K}_{2\epsilon_{I}}$}}, {\color{black}then} the solution $z^{(\rm n)}$ to the problem \eqref{eq:4.40} satisfies
\begin{eqnarray}\label{eq:4.41}
\begin{split}
&\|\delta z^{(\rm n)}_{-}\|_{1,\alpha;\tilde{\Omega}^{(\rm h)}}+\|\delta z^{(\rm n)}_{+}\|_{1,\alpha;\tilde{\Omega}^{(\rm h)}}\\[5pt]
\leq& C^{*}_{41}\bigg(\big\|\delta z^{(\rm n)}_{0}\big\|_{1,\alpha;\tilde{\Gamma}^{(\rm h)}_{\rm in}}
+\sum_{k=\rm e, h}\|\delta\tilde{B}^{(\rm k)}_{0}\|_{1,\alpha; \tilde{\Gamma}^{(\rm k)}_{\rm in}}\\[5pt]
&\qquad\qquad +\sum_{k=\rm e, h}\|\delta\tilde{S}^{(\rm k)}_{0}\|_{1,\alpha; \tilde{\Gamma}^{(\rm k)}_{\rm in}}+\big\|g_{-}+1\big\|_{2,\alpha; \tilde{\Gamma}_{-}}
+{\color{black}\big\|\delta \varphi^{(\rm n)}\big\|^{(-1-\alpha, \{\mathcal{P}_{\rm e}\})}_{2,\alpha; \overline{\tilde{\Gamma}}_{\rm cd}}}\bigg),
\end{split}
\end{eqnarray}
where $C^{*}_{41}>0$ depends only on $\underline{\tilde{U}}$, $L$ and $\alpha$.
\end{proposition}

\vspace{2pt}
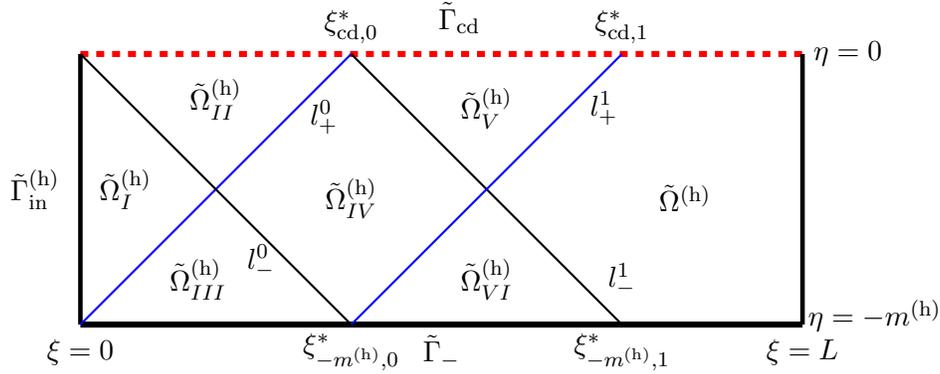
\begin{figure}[ht]
\begin{center}
\begin{tikzpicture}[scale=1.2]
\draw [line width=0.08cm][dashed][red](-5,1.0) --(3.0,1.0);
\draw [line width=0.08cm] (-5,-2) --(3.0,-2);
\draw [line width=0.06cm](-5,-2) --(-5,1.0);
\draw [line width=0.06cm](3.0,-2) --(3.0,1.0);
\draw [thick][blue](-5,-2) --(-2,1);
\draw [thick][black](-5,1)--(-2,-2) ;
\draw [thick][black](-2,1)--(1,-2) ;
\draw [thick][blue](-2,-2) --(1,1);
\node at (3.0, -2.3) {$\xi=L$};
\node at (-5, -2.3) {$\xi=0$};
\node at (3.5, 1) {$\eta=0$};
\node at (3.8, -1.9) {$\eta=-m^{(\rm h)}$};
\node at (-2.0, 1.3) {$\xi^{*}_{\rm cd, 0}$};
\node at (1, 1.3) {$\xi^{*}_{\rm cd, 1}$};
\node at (-2.0, -2.3) {$\xi^{*}_{-m^{(\rm h)}, 0}$};
\node at (1, -2.3) {$\xi^{*}_{-m^{(\rm h)}, 1}$};
\node at (-3.0, -1.3) {$l^{0}_{-}$};
\node at (-2.3, 0.3) {$l^{0}_{+}$};
\node at (0.8, 0.4) {$l^{1}_{+}$};
\node at (1.0, -1.5) {$l^{1}_{-}$};
\node at (1.7, -0.6) {$\tilde{\Omega}^{(\rm h)}$};
\node at (-4.5, -0.5) {$\tilde{\Omega}^{(\rm h)}_{ I}$};
\node at (-3.5,0.5) {$\tilde{\Omega}^{(\rm h)}_{ II}$};
\node at (-3.7, -1.5) {$\tilde{\Omega}^{(\rm h)}_{III}$};
\node at (-2.0, -0.6) {$\tilde{\Omega}^{(\rm h)}_{IV}$};
\node at (-0.5, 0.4) {$\tilde{\Omega}^{(\rm h)}_{V}$};
\node at (-0.5, -1.5) {$\tilde{\Omega}^{(\rm h)}_{VI}$};
\node at (-5.5, -0.5) {$\tilde{\Gamma}^{(\rm h)}_{\rm in}$};
\node at (-1, -2.3) {$\tilde{\Gamma}_{-}$};
\node at (-0.8, 1.4) {$\tilde{\Gamma}_{\rm cd}$};
\end{tikzpicture}
\caption{Initial-boundary value problem for $(\widetilde{\mathbf{FP}})_{\rm n}$ in $\tilde{\Omega}^{(\rm h)}$}\label{fig:4.6}
\end{center}
\end{figure}
The proof of the Proposition \ref{prop:4.2} will be divided into several steps to deal with the different sub-domains
which are constructed as follows (see Fig.\ref{fig:4.6}). 
Denoted by $\tilde{\Omega}_I^{(\rm h)}$   the triangle which is bounded by the inlet $\tilde{\Gamma}^{(\rm h)}_{\rm in}$, the characteristics of \eqref{eq:4.40} 
with the speed $\lambda_+$ issuing 
from the point $(0,-m^{(\rm h)})$ (which we denote by  $l^0_+$) and 
another characteristics of \eqref{eq:4.40} which corresponding to $\lambda_-$ 
issuing from the point $(0,0)$ (which we denote by $l_-^0$).
Let $\tilde{\Omega}^{(\rm h)}_{II}$ be the triangle bounded by $l^0_{\pm}$ and the contact discontinuity $\tilde{\Gamma}_{\rm cd}$. Let $\tilde{\Omega}^{(\rm h)}_{III}$ be the triangle bounded by the lower nozzle walls $\tilde{\Gamma}^{(\rm h)}_{-}$, $l_+^0$ and $l_-^0$.
Let $\tilde{\Omega}^{(\rm h)}_{IV}$ be the diamond bounded by $l^0_{\pm}$, the characteristics $l^1_+$   
determined by $\lambda_{+}$ issuing from the point where $l_-^0$ and the lower nozzle walls $\tilde{\Gamma}^{(\rm h)}_{-}$
intersect,
as well as the characteristics $l^1_-$ 
with the speed $\lambda_{-}$ issuing from the intersection point of $l_+^0$ and the contact discontinuity  $\tilde{\Gamma}_{\rm cd}$. 
Let $\tilde{\Omega}^{(\rm h)}_{V}$ be the triangle bounded by $l^1_{+}$, $l^1_{-}$ and the contact discontinuity $\tilde{\Gamma}_{\rm cd}$.
Let $\tilde{\Omega}^{(\rm h)}_{VI}$ be the triangle bounded by $l^1_{-}$, $l^1_{+}$ and the lower nozzle wall $\tilde{\Gamma}^{(\rm h)}_{-}$.
{Finally, we remark that different from the argument made in \cite{hkwx}, we need to consider the H\"{o}lder estimates, \emph{i.e.}, $C^{1,\alpha}$-norm on the charateristic curves and the solutions instead of the $C^1$-norm used in \cite{hkwx}.}

\subsubsection{Estimates on the characteristic curves} First, let us introduce several lemmas on the H\"{o}lder regularity of the characteristic curves.
For any ${\color{black}\eta_{\pm, 0}}\in [-m^{(\rm h)},0]$, let $\eta=\phi^{(\rm n)}_{\pm}({\color{black}\xi,\eta_{\pm, 0}})$ be the characteristic curves corresponding to $\lambda^{(\rm n-1)}_{\pm}$ issuing from the point $(0,{\color{black}\eta_{\pm, 0}})$:
\begin{eqnarray}\label{eq:4.43}
\begin{split}
\left\{
\begin{array}{llll}
\frac{d\phi^{(\rm n)}_{\pm}}{d\xi}=\lambda^{(\rm n-1)}_{\pm}(\xi,\phi^{(\rm n)}_{\pm}(\tau, {\color{black}\eta_{\pm, 0}})),\\[5pt]
\phi^{(\rm n)}_{\pm}(0,{\color{black}\eta_{\pm, 0}})={\color{black}\eta_{\pm, 0}}.
\end{array}
\right.
\end{split}
\end{eqnarray}
Along the characteristics, we have
\begin{eqnarray}\label{eq:4.44}
\eta={\color{black}\eta_{\pm, 0}}+\int^{\xi}_{0}\lambda^{(\rm n-1)}_{\pm}(\tau,\phi^{(\rm n)}_{\pm}(\tau, {\color{black}\eta_{\pm, 0}}))d\tau.
\end{eqnarray}
From \eqref{eq:4.44}, we can regard $\eta_{+, 0}$ (or $\eta_{-, 0}$) as a function of $(\xi,\eta)$,
\emph{i.e.}, $\eta_{+,0}=\eta_{+, 0}(\xi,\eta)$ (or $\eta_{-,0}=\eta_{-, 0}(\xi,\eta)$).

{\color{black}For any two points $P=(\xi_{P},\eta_{P}), \  Q=(\xi_{Q},\eta_{Q}) \in \tilde{\Omega}^{(\rm h)}_I\cup \tilde{\Omega}^{(\rm h)}_{II}$
(or $P,\, Q \in \tilde{\Omega}^{(\rm h)}_I\cup \tilde{\Omega}^{(\rm h)}_{III}$), there exist unique
points $P_{+}(0)=(0,\eta_{P_{+}(0)}), Q_{+}(0)=(0,\eta_{Q_{+}(0)})\in \tilde{\Gamma}^{(\rm h)}_{\rm in}$
(or  $P_{-}(0)=(0,\eta_{P_{-}(0)}), Q_{-}(0)=(0,\eta_{Q_{-}(0)})\in \tilde{\Gamma}^{(\rm h)}_{\rm in}$)
such that the characteristics corresponding to $\lambda^{(\rm n-1)}_{+}$ (or $\lambda^{(\rm n-1)}_{-}$)
issuing from $P_{+}(0)$ and $Q_{+}(0)$(or $P_{-}(0)$ and $Q_{-}(0)$) pass through $P$ and $Q$.

Define $P_{+}(\xi)=(\xi, \phi_{+}(\xi,\eta_{P_{+}(0)}))$ and $Q_{+}(\xi)=(\xi, \phi_{+}(\xi,\eta_{Q_{+}(0)}))$
(or $P_{-}(\xi)=(\xi, \phi_{-}(\xi,\eta_{P_{-}(0)}))$ and $Q_{-}(\xi)=(\xi, \phi_{-}(\xi,\eta_{Q_{-}(0)}))$). Then
\begin{eqnarray}\label{eq:4.45}
\begin{split}
&\eta_{P_{+}(\xi)}=\phi_{+}(\xi,\eta_{P_{+}(0)})=\eta_{P_{+}(0)}+\int^{\xi}_{0}\lambda^{(\rm n-1)}_{+}(P_{+}(\tau))d\tau, \\[5pt]
&\eta_{Q_{+}(\xi)}=\phi_{+}(\xi,\eta_{Q_{+}(0)})=\eta_{Q_{+}(0)}+\int^{\xi}_{0}\lambda^{(\rm n-1)}_{+}(Q_{+}(\tau))d\tau,
\end{split}
\end{eqnarray}
for $P_{+}, Q_{+}\in \tilde{\Omega}^{(\rm h)}_I\cup \tilde{\Omega}^{(\rm h)}_{II}$, (or
\begin{eqnarray}\label{eq:4.46}
\begin{split}
&\eta_{P_{-}(\xi)}=\phi_{-}(\xi,\eta_{P_{-}(0)})=\eta_{P_{-}(0)}+\int^{\xi}_{0}\lambda^{(\rm n-1)}_{-}(P_{-}(\tau))d\tau, \\[5pt]
&\eta_{Q_{-}(\xi)}=\phi_{-}(\xi,\eta_{Q_{-}(0)})=\eta_{Q_{-}(0)}+\int^{\xi}_{0}\lambda^{(\rm n-1)}_{-}(Q_{-}(\tau))d\tau,
\end{split}
\end{eqnarray}
for $P_{-}, Q_{-}\in \tilde{\Omega}^{(\rm h)}_I\cup \tilde{\Omega}^{(\rm h)}_{III}$).

From the above arguments, we also have $\eta_{P_{\pm}(0)}=\eta_{\pm,0}(P)$ and $\eta_{Q_{\pm}(0)}=\eta_{\pm,0}(Q)$.
Finally, we define the distance in the Euclidean norm as
\begin{align}\label{eq:4.47x}
d(P_{\pm}(\xi),Q_{\pm}(\xi))=\big|P_{\pm}(\xi)-Q_{\pm}(\xi)\big|,\  d^{\alpha}(P_{\pm}(\xi),Q_{\pm}(\xi))=\big|P_{\pm}(\xi)-Q_{\pm}(\xi)\big|^{\alpha}, \ \forall \alpha\in(0,1).
\end{align}
}
Then, we have
\begin{lemma}\label{lem:4.1}
For any $\alpha\in (0,1)$ and $\delta z^{(\rm n-1)}
\in \mathscr{K}_{2\epsilon_{I}}$,
there exist constants $\mathcal{C}_{\rm h,1}>0$ and $\epsilon_{\rm h,1}>0$
$C_{4i}, (i=1,2,3)$ depending only on $\tilde{\underline{U}}$, $L$ and $\alpha$, such that
for $\tilde{\epsilon}<\epsilon_{I} \in (\mathcal{C}_{\rm h,1}\tilde{\epsilon}, \epsilon_{\rm h,1})$ with $\tilde{\epsilon}>0$ sufficiently small, one has
\begin{eqnarray}\label{eq:4.47}
\begin{split}
	(\rm i)\ \ & d^{\alpha}\big({\color{black}P_{\pm}}(\xi),{\color{black}Q_{\pm}}(\xi)\big)\leq C_{41}d^{\alpha}\big(P,Q\big),\quad & \xi\in[0, \xi_{PQ}], \\[5pt]
	(\rm ii)\ \ & d^{\alpha}\big({\color{black}P_{\pm}}(\xi),{\color{black}Q_{\pm}}(\xi)\big)\leq C_{42}d^{\alpha}\big({\color{black}P_{\pm}}(\xi_{PQ}),{\color{black}Q_{\pm}}(\xi_{PQ})\big),\quad & \xi\in[0, \xi_{PQ}], \\[5pt]
	(\rm iii)\ \  & {C^{-1}_{43}|\xi_{PQ}-\xi_{Q}|^{\alpha}\leq d^{\alpha}\big({\color{black}Q_{\pm}(\xi_{PQ})},Q\big)\leq C_{43}|\xi_{PQ}-\xi_{Q}|^{\alpha},}
	\end{split}
	\end{eqnarray}
	for any $P, Q \in \tilde{\Omega}^{(\rm h)}_I\cup \tilde{\Omega}^{(\rm h)}_{II}$ (or $P, Q \in \tilde{\Omega}^{(\rm h)}_I\cup \tilde{\Omega}^{(\rm h)}_{III}$),
	where $\xi_{PQ}=\min\{\xi_{P},\ \xi_{Q} \}$, $P_{\pm}(\xi)=\big(\xi, \phi^{(\rm n)}_{\pm}(\xi, \eta_{P_{\pm}(0)})\big),
	Q_{\pm}(\xi)=\big(\xi, \phi^{(\rm n)}_{\pm}(\xi, \eta_{Q_{\pm}(0)})\big)$, {\color{black}$P_{\pm}(\xi_{PQ})=\big(\xi_{PQ}, \phi^{(\rm n)}_{\pm}(\xi_{PQ}, \eta_{P_{\pm}(0)})\big)$, and $\ Q_{\pm}(\xi_{PQ}) = \big(\xi_{PQ}, \phi^{(\rm n)}_{\pm}(\xi_{PQ}, \eta_{Q_{\pm}(0)})\big)$}.
\end{lemma}

\begin{proof}
	Without loss of generality, we only consider the case that $P, Q \in \tilde{\Omega}^{(\rm h)}_I\cup \tilde{\Omega}^{(\rm h)}_{II}$ and $\xi_{PQ}=\xi_{P}$, since the other cases can be treated similarly.
	By \eqref{eq:4.43}, we know that
	\begin{eqnarray*}
	\begin{split}
	\frac{\partial}{\partial \xi}\Big(\frac{\partial \phi^{(\rm n)}_{+}}{\partial \eta_{+,0}}\Big)
			=\partial_{\eta }\lambda^{(\rm n-1)}_{+}\frac{\partial \phi^{(\rm n)}_{+}}{\partial \eta_{+,0}},
			\ \ \  \frac{\partial \phi^{(\rm n)}_{+}}{\partial \eta_{+,0}}(0,\eta_{+,0}) =1,
		\end{split}
	\end{eqnarray*}
	which leads to
	\begin{eqnarray*}
		\begin{split}
			\frac{\partial \phi^{(\rm n)}_{+}}{\partial \eta_{+,0}}=e^{\int^{\xi}_{0}\partial_{\eta}\lambda^{(\rm n-1)}_{+}(\tau, \phi^{(\rm n)}_{+}(\tau, \eta_{+,0}))d\tau}.
		\end{split}
	\end{eqnarray*}

{\color{black}Noticing that
\begin{eqnarray}\label{eq:4.47a}
\lambda^{(\rm n-1)}_{+}\big|_{\delta z^{(\rm n-1)}_{+}=\delta \tilde{B}^{(\rm h)}_{0}=\delta\tilde{S}^{(\rm h)}_{0}=0}=
-\lambda^{(\rm n-1)}_{-}\big|_{\delta z^{(\rm n-1)}_{-}=\delta \tilde{B}^{(\rm h)}_{0}=\delta\tilde{S}^{(\rm h)}_{0}=0}=\underline{\lambda}_{+}>0,
\end{eqnarray}
one has
\begin{eqnarray}\label{eq:4.47b}
\begin{split}
\|\lambda^{(\rm n-1)}_{\pm}-\underline{\lambda}_{\pm}\|_{1,\alpha; \tilde{\Omega}^{(\rm h)}}
&\leq \mathcal{C}\Big(\|\delta z^{(\rm n-1)}\|_{1,\alpha;\tilde{\Omega}^{(\rm h)}}+\|\delta \tilde{B}^{(\rm h)}_{0}\|_{1,\alpha;\tilde{\Gamma}^{(\rm h)}_{\rm in}}+\|\delta\tilde{S}^{(\rm h)}_{0}\|_{1,\alpha;\tilde{\Gamma}^{(\rm h)}_{\rm in}}\Big)\\[5pt]
&\leq \mathcal{C}(2\epsilon_{I}+\tilde{\epsilon}),
\end{split}
\end{eqnarray}
and hence there exist constants $\mathcal{C}'_{\rm h,1}>0$ and $\epsilon'_{\rm h,1}>0$ depending only on $\tilde{\underline{U}}$ and $L$, such that
for $\tilde{\epsilon}>0$ sufficiently small, if $\tilde{\epsilon}<\epsilon_{I} \in (\mathcal{C}'_{\rm h,1}\tilde{\epsilon}, \epsilon'_{\rm h,1})$,
the following holds,
\begin{eqnarray}\label{eq:4.47c}
\begin{split}
\frac{1}{2}\underline{\lambda}_{+} <\|\lambda^{(\rm n-1)}_{\pm}\|_{1,\alpha; \tilde{\Omega}^{(\rm h)}}<\frac{3}{2}\underline{\lambda}_{+}, \qquad
 \frac{1}{2}\underline{\lambda}_{+}\leq \pm\lambda^{(\rm n-1)}_{\pm}<\frac{3}{2}\underline{\lambda}_{+}.
\end{split}
\end{eqnarray}
}
Then, by \eqref{eq:4.47c}, it follows that
	\begin{eqnarray}\label{eq:4.48}
	\begin{split}
	\big|\phi^{(\rm n)}_{+}(\xi, \eta_{P_{+}(0)})-\phi^{(\rm n)}_{+}(\xi, \eta_{Q_{+}(0)})|\leq \mathcal{C}| \eta_{P_{+}(0)}-\eta_{Q_{+}(0)}|,
	\end{split}
	\end{eqnarray}
	for $0\leq\xi\leq\xi_{PQ}$.
	On the other hand, from \eqref{eq:4.45} and \eqref{eq:4.47b}, we have
	\begin{eqnarray}\label{eq:4.49}
	\begin{split}
	\big|\eta_{P_{+}(0)}-\eta_{Q_{+}(0)}\big|&\leq \big|\eta_{P}-\eta_{Q}\big|+\bigg|\int^{\xi_{Q}}_{\xi_{P}}\lambda^{(\rm n-1)}_{+}\big(Q_{+}(\tau)\big)d\tau\bigg|\\[5pt]
	&\ \ \ \ +\Bigg|\int^{\xi_{P}}_{0}\Big(\lambda^{(\rm n-1)}_{+}\big(Q_{+}(\tau)\big)
	-\lambda^{(\rm n-1)}_{+}\big(P_{+}(\tau)\big)\Big)d\tau\Bigg|\\[5pt]
	&\leq \mathcal{C}L\bigg(\big\|Dz^{(\rm n-1)}\big\|_{0,0;\tilde{\Omega}^{(\rm h)}}
	+\big\|\tilde{B}^{(\rm h)'}_{0}\big\|_{0,0;\tilde{\Gamma}^{(\rm h)}_{\rm in}}
	+\big\|\tilde{S}^{(\rm h)'}_{0}\big\|_{0,0;\tilde{\Gamma}^{(\rm h)}_{\rm in}}\bigg)\big|\eta_{P_{+}(0)}-\eta_{Q_{+}(0)}\big|\\[5pt]
	&\ \ \ \ +\big|\eta_{P}-\eta_{Q}\big| +\mathcal{C}\big|\xi_{P}-\xi_{Q}\big|\\[5pt]
	&\leq \mathcal{C}\Big(\big|\xi_{P}-\xi_{Q}\big|+\big|\eta_{P}-\eta_{Q}\big|\Big)
	+2\mathcal{C}L\big(2\epsilon_{I}+\tilde{\epsilon}\big)\big|\eta_{P_{+}(0)}-\eta_{Q_{+}(0)}\big|.
	\end{split}
	\end{eqnarray}
	By choosing constants $\mathcal{C}''_{\rm h,1}>0$ and $\epsilon''_{\rm h,1}>0$ depending only on $\tilde{\underline{U}}$ and $L$, so that
{for $\tilde{\epsilon}>0$ sufficiently small and $\tilde{\epsilon}<\epsilon_{I} \in (\mathcal{C}''_{\rm h,1}\tilde{\epsilon}, \epsilon''_{\rm h,1})$, we have $2\mathcal{C}L\big(2\epsilon_{I}+\tilde{\epsilon}\big)<\frac{1}{2}$.}
Then, it follows from \eqref{eq:4.49} that
	\begin{eqnarray}\label{eq:4.50}
	\begin{split}
	\big|\eta_{P_{+}(0)}-\eta_{Q_{+}(0)}\big|\leq 2\mathcal{C}\Big(\big|\xi_{P}-\xi_{Q}\big|+\big|\eta_{P}-\eta_{Q}\big|\Big).
	\end{split}
	\end{eqnarray}
	Combining the estimates \eqref{eq:4.48} and \eqref{eq:4.50}, we get the first estimate (\rm i)
 in \eqref{eq:4.47}.
	
	Next, taking $\xi_{P}=\xi_{Q}$ in \eqref{eq:4.49} and then applying estimates \eqref{eq:4.48} and \eqref{eq:4.50} again,
	we have the second estimate (\rm ii)
in \eqref{eq:4.47}.
	
	Finally, by \eqref{eq:4.45} and the estimate \eqref{eq:4.47c}, we have
	\begin{eqnarray*}
		\begin{split}
			d\big(Q_{+}(\xi_{P}), Q\big)
			=\big|\xi_{P}-\xi_{Q}\big|+\bigg|\int^{\xi_{Q}}_{\xi_{P}}\lambda^{(\rm n-1)}_{+}\big(Q_{+}(\tau)\big)d\tau\bigg|\leq \mathcal{C}\big|\xi_{P}-\xi_{Q}\big|,
		\end{split}
	\end{eqnarray*}
and
\begin{eqnarray*}
		\begin{split}
			d\big(Q_{+}(\xi_{P}), Q\big)
			=\big|\xi_{P}-\xi_{Q}\big|+\bigg|\int^{\xi_{Q}}_{\xi_{P}}\lambda^{(\rm n-1)}_{+}\big(Q_{+}(\tau)\big)d\tau\bigg|\geq \mathcal{C}^{-1}\big|\xi_{P}-\xi_{Q}\big|.
		\end{split}
	\end{eqnarray*}
	Then the estimate (\rm iii) in \eqref{eq:4.47} follows.
Finally, we take  $\mathcal{C}_{\rm h, 1}=\max\{\mathcal{C}'_{\rm h, 1}, \mathcal{C}''_{\rm h, 1}\}$ and $\epsilon_{\rm h,1}=\min\{\epsilon'_{\rm h,1}, \epsilon''_{\rm h,1}\}$,
then, {if $\tilde{\epsilon}>0$ is sufficiently small 
 and $\tilde{\epsilon}<\epsilon_{I}\in (\mathcal{C}_{\rm h, 1}\tilde{\epsilon}, \epsilon_{\rm h,1})$,} the results of the lemma hold.
\end{proof}

{\color{black}Now, we turn to the estimates for $\eta_{\pm, 0}$ given by \eqref{eq:4.44}. For any point $(\xi,\eta)\in \tilde{\Omega}^{(\rm h)}_{I}\cap
\tilde{\Omega}^{(\rm h)}_{II}$ (or $(\xi,\eta)\in \tilde{\Omega}^{(\rm h)}_{I}\cap
\tilde{\Omega}^{(\rm h)}_{III}$), there exists a unique $\eta_{+,0}$,
such that the characteristic curve corresponding to $\lambda^{(\rm n-1)}_{+}$ (or $\lambda^{(\rm n-1)}_{-}$) issuing from $(0, \eta_{+,0})$ (or $(0, \eta_{-,0})$) passes through $(\xi,\eta)$.
Hence, by \eqref{eq:4.47c} and the implicit function theorem, from \eqref{eq:4.44} we can regard $\eta_{+, 0}$ (or $\eta_{+, 0}$) as a function of $(\xi,\eta)$,
\emph{i.e.}, $\eta_{+,0}=\eta_{+, 0}(\xi,\eta)$ (or $\eta_{-,0}=\eta_{-, 0}(\xi,\eta)$).}

\begin{lemma}\label{lem:4.2}
	For any $\alpha\in (0,1)$ and $\delta z^{(\rm n-1)}
\in \mathscr{K}_{2\epsilon_{I}}$,
	{\color{black}there exist constants $\epsilon_{\rm h,2}>0$, $\mathcal{C}_{\rm h,2}>0$ and
$C_{44}>0$ depending only on $\tilde{\underline{U}}$, $L$ and $\alpha$, such that for $\tilde{\epsilon}>0$ sufficiently small 
and $\tilde{\epsilon}<\epsilon_{I} \in (\mathcal{C}_{\rm h,2}\tilde{\epsilon}, \epsilon_{\rm h,2})$,}
	it holds that
	\begin{eqnarray}\label{eq:4.52}
	\begin{split}
	\|D\eta_{+, 0}\|_{0,\alpha; 
		\tilde{\Omega}^{(\rm h)}_{I}\cup\tilde{\Omega}^{(\rm h)}_{II}}
	+\|D\eta_{-,0}\|_{0,\alpha; 
		\tilde{\Omega}^{(\rm h)}_{I}\cup\tilde{\Omega}^{(\rm h)}_{III}}\leq C_{44}.
	\end{split}
	\end{eqnarray}
\end{lemma}

\begin{proof}
We only prove 
the estimate for $D\eta_{+,0}$ in $\tilde{\Omega}^{(\rm h)}_{I}\cup\tilde{\Omega}^{(\rm h)}_{II}$, 
since the estimate for $D\eta_{-,0}$ in $\tilde{\Omega}^{(\rm h)}_{I}\cup\tilde{\Omega}^{(\rm h)}_{III}$ can be treated in the same way.
{First, $\|D \eta_{+,0}\|_{0,0;\tilde{\Omega}^{(\rm h)}_{I}\cup\tilde{\Omega}^{(\rm h)}_{II}}$ can be estimated easily by choosing constants $\mathcal{C}_{\rm h,2}>0$ and $\epsilon_{\rm h,2}>0$ depending only on $\tilde{\underline{U}}$ and $L$ such that $\tilde{\epsilon}<\epsilon_{I} \in (\mathcal{C}_{\rm h,2}\tilde{\epsilon}, \epsilon_{\rm h,2})$ for $\tilde{\epsilon}>0$ sufficiently small as done in \cite[Lemma 4.1]{hkwx}.
So it remains to estimate $[D \eta_{+,0}]_{0,\alpha; \tilde{\Omega}^{(\rm h}_{I}\cup\tilde{\Omega}^{(\rm h)}_{II})}$.}
	Taking derivatives on \eqref{eq:4.44} with respect to $\xi$ and $\eta$, we have
	\begin{eqnarray}\label{eq:4.53x}
	\begin{split}
	&\frac{\partial \eta_{+,0}}{\partial \xi}=-\frac{\lambda^{(\rm n-1)}_{+}(\xi,\eta)}
	{1+\int^{\xi}_{0}\partial_{\eta}\lambda^{(\rm n-1)}_{+}
		e^{\int^{\tau}_{0}\partial_{\eta}\lambda^{(\rm n-1)}_{+}ds}d\tau},\\[5pt]
	&\frac{\partial \eta_{+,0}}{\partial \eta}=\frac{1}
	{1+\int^{\xi}_{0}\partial_{\eta}\lambda^{(\rm n-1)}_{+}
		e^{\int^{\tau}_{0}\partial_{\eta}\lambda^{(\rm -1)}_{+}ds}d\tau}.
	\end{split}
	\end{eqnarray}
	Taking two points $P_{1}=(\xi_{1}, \eta_{1})$ and $P_{2}=(\xi_{2}, \eta_{2})$, which satisfy \eqref{eq:4.45}
	with initial conditions {\color{black}$\eta_{P_{1,+}(0)}=\eta_{+,0}(P_{1})$ and $\eta_{P_{2,+}(0)}=\eta_{+,0}(P_{2})$,} respectively.
Then we need to estimate the term $d^{-\alpha}(P_{1},P_{2})\big|D \eta_{+,0}(P_{1})-D \eta_{+,0}(P_{2})\big|$.
	Notice that
	\begin{eqnarray*}
		\begin{split}
			&d^{-\alpha}(P_{1},P_{2})\big|D \eta_{+,0}(P_{1})-D \eta_{+,0}(P_{2})\big|\\[5pt]
			&\ \ \  \leq d^{-\alpha}(P_{1},P_{2})\big|\partial_{\xi}\eta_{+,0}(P_{1})-\partial_{\xi} \eta_{+,0}(P_{2})\big|
			+  d^{-\alpha}(P_{1},P_{2}) \big|\partial_{\eta}\eta_{+,0}(P_{1})-\partial_{\eta} \eta_{+,0}(P_{2})\big|\\[5pt]
			&\ \ \ =: I_{1}+I_{2}.
		\end{split}
	\end{eqnarray*}
	For $I_{1}$, by \eqref{eq:4.53x}, we have
	\begin{eqnarray*}
		\begin{split}
			I_{1}&=d^{-\alpha}(P_{1},P_{2})\Bigg|\frac{\lambda^{(\rm n-1)}_{+}(P_{1})}
			{1+\int^{\xi_{1}}_{0}\partial_{\eta}\lambda^{(\rm n-1)}_{+}e^{\int^{\tau}_{0}\partial_{\eta}\lambda^{(\rm n-1)}_{+}ds}d\tau}
			-\frac{\lambda^{(\rm n-1)}_{+}(P_{2})}{1+\int^{\xi_{2}}_{0}\partial_{\eta}\lambda^{(\rm n-1)}_{+}e^{\int^{\tau}_{0}\partial_{\eta}\lambda^{(\rm n-1)}_{+}ds}d\tau}\Bigg|\\[5pt]
			&\leq I_{11}+I_{12},
		\end{split}
	\end{eqnarray*}
	where
	\begin{eqnarray*}
		\begin{split}
			I_{11}=\Bigg|\frac{d^{-\alpha}(P_{1},P_{2})\lambda^{(\rm n-1)}_{+}(P_1)\Big(\int^{\xi_{2}}_{0}\partial_{\eta}\lambda^{(\rm n-1)}_{+}
				e^{\int^{\tau}_{0}\partial_{\eta}\lambda^{(\rm n-1)}_{+}ds}d\tau-\int^{\xi_{1}}_{0}\partial_{\eta}\lambda^{(\rm n-1)}_{+}
				e^{\int^{\tau}_{0}\partial_{\eta}\lambda^{(\rm n-1)}_{+}ds}d\tau\Big)}
			{\big(1+\int^{\xi_1}_{0}\partial_{\eta}\lambda^{(\rm n-1)}_{+}e^{\int^{\tau}_{0}\partial_{\eta}\lambda^{(\rm n-1)}_{+}ds}d\tau\big)
				\big(1+\int^{\xi_2}_{0}\partial_{\eta}\lambda^{(\rm n-1)}_{+}e^{\int^{\tau}_{0}\partial_{\eta}\lambda^{(\rm n-1)}_{+}ds}d\tau\big)}\Bigg|,
		\end{split}
	\end{eqnarray*}
	and
	\begin{eqnarray*}
		\begin{split}
			I_{12}&=d^{-\alpha}(P_{1},P_{2})\Bigg|\Big(1+\int^{\xi_{1}}_{0}\partial_{\eta}\lambda^{(\rm n-1)}_{+}
			e^{\int^{\tau}_{0}\partial_{\eta}\lambda^{(\rm n-1)}_{+}ds}d\tau\Big)^{-1}\Bigg|\big|\lambda^{(\rm n-1)}_{+}(P_1)-\lambda^{(\rm n-1)}_{+}(P_2)\big|.
		\end{split}
	\end{eqnarray*}
By Lemma \ref{lem:4.1} and by choosing constants still denoted by $\mathcal{C}_{\rm h,2}>0$ and $\epsilon_{\rm h,2}>0$ depending only on $\tilde{\underline{U}}$, $L$ and $\alpha$ such that {for
$\tilde{\epsilon}<\epsilon_{I} \in (\mathcal{C}_{\rm h,2}\tilde{\epsilon}, \epsilon_{\rm h,2})$ with $\tilde{\epsilon}>0$ is sufficiently small,} we have
	\begin{eqnarray*}
		\begin{split}
			I_{11}&=\mathcal{C}d^{-\alpha}(P_{1},P_{2})\bigg|\int^{\xi_{2}}_{\xi_{1}}\big(\partial_{\eta}\lambda^{(\rm n-1)}_{+}
			e^{\int^{\tau}_{0}\partial_{\eta}\lambda^{(\rm n-1)}_{+}ds}\big)(P_{2,+}(\tau))d\tau\bigg|\\[5pt]
			&\ \ \ +\mathcal{C}d^{-\alpha}(P_{1},P_{2})\int^{\xi_{1}}_{0}\Big|\big(\partial_{\eta}\lambda^{(\rm n-1)}_{+}
			e^{\int^{\tau}_{0}\partial_{\eta}\lambda^{(\rm n-1)}_{+}ds}\big)(P_{1,+}(\tau))-\big(\partial_{\eta}\lambda^{(\rm n-1)}_{+}
			e^{\int^{\tau}_{0}\partial_{\eta}\lambda^{(\rm n-1)}_{+}ds}\big)(P_{2,+}(\tau))\Big|d\tau\\[5pt]
			&\leq \mathcal{C},
		\end{split}
	\end{eqnarray*}
where $\mathcal{C}>0$ depends only on $\tilde{\underline{U}}$, $L$ and $\alpha$.
	Similarly, we can get the estimates for $I_{12}$ and $I_{2}$ to obtain $I_{12}+I_{2}\leq \mathcal{C}$ for some constant $\mathcal{C}>0$ depending only on $\tilde{\underline{U}}$, $L$ and $\alpha$.
	Then, with the estimates for $I_{11}$, $I_{12}$ and $I_{2}$, we have the desired estimate on $[D \eta_{+,0}]_{0,\alpha;
		\tilde{\Omega}^{(\rm h)}_{I}\cup\tilde{\Omega}^{(\rm h)}_{II}}$.
	This completes the proof of the lemma.
\end{proof}

Next, let us consider the characteristic curves issuing from the lower nozzle wall. Let $\xi^{*}_{\rm cd,0}$ be the $\xi$-coordinate of the intersection point of characteristic
$\eta=\phi^{(\rm n)}_{+}(\xi, {\color{black}-m^{(\rm h)}})$
and the transonic contact discontinuity $\eta=0$, and let {\color{black}$\xi^{*}_{-m^{(\rm h)},0}$} be the $\xi$-coordinate of the intersection point of the characteristics $\eta=\phi^{(\rm n)}_{-}(\xi, {\color{black}0})$
and the nozzle wall $\eta=-m^{(\rm h)}$, i.e.,
\begin{eqnarray}\label{eq:4.54}
\begin{split}
&\int^{\xi^{*}_{\rm cd, 0}}_{0}\lambda^{(\rm n-1)}_{+}(\tau,\phi^{(\rm n)}_{+}(\tau, -m^{(\rm h)}))d\tau=m^{(\rm h)},\\[5pt]
&\int^{{\color{black}\xi^{*}_{-m^{(\rm h)},0}}}_{0}\lambda^{(\rm n-1)}_{-}(\tau,\phi^{(\rm n)}_{-}(\tau, 0))d\tau= -m^{(\rm h)}.
\end{split}
\end{eqnarray}
Let $\xi=\chi^{(\rm n)}_{+}(\eta, {\color{black}\xi_{-m^{(\rm h)},0}})$ be the characteristic curves corresponding to $\lambda^{(\rm n-1)}_+$ issuing from the point $({\color{black}\xi_{-m^{(\rm h)},0}},-m^{(\rm h)})$, i.e., 
\begin{eqnarray}\label{eq:4.71}
\begin{split}
\left\{
\begin{array}{llll}
\frac{d\chi^{(\rm n)}_{+}}{d\eta}=\frac{1}{\lambda^{(\rm n-1)}_{+}(\chi^{(\rm n)}_{+}(\eta,  {\color{black}\xi_{-m^{(\rm h)},0}}),\eta)},\\[5pt]
\chi^{(\rm n)}_{+}(-m^{(\rm h)},{\color{black}\xi_{-m^{(\rm h)},0}})={\color{black}\xi_{-m^{(\rm h)},0}},
\end{array}
\right.
\end{split}
\end{eqnarray}
for ${\color{black}\xi_{-m^{(\rm h)},0}}\in [0, {\color{black}\xi^{*}_{-m^{(\rm h)},0}}]$, where ${\color{black}\xi^{*}_{-m^{(\rm h)},0}}$ is given by \eqref{eq:4.54}.
Integrating  \eqref{eq:4.71} from $-m^{(\rm h)}$ to $\eta$, one has 
\begin{eqnarray}\label{eq:4.72}
\begin{split}
\xi={\color{black}\xi_{-m^{(\rm h)},0}}+\int^{\eta}_{-m^{(\rm h)}}\frac{d\tau}{\lambda^{(\rm n-1)}_{+}(\chi^{(\rm n)}_{+}(\tau,{\color{black}\xi_{-m^{(\rm h)},0}}),\tau)}.
\end{split}
\end{eqnarray}

Given any two points $\hat{P}=(\xi_{\hat{P}},\eta_{\hat{P}}),\ \hat{Q}=(\xi_{\hat{Q}},\eta_{\hat{Q}}) \in 
\tilde{\Omega}^{(\rm h)}_{III}\cup \tilde{\Omega}^{(\rm h)}_{IV}\cup\tilde{\Omega}^{(\rm h)}_{V}$,
there exist unique points $\hat{P}_{+}(-m^{(\rm h)})=(\xi_{\hat{P}_{+}(-m^{(\rm h)})},-m^{(\rm h)})$ and $\hat{Q}_{+}(-m^{(\rm h)})=(\xi_{\hat{Q}_{+}(-m^{(\rm h)})},-m^{(\rm h)})$
such that the characteristics corresponding to $\lambda^{(\rm n-1)}_{+}$ issuing from them  pass through $\hat{P}$ and $\hat{Q}$, respectively.
Define $\hat{P}_{+}(\eta)=(\chi^{(\rm n)}_{+}(\eta, \xi_{\hat{P}_{+}(-m^{(\rm h)})}), \eta)$ and $\hat{Q}_{+}(\eta)=(\chi^{(\rm n)}_{+}(\eta, \xi_{\hat{Q}_{+}(-m^{(\rm h)})}), \eta)$. Then

\begin{eqnarray}\label{eq:4.73}
\begin{split}
&\xi_{\hat{P}_{+}(\eta)}=\chi^{(\rm n)}_{+}(\eta, \xi_{\hat{P}_{+}(-m^{(\rm h)})})=\xi_{\hat{P}_{+}(-m^{(\rm h)})}
+\int^{\eta}_{-m^{(\rm h)}}\frac{d\tau}{\lambda^{(\rm n-1)}_{+}(\hat{P}_{+}(\tau))}, \\[5pt]
&\xi_{\hat{Q}_{+}(\eta)}=\chi^{(\rm n)}_{+}(\eta, \xi_{\hat{Q}_{+}(-m^{(\rm h)})})=\xi_{\hat{Q}_{+}(-m^{(\rm h)})}
+\int^{\eta}_{-m^{(\rm h)}}\frac{d\tau}{\lambda^{(\rm n-1)}_{+}(\hat{Q}_{+}(\tau))}.
\end{split}
\end{eqnarray}
Now, we define the distance in the Euclidean norm for this case as
\begin{eqnarray}\label{eq:4.74}
\begin{split}
d\big(\hat{P}_{+}(\eta),\hat{Q}_{+}(\eta)\big)=\big|\hat{P}_{+}(\eta)-\hat{Q}_{+}(\eta)\big|,\quad  \ d^{\alpha}\big(\hat{P}_{+}(\eta),\hat{Q}_{+}(\eta)\big)=\big|\hat{P}_{+}(\eta)-\hat{Q}_{+}(\eta)\big|^{\alpha}.
\end{split}
\end{eqnarray}

We have the following lemma.
\begin{lemma}\label{lem:4.3}
	For any $\alpha\in (0,1)$ and $\delta z^{(\rm n-1)}
\in \mathscr{K}_{2\epsilon_{I}}$,
	there exist positive constants $\mathcal{C}_{\rm h,3}>0$, $\epsilon_{\rm h,3}>0$ and $\hat{C}_{4j}, (j=1,2,3)$ depending only on $\tilde{\underline{U}}$, $L$ and $\alpha$,
{\color{black}such that for $\tilde{\epsilon}>0$ sufficiently small 
and $\tilde{\epsilon}<\epsilon_{I} \in (\mathcal{C}_{\rm h,3}\tilde{\epsilon}, \epsilon_{\rm h,3})$,}  we have
	\begin{eqnarray}\label{eq:4.75}
	\begin{split}
	(\rm i)\ \ & d^{\alpha}\big(\hat{P}_{+}(\eta),\hat{Q}_{+}(\eta)\big)\leq \hat{C}_{41}d^{\alpha}\big(\hat{P},\hat{Q}\big),\quad & \eta\in[-m^{(\rm h)}, \eta_{\hat{P}\hat{Q}}], \\[5pt]
	(\rm ii)\ \ & d^{\alpha}\big(\hat{P}_{+}(\eta),\hat{Q}_{+}(\eta)\big)\leq\hat{C}_{42}d^{\alpha}\big(\hat{P}_{+}(\eta_{\hat{P}\hat{Q}}),\hat{Q}_{+}(\eta_{\hat{P}\hat{Q}})\big),\quad & \eta\in[-m^{(\rm h)}, \eta_{\hat{P}\hat{Q}}], \\[5pt]
	(\rm iii)\ \ & {\hat{C}^{-1}_{43}\big|\eta_{\hat{P}\hat{Q}}-\eta_{\hat{Q}}\big|^{\alpha}\leq d^{\alpha}\big(\hat{Q}_{+}(\eta_{\hat{P}\hat{Q}}),\hat{Q}\big)\leq \hat{C}_{43}\big|\eta_{\hat{P}\hat{Q}}-\eta_{\hat{Q}}\big|^{\alpha},}
	\end{split}
	\end{eqnarray}
where $\eta_{\hat{P}\hat{Q}}=\min\{\eta_{\hat{P}}, \ \eta_{\hat{Q}}\}$, and $\hat{P}_{+}(\eta_{\hat{P}\hat{Q}})=\big(\chi^{(\rm n)}_{+}(\eta_{\hat{P}\hat{Q}}, \xi_{\hat{P}_{+}(-m^{(\rm h)})}), \eta_{\hat{P}\hat{Q}}\big)$
and $\hat{Q}_{+}(\eta_{\hat{P}\hat{Q}})=\big(\chi^{(\rm n)}_{+}(\eta_{\hat{P}\hat{Q}}, \xi_{\hat{Q}_{+}(-m^{(\rm h)})}), \eta_{\hat{P}\hat{Q}}\big)$.
\end{lemma}

\begin{proof}
	Without loss of the generality, we only consider the case that $\eta_{\hat{P}\hat{Q}}=\eta_{\hat{P}}$ because the other case can be dealt similarly.
	First, taking derivatives with respect to $\xi_{-m^{(\rm h)},0}$ in \eqref{eq:4.72}, we have
	\begin{eqnarray*}
		\begin{split}
			\left\{
			\begin{array}{llll}
				\frac{\partial}{\partial \eta}\Big(\frac{\partial \chi^{(\rm n)}_{+}}{\partial \xi_{-m^{(\rm h)},0}}\Big)
				=\partial_{\xi }\Big(\frac{1}{\lambda^{(\rm n-1)}_{+}(\chi^{(\rm n)}_{+}(\eta,\xi_{-m^{(\rm h)},0}),\eta)}\Big)\frac{\partial \chi^{(\rm n)}_{+}}{\partial \xi_{-m^{(\rm h)},0}},\\[5pt]
				\frac{\partial \chi^{(\rm n)}_{+}}{\partial \xi_{-m^{(\rm h)},0}}(-m^{(\rm h)},\xi_{-m^{(\rm h)},0}) =1,
			\end{array}
			\right.
		\end{split}
	\end{eqnarray*}
	which implies that
	\begin{eqnarray*}
		\begin{split}
			\frac{\partial \chi^{(\rm n)}_{+}}{\partial \xi_{-m^{(\rm h)},0}}=e^{\int^{\eta}_{-m^{(\rm h)}}\partial_{\xi }\big(\frac{1}{\lambda^{(\rm n-1)}_{+}(\chi^{(\rm n)}_{+}(\tau,\xi_{-m^{(\rm h)},0}),\tau)}\big)d\tau}.
		\end{split}
	\end{eqnarray*}

{\color{black} By \eqref{eq:4.47a}-\eqref{eq:4.47c}, we can obtain
\begin{eqnarray}\label{eq:4.75a}
\begin{split}
\bigg\|\frac{1}{\lambda^{(\rm n-1)}_{\pm}}-\frac{1}{\underline{\lambda}_{\pm}}\bigg\|_{1,\alpha; \tilde{\Omega}^{(\rm h)}}
&\leq \mathcal{C}\Big(\|\delta z^{(\rm n-1)}\|_{1,\alpha;\tilde{\Omega}^{(\rm h)}}+\|\delta \tilde{B}^{(\rm h)}_{0}\|_{1,\alpha;\tilde{\Gamma}^{(\rm h)}_{\rm in}}+\|\delta\tilde{S}^{(\rm h)}_{0}\|_{1,\alpha;\tilde{\Gamma}^{(\rm h)}_{\rm in}}\Big).
\end{split}
\end{eqnarray}
Then, there exist some constants $\mathcal{C}_{\rm h,3}>0$ and $\epsilon_{\rm h,3}>0$ depending only on $\tilde{\underline{U}}$ and $L$, such that for $\tilde{\epsilon}>0$ sufficiently small and
$\tilde{\epsilon}<\epsilon_{I} \in (\mathcal{C}_{\rm h,3}\tilde{\epsilon}, \epsilon_{\rm h,3})$, we can further get
\begin{eqnarray}\label{eq:4.75b}
\begin{split}
\frac{1}{2|\underline{\lambda}_{\pm}|} <\bigg\|\frac{1}{\lambda^{(\rm n-1)}_{\pm}}\bigg\|_{1,\alpha; \tilde{\Omega}^{(\rm h)}}<\frac{3}{2|\underline{\lambda}_{\pm}|}.
\end{split}
\end{eqnarray}
}
Thus, for $0\leq\eta\leq\ \eta_{\hat{P}}$, {\color{black}by \eqref{eq:4.75b}, we have}
	\begin{eqnarray}\label{eq:4.76}
	\begin{split}
	\big|\chi^{(\rm n)}_{+}(\eta, \xi_{\hat{P}_{+}(-m^{(\rm h)})})-\chi^{(\rm n)}_{+}(\eta, \xi_{\hat{Q}_{+}(-m^{(\rm h)})})|
	\leq \mathcal{C}\big|\xi_{\hat{P}_{+}(-m^{(\rm h)})}-\xi_{\hat{Q}_{+}(-m^{(\rm h)})}\big|.
	\end{split}
	\end{eqnarray}
 On the other hand,  we can also deduce from \eqref{eq:4.73} that
	\begin{align}\label{eq:4.77}
	\begin{split}
	&\quad\ \big|\xi_{\hat{P}_{+}(-m^{(\rm h)})}-\xi_{\hat{Q}_{+}(-m^{(\rm h)})}\big|\\[5pt]
&\leq \big|\xi_{\hat{P}}-\xi_{\hat{Q}}\big|+\int^{\eta_{\hat{Q}}}_{\eta_{\hat{P}}}
	\bigg|\frac{1}{\lambda^{(\rm n-1)}_{+}\big(\hat{Q}_{+}(\tau)\big)}\bigg|d\tau\\[5pt]
	&\ \ \ \ +\int^{\eta_{\hat{P}}}_{-m^{(\rm h)}}\bigg|\frac{1}{\lambda^{(\rm n-1)}_{+}\big(\hat{Q}_{+}(\tau)\big)}
	-\frac{1}{\lambda^{(\rm n-1)}_{+}\big(\hat{P}_{+}(\tau)\big)}\bigg|d\tau\\[5pt]
	&\leq \mathcal{C}m^{(\rm h)}\bigg(\big\|D\delta z^{(\rm n-1)}\big\|_{0,0;\tilde{\Omega}^{(\rm h)}}
	+\big\|\big(\delta\tilde{B}^{(\rm h)'}_{0},\delta\tilde{S}^{(\rm h)'}_{0}\big)\big\|_{0,0;\tilde{\Gamma}^{(\rm h)}_{\rm in}}\bigg)\big|\xi_{\hat{P}_{+}(-m^{(\rm h)})}-\xi_{\hat{Q}_{+}(-m^{(\rm h)})}\big|\\[5pt]
	&\ \ \ \ +\big|\xi_{\hat{P}}-\xi_{\hat{Q}}\big| +\mathcal{C}\big|\eta_{\hat{P}}-\eta_{\hat{Q}}\big|\\[5pt]
	&\leq \mathcal{C}\Big(\big|\xi_{\hat{P}}-\xi_{\hat{Q}}\big|+\big|\eta_{\hat{P}}-\eta_{\hat{Q}}\big|\Big)
	+\mathcal{C}m^{(\rm h)}\big(2\epsilon_{I}+\tilde{\epsilon}\big)\big|\xi_{\hat{P}_{+}(-m^{(\rm h)})}-\xi_{\hat{Q}_{+}(-m^{(\rm h)})}\big|,
	\end{split}
	\end{align}
	where $\mathcal{C}$ is a positive constant depending only on $\tilde{\underline{U}}$ and $L$.
	Therefore, by choosing constants still denoted by $\mathcal{C}_{\rm h,3}>0$ and $\epsilon_{\rm h,3}>0$ depending only on $\tilde{\underline{U}}$ and $L$, such that {for $\tilde{\epsilon}>0$ sufficiently small 
and $\tilde{\epsilon}<\epsilon_{I} \in (\mathcal{C}_{\rm h,3}\tilde{\epsilon}, \epsilon_{\rm h,3})$, we have $\mathcal{C}m^{(\rm h)}\big(2\epsilon_{I}+\tilde{\epsilon}\big)<\frac{1}{2}$}, which implies that
	\begin{eqnarray}\label{eq:4.78}
	\begin{split}
	\big|\xi_{\hat{P}_{+}(-m^{(\rm h)})}-\xi_{\hat{Q}_{+}(-m^{(\rm h)})}\big|
	\leq 2\mathcal{C}\Big(\big|\xi_{\hat{P}}-\xi_{\hat{Q}}\big|+\big|\eta_{\hat{P}}-\eta_{\hat{Q}}\big|\Big).
	\end{split}
	\end{eqnarray}
	Hence, the estimate (\rm i) in \eqref{eq:4.75} follows from the estimates \eqref{eq:4.76} and \eqref{eq:4.78}.

Next, by taking $\eta_{\hat{P}}=\eta_{\hat{Q}}$ in \eqref{eq:4.77}, the estimate  (\rm ii) in \eqref{eq:4.75} follows from the estimates
	\eqref{eq:4.76} and \eqref{eq:4.78}.
	\par Finally, by \eqref{eq:4.73} with $k=\hat{Q}$, and the estimates \eqref{eq:4.47c} and \eqref{eq:4.75b}, we have 
	\begin{eqnarray*}
		\begin{split}
			d\big(\hat{Q}_{+}(\eta_{\hat{P}}), \hat{Q}\big)
			&=\big|\eta_{\hat{P}}-\eta_{\hat{Q}}\big|+\bigg|\int^{\eta_{\hat{Q}}}_{\eta_{\hat{P}}}
			\frac{d\tau}{\lambda^{(\rm n-1)}_{+}\big(\hat{Q}_{+}(\tau)\big)}\bigg|\\[5pt]
			&\leq \mathcal{C}\big|\eta_{\hat{P}}-\eta_{\hat{Q}}\big|,
		\end{split}
	\end{eqnarray*}

and
\begin{eqnarray*}
		\begin{split}
			d\big(\hat{Q}_{+}(\eta_{\hat{P}}), \hat{Q}\big)
			&=\big|\eta_{\hat{P}}-\eta_{\hat{Q}}\big|+\bigg|\int^{\eta_{\hat{Q}}}_{\eta_{\hat{P}}}
			\frac{d\tau}{\lambda^{(\rm n-1)}_{+}\big(\hat{Q}_{+}(\tau)\big)}\bigg|\\[5pt]
			&\geq \mathcal{C}^{-1}\big|\eta_{\hat{P}}-\eta_{\hat{Q}}\big|,
		\end{split}
	\end{eqnarray*}
	which gives   (\rm iii) in \eqref{eq:4.75}. The proof is completed.
\end{proof}

Based on the definition, for any point $(\xi,\eta)$ in $\tilde{\Omega}^{(\rm h)}_{III}\cap
\tilde{\Omega}^{(\rm h)}_{IV}\cup\tilde{\Omega}^{(\rm h)}_{V}$, there exists a unique $\xi_{-m^{(\rm h)}, 0}$,
such that the characteristics corresponding to $\lambda^{(\rm n-1)}_{+}$ issuing from $(\xi_{-m^{(\rm h)},0},-m^{(h)})$ passes through $(\xi,\eta)$.
{Hence, by \eqref{eq:4.75b} and implicit function theorem, we can regard $\xi_{-m^{(\rm h)}, 0}$ as a function of $(\xi,\eta)$ in $\tilde{\Omega}^{(\rm h)}_{III}\cap
\tilde{\Omega}^{(\rm h)}_{IV}\cup\tilde{\Omega}^{(\rm h)}_{V}$, \emph{i.e.}, $\xi_{-m^{(\rm h)},0}=\xi_{-m^{(\rm h)}, 0}(\xi,\eta)$.}
We have the following lemma. 

\begin{lemma}\label{lem:4.4}
For any $\alpha\in (0,1)$ and $\delta z^{(\rm n-1)}
\in \mathscr{K}_{2\epsilon_{I}}$, {\color{black}there exist constants $\mathcal{C}_{\rm h,4}>0$ and $\epsilon_{\rm h,4}>0$ depending only on $\tilde{\underline{U}}$, $L$ and $\alpha$,
such that for $\tilde{\epsilon}>0$ sufficiently small and $\tilde{\epsilon}<\epsilon_{I} \in (\mathcal{C}_{\rm h,4}\tilde{\epsilon}, \epsilon_{I,4})$, one has 
}
	\begin{eqnarray}\label{eq:4.79}
	\begin{split}
	\|D \xi_{-m^{(\rm h)},0}\|_{0, \alpha; 
		\tilde{\Omega}^{(\rm h)}_{III}
		\cup \tilde{\Omega}^{(\rm h)}_{IV}\cup\tilde{\Omega}^{(\rm h)}_{V}}\leq \hat{C}_{44},
	\end{split}
	\end{eqnarray}
where the constant $\hat{C}_{44}=\hat{C}_{44}(\underline{\tilde{U}}, L, \alpha)>0$ is independent of $\epsilon_{I}$ and $\rm n$.
\end{lemma}

\begin{proof}
{By choosing the {constants $\mathcal{C}_{\rm h,4}>0$ and $\epsilon_{\rm h,4}>0$ depending only on $\tilde{\underline{U}}$ and $L$,
such that for $\tilde{\epsilon}<\epsilon_{I} \in (\mathcal{C}_{\rm h,4}\tilde{\epsilon}, \epsilon_{\rm h,4})$ with $\tilde{\epsilon}>0$ sufficiently small},
the estimate for $\|D\xi_{-m^{(\rm h)},0}\|_{0,0; \tilde{\Omega}^{(\rm h)}_{III}\cap \tilde{\Omega}^{(\rm h)}_{IV}\cup\tilde{\Omega}^{(\rm h)}_{V}}$ follows easily as done in \cite[Lemma 4.2]{hkwx}.
Next, let us consider $[D\xi_{-m^{(\rm h)},0}]_{0,\alpha;\tilde{\Omega}^{(\rm h)}_{III}\cap \tilde{\Omega}^{(\rm h)}_{IV}\cup\tilde{\Omega}^{(\rm h)}_{V}}$.}
By $\eqref{eq:4.72}$, the straightforward computation shows   that
\begin{eqnarray*}
		\begin{split}
			&\frac{\partial \xi_{-m^{(\rm h)},0}}{\partial \xi}=\frac{1}{1+\int^{\eta}_{-m^{(\rm h)}}
				\partial_{\xi}\big(\frac{1}{\lambda^{(\rm n-1)}_{+}}\big)e^{\int^{\tau}_{-m^{(\rm h)}}
					\partial_{\xi}\big(\frac{1}{\lambda^{(\rm n-1)}_{+}}\big) ds} d\tau},\\[5pt]
			&\frac{\partial \xi_{-m^{(\rm h)},0}}{\partial \eta}=-\frac{1}{\Big(1+\int^{\eta}_{-m^{(\rm h)}}
				\partial_{\xi}\big(\frac{1}{\lambda^{(\rm n-1)}_{-}}\big)e^{\int^{\tau}_{-m^{(\rm h)}}
					\partial_{\xi}\big(\frac{1}{\lambda^{(\rm n-1)}_{-}}\big) ds} d\tau\Big)\lambda^{(\rm n-1)}_{+}}.
		\end{split}
\end{eqnarray*}
For the two points $\hat{P}=(\xi_{\hat{P}},\eta_{\hat{P}}),\ \hat{Q}=(\xi_{\hat{Q}},\eta_{\hat{Q}}) \in 
\tilde{\Omega}^{(\rm h)}_{III}\cup \tilde{\Omega}^{(\rm h)}_{IV}\cup\tilde{\Omega}^{(\rm h)}_{V}$,
there exist unique points $\hat{P}_{+}(-m^{(\rm h)})=(\xi_{\hat{P}_{+}(-m^{(\rm h)})},-m^{(\rm h)})=(\xi_{-m^{(\rm h)}, 0}(\hat{P}), -m^{(\rm h)})$ and $\hat{Q}_{+}(-m^{(\rm h)})=(\xi_{\hat{Q}_{+}(-m^{(\rm h)})},-m^{(\rm h)})=(\xi_{-m^{(\rm h)}, 0}(\hat{Q}), -m^{(\rm h)})$ so that the characteristics determined by $\lambda^{(\rm n-1)}_{+}$ issuing from them pass through $\hat{P}$ and $\hat{Q}$, respectively.
Direct computation yields that
\begin{eqnarray*}
		\begin{split}
			&d^{-\alpha}(\hat{P},\hat{Q})\Big|D \xi_{-m^{(\rm h)},0}(\hat{P})-D \xi_{-m^{(\rm h)},0}(\hat{Q})\Big|\\[5pt]
			\leq& d^{-\alpha}(\hat{P},\hat{Q})\Big|\partial_{\xi}\xi_{-m^{(\rm h)},0}(\hat{P})
			-\partial_{\xi} \xi_{-m^{(\rm h)},0}(\hat{Q})\Big|\\[5pt]
			&\qquad \
			+d^{-\alpha}(\hat{P},\hat{Q})\Big|\partial_{\eta}\xi_{-m^{(\rm h)},0}(\hat{P})
			-\partial_{\eta} \xi_{-m^{(\rm h)},0}(\hat{Q})\Big|\\[5pt]
=:&\hat{I}_{1}+\hat{I}_{2}.
		\end{split}
\end{eqnarray*}
For $\hat{I}_{1}$, by \eqref{eq:4.73}, we have
	\begin{eqnarray*}
		\begin{split}
&\hat{I}_{1}= d^{-\alpha}(\hat{P},\hat{Q})\big|\partial_{\xi}\xi_{-m^{(\rm h)},0}(\hat{P})-\partial_{\xi} \xi_{-m^{(\rm h)},0}(\hat{Q})\big|\\[5pt]
=&\frac{d^{-\alpha}(\hat{P},\hat{Q})\Big|\int^{\eta_{\hat{P}}}_{-m^{(\rm h)}}\partial_{\xi}\big(\frac{1}{\lambda^{(\rm n-1)}_{+}}\big)e^{\int^{\tau}_{-m^{(\rm h)}}
\partial_{\xi}\big(\frac{1}{\lambda^{(\rm n-1)}_{+}}\big) ds} d\tau -\int^{\eta_{\hat{Q}}}_{-m^{(\rm h)}}
\partial_{\xi}\big(\frac{1}{\lambda^{(\rm n-1)}_{+}}\big)e^{\int^{\tau}_{-m^{(\rm h)}}
\partial_{\xi}\big(\frac{1}{\lambda^{(\rm n-1)}_{+}}\big) ds} d\tau\Big|}{\bigg|\Big(1+\int^{\eta_{\hat{P}}}_{-m^{(\rm h)}}
\partial_{\xi}\big(\frac{1}{\lambda^{(\rm n-1)}_{+}}\big)e^{\int^{\tau}_{-m^{(\rm h)}}
\partial_{\xi}\big(\frac{1}{\lambda^{(\rm n-1)}_{+}}\big) ds} d\tau\Big)\Big(1+\int^{\eta_{\hat{Q}}}_{-m^{(\rm h)}}
\partial_{\xi}\big(\frac{1}{\lambda^{(\rm n-1)}_{+}}\big)e^{\int^{\tau}_{-m^{(\rm h)}}
\partial_{\xi}\big(\frac{1}{\lambda^{(\rm n-1)}_{+}}\big) ds} d\tau\Big)\bigg|}\\[5pt]
\leq& \mathcal{C}d^{-\alpha}\big(\hat{P},\hat{Q}\big)
\Bigg\{\bigg|\int^{\eta_{\hat{Q}}}_{\eta_{\hat{P}}}\partial_{\xi}\Big(\frac{1}{\lambda^{(\rm n-1)}_{+}(\hat{Q}_{+}(\tau))}\Big)e^{\int^{\tau}_{-m^{(\rm h)}}
\partial_{\xi}\big(\frac{1}{\lambda^{(\rm n-1)}_{+}}\big) ds}d\tau\bigg|\\[5pt]
&\quad +\int^{\eta_{\hat{P}}}_{-m^{(\rm h)}}\bigg|\partial_{\xi}\Big(\frac{1}{\lambda^{(\rm n-1)}_{+}(\hat{P}_{+}(\tau))}\Big)e^{\int^{\tau}_{-m^{(\rm h)}}
\partial_{\xi}\big(\frac{1}{\lambda^{(\rm n-1)}_{+}}\big) ds}-
\partial_{\xi}\Big(\frac{1}{\lambda^{(\rm n-1)}_{+}(\hat{Q}_{+}(\tau))}\Big)e^{\int^{\tau}_{-m^{(\rm h)}}
\partial_{\xi}\big(\frac{1}{\lambda^{(\rm n-1)}_{+}}\big) ds}\bigg|d\tau\Bigg\}.
		\end{split}
\end{eqnarray*}
Then, we can get, for $\tilde{\epsilon}<\epsilon_{I} \in (\mathcal{C}_{\rm h,4}\tilde{\epsilon}, \epsilon_{\rm h,4})$ with $\tilde{\epsilon}>0$ sufficiently small,
\begin{eqnarray*}
		\begin{split}
\hat{I}_{1}\leq \mathcal{C}\big|\eta_{\hat{P}}-\eta_{\hat{Q}}\big|^{1-\alpha}
+\mathcal{C}m^{(\rm h)}\bigg(\big\|D\delta z^{(\rm n-1)}\big\|_{0,0;\tilde{\Omega}^{(\rm h)}}
+\big\|\big(\delta\tilde{B}^{(\rm h)'}_{0},\delta\tilde{S}^{(\rm h)'}_{0}\big)\big\|_{0,0;\tilde{\Gamma}^{(\rm h)}_{\rm in}}\bigg)
\leq\mathcal{C},
		\end{split}
\end{eqnarray*}
{by \eqref{eq:4.75a} and by choosing constants $\mathcal{C}_{\rm h,4}>0$ and $\epsilon_{\rm h,4}>0$ depending only on $\tilde{\underline{U}}$ and $L$.} 
  Similarly, one can also get $\hat{I}_{2}\leq \mathcal{C}$. We therefore have
	\begin{eqnarray*}
		\begin{split}
			&[D\xi_{-m^{(\rm h)},0}]_{0,\alpha;
				\tilde{\Omega}^{(\rm h)}_{III}\cup\tilde{\Omega}^{(\rm h)}_{IV}
				\cup\tilde{\Omega}^{(\rm h)}_{V}}\\[5pt]
=&\sup_{\hat{P},\hat{Q} \in 
\tilde{\Omega}^{(\rm h)}_{III}\cup \tilde{\Omega}^{(\rm h)}_{IV}\cup\tilde{\Omega}^{(\rm h)}_{V}}
			d^{-\alpha}(\hat{P},\hat{Q})
			\Big|D \xi_{-m^{(\rm h)},0}(\hat{P})-D\xi_{-m^{(\rm h)},0}(\hat{Q})\Big|\\[5pt]
\quad& \leq \mathcal{C}.
		\end{split}
	\end{eqnarray*}
This completes the proof of this lemma.
\end{proof}

Now we consider the   characteristics issuing from the transonic contact discontinuity $\tilde{\Gamma}_{\rm cd}$. Let $\xi=\chi^{(\rm n)}_{-}(\eta, \xi_{\rm cd, 0})$ be the characteristic curves corresponding to
$\lambda^{(\rm n-1)}_-$, and issuing from point $(\xi_{\rm cd, 0},0)$, where $\xi_{\rm cd, 0}\in [0, \xi^{*}_{\rm cd,0}]$ and $\xi^{*}_{\rm cd,0}$ is given by \eqref{eq:4.54}, i.e., they are defined by
\begin{eqnarray}\label{eq:4.98}
\begin{split}
\left\{
\begin{array}{llll}
\frac{d\chi^{(\rm n)}_{-}}{d\eta}=\frac{1}{\lambda^{(\rm n-1)}_{-}(\chi^{(\rm n)}_{-}(\eta,\xi_{\rm cd, 0}),\eta)},\\[5pt]
\chi^{(\rm n)}_{-}(0,\xi_{\rm cd, 0})=\xi_{\rm cd, 0},
\end{array}
\right.
\end{split}
\end{eqnarray}
where $\xi_{\rm cd, 0}\in [0, \xi^{*}_{\rm cd,0}]$. Then, 
\begin{eqnarray}\label{eq:4.99}
\begin{split}
\xi=\xi_{\rm cd, 0}+\int^{\eta}_{0}\frac{d\tau}{\lambda^{(\rm n-1)}_{-}(\chi^{(\rm n)}_{-}(\tau,\xi_{\rm cd,0}),\tau)}.
\end{split}
\end{eqnarray}
For any two points $\check{P}=(\xi_{\check{P}},\eta_{\check{P}}),\ \check{Q}=(\xi_{\check{Q}},\eta_{\check{Q}}) \in 
\tilde{\Omega}^{(\rm h)}_{II}\cup \tilde{\Omega}^{(\rm h)}_{IV}\cup \tilde{\Omega}^{(\rm h)}_{VI}$, there exist unique points $\check{P}_{-}(0)=(\xi_{\check{P}_{-}(0)},0),\
\check{Q}_{-}(0)=(\xi_{\check{Q}_{-}(0)},0)$ so that the curve
$\xi=\chi^{(\rm n)}_{-}(\eta, \xi_{\rm cd, 0})$ issuing from
them pass through $\check{P}$ and $\check{Q}$, respectively. Define $\check{P}_{-}(\eta)=(\chi^{(\rm n)}_{-}(\eta, \xi_{\check{P}_{-}(0)}), \eta)$
and $\check{Q}_{-}(\eta)=(\chi^{(\rm n)}_{-}(\eta, \xi_{\check{Q}_{-}(0)}), \eta)$, then
\begin{eqnarray}\label{eq:4.100}
\begin{split}
&\xi_{\check{P}_{-}(\eta)}=\chi^{(\rm n)}_{-}(\eta, \xi_{\check{P}_{-}(0)})=\xi_{\check{P}_{-}(0)}+\int^{\eta}_{0}\frac{d\tau}{\lambda^{(\rm n-1)}_{-}(\check{P}_{-}(\tau))}, \\[5pt]
&\xi_{\check{Q}_{-}(\eta)}=\chi^{(\rm n)}_{-}(\eta, \xi_{\check{Q}_{-}(0)})=\xi_{\check{Q}_{-}(0)}+\int^{\eta}_{0}\frac{d\tau}{\lambda^{(\rm n-1)}_{-}(\check{Q}_{-}(\tau))}.
\end{split}
\end{eqnarray}
%
Define
\begin{align}\label{eq:4.101}
\begin{split}
d\big(\check{P}_{-}(\eta),\check{Q}_{-}(\eta)\big)=\big|\check{P}_{-}(\eta)-\check{Q}_{-}(\eta)\big|,\
\ d^{\alpha}\big(\check{P}_{-}(\eta),\check{Q}_{-}(\eta)\big)=\big|\check{P}_{-}(\eta)-\check{Q}_{-}(\eta)\big|^{\alpha}, \ \forall \alpha\in (0,1).
\end{split}
\end{align}
Then, we have the following lemma,  the proof  of which is analogous to that of Lemma \ref{lem:4.3} and thus omitted.
\begin{lemma}\label{lem:4.7}
For any $\alpha\in (0,1)$ and $\delta z^{(\rm n-1)}  
\in \mathscr{K}_{2\epsilon_{I}}$, {\color{black}there exist constants $\mathcal{C}_{\rm h,5}>0$ and $\epsilon_{\rm h,5}>0$ depending only on $\tilde{\underline{U}}$, $L$ and $\alpha$,
such that for $\tilde{\epsilon}>0$ sufficiently small and $\tilde{\epsilon}<\epsilon_{I} \in (\mathcal{C}_{\rm h,5}\tilde{\epsilon}, \epsilon_{I,5})$, the following holds}
	\begin{eqnarray}\label{eq:4.102}
	\begin{split}
	(\rm a)\ \ & d^{\alpha}\big(\check{P}_{-}(\eta),\check{Q}_{-}(\eta)\big)\leq \check{C}_{41}d^{\alpha}\big(\check{P},\check{Q}\big),\quad & \eta\in[0,\eta_{\check{P}\check{Q}}], \\[5pt]
	(\rm b)\ \ & d^{\alpha}\big(\check{P}_{-}(\eta),\check{Q}_{-}(\eta)\big)\leq\check{C}_{42}d^{\alpha}\big(\check{P}_{-}(\eta_{\check{P}\check{Q}}),\check{Q}_{-}(\eta_{\check{P}\check{Q}})\big),
	\quad & \eta\in[0,\eta_{\check{P}\check{Q}}], \\[5pt]
	(\rm c)\ \ & {\check{C}^{-1}_{43}\big|\eta_{\check{P}\check{Q}}-\eta_{\check{Q}}\big|^{\alpha}\leq d^{\alpha}\big(\check{Q}_{-}(\eta_{\check{P}\check{Q}}),\check{Q}\big)\leq \check{C}_{43}\big|\eta_{\check{P}\check{Q}}-\eta_{\check{Q}}\big|^{\alpha},}
	\end{split}
	\end{eqnarray}
where {\color{black} the constants $\check{C}_{4j}>0, (j=1,2,3)$ depend only on $\tilde{\underline{U}}$, $L$ and $\alpha$,} and $\eta_{\check{P}\check{Q}}=\min\{\eta_{\check{P}}, \ \eta_{\check{Q}}\}$ and $\check{P}_{-}(\eta_{\check{P}\hat{Q}})=\big(\chi^{(\rm n)}_{-}(\eta_{\check{P}\check{Q}}, \xi_{\check{P}_{-}(0)}), \eta_{\check{P}\check{Q}}\big)$ and $\check{Q}_{-}(\eta_{\check{P}\check{Q}})=\big(\chi^{(\rm n)}_{-}(\eta_{\check{P}\check{Q}}, \xi_{\check{Q}_{-}(0)}), \eta_{\check{P}\check{Q}}\big)$.

\end{lemma}

Finally, for any point $(\xi,\eta)$ in $
\tilde{\Omega}^{(\rm h)}_{II}\cup \tilde{\Omega}^{(\rm h)}_{IV}\cup \tilde{\Omega}^{(\rm h)}_{VI}$,
there exists a unique $\xi_{\rm cd, 0}$, such that the characteristics corresponding to $\lambda^{(\rm n-1)}_{-}$ issuing from $(\xi_{\rm cd,0},0)$ passes through $(\xi,\eta)$.
So we can regard $\xi_{\rm cd, 0}$ as a function of $(\xi,\eta)$ in $\tilde{\Omega}^{(\rm h)}_{II}\cup \tilde{\Omega}^{(\rm h)}_{IV}\cup \tilde{\Omega}^{(\rm h)}_{VI}$, \emph{i.e.},
$\xi_{\rm cd, 0}=\xi_{\rm cd,0}(\xi, \eta)$.
Then, we also have the following lemma, the  proof of which is similar to that of Lemma \ref{lem:4.4} and thus omitted as well.
\begin{lemma}\label{lem:4.6}
	For any $\alpha\in (0,1)$ and $\delta z^{(\rm n-1)}
\in \mathscr{K}_{2\epsilon_{I}}$, {\color{black}there exist constants $\mathcal{C}_{\rm h,6}>0$ and $\epsilon_{\rm h,6}>0$ depending only on $\tilde{\underline{U}}$, $L$ and $\alpha$,
such that for $\tilde{\epsilon}>0$ sufficiently small and $\tilde{\epsilon}<\epsilon_{I} \in (\mathcal{C}_{\rm h,6}\tilde{\epsilon}, \epsilon_{\rm h,6})$, one has}
	\begin{eqnarray}\label{eq:4.103}
	\begin{split}
	\|D \xi_{\rm cd, 0}\|_{0, \alpha; 
\tilde{\Omega}^{(\rm h)}_{II}\cup \tilde{\Omega}^{(\rm h)}_{IV}
		\cup \tilde{\Omega}^{(\rm h)}_{VI}}\leq \check{C}_{44},
	\end{split}
	\end{eqnarray}
{\color{black}where the constant $\check{C}_{44}>0$ depends only on $\tilde{\underline{U}}$, $L$ and $\alpha$.}
\end{lemma}

\subsubsection{Estimates of the solutions to the problem $(\widetilde{\mathbf{FP}})_{\rm n}$ in $\tilde{\Omega}^{(\rm h)}_{I}\cup\tilde{\Omega}^{(\rm h)}_{II}\cup\tilde{\Omega}^{(\rm h)}_{III}$}
In this subsection, we will consider the estimate of solutions in the supersonic region $\Omega^{(\rm h)}$ for the problem $(\widetilde{\mathbf{FP}})_{\rm n}$
near the background state $\underline{z}$, which 
can be written as the following:
\begin{eqnarray}\label{eq:4.42}
\begin{split}
\left\{
\begin{array}{llll}
\partial_{\xi}\delta z^{(n)}_{-}+\lambda^{(\rm n-1)}_{+}\partial_{\eta}\delta z^{(\rm n)}_{-}=0,
& \qquad \ \  \mbox{in}\ \ \ 
\tilde{\Omega}^{(\rm h)}_{I}\cup\tilde{\Omega}^{(\rm h)}_{II}\cup\tilde{\Omega}^{(\rm h)}_{III}, \\[5pt]
\partial_{\xi}\delta z^{(\rm n)}_{+}+\lambda^{(\rm n-1)}_{-}\partial_{\eta}\delta z^{(\rm n)}_{+}=0,
& \qquad \ \  \mbox{in}\ \ \ 
\tilde{\Omega}^{(\rm h)}_{I}\cup\tilde{\Omega}^{(\rm h)}_{II}\cup\tilde{\Omega}^{(\rm h)}_{III},\\[5pt]
\delta z^{(\rm n)}=\delta z_{0}, & \qquad \ \  \mbox{on} \ \ \ \tilde{\Gamma}^{(\rm h)}_{\rm in}.
\end{array}
\right.
\end{split}
\end{eqnarray}

\par Our main result in this subsection can be stated below. 
\begin{proposition}\label{prop:4.3}
For any given $\alpha\in(0,1)$, 
{\color{black}there exist constants $\mathcal{C}_{\rm h,7}>0$ and $\epsilon_{\rm h,7}>0$ depending only on $\tilde{\underline{U}}$, $L$ and $\alpha$,
such that for $\tilde{\epsilon}<\epsilon_{I} \in (\mathcal{C}_{\rm h,7}\tilde{\epsilon}, \epsilon_{\rm h,7})$ with $\tilde{\epsilon}>0$ sufficiently small,
if $\delta z^{(\rm n-1)}
\in \mathscr{K}_{2\epsilon_{I}}$,} the solution $\delta z^{(\rm n)}$ of the problem \eqref{eq:4.42} satisfies
\begin{eqnarray}\label{eq:4.55}
\begin{split}
\|\delta z^{(\rm n)}_{-}\|_{1,\alpha;
	\tilde{\Omega}^{(\rm h)}_{I}\cup\tilde{\Omega}^{(\rm h)}_{II}}
+\|\delta z^{(\rm n)}_{+}\|_{1,\alpha;
	\tilde{\Omega}^{(\rm h)}_{I}\cup\tilde{\Omega}^{(\rm h)}_{III}}
\leq \tilde{C}^{*}_{41}\Big(\|\delta z_{-,0}\|_{1,\alpha; \tilde{\Gamma}^{(\rm h)}_{\rm in}}+\|\delta z_{+,0}\|_{1,\alpha; \tilde{\Gamma}^{(\rm h)}_{\rm in}}\Big),
\end{split}
\end{eqnarray}
where $\tilde{C}^{*}_{41}>0$ depends only on $\tilde{\underline{U}}$, $L$ and $\alpha$.
\end{proposition}

\begin{proof}
We only consider the estimate for 
the solution $z^{(\rm n)}_{-}$ in $\tilde{\Omega}^{(\rm h)}_I\cup\tilde{\Omega}^{(\rm h)}_{II}$,
because the proof for $z^{(\rm n)}_{+}$ in $\tilde{\Omega}^{(\rm h)}_I\cup\tilde{\Omega}^{(\rm h)}_{III}$ is similar.
{\color{black}First, for any point $(\xi, \eta)\in \tilde{\Omega}^{(\rm h)}_I\cup\tilde{\Omega}^{(\rm h)}_{II}$, we can draw a characteristic line $\eta=\phi^{(\rm n)}_{+}(\xi,\eta_{+,0})$
 defined by \eqref{eq:4.43}-\eqref{eq:4.44} such that it intersects with  $\tilde{\Gamma}^{(\rm h)}_{\rm in}$ at the point $(0, \eta_{+,0})$. Then, one has $\delta z^{(\rm n)}_{-}(\xi, \eta)=\delta z_{-,0}(\eta_{+,0})$. Moreover, differentiating the equation $\eqref{eq:4.42}_{1}$ with respect to $\xi$ and $\eta$, and then integrating them along the characteristics $\eta=\phi^{(\rm n)}_{+}(\xi,\eta_{+,0})$, we have
\begin{eqnarray}\label{eq:4.55a}
\begin{split}
\partial_{\xi} \delta z^{(\rm n)}_{-}(\xi,\eta)&=\partial_{\xi}\delta z^{(\rm n)}_{-}(0, \eta_{+,0})
-\int^{\xi}_{0}\big(\partial_{\tau}\lambda^{(\rm n-1)}_{+}\partial_{\eta}\delta z^{(\rm n)}_{-}\big)(\tau, \phi^{(\rm n)}_+(\tau,\eta_{+,0}))d\tau\\[5pt]
&=\partial_{\xi} \eta_{+,0}\partial_{\eta_{+,0}}\delta z_{-}(0,\eta_{+,0})
-\int^{\xi}_{0}\big(\partial_{\tau}\lambda^{(\rm n-1)}_{+}\partial_{\eta}\delta z^{(\rm n)}_{-}\big)(\tau, \phi^{(\rm n)}_+(\tau,\eta_{+,0}))d\tau,
\end{split}
\end{eqnarray}
and
\begin{eqnarray}\label{eq:4.55b}
\begin{split}
\partial_{\eta} \delta z^{(\rm n)}_{-}(\xi, \eta)&=\partial_{\eta}\delta z^{(\rm n)}_{-}(0, \eta_{+,0})
-\int^{\xi}_{0}\big(\partial_{\eta}\lambda^{(\rm n-1)}_{+}\partial_{\eta}\delta z^{(\rm n)}_{-}\big)(\tau, \phi^{(\rm n)}_+(\tau,\eta_{+,0}))d\tau\\[5pt]
&=\partial_{\eta} \eta_{+,0}\partial_{\eta_{+,0}}\delta z_{-}(0,\eta_{+,0})
-\int^{\xi}_{0}\big(\partial_{\eta}\lambda^{(\rm n-1)}_{+}\partial_{\eta}\delta z^{(\rm n)}_{-}\big)(\tau, \phi^{(\rm n)}_+(\tau,\eta_{+,0}))d\tau.
\end{split}
\end{eqnarray}
}

Therefore, we can follow the approach in \cite{hkwx} for proving the Proposition 4.1 and apply the Lemma \ref{lem:4.2} as well as the estimate \eqref{eq:4.47c} to get 
\begin{eqnarray}\label{eq:4.62}
\begin{split}
\big\|\delta z^{(\rm n)}_{-}\big\|_{1,0;
\tilde{\Omega}^{(\rm h)}_I\cup\tilde{\Omega}^{(\rm h)}_{II}}\leq
2\mathcal{C}\big\|\delta z_{-,0}\big\|_{1,0;\tilde{\Gamma}^{(\rm h)}_{\rm in}},
\end{split}
\end{eqnarray}
provided that {$\tilde{\epsilon}<\epsilon_{I}\in (\mathcal{C}'_{\rm h,7}\tilde{\epsilon}, \epsilon'_{\rm h,7})$ with $\tilde{\epsilon}>0$ is sufficiently small,} where the constants $\mathcal{C}'_{\rm h,7}>\max\{\mathcal{C}_{\rm h,1}, \mathcal{C}_{\rm h,2}\}$, $0<\epsilon'_{\rm h,7}<\min\{\epsilon_{\rm h,1}, \epsilon_{\rm h,2}\}$ and
$\mathcal{C}>0$ depends only on 
$\tilde{\underline{U}}$ and $L$.

\par Next, let us turn to  
the estimate of $[D\delta z^{(\rm n)}_{-}]_{0,\alpha; \tilde{\Omega}^{(\rm h)}
\cap(\tilde{\Omega}_I\cup\tilde{\Omega}_{II})}$.
For the two points $P=(\xi_{P},\eta_{P})$, $Q=(\xi_{Q},\eta_{Q})\in \tilde{\Omega}^{(\rm h)}_I\cup \tilde{\Omega}^{(\rm h)}_{II}$ lying on the
characteristics corresponding to $\lambda^{(\rm n-1)}_{+}$ and intersecting with $\tilde{\Gamma}^{(\rm h)}_{\rm in}$ at
the points $P_{+}(0)=(0,\eta_{P_{+}(0)})$ and $Q_{+}(0)=(0,\eta_{Q_{+}(0)})$, respectively,
  by \eqref{eq:4.55a}-\eqref{eq:4.55b}, we have
\begin{eqnarray}\label{eq:4.63}
\begin{split}
\partial_{\xi} \delta z^{(\rm n)}_{-}(P)
&=\partial_{\xi}\eta_{+,0}(P)\partial_{\eta_{+,0}}\delta z_{-}(P_{+}(0))
-\int^{\xi_{P}}_{0}\big(\partial_{\tau}\lambda^{(\rm n-1)}_{+}\partial_{\eta}
\delta z^{(\rm n)}_{-}\big)(P_{+}(\tau))d\tau,\\[5pt]
\partial_{\eta} \delta z^{(\rm n)}_{-}(P)
&=\partial_{\eta}\eta_{+,0}(P)\partial_{\eta_{+,0}}\delta z_{-}(P_{+}(0))
-\int^{\xi_{P}}_{0}\big(\partial_{\eta}\lambda^{(\rm n-1)}_{+}\partial_{\eta}\delta z^{(\rm n)}_{-}\big)(P_{+}(\tau))d\tau,
\end{split}
\end{eqnarray}
and
\begin{eqnarray}\label{eq:4.64}
\begin{split}
\partial_{\xi} \delta z^{(\rm n)}_{-}(Q)
&=\partial_{\xi}\eta_{+,0}(Q)\partial_{\eta_{+,0}}\delta z_{-}(Q_{+}(0))
-\int^{\xi_{Q}}_{0}\big(\partial_{\tau}\lambda^{(\rm n-1)}_{+}\partial_{\eta}
\delta z^{(\rm n)}_{-}\big)(Q_{+}(\tau))d\tau,\\[5pt]
\partial_{\eta} \delta z^{(\rm n)}_{-}(Q)
&=\partial_{\eta}\eta_{+,0}(Q)\partial_{\eta_{+,0}}\delta z_{-}(Q_{+}(0))
-\int^{\xi_{Q}}_{0}\big(\partial_{\eta}\lambda^{(\rm n-1)}_{+}\partial_{\eta}\delta z^{(\rm n)}_{-}\big)(Q_{+}(\tau))d\tau,
\end{split}
\end{eqnarray}
where $\eta_{P_{+}(0)}=\eta_{+,0}(P)$ and $\eta_{Q_{+}(0)}=\eta_{+,0}(Q)$.
%

Without loss of the generality, we asume that $\xi_{P}<\xi_{Q}$ and
let $Q_{+}(\xi_{P})=(\xi_{P}, \phi^{(\rm n)}_{+}(\xi_{P},\eta_{Q_{+}(0)}))$.
Then, by a direct computation and by Lemma \ref{lem:4.1}, one can get that
\begin{align}
&d^{-\alpha}(P,Q)\Big|D\delta z^{(\rm n)}_{-}(P)-D\delta z^{(\rm n)}_{-}(Q)\Big|\nonumber\\[5pt]
& \quad \ \  \leq d^{-\alpha}(P,Q)\Big|D\delta z^{(\rm n)}_{-}(P)-D\delta z^{(\rm n)}_{-}(Q_{+}(\xi_{P}))\Big|\nonumber\\[5pt]
& \quad \quad \quad \ \ + d^{-\alpha}(P,Q)\Big|D\delta z^{(\rm n)}_{-}(Q_{+}(\xi_{P}))-D\delta z^{(\rm n)}_{-}(Q)\Big|\label{eq:4.65}\\[5pt]
& \quad \ \  \leq \mathcal{C}d^{-\alpha}\big(P,Q_{+}(\xi_{P})\big)\Big|D\delta z^{(\rm n)}_{-}(P)-D\delta z^{(\rm n)}_{-}(Q_{+}(\xi_{P}))\Big|\nonumber\\[5pt]
& \quad \quad \quad \ \ + \mathcal{C}d^{-\alpha}\big(Q_{+}(\xi_{P}),Q\big)\Big|D\delta z^{(\rm n)}_{-}(Q_{+}(\xi_{P}))-D\delta z^{(\rm n)}_{-}(Q)\Big|\nonumber\\[5pt]
& \quad \ \  =:\mathcal{C}(J_{1}+J_{2}),\nonumber
\end{align}
where constant $\mathcal{C}>0$ depends only on $\tilde{\underline{U}}$, $L$ and $\alpha$.

Now, we will estimate $J_{1}$ and $J_{2}$. 
For the term $J_{1}$, by \eqref{eq:4.63} and \eqref{eq:4.55a}-\eqref{eq:4.55b} for $Q_{+}(\xi_{p})$, we have
\begin{eqnarray*}
\begin{split}
J_{1}&\leq d^{-\alpha}\big(P,Q_{+}(\xi_{P})\big)\Big|D\eta_{+,0}(P)\partial_{\eta_{+,0}}\delta z_{-}(P_{+}(0))
-D\eta_{+,0}(Q_{+}(\xi_{P}))\partial_{\eta_{+,0}}\delta z_{-}(Q_{+}(0))\Big|\\[5pt]
&\quad \ \ +d^{-\alpha}\big(P,Q_{+}(\xi_{P})\big)\int^{\xi_{P}}_{0}\Big|\big(D\lambda^{(\rm n-1)}_{+}\partial_{\tau}\delta z^{(\rm n)}_{-}\big)(P_{+}(\tau))
-\big(D\lambda^{(\rm n-1)}_{+}\partial_{\tau}\delta z^{(\rm n)}_{-}\big)(Q_{+}(\tau))\Big|d\tau\\[5pt]
&=: J_{11}+J_{12}.
\end{split}
\end{eqnarray*}
For  $J_{11}$, it follows from Lemma \ref{lem:4.1} and Lemma \ref{lem:4.2} that
\begin{eqnarray*}
\begin{split}
J_{11} &=d^{-\alpha}\big(P,Q_{+}(\xi_{P})\big)\Big|D\eta_{+,0}(P)\partial_{\eta_{+,0}}\delta z_{-}(P_{+}(0))-D\eta_{+,0}(Q_{+}(\xi_{P}))\partial_{\eta_{+,0}}\delta z_{-}(Q_{+}(0))\Big|\\[5pt]
&\leq d^{-\alpha}\big(P,Q_{+}(\xi_{P})\big)\Big|D\eta_{+,0}(P)-D\eta_{+,0}(Q_{+}(\xi_{P}))\Big| \big|\partial_{\eta_{+,0}}\delta z_{-}(P_{+}(0))\big|\\[5pt]
&\quad \ \ +d^{-\alpha}\big(P,Q_{+}(\xi_{P})\big)\Big|\partial_{\eta_{+,0}}\delta z_{-}(P_{+}(0))-\partial_{\eta_{+,0}}\delta z_{-}(Q_{+}(0))\Big| \big|D\eta_{0,+}(Q_{+}(\xi_{P}))\big|\\[5pt]
&\leq \big|\partial_{\eta_{+,0}}\delta z_{-}(P_{+}(0))\big|d^{-\alpha}\big(P,Q_{+}(\xi_{P})\big)\Big|D\eta_{+,0}(P)-D\eta_{+,0}(Q_{+}(\xi_{P}))\Big| \\[5pt]
&\quad \ \ +\mathcal{C}\big|D\eta_{+,0}(Q_{+}(\xi_{P}))\big|d^{-\alpha}\big(P_{+}(0),Q_{+}(0)\big)\Big|\partial_{\eta_{+,0}}\delta z_{-}(P_{+}(0))-\partial_{\eta_{+,0}}\delta z_{-}(Q_{+}(0))\Big| \\[5pt]
&\leq \mathcal{C}\big\|\delta z'_{-,0}\big\|_{0,\alpha;\tilde{\Gamma}^{(\rm h)}_{\rm in}},
\end{split}
\end{eqnarray*}
where $\mathcal{C}=\mathcal{C}(\tilde{\underline{U}}, L,\alpha)>0$ is independent of $\epsilon_{I}$ and $\rm n$.
Next, for $J_{12}$, a direct computation shows that
\begin{eqnarray*}
\begin{split}
J_{12}&=d^{-\alpha}\big(P,Q_{+}(\xi_{P})\big)\int^{\xi_{P}}_{0}\Big|\Big(D\lambda^{(\rm n-1)}_{+}\partial_{\tau}\delta z^{(\rm n)}_{-}\Big)(P_{+}(\tau))-\Big(D\lambda^{(\rm n-1)}_{+}\partial_{\tau}\delta z^{(\rm n)}_{-}\Big)(Q_{+}(\tau))\Big|d\tau\\[5pt]
&\leq d^{-\alpha}\big(P,Q_{+}(\xi_{P})\big)\int^{\xi_{P}}_{0}\Big|\Big(D\lambda^{(\rm n-1)}_{+}(P_{+}(\tau))
-D\lambda^{(\rm n-1)}_{+}(Q_{+}(\tau))\Big)\partial_{\tau}\delta z^{(\rm n)}_{-}(P_{+}(\tau))\Big|d\tau\\[5pt]
&\quad \ \ +d^{-\alpha}\big(P,Q_{+}(\xi_{P})\big)\int^{\xi_{P}}_{0}\Big|\Big(\partial_{\tau}\delta z^{(\rm n)}_{-}(P_{+}(\tau))
-\partial_{\tau}\delta z^{(\rm n)}_{-}(Q_{+}(\tau))\Big) D\lambda^{(\rm n-1)}_{+}(Q_{+}(\tau))\Big|d\tau.
\end{split}
\end{eqnarray*}
Then, by the estimates \eqref{eq:4.47c} and \eqref{eq:4.62}, we obtain
\begin{eqnarray*}
\begin{split}
J_{12}&\leq \mathcal{C}\|\lambda^{(\rm n-1)}_{+}\|_{1,\alpha; \tilde{\Omega}^{(\rm h)}}
\int^{\xi_{P}}_{0}\big|D\delta z^{(\rm n)}_{-}(P(\tau))\big|d\tau\\[5pt]
&\quad \ \ +\mathcal{C}\Big(\big\|\delta z^{(\rm n-1)}\big\|_{1,0; \tilde{\Omega}^{(\rm h)}}
+\big\|\big(\delta\tilde{B}^{(\rm h)'}_{0},\delta\tilde{S}^{(\rm h)'}_{0}\big)\big\|_{1,0;\tilde{\Gamma}^{(\rm h)}_{\rm in}}\Big)\\[5pt]
&\qquad \quad \ \times \int^{\xi_{P}}_{0}d^{-\alpha}\big(P_{+}(\tau),Q_{+}(\tau)\big)\Big|D\delta z^{(\rm n)}_{-}(P_{+}(\tau))-D\delta z^{(\rm n)}_{-}(Q_{+}(\tau))\Big|d\tau\\[5pt]
&\leq \mathcal{C}{\big\|\delta z_{-,0}\big\|_{1,0;\tilde{\Gamma}^{(\rm h)}_{\rm in}}}
+\mathcal{C}\big(2\epsilon_{I}+\tilde{\epsilon}\big)\int^{\xi_{P}}_{0}d^{-\alpha}\big(P_{+}(\tau),Q_{+}(\tau)\big)\Big|D\delta z^{(\rm n)}_{-}(P_{+}(\tau))-D\delta z^{(\rm n)}_{-}(Q_{+}(\tau))\Big|d\tau,
\end{split}
\end{eqnarray*}
where $\mathcal{C}=\mathcal{C}(\tilde{\underline{U}}, L,\alpha)>0$ is independent of $\epsilon_{I}$ and $\rm n$.

Combining the estimates on $J_{11}$ with $J_{12}$, we can deduce by the Gronwall inequality that
\begin{eqnarray*}
\begin{split}
J_{1}=d^{-\alpha}\big(P,Q_{+}(\xi_{P})\big)\big|D\delta z^{(\rm n)}_{-}(P)-D\delta z^{(\rm n)}_{-}(Q_{+}(\xi_{P}))\big|
\leq \mathcal{C}\big\|\delta z'_{-,0}\big\|_{0,\alpha;\tilde{\Gamma}^{(\rm h)}_{\rm in}}e^{2\mathcal{C}(2\epsilon_{I}+\tilde{\epsilon})\xi_{P}},
\end{split}
\end{eqnarray*}
which implies
\begin{equation}\label{eq:4.66}
\begin{split}
J_{1}&\leq \sup_{P,Q_{+}(\xi_{P})\in
	\tilde{\Omega}^{(\rm h)}_I\cup\tilde{\Omega}^{(\rm h)}_{II}}
d^{-\alpha}\big(P,Q_{+}(\xi_{P})\big)\Big|D\delta z^{(\rm n)}_{-}(P)-D\delta z^{(\rm n)}_{-}(Q_{+}(\xi_{P}))\Big|\\[5pt]
&\leq 2\mathcal{C}{\big\|\delta z_{-,0}\big\|_{1,\alpha;\tilde{\Gamma}^{(\rm h)}_{\rm in}}}.
\end{split}
\end{equation}
Here, we also choose the constants $\mathcal{C}''_{\rm h,7}>0$ and $\epsilon''_{\rm h,7}>0$ depends only on 
$\tilde{\underline{U}}$ and $L$, such that for $\tilde{\epsilon}>0$ sufficiently small and $\tilde{\epsilon}<\epsilon_{I}\in (\mathcal{C}''_{\rm h,7}\tilde{\epsilon}, \epsilon''_{\rm h,7})$,
it holds $e^{2\mathcal{C}L(2\epsilon_{I}+\tilde{\epsilon})}\leq 2$.

\par The remaining task is to deal with $J_{2}$.
By the equations in \eqref{eq:4.64}, \eqref{eq:4.55a}-\eqref{eq:4.55b} for $Q_{+}(\xi_{p})$ and the estimate \eqref{eq:4.47c}, and Lemmas \ref{lem:4.1}-\ref{lem:4.2}, we have
{\color{black}
\begin{align}\label{eq:4.67}
\begin{split}
J_{2}&= d^{-\alpha}\big(Q_{+}(\xi_{P}),Q\big)
\Big|D\delta z^{(\rm n)}_{-}(Q_{+}(\xi_{P}))-D\delta z^{(\rm n)}_{-}(Q)\Big|\\[5pt]
&\leq d^{-\alpha}\big(Q_{+}(\xi_{P}),Q\big)\Big|D\eta_{+,0}(Q_{+}(\xi_{P}))-D\eta_{+,0}(Q)\Big||\partial_{\eta_{+,0}}\delta z_{-}(Q_{+}(0))|\\[5pt]
&\quad +d^{-\alpha}\big(Q_{+}(\xi_{P}),Q\big)\int^{\xi_{Q}}_{\xi_{P}}\Big|\big(D\lambda^{(\rm n-1)}_{+}\partial_{\eta}\delta z^{(\rm n)}_{-}\big)(Q_{+}(\tau))\Big|d\tau\\[5pt]
&\leq \mathcal{C}\Big(\|D\eta_{+,0}\|_{0, \alpha; \tilde{\Omega}^{(\rm h)}_{I}\cup\tilde{\Omega}^{(\rm h)}_{II}}\big\|\delta z'_{-,0}\big\|_{0,0;\tilde{\Gamma}^{(\rm h)}_{\rm in}}
+\big|\xi_{P}-\xi_{Q}\big|^{1-\alpha}
\big\|\lambda^{(\rm n-1)}_{+}\big\|_{1,0;\tilde{\Omega}^{(\rm h)}}
\big\|\delta z^{(\rm n-1)}\big\|_{1,0;
	\tilde{\Omega}^{(\rm h)}_I\cup\tilde{\Omega}^{(\rm h)}_{II}}\Big)\\[5pt]
&\leq \mathcal{C}\big\|\delta z_{-,0}\big\|_{1,0;\tilde{\Gamma}^{(\rm h)}_{\rm in}}.
\end{split}
\end{align}
}
  Combining all the estimates from \eqref{eq:4.65} to \eqref{eq:4.67} above together,
we therefore get that
\begin{eqnarray}\label{eq:4.68}
\begin{split}
&[D\delta z^{(\rm n)}_{-}]_{0,\alpha;
	\tilde{\Omega}^{(\rm h)}_I\cup\tilde{\Omega}^{(\rm h)}_{II}}\\[5pt]
=&\sup_{P,Q\in
	\tilde{\Omega}^{(\rm h)}_I\cup\tilde{\Omega}^{(\rm h)}_{II}}
d^{-\alpha}(P,Q)\Big|D\delta z^{(\rm n)}_{-}(P)-D\delta z^{(\rm n)}_{-}(Q)\Big|\\[5pt]
\leq &\mathcal{C}\sup_{P,Q_{+}(\xi_{P})\in
	\tilde{\Omega}^{(\rm h)}_I\cup\tilde{\Omega}^{(\rm h)}_{II}}
d^{-\alpha}\big(P,Q_{+}(\xi_{P})\big)\Big|D\delta z^{(\rm n)}_{-}(P)-D\delta z^{(\rm n)}_{-}(Q_{+}(\xi_{P}))\Big|\\[5pt]
&\quad +\sup_{Q_{+}(\xi_{P}),Q\in
	\tilde{\Omega}^{(\rm h)}_I\cup\tilde{\Omega}^{(\rm h)}_{II}}
d^{-\alpha}(Q_{+}(\xi_{P}),Q\big)\Big|D\delta z^{(\rm n)}_{-}(Q_{+}(\xi_{P}))-D\delta z^{(\rm n)}_{-}(Q)\Big|\\[5pt]
\leq& \mathcal{C}\big\|\delta z'_{-,0}\big\|_{0,\alpha;\tilde{\Gamma}^{(\rm h)}_{\rm in}},
\end{split}
\end{eqnarray}
where 
the constant $\mathcal{C}>0$ depends only on $\tilde{\underline{U}}$, $L$ and $\alpha$.

\par Finally, it follows from the estimates \eqref{eq:4.62} and \eqref{eq:4.68} that 
{we can choose the constants $\tilde{C}^{*}_{41}>0$, $\mathcal{C}_{\rm h,7}=\max\{\mathcal{C}'_{\rm h,7}, \mathcal{C}''_{\rm h,7}\}$ and
$\epsilon_{\rm h,7}=\min\{\epsilon'_{\rm h,7},\epsilon''_{\rm h,7} \}$ depending  only on $\tilde{\underline{U}}$, $L$ and $\alpha$
such that for $\tilde{\epsilon}<\epsilon_{I} \in (\mathcal{C}_{\rm h,7}\tilde{\epsilon}, \epsilon_{\rm h,7})$ with $\tilde{\epsilon}>0$ sufficiently small, }
it holds that
\begin{eqnarray*}
\begin{split}
\|\delta z^{(\rm n)}_{-}\|_{1,\alpha;
	\tilde{\Omega}^{(\rm h)}_I\cup\tilde{\Omega}^{(\rm h)}_{II}}
\leq  \tilde{C}^{*}_{41} \|\delta z_{-,0}\|_{1,\alpha;\tilde{\Gamma}^{(\rm h)}_{\rm in}}.
\end{split}
\end{eqnarray*}
The proof of this proposition is completed.
\end{proof}

\begin{remark}\label{rem:4.2}
By the estimate \eqref{eq:4.55}, we actually have the estimates of $\delta z_{\pm}^{(\rm n)}$ in $\tilde{\Omega}^{(\rm h)}_I$. Moreover, 
on the contact discontinuity $\tilde{\Gamma}_{\rm cd}\cap\tilde{\Omega}^{(\rm h)}_{II}$
and the lower wall of the nozzle $\tilde{\Gamma}_{-}\cap\tilde{\Omega}^{(\rm h)}_{III}$, we also have
\begin{eqnarray}\label{eq:4.69}
\begin{split}
\big\|\delta z^{(\rm n)}_{-}\big\|_{1,\alpha;\tilde{\Gamma}_{\rm cd}\cap\tilde{\Omega}^{(\rm h)}_{II}}
+\big\|\delta z^{(\rm n)}_{+}\big\|_{1,\alpha;\tilde{\Gamma}_{-}\cap\tilde{\Omega}^{(\rm h)}_{III}}
\leq \tilde{C}^{*}_{42}\bigg(\big\|\delta z_{-,0}\big\|_{1,\alpha; \tilde{\Gamma}^{(\rm h)}_{\rm in}}
+\big\|\delta z_{+,0}\big\|_{1,\alpha; \tilde{\Gamma}^{(\rm h)}_{\rm in}}\bigg).
\end{split}
\end{eqnarray}
\end{remark}

\subsubsection{Estimates of the solutions to the problem $(\widetilde{\mathbf{FP}})_{\rm n}$ in $\tilde{\Omega}^{(\rm h)}_{III}\cup\tilde{\Omega}^{(\rm h)}_{IV}\cup\tilde{\Omega}^{(\rm h)}_{V}$}
Problem $(\widetilde{\mathbf{FP}})_{\rm n}$ in $\tilde{\Omega}^{(\rm h)}_{III}\cup\tilde{\Omega}^{(\rm h)}_{IV}\cup\tilde{\Omega}^{(\rm h)}_{V}$ can be formulated as the following:
\begin{eqnarray}\label{eq:4.70}
\begin{split}
\left\{
\begin{array}{llll}
\partial_{\xi}\delta z^{(\rm n)}_{-}+\lambda^{(\rm n-1)}_{+}\partial_{\eta}\delta z^{(\rm n)}_{-}=0,
& \qquad \ \  \mbox{in} \ \ \ 
\tilde{\Omega}^{(\rm h)}_{III}
\cup\tilde{\Omega}^{(\rm h)}_{IV}\cup\tilde{\Omega}^{(\rm h)}_{V}, \\[5pt]
\delta z^{(\rm n)}_{-}+\delta z^{(\rm n)}_{+}=2\arctan g'_{-},  & \qquad \ \   \mbox{on} \ \ \
\tilde{\Gamma}_{-}\cap\{0\leq \xi\leq \xi^{*}_{-m^{(\rm h)},0}\}.
\end{array}
\right.
\end{split}
\end{eqnarray}
  Similar to the argument in Section 4.2.2, we also study problem \eqref{eq:4.70} via the characteristics method, but the difference 
is that there are reflections of the characteristics on the lower wall $\tilde{\Gamma}_{-}$ of the nozzle. Let $\eta=\chi^{(\rm n)}_{+}(\xi, \xi^{*}_{-m^{(\rm h)},0})$ be the reflected characteristics which issues from $(\xi^{*}_{-m^{(\rm h)},0}, -m^{(\rm h)})$ and
intersects with the transonic contact discontinuity $\tilde{\Gamma}_{\rm cd}$ at the point $(\xi^{*}_{\rm cd,1},0)$. Then, it holds that
\begin{eqnarray}\label{eq:4.80}
\begin{split}
&\int^{0}_{-m^{(\rm h)}}\frac{d\tau}{\lambda^{(\rm n-1)}_{+}(\chi^{(\rm n)}_{+}(\tau, \xi^{*}_{-m^{(\rm h)},0}),\tau)}=\xi^{*}_{\rm cd,1}.
\end{split}
\end{eqnarray}

We have the following estimates for $\delta z^{(\rm n)}_{-}$ in $\tilde{\Omega}^{(\rm h)}_{III}\cup\tilde{\Omega}^{(\rm h)}_{IV}\cup\tilde{\Omega}^{(\rm h)}_{V}$.
\begin{proposition}\label{prop:4.4}
For any given $\alpha\in (0,1)$, 
{\color{black}there exist constants $\mathcal{C}_{\rm h,8}>0$ and $\epsilon_{\rm h,8}>0$ depending only on $\tilde{\underline{U}}$, $L$ and $\alpha$,
such that for $\tilde{\epsilon}<\epsilon_{I} \in (\mathcal{C}_{\rm h,8}\tilde{\epsilon}, \epsilon_{\rm h,8})$ with $\tilde{\epsilon}>0$ sufficiently small, if $\delta z^{(\rm n-1)}\in \mathscr{K}_{2\epsilon_{I}}$,}
  the solution $z^{(\rm n)}_{-}$ to the problem \eqref{eq:4.70} satisfies
\begin{align}\label{eq:4.81}
\begin{split}
&\big\|\delta z^{(\rm n)}_{-}\big\|_{1,\alpha;
\tilde{\Omega}^{(\rm h)}_{III}\cup\tilde{\Omega}^{(\rm h)}_{IV} \cup\tilde{\Omega}^{(\rm h)}_{V}}\\[5pt]
&\quad \leq
\tilde{C}^{*}_{43}\bigg(\big\|\delta z^{(\rm n)}_{+}\big\|_{1,\alpha;\tilde{\Gamma}_{-}\cap\{0\leq \xi\leq \xi^{*}_{-m^{(\rm h)},0}\}}
+\big\|g_{-}+1\big\|_{2,\alpha; \tilde{\Gamma}_{-}\cap\{0\leq \xi\leq \xi^{*}_{-m^{(\rm h)},0}}\bigg),
\end{split}
\end{align}
where $\xi^{*}_{-m^{(\rm h)},0}$ is given in \eqref{eq:4.54} and the constant $\tilde{C}^{*}_{43}>0$ depends only on $\underline{\tilde{U}}$, $L$ and $\alpha$.
\end{proposition}


\begin{proof}
By $\eqref{eq:4.70}_{1}$, we have
\begin{eqnarray}\label{eq:4.83}
\begin{split}
\partial_{\eta}\delta z^{(\rm n)}_{-}+\frac{1}{{\lambda}^{(\rm n-1)}_{+}}\partial_{\xi}\delta z^{(\rm n)}_{-}=0.
\end{split}
\end{eqnarray}

{\color{black}For any point $(\xi, \eta)\in \tilde{\Omega}^{(\rm h)}_{III}\cup\tilde{\Omega}^{(\rm h)}_{IV} \cup\tilde{\Omega}^{(\rm h)}_{V}$, we can define a characteristic line
$\xi=\chi^{(\rm n)}_{+}(\eta,\xi_{-m^{(\rm h)},0})$ by \eqref{eq:4.71}-\eqref{eq:4.72} such that it intersects with the boundary $\tilde{\Gamma}_{-}$ at the point $(\xi_{-m^{(\rm h)},0},-m^{(\rm h)})$.
Then, 
by the boundary condition on $\tilde{\Gamma}_{-}\cap\{0\leq \xi\leq \xi^{*}_{-m^{(\rm h)},0}\}$,
we get
\begin{eqnarray}\label{eq:4.83a}
\begin{split}
\delta z^{(\rm n)}_{-}(\xi, \eta)=\delta z_{-}(\xi_{-m^{(\rm h)},0},-m^{(\rm h)})=2\arctan g'_{-}(\xi_{-m^{(\rm h)},0})-\delta z^{(\rm n)}_{+}(\xi_{-m^{(\rm h)},0},-m^{(\rm h)}).
\end{split}
\end{eqnarray}

{\color{black}Next, differentiating the equation $\eqref{eq:4.83}$ with respect to $\xi$ and $\eta$, and then integrating them along the characteristics $\xi=\chi^{(\rm n)}_{+}(\eta,\xi_{-m^{(\rm h)},0})$,} we obatin
\begin{align}\label{eq:4.83b}
\begin{split}
&\partial_{\xi} \delta z^{(\rm n)}_{-}(\xi,\eta)\\[5pt]
&=\partial_{\xi}\delta z^{(\rm n)}_{-}(\xi_{-m^{(\rm h)},0}, -m^{(\rm h)})
-\int^{\eta}_{-m^{(\rm h)}}\partial_{\xi}\Big(\frac{1}{\lambda^{(\rm n-1)}_{+}}\Big)
\partial_{\xi}\delta z^{(\rm n)}_{-}\big(\chi^{(\rm n)}_{+}(\tau,\xi_{-m^{(\rm h)},0}),\tau \big)d\tau\\[5pt]
&=\partial_{\xi} \xi_{-m^{(\rm h)},0}\partial_{\xi_{-m^{(\rm h)},0}}\delta z^{(\rm n)}_{-}(\xi_{-m^{(\rm h)},0}, -m^{(\rm h)})
-\int^{\eta}_{-m^{(\rm h)}}\partial_{\xi}\Big(\frac{1}{\lambda^{(\rm n-1)}_{+}}\Big)
\partial_{\xi}\delta z^{(\rm n)}_{-}\big(\chi^{(\rm n)}_{+}(\tau,\xi_{-m^{(\rm h)},0}),\tau \big)d\tau,
\end{split}
\end{align}
and
\begin{align}\label{eq:4.83c}
\begin{split}
&\partial_{\eta} \delta z^{(\rm n)}_{-}(\xi,\eta)\\[5pt]
&=\partial_{\eta}\delta z^{(\rm n)}_{-}(\xi_{-m^{(\rm h)},0}, -m^{(\rm h)})
-\int^{\eta}_{-m^{(\rm h)}}\partial_{\tau}\Big(\frac{1}{\lambda^{(\rm n-1)}_{+}}\Big)
\partial_{\xi}\delta z^{(\rm n)}_{-}\big(\chi^{(\rm n)}_{+}(\tau,\xi_{-m^{(\rm h)},0}),\tau \big)d\tau\\[5pt]
&=\partial_{\eta} \xi_{-m^{(\rm h)},0}\partial_{\xi_{-m^{(\rm h)},0}}\delta z^{(\rm n)}_{-}(\xi_{-m^{(\rm h)},0}, -m^{(\rm h)})
-\int^{\eta}_{-m^{(\rm h)}}\partial_{\tau}\Big(\frac{1}{\lambda^{(\rm n-1)}_{+}}\Big)
\partial_{\xi}\delta z^{(\rm n)}_{-}\big(\chi^{(\rm n)}_{+}(\tau,\xi_{-m^{(\rm h)},0}),\tau \big)d\tau.
\end{split}
\end{align}
}
By using the boundary condition $\eqref{eq:4.70}_2$ on $\tilde{\Gamma}_{-}\cap\{0\leq \xi\leq \xi^{*}_{-m^{(\rm h)},0}\}$, we further have for
$\partial_{\xi_{-m^{(\rm h)},0}}\delta z^{(\rm n)}_{-}(\xi_{-m^{(\rm h)},0}, -m^{(\rm h)})$ that
\begin{align}\label{eq:4.83d}
\begin{split}
\partial_{\xi_{-m^{(\rm h)},0}}\delta z^{(\rm n)}_{-}(\xi_{-m^{(\rm h)},0}, -m^{(\rm h)})
=\frac{2g''_{-}}{1+(g'_{-})^{2}}-\partial_{\xi_{-m^{(\rm h)},0}}\delta z^{(\rm n)}_{+}(\xi_{-m^{(\rm h)},0}, -m^{(\rm h)}).
\end{split}
\end{align}

Then, following the same argument as in \cite{hkwx} for the proof of Proposition \ref{prop:4.2} and by Lemma \ref{lem:4.4},
there exist   constants $\mathcal{C}'_{\rm h,8}>\max\{\mathcal{C}_{\rm h,3}, \mathcal{C}_{\rm h,4} \}$, $\epsilon'_{\rm h,8}<\min\{\epsilon_{\rm h,3}, \epsilon_{\rm h,4}\}$ and
$\mathcal{C}>0$ depending  only on
$\tilde{\underline{U}}$ and $L$,  such that for {$\tilde{\epsilon}<\epsilon_{I}\in (\mathcal{C}'_{\rm h,8}\tilde{\epsilon}, \epsilon'_{\rm h, 8})$ with $\tilde{\epsilon}>0$ sufficiently small,} we have
\begin{eqnarray}\label{eq:4.84}
\begin{split}
&\big\|\delta z^{(\rm n)}_{-}\big\|_{1,0;
\tilde{\Omega}^{(\rm h)}_{III}
\cup \tilde{\Omega}^{(\rm h)}_{IV}\cup\tilde{\Omega}^{(\rm h)}_{V}}
\leq \mathcal{C}\Big(\big\|\delta z^{(\rm n)}_{+}\big\|_{1,0; \tilde{\Gamma}_{-}\cap\{0\leq \xi\leq\xi^{*}_{-m^{(\rm h)},0}\}}
+\big\|g'_{-}\big\|_{1,0; \tilde{\Gamma}_{-}\cap\{0\leq \xi\leq\xi^{*}_{-m^{(\rm h)},0}\}}\Big).
\end{split}
\end{eqnarray}

Therefore, it remains to obtain the estimate of $\big[D\delta z^{(\rm n)}_{-}\big]_{0,\alpha; \tilde{\Omega}^{(\rm h)}_{III}\cup\tilde{\Omega}^{(\rm h)}_{IV}\cup\tilde{\Omega}^{(\rm h)}_{V}}$. 
For the given two points $\hat{P}=(\xi_{\hat{P}},\eta_{\hat{P}}),\ \hat{Q}=(\xi_{\hat{Q}},\eta_{\hat{Q}}) \in \tilde{\Omega}^{(\rm h)}_{III}\cup \tilde{\Omega}^{(\rm h)}_{IV}\cup\tilde{\Omega}^{(\rm h)}_{V}$, satisfying \eqref{eq:4.73} with $\hat{P}_{+}(-m^{(\rm h)})=(\xi_{\hat{P}_{+}(-m^{(\rm h)})},-m^{(\rm h)})$ and $\hat{Q}_{+}(-m^{(\rm h)})=(\xi_{\hat{Q}_{+}(-m^{(\rm h)})},-m^{(\rm h)})$,
  substituting them into \eqref{eq:4.83b}-\eqref{eq:4.83c}, we obtain
\begin{align}\label{eq:4.90}
\begin{split}
\partial_{\xi} \delta z^{(\rm n)}_{-}(\hat{P})
&=\partial_{\xi}\xi_{-m^{(\rm h)},0}(\hat{P})\partial_{\xi_{-m^{(\rm h)},0}}\delta z^{(\rm n)}_{-}(\hat{P}_{+}(-m^{(\rm h)})) \\[5pt]
&\qquad\qquad\qquad-\int^{\eta_{\hat{P}}}_{-m^{(\rm h)}}\partial_{\xi}\Big(\frac{1}{\lambda^{(\rm n-1)}_{+}}\Big)\partial_{\xi}\delta z^{(\rm n)}_{-}\big(\hat{P}_{+}(\tau)\big)d\tau,\\[5pt]
\partial_{\eta} \delta z^{(\rm n)}_{-}(\hat{P})
&=\partial_{\eta}\xi_{-m^{(\rm h)},0}(\hat{P})\partial_{\xi_{-m^{(\rm h)},0}}\delta z^{(\rm n)}_{-}(\hat{P}_{+}(-m^{(\rm h)})) \\[5pt]
&\qquad\qquad\qquad-\int^{\eta_{\hat{P}}}_{-m^{(\rm h)}}\partial_{\tau}\Big(\frac{1}{\lambda^{(\rm n-1)}_{+}}\Big)
\partial_{\xi}\delta z^{(\rm n)}_{-}\big(\hat{P}_{+}(\tau)\big)d\tau,
\end{split}
\end{align}
and
\begin{align}\label{eq:4.91}
\begin{split}
\partial_{\xi} \delta z^{(\rm n)}_{-}(\hat{Q})
&=\partial_{\xi}\xi_{-m^{(\rm h)},0}(\hat{Q})\partial_{\xi_{-m^{(\rm h)},0}}\delta z^{(\rm n)}_{-}(\hat{Q}_{+}(-m^{(\rm h)}))\\[5pt]
&\qquad\qquad\qquad-\int^{\eta_{\hat{Q}}}_{-m^{(\rm h)}}\partial_{\xi}\Big(\frac{1}{\lambda^{(\rm n-1)}_{+}}\Big)
\partial_{\xi}\delta z^{(\rm n)}_{-}\big(\hat{Q}_{+}(\tau)\big)d\tau,\\[5pt]
\partial_{\eta} \delta z^{(\rm n)}_{-}(\hat{Q})
&=\partial_{\eta}\xi_{-m^{(\rm h)},0}(\hat{Q})\partial_{\xi_{-m^{(\rm h)},0}}\delta z^{(\rm n)}_{-}(\hat{Q}_{+}(-m^{(\rm h)})) \\[5pt]
&\qquad\qquad\qquad-\int^{\eta_{\hat{Q}}}_{-m^{(\rm h)}}\partial_{\tau}\Big(\frac{1}{\lambda^{(\rm n-1)}_{+}}\Big)
\partial_{\xi}\delta z^{(\rm n)}_{-}\big(\hat{Q}_{+}(\tau)\big)d\tau.
\end{split}
\end{align}

Similarly to the arguments in 
proving the Proposition \ref{prop:4.3}, we first 
derive the estimate for
$d^{-\alpha}\big(\hat{P},\hat{Q}\big)\big|D\delta z^{(\rm n)}_{-}(\hat{P})
-D\delta z^{(\rm n)}_{-}(\hat{Q})\big|$.
To this end, we assume that $\eta_{\hat{P}}<\eta_{\hat{Q}}$
and let $\hat{Q}_{+}(\eta_{\hat{P}})=(\chi^{(\rm n)}_{+}(\eta_{\hat{P}}, \xi_{\hat{Q}_{+}(-m^{(\rm h)})}),\eta_{\hat{P}})$, and the other cases can be treated  in the same way.
By \eqref{eq:4.90} and \eqref{eq:4.83c}--\eqref{eq:4.83d}, we have
\begin{align}\label{eq:4.92}
\begin{split}
& d^{-\alpha}(\hat{P},\hat{Q}_{+}(\eta_{\hat{P}}))\Big|D\delta z^{(\rm n)}_{-}(\hat{P})-D\delta z^{(\rm n)}_{-}(\hat{Q}_{+}(\eta_{\hat{P}}))\Big|\\[5pt]
\leq &d^{-\alpha}(\hat{P},\hat{Q}_{+}(\eta_{\hat{P}}))
\Big|D\xi_{-m^{(\rm h)},0}(\hat{P})\partial_{\xi_{-m^{(\rm h)},0}}\delta z^{(\rm n)}_{-}(\hat{P}_{+}(-m^{(\rm h)}))\\[5pt]
&\qquad\qquad\qquad\quad -D\xi_{-m^{(\rm h)},0}(\hat{Q}_{+}(\eta_{\hat{P}}))\partial_{\xi_{-m^{(\rm h)},0}}\delta z^{(\rm n)}_{-}(\hat{Q}_{+}(-m^{(\rm h)}))\Big|\\[5pt]
&   +d^{-\alpha}(\hat{P},\hat{Q}_{+}(\eta_{\hat{P}}))\int^{\eta_{\hat{P}}}_{-m^{(\rm h)}}\bigg|\Big(D\big(\frac{1}{\lambda^{(\rm n-1)}_{+}}\big)
\partial_{\xi}\delta z^{(\rm n)}_{-}\Big)(\hat{P}_{+}(\tau))-\Big(D\big(\frac{1}{\lambda^{(\rm n-1)}_{+}}\big)
\partial_{\xi}\delta z^{(\rm n)}_{-}\Big)(\hat{Q}_{+}(\tau))\bigg|d\tau\\[5pt]
&=:\hat{J}_{1}+\hat{J}_{2}.
\end{split}
\end{align}

For $\hat{J}_{1}$, by Lemma \ref{lem:4.3} and Lemma \ref{lem:4.4}, one has
\begin{eqnarray}\label{eq:4.92a}
\begin{split}
\hat{J}_{1} & \leq d^{-\alpha}\big(\hat{P},\hat{Q}_{+}(\eta_{\hat{P}})\big)
\Big|D\xi_{-m^{(\rm h)}, 0}(\hat{P})-D\xi_{-m^{(\rm h)}, 0}(\hat{Q}_{+}(\eta_{\hat{P}}))\Big|\big|\partial_{\xi_{-m^{(\rm h)}, 0}}\delta z^{(\rm n)}_{-}(\hat{P}_{+}(-m^{(\rm h)}))\big|\\[5pt]
&\quad \ \ + \mathcal{C}\big|D\xi_{-m^{(\rm h)}, 0}(\hat{Q}_{+}(\eta_{\hat{P}}))\big|d^{-\alpha}\big(\hat{P}(-m^{(\rm h)}),\hat{Q}(-m^{(\rm h)})\big)\\[5pt]
&\qquad\qquad \ \ \times\Big|\partial_{\xi_{-m^{(\rm h)}, 0}}\delta z^{(\rm n)}_{-}(\hat{P}_{+}(-m^{(\rm h)}))-\partial_{\xi_{-m^{(\rm h)}, 0}}\delta z^{(\rm n)}_{-}(\hat{Q}_{+}(-m^{(\rm h)}))\Big|
\\[5pt]
& \leq \mathcal{C}\|D\xi_{-m^{(\rm h)}, 0}\|_{0,\alpha;\tilde{\Gamma}_{-}\cap\{0\leq \xi\leq\xi^{*}_{-m^{(\rm h)},0}\}}
\big\|\partial_{\xi}\delta z^{(\rm n)}_{-}\big\|_{0,0;\tilde{\Gamma}_{-}\cap\{0\leq \xi\leq\xi^{*}_{-m^{(\rm h)},0}\}}\\[5pt]
&\quad \ \ + \mathcal{C}\big\|D\xi_{-m^{(\rm h)}, 0}\big\|_{0,0;\tilde{\Gamma}_{-}\cap\{0\leq \xi\leq\xi^{*}_{-m^{(\rm h)},1}\}}d^{-\alpha}\big(\hat{P}(-m^{(\rm h)}),\hat{Q}(-m^{(\rm h)})\big)\\[5pt]
&\qquad\qquad \ \ \times\Big|\partial_{\xi_{-m^{(\rm h)}, 0}}\delta z^{(\rm n)}_{-}(\hat{P}_{+}(-m^{(\rm h)}))-\partial_{\xi_{-m^{(\rm h)}, 0}}\delta z^{(\rm n)}_{-}(\hat{Q}_{+}(-m^{(\rm h)}))\Big|
\\[5pt]
&\leq \mathcal{C}\big\|\partial_{\xi}\delta z^{(\rm n)}_{-}\big\|_{0,0;\tilde{\Gamma}_{-}\cap\{0\leq \xi\leq\xi^{*}_{-m^{(\rm h)},0}\}}+\mathcal{C}d^{-\alpha}\big(\hat{P}(-m^{(\rm h)}),\hat{Q}(-m^{(\rm h)})\big)\\[5pt]
&\qquad\qquad\qquad\qquad \times \Big|\partial_{\xi_{-m^{(\rm h)}, 0}}\delta z^{(\rm n)}_{-}(\hat{P}_{+}(-m^{(\rm h)}))-\partial_{\xi_{-m^{(\rm h)}, 0}}\delta z^{(\rm n)}_{-}(\hat{Q}_{+}(-m^{(\rm h)}))\Big|,
\end{split}
\end{eqnarray}
where 
the constant $\mathcal{C}=\mathcal{C}(\tilde{\underline{U}}, L, \alpha)>0$ is independent of $\epsilon_{I}$ and $\rm n$.
%
By \eqref{eq:4.83d}, we know that
\begin{align*}
\begin{split}
\big\|\partial_{\xi}\delta z^{(\rm n)}_{-}\big\|_{0,0;\tilde{\Gamma}_{-}\cap\{0\leq \xi\leq\xi^{*}_{-m^{(\rm h)},0}\}}\leq
\big\|g''_{-}\big\|_{0,0;\tilde{\Gamma}_{-}\cap\{0\leq \xi\leq\xi^{*}_{-m^{(\rm h)},0}\}}
+\big\|\partial_{\xi}\delta z^{(\rm n)}_{+}\big\|_{0,0;\tilde{\Gamma}_{-}\cap\{0\leq \xi\leq\xi^{*}_{-m^{(\rm h)},0}\}}.
\end{split}
\end{align*}
For the last term in \eqref{eq:4.92a}, by using  \eqref{eq:4.83d} again and choosing the constants $\mathcal{C}''_{\rm h,8}>0$, $\epsilon''_{\rm h,8}>0$ depending only on
$\tilde{\underline{U}}$ and $L$, such that for $\tilde{\epsilon}<\epsilon_{I}\in (\mathcal{C}''_{\rm h,8}\tilde{\epsilon}, \epsilon''_{\rm h,8})$ with {$\tilde{\epsilon}>0$ be sufficiently small,}
we have
\begin{align*}
&d^{-\alpha}\big(\hat{P}(-m^{(\rm h)}),\hat{Q}(-m^{(\rm h)})\big)
\Big|\partial_{\xi_{-m^{(\rm h)}, 0}}\delta z^{(\rm n)}_{-}(\hat{P}_{+}(-m^{(\rm h)}))-\partial_{\xi_{-m^{(\rm h)}, 0}}\delta z^{(\rm n)}_{-}(\hat{Q}_{+}(-m^{(\rm h)}))\Big|\\[5pt]
\leq&d^{-\alpha}\big(\hat{P}(-m^{(\rm h)}),\hat{Q}(-m^{(\rm h)})\big)\Bigg|\frac{g''_{-}(\xi_{\hat{P}_{+}(-m^{(\rm h)})})}{1+\big(g'_{-}(\xi_{\hat{P}_{+}(-m^{(\rm h)})})\big)^{2}}
-\frac{g''_{-}(\xi_{\hat{Q}_{+}(-m^{(\rm h)})})}{1+\big(g'_{-}(\xi_{\hat{Q}_{+}(-m^{(\rm h)})})\big)^{2}}\Bigg|\\[5pt]
&+d^{-\alpha}\big(\hat{P}(-m^{(\rm h)}),\hat{Q}(-m^{(\rm h)})\big)\Big|\partial_{\xi_{-m^{(\rm h)}, 0}}\delta z^{(\rm n)}_{+}(\hat{P}_{+}(-m^{(\rm h)}))-\partial_{\xi_{-m^{(\rm h)}, 0}}\delta z^{(\rm n)}_{+}(\hat{Q}_{+}(-m^{(\rm h)}))\Big|\\[5pt]
\leq& \mathcal{C}\big(1+(g'_{-})^{2}\big) d^{-\alpha}\big(\hat{P}(-m^{(\rm h)}),\hat{Q}(-m^{(\rm h)})\big)
\Big|g''_{-}(\xi_{\hat{P}_{+}(-m^{(\rm h)})})-g''_{-}(\xi_{\hat{Q}_{+}(-m^{(\rm h)})})\Big|\\[5pt]
&+\mathcal{C}\big|g''_{-}\big| d^{-\alpha}\big(\hat{P}(-m^{(\rm h)}),\hat{Q}(-m^{(\rm h)})\big)\Big|\big(g'_{-}(\xi_{\hat{P}_{+}(-m^{(\rm h)})})\big)^{2}-\big(g'_{-}(\xi_{\hat{Q}_{+}(-m^{(\rm h)})})\big)^{2}\Big|\\[5pt]
&\qquad +\big\|\partial_{\xi}\delta z^{(\rm n)}_{+}\big\|_{0,\alpha;\tilde{\Gamma}_{-}\cap\{0\leq \xi\leq\xi^{*}_{-m^{(\rm h)},0}\}}\\[5pt]
\leq& \mathcal{C}\Big(\big\|\partial_{\xi}\delta z^{(\rm n)}_{+}\big\|_{0,\alpha;\tilde{\Gamma}_{-}\cap\{0\leq \xi\leq\xi^{*}_{-m^{(\rm h)},0}\}}
+\big\|g''_{-}\big\|_{0,\alpha; \tilde{\Gamma}_{-}\cap\{0\leq \xi\leq\xi^{*}_{-m^{(\rm h)},0}\}}\Big),
\end{align*}
where the constant $\mathcal{C}=\mathcal{C}(\tilde{\underline{U}},L,\alpha)>0$ is independent of $\epsilon_{I}$ and $\rm n$. 
Hence, substituting the above two estimates into \eqref{eq:4.92a}, we obtain the estimate for $\hat{J}_{1}$ that
\begin{eqnarray}\label{eq:4.93}
\begin{split}
\hat{J}_{1} &\leq \mathcal{C}\Big(\big\|\partial_{\xi}\delta z^{(\rm n)}_{+}\big\|_{0,\alpha;\tilde{\Gamma}_{-}\cap\{0\leq \xi\leq\xi^{*}_{-m^{(\rm h)},0}\}}
+\big\|g''_{-}\big\|_{0,\alpha; \tilde{\Gamma}_{-}\cap\{0\leq \xi\leq\xi^{*}_{-m^{(\rm h)},0}\}}\Big),
\end{split}
\end{eqnarray}
where the constant $\mathcal{C}>0$ depends only on $\tilde{\underline{U}}$, $L$ and $\alpha$.

Let us consider $\hat{J}_{2}$ now. By Lemma \ref{lem:4.3},  a direct computation yields
\begin{eqnarray*}
\begin{split}
\hat{J}_{2}&\leq d^{-\alpha}\big(\hat{P},\hat{Q}_{+}(\eta_{\hat{P}})\big)
\int^{\eta_{\hat{P}}}_{-m^{(\rm h)}}\bigg|D\Big(\frac{1}{\lambda^{(\rm n-1)}_{+}}\Big)(\hat{P}_{+}(\tau))\bigg|\Big|\partial_{\xi}\delta z^{(\rm n)}_{-}(\hat{P}_{+}(\tau))
-\partial_{\xi}\delta z^{(\rm n)}_{-}(\hat{Q}_{+}(\tau))\Big|d\tau\\[5pt]
& \ \ \ \ +d^{-\alpha}\big(\hat{P},\hat{Q}_{+}(\eta_{\hat{P}})\big)\int^{\eta_{\hat{P}}}_{-m^{(\rm h)}}
\bigg|D\Big(\frac{1}{\lambda^{(\rm n-1)}_{+}}\Big)(\hat{P}_{+}(\tau))-D\Big(\frac{1}{\lambda^{(\rm n-1)}_{+}}\Big)(\hat{Q}_{+}(\tau))\bigg|
\Big|\partial_{\xi}\delta z^{(\rm n)}_{-}(\hat{Q}_{+}(\tau))\Big|d\tau\\[5pt]
&=: \hat{J}_{21}+\hat{J}_{22}.
\end{split}
\end{eqnarray*}
For $\hat{J}_{21}$, by Lemma \ref{lem:4.3} and \eqref{eq:4.75}, there exists a constant $\mathcal{C}>0$ depending only on $\tilde{\underline{U}}$, $L$ and $\alpha$
such that
\begin{eqnarray*}
\begin{split}
\hat{J}_{21}&\leq \mathcal{C}\Big(\|\delta z^{(\rm n-1)}\|_{1,0; \tilde{\Omega}^{(\rm h)}}
+\big\|\big(\delta\tilde{B}^{(\rm h)'}_{0},\delta\tilde{S}^{(\rm h)'}_{0}\big)\big\|_{1,0;\tilde{\Gamma}^{(\rm h)}_{\rm in}}\Big)\\[5pt]
&\quad \ \ \   \times\int^{\eta_{\hat{P}}}_{-m^{(\rm h)}}d^{-\alpha}\big(\hat{P}_{+}(\tau),\hat{Q}_{+}(\tau)\big)
\Big|D\delta z^{(\rm n)}_{-}(\hat{P}_{+}(\tau))-D\delta z^{(\rm n)}_{-}(\hat{Q}_{+}(\tau))\Big|d\tau\\[5pt]
&\leq \mathcal{C}\big(2\epsilon_{I}+\tilde{\epsilon}\big)\int^{\eta_{\hat{P}}}_{-m^{(\rm h)}}d^{-\alpha}\big(\hat{P}_{+}(\tau),\hat{Q}_{+}(\tau)\big)
\Big|D\delta z^{(\rm n)}_{-}(\hat{P}_{+}(\tau))-D\delta z^{(\rm n)}_{-}(\hat{Q}_{+}(\tau))\Big|d\tau.
\end{split}
\end{eqnarray*}
For $\hat{J}_{22}$, by \eqref{eq:4.75b}\eqref{eq:4.84} and Lemma \ref{lem:4.3}, we have
\begin{eqnarray*}
\begin{split}
\hat{J}_{22}&\leq \mathcal{C}\big\|\big(\lambda^{(\rm n-1)}_{+}\big)^{-1}\big\|_{1,\alpha; \tilde{\Omega}^{(\rm h)}}
\int^{\eta_{\hat{P}}}_{-m^{(\rm h)}}\Big|D\delta z^{(\rm n)}_{-}(\hat{Q}_{+}(\tau))\Big|d\tau\\[5pt]
&\leq \mathcal{C}\Big( \big\|\delta z^{(\rm n)}_{+}\big\|_{1,0;\tilde{\Gamma}_{-}\cap\{0\leq\xi\leq\xi^{*}_{-m^{(\rm h)},0}\}}
+\big\|g'_{-}\big\|_{1,0;\tilde{\Gamma}_{-}\cap\{0\leq\xi\leq\xi^{*}_{-m^{(\rm h)},0}\}}\Big),
\end{split}
\end{eqnarray*}
where $\mathcal{C}>0$ depends only on $\tilde{\underline{U}}$, $L$ and $\alpha$. 
 By \eqref{eq:4.93} as well as the estimates for $\hat{J}_{21}$ and $\hat{J}_{22}$,
we have
\begin{eqnarray*}
\begin{split}
&d^{-\alpha}\big(\hat{P},\hat{Q}_{+}(\eta_{\hat{P}})\big)
\Big|D\delta z^{(\rm n)}_{-}(\hat{P})-D\delta z^{(\rm n)}_{-}(\hat{Q}_{+}(\eta_{\hat{P}}))\Big|\\[5pt]
&\qquad \leq \mathcal{C}\Big(\big\|\delta z^{(\rm n)}_{+}\big\|_{1,\alpha;\tilde{\Gamma}_{-}\cap\{0\leq\xi\leq\xi^{*}_{-,0}\}}
+\big\|g'_{+}\big\|_{1,\alpha;\tilde{\Gamma}_{-}\cap\{0\leq\xi\leq\xi^{*}_{-,0}\}}\Big)\\[5pt]
&\qquad\qquad\ \   +\mathcal{C}\big(2\epsilon_{I}+\tilde{\epsilon}\big)\int^{\eta_{\hat{P}}}_{-m^{(\rm h)}}d^{-\alpha}\big(\hat{P}(\tau),\hat{Q}(\tau)\big)\Big|D\delta z^{(\rm n)}_{-}(\hat{P}(\tau))
-D\delta z^{(\rm n)}_{-}(\hat{Q}(\tau))\Big|d\tau.
\end{split}
\end{eqnarray*}
Then it follows from the Gronwall inequality that
\begin{eqnarray*}
\begin{split}
&d^{-\alpha}\big(\hat{P},\hat{Q}_{+}(\eta_{\hat{P}})\big)
\Big|D\delta z^{(\rm n)}_{-}(\hat{P})-D\delta z^{(\rm n)}_{-}(\hat{Q}_{+}(\eta_{\hat{P}}))\Big|\\[5pt]
&\quad \quad \ \
\leq \mathcal{C}\Big( \big\|\delta z^{(\rm n)}_{+}\big\|_{1,\alpha;\tilde{\Gamma}_{-}\cap\{0\leq\xi\leq\xi^{*}_{-m^{(\rm h)},0}\}}
+\big\|g'_{+}\big\|_{1,\alpha;\tilde{\Gamma}_{-}\cap\{0\leq\xi\leq\xi^{*}_{-m^{(\rm h)},0}\}}\Big)e^{\mathcal{C}(2\epsilon_{I}+\tilde{\epsilon})m^{(\rm h)}}.
\end{split}
\end{eqnarray*}
Taking the constants $\mathcal{C}'''_{\rm h,8}>0$, $\epsilon'''_{\rm h,8}>0$ depending only on
$\tilde{\underline{U}}$, $L$ and $\alpha$, and {letting $\tilde{\epsilon}>0$ be sufficiently small such that for $\tilde{\epsilon}<\epsilon_{I}\in (\mathcal{C}'''_{\rm h,8}\tilde{\epsilon}, \epsilon'''_{\rm h,8})$,} one has
\begin{eqnarray}\label{eq:4.94}
\begin{split}
&\sup_{\hat{P},\hat{Q}_{+}(\eta_{\hat{P}})\in 
\tilde{\Omega}^{(\rm h)}_{III}\cup \tilde{\Omega}^{(\rm h)}_{IV}\cup\tilde{\Omega}^{(\rm h)}_{V}}d^{-\alpha}\big(\hat{P}_{+},\hat{Q}_{+}(\eta_{\hat{P}})\big)
\Big|D\delta z^{(\rm n)}_{-}(\hat{P})-D\delta z^{(\rm n)}_{-}(\hat{Q}_{+}(\eta_{\hat{P}}))\Big|\\[5pt]
&\quad \ \ \
\leq 2\mathcal{C}\Big( \big\|\delta z^{(\rm n)}_{+}\big\|_{1,\alpha;\tilde{\Gamma}_{-}\cap\{0\leq\xi\leq\xi^{*}_{-m^{(\rm h)},0}\}}
+\big\|g'_{+}\big\|_{1,\alpha;\tilde{\Gamma}_{-}\cap\{0\leq\xi\leq\xi^{*}_{-m^{(\rm h)},0}\}}\Big).
\end{split}
\end{eqnarray}

\par On the other hand, by \eqref{eq:4.83c}--\eqref{eq:4.83d} and \eqref{eq:4.91}, we can also deduce that
\begin{eqnarray*}\label{eq:4.95}
\begin{split}
&\ d^{-\alpha}\big(\hat{Q}_{+}(\eta_{\hat{P}}),\hat{Q}\big)
\Big|D\delta z^{(\rm n)}_{-}(\hat{Q}_{+}(\eta_{\hat{P}}))-D\delta z^{(\rm n)}_{-}(\hat{Q})\Big|\\[5pt]
\leq &d^{-\alpha}\big(\hat{Q}_{+}(\eta_{\hat{P}}),\hat{Q}\big)\big|D\xi_{-m^{(\rm h)},0}(\hat{Q}_{+}(\eta_{\hat{P}}))-D\xi_{-m^{(\rm h)},0}(\hat{Q})\big||\partial_{\xi_{-m^{(\rm h)},0}}D\delta z_{-}(Q_{+}(0))|\\[5pt]
&\quad+d^{-\alpha}\big(\hat{Q}_{+}(\eta_{\hat{P}}),\hat{Q}\big)
\int^{\eta_{\hat{Q}}}_{\eta_{\hat{P}}}\Big|D\Big(\frac{1}{\lambda^{(\rm n-1)}_{+}}\Big)
\partial_{\xi}\delta z^{(\rm n)}_{-}(\hat{Q}_{+}(\tau))\Big|d\tau,
\end{split}
\end{eqnarray*}
which  implies by Lemma \ref{lem:4.3} and \eqref{eq:4.84} that
\begin{eqnarray*}
\begin{split}
&\sup_{\hat{Q}_{+}(\eta_{\hat{P}}),\hat{Q}\in 
\tilde{\Omega}^{(\rm h)}_{III}\cup \tilde{\Omega}^{(\rm h)}_{IV}\cup\tilde{\Omega}^{(\rm h)}_{V}}d^{-\alpha}\big(\hat{Q}_{+}(\eta_{\hat{P}}),\hat{Q}\big)
\Big|D\delta z^{(\rm n)}_{-}(\hat{Q}_{+}(\eta_{\hat{P}}))-D\delta z^{(\rm n)}_{-}(\hat{Q})\Big|\\[5pt]
&\qquad\qquad \qquad\qquad \leq\mathcal{C}\Big(\|D\delta z^{(\rm n)}_{+}\|_{0,\alpha;\tilde{\Gamma}_{-}\cap\{0\leq\xi\leq\xi^{*}_{-m^{(\rm h)},0}\}}
+\big\|g''_{+}\big\|_{0,\alpha;\tilde{\Gamma}_{-}\cap\{0\leq\xi\leq\xi^{*}_{-m^{(\rm h)},0}\}}\Big),
\end{split}
\end{eqnarray*}
provided that {$\tilde{\epsilon}<\epsilon_{I}\in (\mathcal{C}'_{\rm h,8}\tilde{\epsilon}, \epsilon'_{\rm h,8})$ with $\tilde{\epsilon}>0$ be sufficiently small.}
Here, the constant $\mathcal{C}=\mathcal{C}(\tilde{\underline{U}},L, \alpha)>0$ is independent of $\epsilon_{I}$ and $\rm n$.
Combining \eqref{eq:4.93} and \eqref{eq:4.94} together, we get
\begin{eqnarray}\label{eq:4.96}
\begin{split}
&\sup_{\hat{P},\hat{Q}\in 
\tilde{\Omega}^{(\rm h)}_{III}\cup \tilde{\Omega}^{(\rm h)}_{IV}\cup\tilde{\Omega}^{(\rm h)}_{V}}d^{-\alpha}\big(\hat{P},\hat{Q}\big)
\Big|D\delta z^{(\rm n)}_{-}(\hat{P})-D\delta z^{(\rm n)}_{-}(\hat{Q})\Big|\\[5pt]
& \ \ \ \leq \sup_{\hat{P},\hat{Q}_{+}(\eta_{\hat{P}})\in 
\tilde{\Omega}^{(\rm h)}_{III}\cup \tilde{\Omega}^{(\rm h)}_{IV}\cup\tilde{\Omega}^{(\rm h)}_{V}}d^{-\alpha}\big(\hat{P},\hat{Q}_{+}(\eta_{\hat{P}})\big)
\Big|D\delta z^{(\rm n)}_{-}(\hat{P})-D\delta z^{(\rm n)}_{-}(\hat{Q}_{+}(\eta_{\hat{P}}))\Big|\\[5pt]
&\quad  \ \ \ \ +\sup_{\hat{Q}_{+}(\eta_{\hat{P}}),\hat{Q}\in 
\tilde{\Omega}^{(\rm h)}_{III}\cup \tilde{\Omega}^{(\rm h)}_{IV}\cup\tilde{\Omega}^{(\rm h)}_{V}}d^{-\alpha}\big(\hat{Q}_{+}(\eta_{\hat{P}}),\hat{Q}\big)
\Big|D\delta z^{(\rm n)}_{-}(\hat{Q}_{+}(\eta_{\hat{P}}))-D\delta z^{(\rm n)}_{-}(\hat{Q})\Big|\\[5pt]
& \ \ \ \leq\mathcal{C}\Big(\big\|\delta z^{(\rm n)}_{+}\big\|_{1,\alpha;\tilde{\Gamma}_{-}\cap\{0\leq\xi\leq\xi^{*}_{-m^{(\rm h)},0}\}}
+\big\|g'_{+}\big\|_{1,\alpha;\tilde{\Gamma}_{-}\cap\{0\leq\xi\leq\xi^{*}_{-m^{(\rm h)},0}\}}\Big).
\end{split}
\end{eqnarray}

Finally, by \eqref{eq:4.84} and \eqref{eq:4.96}, 
we can choose constants $\tilde{C}^{*}_{44}>0$, $\mathcal{C}_{\rm h,8}=\max\{\mathcal{C}'_{\rm h,8}, \mathcal{C}''_{\rm h,8},\mathcal{C}'''_{\rm h,8}\}$
 and $\epsilon_{\rm h,8}=\min\{\epsilon'_{\rm h,8}, \epsilon''_{\rm h,8}, \epsilon'''_{\rm h,8}\}$ depending only on
$\tilde{\underline{U}}$ and $L$ such that for {$\tilde{\epsilon}<\epsilon_{I}\in (\mathcal{C}_{\rm h,8}\tilde{\epsilon}, \epsilon_{\rm h,8})$ with {$\tilde{\epsilon}>0$ be sufficiently small},
the inequality \eqref{eq:4.81} holds.} 
\end{proof}

\subsubsection{Estimates of the solutions to the problem $(\widetilde{\mathbf{FP}})_{\rm n}$ in the supersonic region $\tilde{\Omega}^{(\rm h)}_{II}\cup\tilde{\Omega}^{(\rm h)}_{IV}\cup\tilde{\Omega}^{(\rm h)}_{VI}$}
Based on the results in Sections 4.1-- 4.2.3, we will consider 
the 
problem $(\widetilde{\mathbf{FP}})_{\rm n}$ in $\tilde{\Omega}^{(\rm h)}_{II}\cup\tilde{\Omega}^{(\rm h)}_{IV}\cup\tilde{\Omega}^{(\rm h)}_{VI}$,
which involves the transonic contact discontinuity $\tilde{\Gamma}_{\rm cd}$ as well as the refection on it.
Problem
$(\widetilde{\mathbf{FP}})_{\rm n}$ in this region is
\begin{eqnarray}\label{eq:4.97}
\left\{
\begin{array}{llll}
\partial_{\xi}\delta z^{(\rm n)}_{+}+ \lambda^{(\rm n-1)}_{-}\partial_{\eta}\delta z^{(\rm n)}_{+}=0,
&\  \mbox{in} \ \ 
\tilde{\Omega}^{(\rm h)}_{II}\cup\tilde{\Omega}^{(\rm h)}_{IV}
\cup\tilde{\Omega}^{(\rm h)}_{VI},  \\[5pt]
\delta z^{(\rm n)}_{+}=\delta z^{(\rm n)}_{-}-2\beta^{(\rm n-1)}_{\rm cd,1}\partial_{\xi}\delta\varphi^{(\rm n)}
-2\beta^{(\rm n-1)}_{\rm cd, 2}\partial_{\eta}\delta\varphi^{(\rm n)}-2c_{\rm cd}(\xi),
&\  \mbox{on} \ \ \tilde{\Gamma}_{\rm cd}\cap\{0\leq \xi \leq \xi^{*}_{\rm cd,0}\}.
\end{array}
\right.
\end{eqnarray}
%

\par 
Suppose that the characteristics $\xi=\chi^{(\rm n)}_{-}(\eta, {\color{black}\xi^{*}_{\rm cd, 0}})$ issuing from the point {\color{black}$(\xi^{*}_{\rm cd,0}, 0)$} on $\tilde{ \Gamma}_{\rm cd}$ intersects with the lower nozzle wall $\tilde{\Gamma}_{-}$ at the point $(\xi^{*}_{-,1}, -m^{(\rm h)})$. Then, it satisfies that
\begin{eqnarray}\label{eq:4.104}
\begin{split}
&\int^{-m^{(\rm h)}}_{0}\frac{d\tau}{\lambda^{(\rm n-1)}_{-}(\chi^{(\rm n)}_{-}(\tau, \xi^{*}_{\rm cd,0}),\tau)}=\xi^{*}_{-m^{(\rm h)},1}.
\end{split}
\end{eqnarray}

We first have the following proposition for $\delta z^{(\rm n)}_{+}$ in $\tilde{\Omega}^{(\rm h)}_{II}\cup\tilde{\Omega}^{(\rm h)}_{IV}\cup\tilde{\Omega}^{(\rm h)}_{VI}$.
\begin{proposition}\label{prop:4.5}
For any given $\alpha\in(0,1)$ and for {\color{black} $\tilde{\epsilon}>0$ sufficiently small}, {\color{black}there exist constants 
$\tilde{C}^{*}_{46}>0$, $\mathcal{C}_{\rm h,9}>0$ and $\epsilon_{\rm h,9}>0$ depending only on $\underline{\tilde{U}}$, $L$ and $\alpha$,
such that for $\tilde{\epsilon}<\epsilon_{I}\in (\mathcal{C}_{\rm h,9}\tilde{\epsilon},\epsilon_{\rm h,9})$, if $(\delta\varphi^{(\rm n-1)}, \delta z^{(\rm n-1)})\in \mathscr{K}_{2\epsilon_{I}}$},
  the solution $z^{(\rm n)}_{+}$ of the problem \eqref{eq:4.97} satisfies
\begin{eqnarray}\label{eq:4.105}
\begin{split}
&\|\delta z^{(\rm n)}_{+}\|_{1,\alpha;
\tilde{\Omega}^{(\rm h)}_{II}\cup\tilde{\Omega}^{(\rm h)}_{IV}\cup\tilde{\Omega}^{(\rm h)}_{VI}}\\[5pt]
\leq& \tilde{C}^{*}_{46}\bigg(\|\delta \varphi^{(\rm n)}\|_{2,\alpha;\overline{\tilde{\Gamma}}_{\rm cd}\cap\{0\leq \xi\leq \xi^{*}_{\rm cd, 0}\}}
+\|\delta z^{(\rm n)}_{-}\|_{1,\alpha;\tilde{\Gamma}_{\rm cd}\cap\{0\leq \xi\leq \xi^{*}_{\rm cd, 0}\}}\\[5pt]
&\quad \ \ \
+\sum_{k=\rm e, \rm h}\big\|\delta\tilde{B}^{(\rm k)}_{0}\big\|_{1,\alpha;\tilde{\Gamma}^{(\rm k)}_{\rm in}}
+\sum_{k=\rm e, \rm h}\big\|\delta\tilde{S}^{(\rm k)}_{0}\big\|_{1,\alpha;\tilde{\Gamma}^{(\rm k)}_{\rm in}}\bigg),
\end{split}
\end{eqnarray}
where $\xi^{*}_{\rm cd, 0}$ is defined by \eqref{eq:4.54}.
\end{proposition}


\begin{proof}
We first rewrite the equation for $\delta z^{(\rm n)}_{+}$ as
\begin{eqnarray}\label{eq:4.107}
\begin{split}
\partial_{\eta}\delta z^{(\rm n)}_{+}+\frac{1}{\lambda^{(\rm n-1)}_{-}}\partial_{\xi}\delta z^{(\rm n)}_{+}=0,
\end{split}
\end{eqnarray}
{and still consider $\eta>0$, otherwise we can replace $\eta$ by $-\eta$.}

For any point $(\xi, \eta)\in \tilde{\Omega}^{(\rm h)}_{II}\cup\tilde{\Omega}^{(\rm h)}_{IV}\cup\tilde{\Omega}^{(\rm h)}_{VI}$, we can define a characteristic line
$\xi=\chi^{(\rm n)}_{-}(\eta,\xi_{\rm cd,0})$ by \eqref{eq:4.98} and \eqref{eq:4.99} such that it  intersects with the boundary $\tilde{\Gamma}_{\rm cd}\cap\{0\leq\xi\leq \xi^{*}_{\rm cd,0}\}$
at the point $(\xi_{\rm cd,0},0)$. Then, 
by the boundary condition $\eqref{eq:4.97}_{2}$ on $\tilde{\Gamma}_{\rm cd}\cap\{0\leq \xi\leq \xi^{*}_{\rm cd,0}\}$, we can derive that
\begin{eqnarray}\label{eq:4.107b}
\begin{split}
\delta z^{(\rm n)}_{+}(\xi, \eta)=\delta z_{+}(\xi_{\rm cd,0},0)&=\delta z^{(\rm n)}_{-}(\xi_{\rm cd,0},0)-2\beta^{(\rm n-1)}_{\rm cd,1}\partial_{\xi}\xi_{\rm cd,0}\partial_{\xi_{\rm cd,0}}\delta\varphi^{(\rm n)}(\xi_{\rm cd,0},0)\\[5pt]
&\qquad  \ \
-2\beta^{(\rm n-1)}_{\rm cd, 2}\partial_{\eta}\xi_{\rm cd,0}\partial_{\xi_{\rm cd,0}}\delta\varphi^{(\rm n)}(\xi_{\rm cd,0},0)-2c_{\rm cd}(\xi_{\rm cd,0}).
\end{split}
\end{eqnarray}
Thus, by Lemma \ref{lem:4.6}, there exist positive constants $\mathcal{C}'_{\rm h,9}>\max\{\mathcal{C}_{\rm h,5}, \mathcal{C}_{\rm h,6}\}$ and $0<\epsilon'_{\rm h,9}<\min\{\epsilon_{\rm h,5}, \epsilon_{\rm h,6}\}$ depending only on $\tilde{\underline{U}}$ and $L$,  such that {for $\tilde{\epsilon}<\epsilon_{I}\in (\mathcal{C}'_{\rm h,9}\tilde{\epsilon}, \epsilon'_{\rm h,9})$ with $\tilde{\epsilon}>0$ sufficiently small}, it holds that
{\color{black}
\begin{align}\label{eq:4.107b2}
\begin{split}
&\|\delta z^{(\rm n)}_{+}\|_{0,0;\tilde{\Gamma}_{\rm cd}\cap\{0\leq \xi\leq \xi^{*}_{\rm cd,0}\}}\\
&\leq \|\delta z^{(\rm n)}_{-}\|_{0,0;\tilde{\Gamma}_{\rm cd}\cap\{0\leq \xi\leq \xi^{*}_{\rm cd,0}\}}
+2\|\beta^{(\rm n-1)}_{\rm cd,1}\partial_{\xi}\xi_{\rm cd,0}\partial_{\xi_{\rm cd,0}}\delta\varphi^{(\rm n)}\|_{0,0;\overline{\tilde{\Gamma}}_{\rm cd}\cap\{0\leq \xi\leq \xi^{*}_{\rm cd,0}\}}\\[5pt]
&\quad
+2\|\beta^{(\rm n-1)}_{\rm cd, 2}\partial_{\eta}\xi_{\rm cd,0}\partial_{\xi_{\rm cd,0}}\delta\varphi^{(\rm n)}\|_{0,0;\overline{\tilde{\Gamma}}_{\rm cd}\cap\{0\leq \xi\leq \xi^{*}_{\rm cd,0}\}}+2\|c_{\rm cd}\|_{0,0;\tilde{\Gamma}_{\rm cd}\cap\{0\leq \xi\leq \xi^{*}_{\rm cd,0}\}}\\[5pt]
&\leq \mathcal{C}\Big(\|\delta z^{(\rm n)}_{-}\|_{0,0;\tilde{\Gamma}_{\rm cd}\cap\{0\leq \xi\leq \xi^{*}_{\rm cd,0}\}}
+\|D\delta\varphi^{(\rm n)}\|_{0,0;\overline{\tilde{\Gamma}}_{\rm cd}\cap\{0\leq \xi\leq \xi^{*}_{\rm cd,0}\}}\Big)\\[5pt]
&\quad +\mathcal{C}\Big(\sum_{k=\rm e, \rm h}\big\|\delta\tilde{B}^{(\rm k)}_{0}\big\|_{0,0;\tilde{\Gamma}^{(\rm k)}_{\rm in}}
+\sum_{k=\rm e, \rm h}\big\|\delta\tilde{S}^{(\rm k)}_{0}\big\|_{0,0;\tilde{\Gamma}^{(\rm k)}_{\rm in}}\Big),
\end{split}
\end{align}
}
where $\mathcal{C}=\mathcal{C}(\tilde{\underline{U}},L)>0$.

Next, differentiating the equation $\eqref{eq:4.107}$ with respect to $\xi$ and $\eta$, and then integrating them along the characteristics $\xi=\chi^{(\rm n)}_{+}(\eta,\xi_{\rm cd,0})$, we have
{\color{black}
\begin{align}\label{eq:4.107c}
\begin{split}
\partial_{\xi} \delta z^{(\rm n)}_{+}(\xi,\eta)&=\partial_{\xi}\delta z^{(\rm n)}_{+}(\xi_{\rm cd,0},0)
-\int^{\eta}_{0}\partial_{\xi}\Big(\frac{1}{\lambda^{(\rm n-1)}_{-}}\Big)
\partial_{\xi}\delta z^{(\rm n)}_{+}\big(\chi^{(\rm n)}_{-}(\tau,\xi_{\rm cd,0}),\tau \big)d\tau\\[5pt]
&=\partial_{\xi}\xi_{\rm cd,0}\partial_{\xi_{\rm cd,0}}\delta z^{(\rm n)}_{+}(\xi_{\rm cd,0},0)
-\int^{\eta}_{0}\partial_{\xi}\Big(\frac{1}{\lambda^{(\rm n-1)}_{-}}\Big)
\partial_{\xi}\delta z^{(\rm n)}_{+}\big(\chi^{(\rm n)}_{-}(\tau,\xi_{\rm cd,0}),\tau \big)d\tau,
\end{split}
\end{align}
and
\begin{align}\label{eq:4.107d}
\begin{split}
\partial_{\eta} \delta z^{(\rm n)}_{+}(\xi,\eta)&=\partial_{\eta}\delta z^{(\rm n)}_{+}(\xi_{\rm cd,0},0)
-\int^{\eta}_{0}\partial_{\tau}\Big(\frac{1}{\lambda^{(\rm n-1)}_{-}}\Big)
\partial_{\xi}\delta z^{(\rm n)}_{+}\big(\chi^{(\rm n)}_{-}(\tau,\xi_{\rm cd,0}),\tau \big)d\tau\\[5pt]
&=\partial_{\eta}\xi_{\rm cd,0}\partial_{\xi_{\rm cd,0}}\delta z^{(\rm n)}_{+}(\xi_{\rm cd,0},0)
-\int^{\eta}_{0}\partial_{\tau}\Big(\frac{1}{\lambda^{(\rm n-1)}_{-}}\Big)
\partial_{\xi}\delta z^{(\rm n)}_{+}\big(\chi^{(\rm n)}_{-}(\tau,\xi_{\rm cd,0}),\tau \big)d\tau.
\end{split}
\end{align}
}
For the terms $\partial_{\xi_{\rm cd,0}}\delta z^{(\rm n)}_{+}(\xi_{\rm cd,0},0)$, 
we employ the boundary condition $\eqref{eq:4.97}_{2}$ on $\tilde{\Gamma}_{\rm cd}\cap\{0\leq \xi\leq \xi^{*}_{\rm cd,0}\}$ again to deduce that
{\color{black}
\begin{align}\label{eq:4.107e}
\begin{split}
&\quad \ \partial_{\xi_{\rm cd,0}}\delta z^{(\rm n)}_{+}(\xi_{\rm cd,0},0)\\[5pt]
&=
\partial_{\xi_{\rm cd,0}}\delta z^{(\rm n)}_{-}(\xi_{\rm cd,0},0)
-2\Big(\partial_{\xi}\xi_{\rm cd,0}\partial_{\xi_{\rm cd,0}}\beta^{(\rm n-1)}_{\rm cd,1}
+\partial_{\eta}\xi_{\rm cd,0}\partial_{\xi_{\rm cd,0}}\beta^{(\rm n-1)}_{\rm cd,2}\Big)
\partial_{\xi_{\rm cd,0}}\delta\varphi^{(\rm n)}(\xi_{\rm cd,0},0) \\[5pt]
&\quad -2\Big(\partial_{\xi}\xi_{\rm cd,0}\beta^{(\rm n-1)}_{\rm cd,1} +\partial_{\eta}\xi_{\rm cd,0}\beta^{(\rm n-1)}_{\rm cd,2}\Big)
\partial^{2}_{\xi_{\rm cd,0}\xi_{\rm cd,0}}\delta\varphi^{(\rm n)}(\xi_{\rm cd,0},0)-2c'_{\rm cd}(\xi_{\rm cd,0}).
\end{split}
\end{align}
}

Thus, by Lemma \ref{lem:4.6} and \eqref{eq:4.107e}, and by letting {$\tilde{\epsilon}>0$} sufficiently small, there exist positive constants $\mathcal{C}'_{I,8}$ and $\epsilon'_{I,8}$
depending only on $\tilde{\underline{U}}$ and $L$,  such that for {$\tilde{\epsilon}<\epsilon_{I}\in (\mathcal{C}'_{I,8}\tilde{\epsilon}, \epsilon'_{I,8})$,} it holds that
{\color{black}
\begin{align}\label{eq:4.107f}
\begin{split}
\|D\delta z^{(\rm n)}_{+}\|_{0,0;\tilde{\Gamma}_{\rm cd}\cap\{0\leq \xi\leq \xi^{*}_{\rm cd, 0}\}}
&\leq \|D\xi_{\rm cd,0}\|_{0,0; \tilde{\Omega}^{(\rm h)}_{II}\cup\tilde{\Omega}^{(\rm h)}_{IV}\cup\tilde{\Omega}^{(\rm h)}_{VI}}\|\partial_{\xi_{\rm cd,0}}\delta z^{(\rm n)}_{+}\|_{0,0;\tilde{\Gamma}_{\rm cd}\cap\{0\leq \xi\leq \xi^{*}_{\rm cd,0}\}}\\[5pt]
&\leq \mathcal{C}\Big(\|D\delta z^{(\rm n)}_{-}\|_{0,0;\tilde{\Gamma}_{\rm cd}\cap\{0\leq \xi\leq \xi^{*}_{\rm cd,0}\}}
+\|D\delta\varphi^{(\rm n)}\|_{1,0;\tilde{\Gamma}_{\rm cd}\cap\{0\leq \xi\leq \xi^{*}_{\rm cd,0}\}}\Big)\\[5pt]
&\leq \mathcal{C}\Big(\|D\delta z^{(\rm n)}_{-}\|_{0,0;\tilde{\Gamma}_{\rm cd}\cap\{0\leq \xi\leq \xi^{*}_{\rm cd,0}\}}
+\|D\delta\varphi^{(\rm n)}\|_{1,0;\tilde{\Gamma}_{\rm cd}\cap\{0\leq \xi\leq \xi^{*}_{\rm cd,0}\}}\Big)\\[5pt]
&\quad\ +\mathcal{C}\Big(\sum_{k=\rm e, \rm h}\big\|\delta\tilde{B}^{(\rm k)'}_{0}\big\|_{0,0;\tilde{\Gamma}^{(\rm k)}_{\rm in}}
+\sum_{k=\rm e, \rm h}\big\|\delta\tilde{S}^{(\rm k)'}_{0}\big\|_{0,0;\tilde{\Gamma}^{(\rm k)}_{\rm in}}\Big),
\end{split}
\end{align}
}
where $\mathcal{C}=\mathcal{C}(\tilde{\underline{U}},L)>0$.

By the estimates \eqref{eq:4.107b} and \eqref{eq:4.107f}, following the proof of Proposition \ref{prop:4.3} in \cite{hkwx},
we can choose positive constants $\mathcal{C}'_{I,8}$ and $\epsilon'_{I,8}$ depending only on $\tilde{\underline{U}}$,  such that
{ for $\tilde{\epsilon}<\epsilon_{I}\in (\mathcal{C}'_{I,8}\tilde{\epsilon}, \epsilon'_{I,8})$ with $\tilde{\epsilon}>0$ sufficiently small,}
it holds that
{\color{black}
\begin{eqnarray}\label{eq:4.108}
\begin{split}
\big\|\delta z^{(\rm n)}_{+}\big\|_{1,0; 
	\tilde{\Omega}^{(\rm h)}_{II}\cup\tilde{\Omega}^{(\rm h)}_{IV}\cup\tilde{\Omega}^{(\rm h)}_{VI} }
&\leq \mathcal{C}\Big(\|\delta z^{(\rm n)}_{+}\|_{1,0;\tilde{\Gamma}_{\rm cd}\cap\{0\leq \xi\leq \xi^{*}_{\rm cd,0}\}}
+\|D\delta z^{(\rm n)}_{+}\|_{1,0;\tilde{\Gamma}_{\rm cd}\cap\{0\leq \xi\leq \xi^{*}_{\rm cd,0}\}}\Big) \\[5pt]
&\leq \mathcal{C}\bigg(\|\delta \varphi^{(\rm n)}\|_{{2,0;\overline{\tilde{\Gamma}}_{\rm cd}\cap\{0\leq \xi\leq \xi^{*}_{\rm cd, 0}\}}}
+\|\delta z^{(\rm n)}_{-}\|_{1,0;\tilde{\Gamma}_{\rm cd}\cap\{0\leq \xi\leq \xi^{*}_{\rm cd, 0}\}}\\[5pt]
&\qquad +\sum_{k=\rm e, \rm h}\big\|\delta\tilde{B}^{(\rm k)}_{0}\big\|_{1,0;\tilde{\Gamma}^{(\rm k)}_{\rm in}}
+\sum_{k=\rm e, \rm h}\big\|\delta\tilde{S}^{(\rm k)}_{0}\big\|_{1,0;\tilde{\Gamma}^{(\rm k)}_{\rm in}}\bigg).
\end{split}
\end{eqnarray}
}
Then we only need to estimate $[D\delta z^{(\rm n)}_{+}]_{0,\alpha; \tilde{\Omega}^{(\rm h)}\cap(\tilde{\Omega}^{(\rm h)}_{II}\cup\tilde{\Omega}^{(\rm h)}_{IV}\cup\tilde{\Omega}^{(\rm h)}_{VI})}$.
As in 
the proof of Proposition \ref{prop:4.4}, we take two points $\check{P}=(\xi_{\check{P}},\eta_{\check{P}}),\ \check{Q}=(\xi_{\check{Q}},\eta_{\check{Q}}) \in 
\tilde{\Omega}^{(\rm h)}_{II}\cup \tilde{\Omega}^{(\rm h)}_{IV}\cup \tilde{\Omega}^{(\rm h)}_{VI}$, satisfying \eqref{eq:4.100} with $\check{P}_{-}(0)=(\xi_{\check{P}_{-}(0)},0)$
and $\check{Q}_{-}(0)=(\xi_{\check{Q}_{-}(0)},0)$.
Next, 
we substitute them into \eqref{eq:4.107c} and \eqref{eq:4.107d} 
to obtain
\begin{align}\label{eq:4.112}
\begin{split}
\partial_{\xi} \delta z^{(\rm n)}_{+}(\check{P})&=\partial_{\xi}\xi_{\rm cd,0}(\check{P})\partial_{\xi_{\rm cd,0}}\delta z^{(\rm n)}_{+}(\check{P}_{-}(0))
-\int^{\eta_{\check{P}}}_{0}\partial_{\xi}\Big(\frac{1}{\lambda^{(\rm n-1)}_{-}}\Big)
\partial_{\xi}\delta z^{(\rm n)}_{+}\big(\check{P}_{-}(\tau)\big)d\tau,\\[5pt]
\partial_{\eta} \delta z^{(\rm n)}_{+}(\check{P})
&=\partial_{\eta}\xi_{\rm cd,0}(\check{P}_{-}(0))\partial_{\xi_{\rm cd,0}}\delta z^{(\rm n)}_{+}(\check{P}_{-}(0))
-\int^{\eta_{\check{P}}}_{0}\partial_{\tau}\big(\frac{1}{\lambda^{(\rm n-1)}_{-}}\big)
\partial_{\xi}\delta z^{(\rm n)}_{+}(\check{P}_{-}(\tau))d\tau,
\end{split}
\end{align}
and
\begin{align}\label{eq:4.113}
\begin{split}
\partial_{\xi} \delta z^{(\rm n)}_{+}(\check{Q})&=\partial_{\xi}\xi_{\rm cd,0}(\check{Q})\partial_{\xi_{\rm cd,0}}\delta z^{(\rm n)}_{+}(\check{Q}_{-}(0))
-\int^{\eta_{\check{Q}}}_{0}\partial_{\xi}\Big(\frac{1}{\lambda^{(\rm n-1)}_{-}}\Big)
\partial_{\xi}\delta z^{(\rm n)}_{+}\big(\check{Q}_{-}(\tau)\big)d\tau,\\[5pt]
\partial_{\eta} \delta z^{(\rm n)}_{+}(\check{Q})
&=\partial_{\eta}\xi_{\rm cd,0}(\check{Q}_{-}(0))\partial_{\xi_{\rm cd,0}}\delta z^{(\rm n)}_{+}(\check{Q}_{-}(0))
-\int^{\eta_{\check{Q}}}_{0}\partial_{\tau}\big(\frac{1}{\lambda^{(\rm n-1)}_{-}}\big)
\partial_{\xi}\delta z^{(\rm n)}_{+}(\check{Q}_{-}(\tau))d\tau.
\end{split}
\end{align}
Without loss of   generality, we also assume  that $\eta_{\check{P}}<\eta_{\check{Q}}$
and let $\check{Q}_{-}(\eta_{\check{P}})=(\chi^{(\rm n)}_{-}(\eta_{\check{P}}, \xi_{\check{Q}_{-}(0)}),\eta_{\check{P}})$.
By \eqref{eq:4.112} and \eqref{eq:4.107f} on the contact discontinuity $\tilde{\Gamma}_{\rm cd}\cap\{0\leq \xi\leq \xi^{*}_{\rm cd,0}\}$, we have
\begin{eqnarray*}
\begin{split}
&d^{-\alpha}\big(\check{P},\check{Q}_{-}(\eta_{\check{P}})\big)
\Big|D\xi_{\rm cd,0}(\check{P}) \partial_{\xi_{\rm cd,0}}\delta z^{(\rm n)}_{+}(\check{P}_{-}(0))-D\xi_{\rm cd,0}(\check{Q}_{-}(\eta_{\check{P}}))
\partial_{\xi_{\rm cd,0}}\delta z^{(\rm n)}_{+}(\check{Q}_{-}(0))\Big|\\[5pt]
\leq& d^{-\alpha}\big(\check{P},\check{Q}_{-}(\eta_{\check{P}})\big)
\Big|D\xi_{\rm cd, 0}(\check{P})-D\xi_{\rm cd, 0}(\check{Q}_{-}(\eta_{\check{P}})\Big| \big|\partial_{\xi_{\rm cd,0}}\delta z^{(\rm n)}_{+}(\check{P}_{-}(0))\big|\\[5pt]
&\quad \ \ + \mathcal{C}\big|D\xi_{\rm cd, 0}(\check{Q}_{-}(0))\big|d^{-\alpha}\big(\check{P}_{-}(0),\check{Q}_{-}(0)\big)
\Big|\partial_{\xi_{\rm cd,0}}\delta z^{(\rm n)}_{+}(\check{P}_{-}(0))-\partial_{\xi_{\rm cd,0}}\delta z^{(\rm n)}_{+}(\check{Q}_{-}(0))\Big|\\[5pt]
\leq& \mathcal{C}\big[D\xi_{\rm cd,0}\big]_{0,\alpha;
\tilde{\Omega}^{(\rm h)}_{II}\cup\tilde{\Omega}^{(\rm h)}_{IV}\cup\tilde{\Omega}^{(\rm h)}_{VI}}
\big\|\partial_{\xi}\delta z^{(\rm n)}_{+}\big\|_{0,0;\tilde{\Gamma}_{\rm cd}\cap\{0\leq \xi\leq \xi^{*}_{\rm cd,0}\}}\\[5pt]
&\quad \ \ + \mathcal{C}\big\|D\xi_{\rm cd, 0}\big\|_{0,0;
\tilde{\Omega}^{(\rm h)}_{II}\cup\tilde{\Omega}^{(\rm h)}_{IV}\cup\tilde{\Omega}^{(\rm h)}_{VI}}
d^{-\alpha}\big(\check{P}_{-}(0),\check{Q}_{-}(0)\big)\Big|\partial_{\xi_{\rm cd,0}}\delta z^{(\rm n)}_{+}(\check{P}_{-}(0))-\partial_{\xi_{\rm cd,0}}\delta z^{(\rm n)}_{+}(\check{Q}_{-}(0))\Big|.
\end{split}
\end{eqnarray*}

For the last term, by \eqref{eq:4.107e}, we can choose positive constants $\mathcal{C}''_{\rm h,9}$ and $\epsilon''_{\rm h,9}$ depending only on $\tilde{\underline{U}}$, $L$ and $\alpha$,  such that
for $\tilde{\epsilon}<\epsilon_{I}\in (\mathcal{C}''_{\rm h,9}\tilde{\epsilon}, \epsilon''_{\rm h,9})$ with  $\tilde{\epsilon}>0$ sufficiently small, it holds that
\begin{eqnarray*}
\begin{split}
&d^{-\alpha}\big(\check{P}_{-}(0),\check{Q}_{-}(0)\big)\Big|\partial_{\xi_{\rm cd,0}}\delta z^{(\rm n)}_{+}(\check{P}_{-}(0))-\partial_{\xi_{\rm cd,0}}\delta z^{(\rm n)}_{+}(\check{Q}_{-}(0))\Big|\\[5pt]
\leq &\mathcal{C}\bigg( \big\|D^{2}\delta\varphi^{(\rm n)}\big\|_{{0,\alpha;\overline{\tilde{\Gamma}}_{\rm cd}
\cap\{0\leq \xi\leq \xi^{*}_{\rm cd,0}\}}}+\big\|D\delta z^{(\rm n)}_{-}\big\|_{0,\alpha;\tilde{\Gamma}_{\rm cd}
\cap\{0\leq \xi\leq \xi^{*}_{\rm cd,0}\}}\\[5pt]
&\qquad\ +\sum_{k=\rm e, \rm h}\big\|(\delta\tilde{B}^{(\rm k)}_{0})'\big\|_{0,\alpha;\tilde{\Gamma}^{(\rm k)}_{\rm in}}
+\sum_{k=\rm e, \rm h}\big\|(\delta\tilde{S}^{(\rm k)}_{0})'\big\|_{0,\alpha;\tilde{\Gamma}^{(\rm k)}_{\rm in}}\bigg),
\end{split}
\end{eqnarray*}
where the constant $\mathcal{C}>0$ 
depends only on $\tilde{\underline{U}}$, $L$ and $\alpha$. 

From the estimate \eqref{eq:4.108} and by Lemma \ref{lem:4.6}, we can obtain
\begin{eqnarray*}
\begin{split}
&d^{-\alpha}\big(\check{P},\check{Q}_{-}(\eta_{\check{P}})\big)
\Big|D\xi_{\rm cd,0}(\check{P}) \partial_{\xi_{\rm cd,0}}\delta z^{(\rm n)}_{+}(\check{P}_{-}(0))-D\xi_{\rm cd,0}(\check{Q}_{-}(\eta_{\check{P}}))
\partial_{\xi_{\rm cd,0}}\delta z^{(\rm n)}_{+}(\check{Q}_{-}(0))\Big|\\[5pt]
\leq& \mathcal{C}\bigg( \big\|\delta\varphi^{(\rm n)}\big\|_{{2,\alpha;\overline{\tilde{\Gamma}}_{\rm cd}
\cap\{0\leq \xi\leq \xi^{*}_{\rm cd,0}\}}}+\big\|\delta z^{(\rm n)}_{-}\big\|_{1,\alpha;\tilde{\Gamma}_{\rm cd}
\cap\{0\leq \xi\leq \xi^{*}_{\rm cd,0}\}}\\[5pt]
&\qquad +\sum_{k=\rm e, \rm h}\big\|\delta\tilde{B}^{(\rm k)}_{0}\big\|_{1,\alpha;\tilde{\Gamma}^{(\rm k)}_{\rm in}}
+\sum_{k=\rm e, \rm h}\big\|\delta\tilde{S}^{(\rm k)}_{0}\big\|_{1,\alpha;\tilde{\Gamma}^{(\rm k)}_{\rm in}}\bigg),
\end{split}
\end{eqnarray*}
where the positive constant $\mathcal{C}$ depends only on $\tilde{\underline{U}}$, $L$ and $\alpha$.

Similarly  to the arguments in the proof of Proposition \ref{prop:4.4}, we can also choose positive constants $\mathcal{C}'''_{\rm h,9}$ and $\epsilon'''_{\rm h,9}$
depending only on $\tilde{\underline{U}}$, $L$ and $\alpha$,  such that for {$\tilde{\epsilon}<\epsilon_{I}\in (\mathcal{C}'''_{\rm h,9}\tilde{\epsilon}, \epsilon'''_{\rm h,9})$ with $\tilde{\epsilon}>0$ sufficiently small,} the following
\begin{eqnarray*}
\begin{split}
&\sup_{\check{P},\check{Q}_{-}(\eta_{\check{P}})\in \tilde{\Omega}^{(\rm h)}_{II}\cup \tilde{\Omega}^{(\rm h)}_{IV}\cup \tilde{\Omega}^{(\rm h)}_{VI}}
d^{-\alpha}\big(\check{P},\check{Q}_{-}(\eta_{\check{P}})\big)
\Big|D\delta z^{(\rm n)}_{+}(\check{P})
-D\delta z^{(\rm n)}_{+}(\check{Q}_{-}(\eta_{\check{P}}))\Big|\\[5pt]
&\quad \ \ \   \leq \mathcal{C}\bigg( \big\|\delta\varphi^{(\rm n)}\big\|_{{2,\alpha;\overline{\tilde{\Gamma}}_{\rm cd}
\cap\{0\leq \xi\leq \xi^{*}_{\rm cd,0}\}}}+\big\|\delta z^{(\rm n)}_{-}\big\|_{1,\alpha;\tilde{\Gamma}_{\rm cd}
\cap\{0\leq \xi\leq \xi^{*}_{\rm cd,0}\}}\\[5pt]
&\qquad\quad +\sum_{k=\rm e, \rm h}\big\|\delta\tilde{B}^{(\rm k)}_{0}\big\|_{1,\alpha;\tilde{\Gamma}^{(\rm k)}_{\rm in}}
+\sum_{k=\rm e, \rm h}\big\|\delta\tilde{S}^{(\rm k)}_{0}\big\|_{1,\alpha;\tilde{\Gamma}^{(\rm k)}_{\rm in}}\bigg),
\end{split}
\end{eqnarray*}
holds by the Gronwall inequality.

\par On the other hand, by the equations \eqref{eq:4.107c}--\eqref{eq:4.107d} and \eqref{eq:4.113}, and the estimates \eqref{eq:4.75a}-\eqref{eq:4.75b} and \eqref{eq:4.108}, we can also deduce that
\begin{eqnarray*}
\begin{split}
&\sup_{\check{Q}_{-}(\eta_{\check{P}}),\check{Q}\in \tilde{\Omega}^{(\rm h)}_{II}\cup \tilde{\Omega}^{(\rm h)}_{IV}\cup \tilde{\Omega}^{(\rm h)}_{VI}}
d^{-\alpha}\big(\check{Q}_{-}(\eta_{\check{P}}),\check{Q}\big)
\Big|D\delta z^{(\rm n)}_{+}(\check{Q}_{-}(\eta_{\check{P}}))-D\delta z^{(\rm n)}_{+}(\check{Q})\Big|\\[5pt]
& \ \ \ \leq \sup_{\check{Q}_{-}(\eta_{\check{P}}),\check{Q}\in \tilde{\Omega}^{(\rm h)}_{II}\cup\tilde{\Omega}^{(\rm h)}_{IV}\cup\tilde{\Omega}^{(\rm h)}_{VI}}
d^{-\alpha}\big(\check{Q}_{-}(\eta_{\check{P}}),\check{Q}\big)
\Big|D\xi_{\rm cd,0}(\check{Q}_{-}(\eta_{\check{P}}))-D\xi_{\rm cd,0}(\check{Q})
\Big||\partial_{\xi_{\rm cd,0}}\delta z^{(\rm n)}_{+}(\check{P}_{-}(0))|\\[5pt]
& \ \ \ + \sup_{\check{Q}_{-}(\eta_{\check{P}}),\check{Q}\in \tilde{\Omega}^{(\rm h)}_{II}\cup\tilde{\Omega}^{(\rm h)}_{IV}\cup\tilde{\Omega}^{(\rm h)}_{VI}}
d^{-\alpha}\big(\check{Q}_{-}(\eta_{\check{P}}),\check{Q}\big)
\int^{\eta_{\check{Q}}}_{\eta_{\check{P}}}\bigg|D\Big(\frac{1}{\lambda^{(\rm n-1)}_{-}}\Big)
\partial_{\xi}\delta z^{(\rm n)}_{+}(\check{Q}_{-}(\tau))\bigg|d\tau\\[5pt]
&\ \ \ \   \leq \mathcal{C}\bigg(\big\|D^{2}\delta \varphi^{(\rm n)}\big\|_{{0,0;\overline{\tilde{\Gamma}}_{\rm cd}\cap\{0\leq \xi\leq \xi^{*}_{\rm cd, 0}\}}}
+\sum_{k=\rm e, \rm h}\big\|\big(\delta\tilde{B}^{(\rm k)}_{0},\delta\tilde{S}^{(\rm k)}_{0}\big)\big\|_{1,0;\tilde{\Gamma}^{(\rm k)}_{\rm in}}\bigg),
\end{split}
\end{eqnarray*}
provided that {$\tilde{\epsilon}<\epsilon_{I}\in (\mathcal{C}'''_{\rm h,9}\tilde{\epsilon}, \epsilon'''_{\rm h,9})$ with $\tilde{\epsilon}>0$ sufficiently small.}
Here, the constant $\mathcal{C}>0$ depending only on $\tilde{\underline{U}}$, $L$ and $\alpha$.

\par With the two estimates above, we thus obtain
\begin{eqnarray*}\label{eq:4.53}
\begin{split}
&\sup_{\hat{P},\hat{Q}\in \tilde{\Omega}^{(\rm h)}_{II}\cup \tilde{\Omega}^{(\rm h)}_{IV}\cup \tilde{\Omega}^{(\rm h)}_{VI}}
d^{-\alpha}\big(\check{P},\check{Q}\big)
\Big|D\delta z^{(\rm n)}_{+}(\check{P})-D\delta z^{(\rm n)}_{+}(\check{Q})\Big|\\[5pt]
& \ \ \ \leq \sup_{\check{P}(\eta_{\check{P}}),\check{Q}(\eta_{\check{P}})\in \tilde{\Omega}^{(\rm h)}_{II}\cup \tilde{\Omega}^{(\rm h)}_{IV}\cup \tilde{\Omega}^{(\rm h)}_{VI}}
d^{-\alpha}\big(\check{P}(\eta_{\check{P}}),\check{Q}(\eta_{\check{P}})\big)
\Big|D\delta z^{(\rm n)}_{+}(\check{P}(\eta_{\check{P}}))-D\delta z^{(\rm n)}_{+}(\check{Q}(\eta_{\check{P}}))\Big|\\[5pt]
&\quad  \ \ \ \ +\sup_{\check{Q}(\eta_{\check{P}}),\check{Q}(\eta_{\check{Q}})\in \tilde{\Omega}^{(\rm h)}_{II}\cup \tilde{\Omega}^{(\rm h)}_{IV}\cup \tilde{\Omega}^{(\rm h)}_{VI}}
d^{-\alpha}\big(\check{Q}(\eta_{\check{P}}),\check{Q}(\eta_{\check{Q}})\big)
\Big|D\delta z^{(\rm n)}_{+}(\check{Q}(\eta_{\check{P}}))-D\delta z^{(\rm n)}_{+}(\check{Q}(\eta_{\check{Q}}))\Big|\\[5pt]
& \ \ \ \leq\mathcal{C}\bigg( \big\|D^{2}\delta\varphi^{(\rm n)}\big\|_{{0,\alpha;(\tilde{\Gamma}_{\rm cd}\cup \{\mathcal{O}\})\cap\{0\leq \xi\leq \xi^{*}_{\rm cd,0}\}}}+\big\|D\delta z^{(\rm n)}_{-}\big\|_{0,\alpha;\tilde{\Gamma}_{\rm cd}
\cap\{0\leq \xi\leq \xi^{*}_{\rm cd,0}\}}\\[5pt]
&\qquad\qquad  +\sum_{k=\rm e, \rm h}\big\|\delta\tilde{B}^{(\rm k)}_{0}\big\|_{1,\alpha;\tilde{\Gamma}^{(\rm k)}_{\rm in}}
+\sum_{k=\rm e, \rm h}\big\|\delta\tilde{S}^{(\rm k)}_{0}\big\|_{1,\alpha;\tilde{\Gamma}^{(\rm k)}_{\rm in}}\bigg),
\end{split}
\end{eqnarray*}
which gives  the estimates on $\big[D\delta z^{(\rm n)}_{+}\big]_{0,\alpha; 
\tilde{\Omega}^{(\rm h)}_{II}\cup \tilde{\Omega}^{(\rm h)}_{IV}\cup \tilde{\Omega}^{(\rm h)}_{VI}}$.

\par Finally, combining the estimates on {$\|\delta z^{(\rm n)}_{+}\|_{1,0; \tilde{\Omega}^{(\rm h)}_{II}\cup \tilde{\Omega}^{(\rm h)}_{IV}\cup \tilde{\Omega}^{(\rm h)}_{VI}}$}
and
$\big[D\delta z^{(\rm n)}_{+}\big]_{0,\alpha; \tilde{\Omega}^{(\rm h)}_{II}\cup \tilde{\Omega}^{(\rm h)}_{IV}\cup \tilde{\Omega}^{(\rm h)}_{VI}}$, there exist constants $\tilde{C}^{*}_{46}>0$, and
$\mathcal{C}_{\rm h,9}=\max\{\mathcal{C}'_{\rm h,9},\mathcal{C}''_{\rm h,9}, \mathcal{C}'''_{\rm h,9}\}$ and $\epsilon_{\rm h,9}=\min\{\epsilon'_{\rm h,9},\epsilon''_{\rm h,9}, \epsilon'''_{\rm h,9}\}$
depending only on $\tilde{\underline{U}}$, $L$ and $\alpha$, such that {for $\tilde{\epsilon}<\epsilon_{I}\in (\mathcal{C}_{\rm h,9}\tilde{\epsilon},\epsilon_{\rm h,9})$ with $\tilde{\epsilon}>0$ is sufficiently small,}
the estimate \eqref{eq:4.105} holds. 
\end{proof}

\par With Propositions \ref{prop:4.3}-\ref{prop:4.5}, we now give the proof of Proposition \ref{prop:4.2}.
\begin{proof}[Proof of Proposition \ref{prop:4.2}]
 We will show the estimate \eqref{eq:4.41} by adopting the induction argument.
It follows from \eqref{eq:4.55}, \eqref{eq:4.69}, \eqref{eq:4.81} and
\eqref{eq:4.104} that
\begin{eqnarray}\label{equ:4.112}
\begin{aligned}
\|\delta z^{(n)}_{-}\|_{1,\alpha; 
	\tilde{\Omega}_I\cup\Omega_{II}\cup\Omega_{III}\cup\Omega_{IV}}
\leq C^{*}_{1,-}\Big(\|\delta z_{0}\|_{1,\alpha; \tilde{\Gamma}^{(\rm h)}_{\rm in}}
+\|g_{-}-1\|_{2,\alpha; \tilde{\Gamma}_{-}\cap\{0\leq \xi\leq \xi^{*}_{-,0}\}}\Big),
\end{aligned}
\end{eqnarray}
and
\begin{eqnarray}\label{equ:4.113}
\begin{aligned}
\|\delta z^{(n)}_{+}\|_{1,\alpha; \tilde{\Omega}_I\cup\Omega_{II}\cup\Omega_{III}\cup\Omega_{IV}}
\leq& C^{*}_{1,+}\Big(\|\delta z_{0}\|_{1,\alpha; \tilde{\Gamma}^{(\rm h)}_{\rm in}}
+\sum_{k=\rm e, h}\|\delta\tilde{B}^{(\rm k)}_{0}\|_{1,\alpha; \tilde{\Gamma}^{(\rm k)}_{\rm in}}\\[5pt]
&\ \ \  +\sum_{k=\rm e, h}\|\delta\tilde{S}^{(\rm k)}_{0}\|_{1,\alpha; \tilde{\Gamma}^{(\rm k)}_{\rm in}}
+\|\delta \varphi^{(\rm n)}\|_{2,\alpha; \tilde{\Gamma}_{\rm cd}\cap\{0\leq \xi\leq \xi^{*}_{\rm cd, 0}\}}\Big).
\end{aligned}
\end{eqnarray}

\par Let $\xi_0^*=\min\{\xi^{*}_{\rm cd,0},\,\xi^{*}_{-m^{(\rm h)},0}\}$ and $\tilde{\Omega}_0^{(\rm h)}:=\tilde{\Omega}^{(\rm h)}\cap\{0\leq\xi\leq\xi^*_0\}$. {Then, there exist constants $C^{*}_{1}>0$ and
$\mathcal{C}_{\rm h,0}=\max\{\mathcal{C}_{\rm h,7},\mathcal{C}_{\rm h,8}, \mathcal{C}_{\rm h,9}\}$ and $\epsilon_{\rm h,0}=\min\{\epsilon_{\rm h,7},\epsilon_{\rm h,8}, \epsilon_{\rm h,9}\}$
depending only on $\tilde{\underline{U}}$, $L$ and $\alpha$, such that {for $\tilde{\epsilon}<\epsilon_{I}\in (\mathcal{C}_{\rm h,0}\tilde{\epsilon},\epsilon_{\rm h,0})$ with $\tilde{\epsilon}>0$ is sufficiently small,}
it holds that}
\begin{eqnarray*}
\begin{aligned}
\|\delta z^{(\rm n)}\|_{1,\alpha; \Omega^{(\rm h)}_{0}}
&\leq C^{*}_{1}\Big(\|\delta z_{0}\|_{1,\alpha; \tilde{\Gamma}^{(\rm h)}_{\rm in}}
+\sum_{k=\rm e, h}\|\delta\tilde{B}^{(\rm k)}_{0}\|_{1,\alpha; \tilde{\Gamma}^{(\rm k)}_{\rm in}}
 +\sum_{k=\rm e, h}\|\delta\tilde{S}^{(\rm k)}_{0}\|_{1,\alpha; \tilde{\Gamma}^{(\rm k)}_{\rm in}}\\[5pt]
&\ \ \ \ \ \ \qquad +\|\delta \varphi^{(\rm n)}\|_{2,\alpha;\overline{ \tilde{\Gamma}}_{\rm cd}\cap\{0\leq \xi\leq \xi^{*}_{0}\}}
+\|g_{-}-1\|_{2,\alpha; \tilde{\Gamma}_{-}\cap\{0\leq \xi\leq \xi^{*}_{0}\}}\Big).
\end{aligned}
\end{eqnarray*}

\par  Let $(\xi^{*}_{\rm cd,1},0)$ be the intersection point 
of the characteristics 
with the speed $\lambda_+$ starting from
$(\xi_0^*,-m^{(\rm h)})$ and the transonic contact discontinuity $\eta=0$.
Let $(\xi^{*}_{-m^{(\rm h)},1}, -m^{(h)})$ stand for the point 
where the characteristics 
with the speed $\lambda_-$ starting from $(\xi_0^*,0)$ and the lower nozzle wall $\tilde{\Gamma}^{(\rm h)}_{-}$ intersect.
Set $\tilde{\Omega}_1^{(\rm h)}:=\tilde{\Omega}^{(\rm h)}\cap \{\xi_0^*\leq\xi\leq \xi^*_1\}$.
Regarding the line $\xi=\xi_0^*$ as 
$\xi=0$ and then repeating the arguments for proving Propositions \ref{prop:4.3}--\ref{prop:4.5} again, we have
\begin{eqnarray}\label{eq:4.114}
\begin{aligned}
\|\delta z^{(\rm n)}\|_{1,\alpha; \tilde{\Omega}_1^{(\rm h)}}
&\leq C^{*}_{2}\Big(\|\delta z^{(\rm n)}\|_{1,\alpha; \{\xi=\xi^{*}_{1}\}}
+\sum_{k=\rm e, h}\|\delta\tilde{B}^{(\rm k)}_{0}\|_{1,\alpha; \tilde{\Gamma}^{(\rm k)}_{\rm in}}
+\sum_{k=\rm e, h}\|\delta\tilde{S}^{(\rm k)}_{0}\|_{1,\alpha; \tilde{\Gamma}^{(\rm k)}_{\rm in}}\\[5pt]
& \ \ \ \ \ +\|\delta \varphi^{(\rm n)}\|_{{2,\alpha; \overline{\tilde{\Gamma}}_{\rm cd}\cap\{\xi^{*}_{0}\leq \xi\leq \xi^{*}_{1}\}}}
+\|g_{-}-1\|_{2,\alpha; \tilde{\Gamma}_{-}\cap\{\xi^{*}_{0}\leq \xi\leq \xi^{*}_{1}\}}\Big)\\[5pt]
&\leq \tilde{C}^{*}_{2}\Big(\|\delta z_{0}\|_{1,\alpha; \tilde{\Gamma}^{(\rm h)}_{\rm in}}
+\sum_{k=\rm e, h}\|\delta\tilde{B}^{(\rm k)}_{0}\|_{1,\alpha; \tilde{\Gamma}^{(\rm k)}_{\rm in}}
 +\sum_{k=\rm e, h}\|\delta\tilde{S}^{(\rm k)}_{0}\|_{1,\alpha; \tilde{\Gamma}^{(\rm k)}_{\rm in}}\\[5pt]
&\ \ \ \ \ +\|\delta \varphi^{(\rm n)}\|_{{ 2,\alpha; \overline{\tilde{\Gamma}}_{\rm cd}\cap\{0\leq \xi\leq \xi^{*}_{1}\}}}
+\|g_{-}-1\|_{2,\alpha; \tilde{\Gamma}_{-}\cap\{0\leq \xi\leq \xi^{*}_{1}\}}\Big).
\end{aligned}
\end{eqnarray}

{\color{black}
Finally, by \eqref{eq:4.54}, we notice that $\xi_0^*$ depends on $\lambda^{(\rm n-1)}_{\pm}$ and $m^{(\rm h)}$. Furthermore, \eqref{eq:4.75b} and the uniform bound of  $m^{(\rm h)}$
imply that $\xi_0^*$ also has a uniform  lower bound independent of $\rm n$ and $\epsilon_{I}$. Thus,
 we can repeat the process $\ell$ times with $\ell=[\frac{L}{\xi^{*}_{0}}]+1$ for the finite length $L$, where $\ell$ has a uniform upper bound.}

Define $\xi_\ell^*$ and $\tilde{\Omega}_\ell^{(\rm h)}$.
Obviously $\tilde{\Omega}^{(\rm h)}=\cup_{0\leq k\leq \ell}\tilde{\Omega}_k^{(\rm h)}$.
By taking the summation of all the estimates as in \eqref{eq:4.114} altogether for $k=0$, $\cdot\cdot\cdot$, $\ell$,
we obtain \eqref{eq:4.41}. This completes the proof of Proposition \ref{prop:4.2}.
\end{proof}

Now we give the proof of the main result in this section.
\begin{proof}[Proof of Theorem \ref{thm:4.1}]
By the estimate \eqref{eq:4.4} and Proposition \ref{prop:4.2}, we have
\begin{eqnarray}\label{eq:4.115}
\begin{split}
\|\delta z^{(\rm n)}\|_{1,\alpha;\tilde{\Omega}^{(\rm h)}}
&\leq\mathcal{C}\bigg(\big\|\delta z_{0}\big\|_{1,\alpha;\tilde{\Gamma}^{(\rm h)}_{\rm in}}
+{\big\|\delta\tilde{p}^{(\rm e)}_{0}\big\|_{1,\alpha;\tilde{\Gamma}^{(\rm e)}_{\rm in}}}
+\sum_{k=\rm e, h}\|\delta\tilde{B}^{(\rm k)}_{0}\|_{1,\alpha; \tilde{\Gamma}^{(\rm k)}_{\rm in}}\\[5pt]
&\qquad  +\sum_{k=\rm e, h}\|\delta\tilde{S}^{(\rm k)}_{0}\|_{1,\alpha; \tilde{\Gamma}^{(\rm k)}_{\rm in}}
+\big\|g_{-}+1\big\|_{2,\alpha; \tilde{\Gamma}_{-}}+\big\|g_{+}-1\big\|_{2,\alpha;\tilde{ \Gamma}_{+}}\\[5pt]
&\qquad +\big\|\tilde{\omega}_{\rm e}\big\|^{(-1-\alpha, \{\mathcal{P}_{\rm e}, \mathcal{Q}_{\rm e}\})}_{2,\alpha;\tilde{\Gamma}^{(\rm e)}_{\rm ex}}
+{\big\|\omega_{\rm cd}\big\|^{(-\alpha, \{\mathcal{P}_{\rm e}\})}_{1,\alpha; \tilde{\Gamma}_{\rm cd}}}\bigg).
\end{split}
\end{eqnarray}

The we can choose  constants $\alpha_0\in(0,1)$ depending only on $\tilde{\underline{U}}$, $L$, and $\varrho_{0}>0$ depending on $\tilde{\epsilon}$, $\tilde{\underline{U}}$,
and let   $\epsilon^{*}_{I}=\min\{\varepsilon_{\rm e, 0},\epsilon^{*}_{\rm h, 0} \}$, $\mathcal{C}^{*}_0=\max\{\mathcal{C}_{\rm e, 0}, \mathcal{C}_{\rm h, 0}\}$,
such that for any $\alpha\in(0,\alpha_0)$, $\varrho\in (0,\varrho_{0})$ and $\tilde{\epsilon}<\epsilon_I\in (\mathcal{C}^{*}_0 \tilde{\epsilon},\epsilon^{*}_{I})$ with $\tilde{\epsilon}>0$ sufficiently small,  if $(\delta \varphi^{(\rm n-1)},\delta z^{(\rm n-1)})\in \mathscr{K}_{2\epsilon_I}$ and $\omega_{\rm cd}\in \mathcal{W}$ with $\|\omega_{\rm cd}\|^{(-\alpha, \{\mathcal{P}_{\rm e}\})}_{1,\alpha; \tilde{\Gamma}_{\rm cd}}\leq \varrho$,  we can get the estimate \eqref{eq:4.1} by Proposition \ref{prop:4.1} and \eqref{eq:4.115}. This completes the proof of the Theorem \ref{thm:4.1}.

\end{proof}

\bigskip

\section{Determination of the Flow Slope Function $\omega_{\rm cd}$ as a Fixed Point}\setcounter{equation}{0}
In this section, we employ the implicit function theorem to show that the map $\mathbf{T}_{\omega}$ admits a unique fixed point
$\omega_{\rm cd}$. {\color{black}For the precise statement of the implicit function theorem, we  refer to Theorem 4.B on page 150 of the book \cite{Zeidler}.}
For a fixed $\rm n$, we define
\begin{eqnarray}\label{eq:5.1x}
\mathscr{F}(\omega_{\rm cd}; \boldsymbol{\omega}_{0})=\mathbf{T}_{\omega}(\omega_{\rm cd}; \boldsymbol{\omega}_{0})-\omega_{\rm cd}
=\tan\Big(\frac{\delta z^{(\rm n)}_{-}+\delta z^{(\rm n)}_{+}}{2}\Big)-\omega_{\rm cd},
\end{eqnarray}
where $\boldsymbol{\omega}_{0}=( \tilde{p}^{(\rm e)}_{0}, \tilde{B}^{(\rm e)}_{0}, \tilde{S}^{(\rm e)}_{0}, z_{0}, \tilde{B}^{(\rm h)}_{0},
\tilde{S}^{(\rm h)}_{0},\tilde{\omega}_{\rm e}, g_{+}, g_{-})$.
Let
\begin{eqnarray}\label{eq:5.2x}
\begin{split}
\mathcal{W}_{0}&=C^{1,\alpha}(\tilde{\Gamma}^{(\rm e)}_{\rm in})\times C^{1,\alpha}(\tilde{\Gamma}^{(\rm e)}_{\rm in})\times C^{1,\alpha}(\tilde{\Gamma}^{(\rm e)}_{\rm in})\times (C^{1,\alpha}(\tilde{\Gamma}^{(\rm h)}_{\rm in}))^{2}
\\[5pt]
&\quad \times C^{1,\alpha}(\tilde{\Gamma}^{(\rm h)}_{\rm in})\times C^{1,\alpha}(\tilde{\Gamma}^{(\rm h)}_{\rm in})\times C^{2,\alpha}_{(-1-\alpha, \{\mathcal{P}_{\rm e}, \mathcal{Q}_{\rm e}\})}(\tilde{\Gamma}^{(\rm e)}_{\rm ex}) \times C^{2,\alpha}(\tilde{\Gamma}_{+})\times C^{2,\alpha}(\tilde{\Gamma}_{-}).
\end{split}
\end{eqnarray}
Then the map
\begin{eqnarray}\label{eq:5.3}
\mathscr{F}:\ {\mathcal{W}}
\times \mathcal{W}_{0} \longrightarrow {\mathcal{W}}
\ \ \ \ \mbox{with} \ \ \  \mathscr{F}(0, \underline{\boldsymbol{\omega}})=0,
\end{eqnarray}
with  $\underline{\boldsymbol{\omega}}=(\underline{p}^{(\rm e)}, \underline{\tilde{B}}^{(\rm e)}, \underline{\tilde{S}}^{(\rm e)}, \underline{z},
\underline{\tilde{B}}^{(\rm h)}, \underline{\tilde{S}}^{(\rm h)}, 0, 1, -1)$ has the property
in the following proposition.

\begin{proposition}\label{prop:5.1}
{\color{black}For any given $m^{(\rm e)} \in \mathcal{M}_{\sigma}$ with $\sigma>0$ sufficiently small, there exists a constant $\tilde{\epsilon}^{*}_{0}>0$  depending only on $\tilde{\underline{U}}$ and $L$,} 
such that if
\begin{eqnarray}\label{eq:5.4}
\underline{M}^{(\rm h)}\geq \sqrt{1+L^{2}/4},
\end{eqnarray}
with $\underline{M}^{(\rm h)}=\frac{\underline{u}^{(\rm h)}}{\underline{c}^{(\rm h)}}$
and 
\begin{eqnarray}\label{eq:5.5}
\|\boldsymbol{\omega}_{0}-\underline{\boldsymbol{\omega}}\|_{\mathcal{W}_{0}}<\tilde{\epsilon},
\end{eqnarray}
with {\color{black}$\tilde{\epsilon}\in (0,\tilde{\epsilon}^{*}_{0})$}, 
 the equation $\mathscr{F}(\omega_{\rm cd}; \boldsymbol{\omega}_{0})=0$
admits a unique solution $\omega_{\rm cd}\in {\color{black}C^{1,\alpha}_{(-\alpha, \{\mathcal{P}_{\rm e}\})}(\tilde{\Gamma}_{\rm cd})}$ 
satisfying
{\color{black}
\begin{eqnarray}\label{eq:5.6}
{\|\omega_{\rm cd}\|^{(-\alpha,{\color{black}\{\mathcal{P}_{\rm e}\}})}_{1,\alpha; \tilde{\Gamma}_{\rm cd}}}\leq C_{\omega}\tilde{\epsilon},
\end{eqnarray}
}
where $C_{\omega}>0$ is a constant depending only on $\underline{\boldsymbol{\omega}}$ and $L$.
\end{proposition}

\begin{proof}
We will divide the proof into three steps. For the simplicity of notation, {\color{black}we write $\mathscr{F}(\omega_{\rm cd}; \boldsymbol{\omega}_{0})$
and its derivatives $\mathscr{F}_{\omega_{\rm cd}}[\omega_{\rm cd}; \boldsymbol{\omega}_{0}]$ with respect to $\omega_{\rm cd}$ as $\mathscr{F}(\omega_{\rm cd})$
and $\mathscr{F}_{\omega_{\rm cd}}[\omega_{\rm cd}]$, respectively.}
\smallskip

\par  {\bf Step\ 1}: Definition of the map {\color{black}$\mathscr{F}_{\omega_{\rm cd}}[\omega_{\rm cd}]$.}
\smallskip
\par For any {\color{black}$\omega_{\rm cd}$, $\delta \omega_{\rm cd}\in \mathcal{W}$ 
satisfying the Theorem \ref{thm:4.1}, and for $\tau>0$,
we consider the limit of
\begin{eqnarray*}
\frac{\big\|\mathscr{F}(\omega_{\rm cd}+\tau\delta\omega_{\rm cd})-\mathscr{F}(\omega_{\rm cd})-\tau\mathscr{F}_{\omega_{\rm cd}}[\omega_{\rm cd}](\delta \omega_{\rm cd})\big\|^{(-\alpha,\{\mathcal{P}_{\rm e}\})}_{1,\alpha; \tilde{\Gamma}_{\rm cd}}}{\tau\big\|\delta \omega_{\rm cd}\big\|^{(-\alpha,\{\mathcal{P}_{\rm e}\})}_{1,\alpha; \tilde{\Gamma}_{\rm cd}}}
\end{eqnarray*}
as $\tau\rightarrow 0$.
} 

Let $(\delta\varphi^{(\rm n)}_{\omega_{\rm cd}+\tau \delta \omega_{\rm cd}}, \delta z^{(\rm n)}_{\omega_{\rm cd}+\tau \delta \omega_{\rm cd}})$ and $(\delta\varphi^{(\rm n)}_{\omega_{\rm cd}},\delta z^{(\rm n)}_{\omega_{\rm cd}})$ be the solutions of the problem {\color{black}$(\widetilde{\mathbf{FP}})_{\rm n}$} with the boundary conditions:
\begin{eqnarray*}\label{eq:5.7x}
\partial_{\xi}\delta\varphi^{(\rm n)}_{\omega_{\rm cd}+\tau \delta \omega_{\rm cd}}=\omega_{\rm cd}+\tau \delta \omega_{\rm cd},\ \ \ \
\partial_{\xi}\delta\varphi^{(\rm n)}_{\omega_{\rm cd}}=\omega_{\rm cd},
\end{eqnarray*}
on $\tilde{\Gamma}_{\rm cd}$, respectively.
Denote
\begin{eqnarray*}\label{eq:5.8x}
\dot{\varphi}^{(\rm n, \tau)}_{\omega_{\rm cd}}
=\frac{\delta\varphi^{(\rm n)}_{\omega_{\rm cd}+\tau\delta\omega_{\rm cd}}-\delta\varphi^{(\rm n)}_{\omega_{\rm cd}}}{\tau},\quad \ \
\dot{z}^{(\rm n, \tau)}_{\omega_{\rm cd}}
=\frac{\delta z^{(\rm n)}_{\omega_{\rm cd}+\tau\delta\omega_{\rm cd}}-\delta z^{(\rm n)}_{\omega_{\rm cd}}}{\tau}.
\end{eqnarray*}
By the direct computation, we know that $(\dot \varphi^{(\rm n, \tau)}_{\omega_{\rm cd}}, \dot z^{(\rm n, \tau)}_{\omega_{\rm cd}})$ satisfies the following boundary value problem:
\begin{eqnarray}\label{eq:5.7}
\left\{
\begin{array}{llll}
\sum_{i,j=1,2}\partial_{i}\big(a^{(\rm n)}_{ij, \omega_{\rm cd}}\partial_{j}\dot{\varphi}^{(\rm n, \tau)}_{\omega_{\rm cd}}\big)
=\sum_{i,j=1,2}\partial_{i}\Big(\frac{a^{(\rm n)}_{ij, \omega_{\rm cd}}-a^{(\rm n)}_{ij, \omega_{\rm cd}+\tau\delta\omega_{\rm cd}}}{\tau}
\partial_{j}\delta \varphi^{(\rm n)}_{\omega_{\rm cd}+\tau\delta\omega_{\rm cd}}\Big),
&\ \  \mbox{in}\ \ \  \tilde{\Omega}^{(\rm e)}, \\[5pt]
\partial_{1} \dot{z}^{(\rm n, \tau)}_{-,\omega_{\rm cd}}+ \lambda^{(\rm n-1)}_{+}\partial_{2} \dot{z}^{(\rm n, \tau)}_{-,\omega_{\rm cd}}
=0, &\ \  \mbox{in}\ \ \ \tilde{\Omega}^{(\rm h)},  \\[5pt]
\partial_{1}\dot{z}^{(\rm n, \tau)}_{+,\omega_{\rm cd}}+ \lambda^{(\rm n-1)}_{-}\partial_{2} \dot{z}^{(\rm n, \tau)}_{+,\omega_{\rm cd}}
=0, &\ \  \mbox{in}\ \ \ \tilde{\Omega}^{(\rm h)},  \\[5pt]
\dot{\varphi}^{(\rm n, \tau)}_{\omega_{\rm cd}}=g^{(\rm n, \tau)}_{0, \omega_{\rm cd}}(\eta), &\ \
\mbox{on}\ \ \ \tilde{\Gamma}^{(\rm e)}_{\rm in},\\[5pt]
\partial_{1}\dot{\varphi}^{(\rm n, \tau)}_{\omega_{\rm cd}}=0,  &\ \   \mbox{on}\ \ \ \tilde{\Gamma}^{(\rm e)}_{\rm ex},\\[5pt]
\dot{z}^{(\rm n, \tau)}_{\omega_{\rm cd}}=0,  &\ \   \mbox{on}\ \ \ \tilde{\Gamma}^{(\rm h)}_{\rm in},\\[5pt]
\dot{\varphi}^{(\rm n, \tau)}_{\omega_{\rm cd}}=g^{(\rm n, \tau)}_{+, \omega_{\rm cd}}(\xi), &\ \  \mbox{on}\ \ \ \tilde{\Gamma}_{+},\\[5pt]
\delta\varphi^{(\rm n, \tau)}_{\omega_{\rm cd}}=\int^{\xi}_{0}\delta\omega_{\rm cd}(\nu)d\nu,
&\ \   \mbox{on}\ \ \ \tilde{\Gamma}_{\rm cd},\\[5pt]
\dot{z}^{(\rm n, \tau)}_{-, \omega_{\rm cd}}- \dot{z}^{(\rm n, \tau)}_{+, \omega_{\rm cd}}
=2\beta^{(\rm n-1)}_{\rm cd, 1}\partial_{1}\dot{\varphi}^{(\rm n,\tau)}_{\omega_{\rm cd}}
+2\beta^{(\rm n-1)}_{\rm cd, 2}\partial_{2}\dot{\varphi}^{(\rm n,\tau)}_{\omega_{\rm cd}},
&\ \   \mbox{on}\ \ \ \tilde{\Gamma}_{\rm cd},\\[5pt]
\dot{z}^{(\rm n, \tau)}_{-, \omega_{\rm cd}}+\dot{z}^{(\rm n, \tau)}_{+, \omega_{\rm cd}}=0, &\ \  \mbox{on}\ \ \ \tilde{\Gamma}_{-},
\end{array}
\right.
\end{eqnarray}
where $\lambda^{(\rm n-1)}_{\pm}$,\ $\beta^{(\rm n-1)}_{0,i}$ and $\beta^{(\rm n-1)}_{\textrm {cd}, i}(i=1,2)$ are defined as
in \eqref{eq:3.70}-\eqref{eq:3.72} and \eqref{eq:3.71},
\begin{eqnarray*}
a^{(\rm n)}_{ij,k}=\int^{1}_{0}\partial_{\partial_{j}\varphi}\mathcal{N}_{i}
\big(D\underline{\varphi}+\varsigma D\delta\varphi^{(\rm n)}_{k};
\tilde{B}^{(\rm e)}_{0},\tilde{S}^{(\rm e)}_{0}\big)d\varsigma, \ \ (i,j=1,2),\ \ k=\omega_{\rm cd},\
\omega_{\rm cd}+\tau\delta\omega_{\rm cd},
\end{eqnarray*}
satisfying \eqref{eq:4.8a}-\eqref{eq:4.13},
\begin{eqnarray*}
\begin{split}
g^{(\rm n, \tau)}_{0,\omega_{\rm cd}}(\eta)
&=-\frac{1}{\underline{\beta}_{0,2}}\sum_{j=1,2}\int^{\eta}_{0}
\bigg(\beta^{(\rm n)}_{0,j, \omega_{\rm cd}+\tau\delta\omega_{\rm cd}}-\underline{\beta}_{0,j}\bigg)
\partial_{j}\dot{\varphi}^{(\rm n,\tau)}_{\omega_{\rm cd}}d\mu\\[5pt]
&\ \ \  -\frac{1}{\underline{\beta}_{0,2}}\sum_{j=1,2}\int^{\eta}_{0}
\bigg(\frac{\beta^{(\rm n)}_{0,j, \omega_{\rm cd}+\tau\delta\omega_{\rm cd}}-\beta^{(\rm n)}_{0,j, \omega_{\rm cd}}}{\tau}\bigg)
\partial_{j}\delta \varphi^{(\rm n)}_{\omega_{\rm cd}}d\mu,
\end{split}
\end{eqnarray*}
and $g^{(\rm n, \tau)}_{+,\omega_{\rm cd}}(\xi)=g^{(\rm n, \tau)}_{0,\omega_{\rm cd}}(m^{(\rm e)})$, with
\begin{eqnarray}\label{eq:3.71b}
\begin{split}
&\beta^{(\rm n)}_{0,j, \omega_{\rm cd}+\tau\delta\omega_{\rm cd}}
=\int^{1}_{0}\partial_{\partial_{j}\varphi}\tilde{p}^{(\rm e)}\big(D\underline{\varphi}+\nu D\delta\varphi^{(\rm n)}_{\omega_{\rm cd}+\tau\delta\omega_{\rm cd}};\tilde{B}^{(\rm e)}_{0},\tilde{S}^{(\rm e)}_{0}\big)d\nu,\\[5pt]
&\beta^{(\rm n)}_{0, j,\omega_{\rm cd} }
=\int^{1}_{0}\partial_{\partial_{j}\varphi}\tilde{p}^{(\rm e)}\big(D\underline{\varphi}+\nu D\delta\varphi^{(\rm n)}_{\omega_{\rm cd}};\tilde{B}^{(\rm e)}_{0},\tilde{S}^{(\rm e)}_{0}\big)d\nu.
\end{split}
\end{eqnarray}
As a special case of  
Section 4, the problem \eqref{eq:5.7} has a unique solution $\dot{\varphi}^{(\rm n, \tau)}_{\omega_{\rm cd}}\in {C^{(-1-\alpha,  \tilde{\Sigma}^{(\rm e)}
\setminus \{\mathcal{O}\})}_{2,\alpha}}
(\tilde{\Omega}^{(\rm e)})$ and $\dot{z}^{(\rm n, \tau)}_{\omega_{\rm cd}}\in C^{1, \alpha}(\tilde{\Omega}^{(\rm h)})$ {\color{black}satisfying   
\begin{align}\label{eq:3.71c}
\begin{split}
&\quad \ \|\dot{\varphi}^{(\rm n, \tau)}_{\omega_{\rm cd}}\|^{{(-1-\alpha,\tilde{\Sigma}^{(\rm e)}
\setminus \{\mathcal{O}\})}}_{2,\alpha;\tilde{\Omega}^{(\rm e)}}
+\|\dot{z}^{(\rm n, \tau)}_{\omega_{\rm cd}}\|_{1,\alpha;\tilde{\Omega}^{(\rm h)}}\\[5pt]
&\leq \mathcal{C}\Big(\|\delta\varphi^{(\rm n)}_{\omega_{\rm cd}+\tau\delta\omega_{\rm cd}}\|^{{(-1-\alpha,\tilde{\Sigma}^{(\rm e)}
\setminus \{\mathcal{O}\})}}_{2,\alpha;\tilde{\Omega}^{(\rm e)}}\|\dot{\varphi}^{(\rm n, \tau)}_{\omega_{\rm cd}}\|^{{(-1-\alpha,\tilde{\Sigma}^{(\rm e)}
\setminus \{\mathcal{O}\})}}_{2,\alpha;\tilde{\Omega}^{(\rm e)}}+\|\delta\omega_{\rm cd}\|^{{(-\alpha,\{\mathcal{P}_{\rm e}\})}}_{1,\alpha;\tilde{\Gamma}_{\rm cd}}\Big)\\[5pt]
&\leq \mathcal{C}\Big(\|\omega_{0}-\underline{\omega}\|_{\mathcal{W}_{0}}+\|\omega_{\rm cd}\|^{{(-\alpha,\{\mathcal{P}_{\rm e}\})}}_{1,\alpha;\tilde{\Gamma}_{\rm cd}}
+\tau\|\delta\omega_{\rm cd}\|^{{(-\alpha,\{\mathcal{P}_{\rm e}\})}}_{1,\alpha;\tilde{\Gamma}_{\rm cd}}\Big)\|\dot{\varphi}^{(\rm n, \tau)}_{\omega_{\rm cd}}\|^{{(-1-\alpha,\tilde{\Sigma}^{(\rm e)}
\setminus \{\mathcal{O}\})}}_{2,\alpha;\tilde{\Omega}^{(\rm e)}}\\[5pt]
&\quad\ +\mathcal{C}\|\delta\omega_{\rm cd}\|^{{(-\alpha,\{\mathcal{P}_{\rm e}\})}}_{1,\alpha;\tilde{\Gamma}_{\rm cd}}\\[5pt]
&\leq \mathcal{C}\Big(\tilde{\epsilon}+\varrho+\tau\|\delta\omega_{\rm cd}\|^{{(-\alpha,\{\mathcal{P}_{\rm e}\})}}_{1,\alpha;\tilde{\Gamma}_{\rm cd}}\Big)\|\dot{\varphi}^{(\rm n, \tau)}_{\omega_{\rm cd}}\|^{{(-1-\alpha,\tilde{\Sigma}^{(\rm e)}
\setminus \{\mathcal{O}\})}}_{2,\alpha;\tilde{\Omega}^{(\rm e)}}+\mathcal{C}\|\delta\omega_{\rm cd}\|^{{(-\alpha,\{\mathcal{P}_{\rm e}\})}}_{1,\alpha;\tilde{\Gamma}_{\rm cd}},
\end{split}
\end{align}
for $\alpha\in(0,\alpha_0)$ with $\alpha_{0}\in(0,1)$ depending only on $\tilde{\underline{U}}$ and $L$, where we have used the estimate \eqref{eq:4.1} in Theorem \ref{thm:4.1} for $\delta\varphi^{(\rm n)}_{\omega_{\rm cd}+\tau\delta\omega_{\rm cd}}$ by changing the boundary data
$\omega_{\rm cd}$ to $\omega_{\rm cd}+\tau\delta\omega_{\rm cd}$ on $\tilde{\Gamma}_{\rm cd}$. Here, the constant $\mathcal{C}>0$ depends only on $\tilde{\underline{U}}$, $L$ and $\alpha$.

First, we choose $\tilde{\epsilon}^{*}_{0}>0$, $\varrho^{*}_{0}>0$ and $\tau_{0}>0$ depending only on $\tilde{\underline{U}}$, $L$ and $\alpha$ such that $\mathcal{C}\tilde{\epsilon}^{*}_{0}\leq \frac{1}{4}$, $\mathcal{C}\varrho^{*}_{0}\leq \frac{1}{4}$ and $\mathcal{C}\tau_{0}\|\delta\omega_{\rm cd}\|^{{(-\alpha,\{\mathcal{P}_{\rm e}\})}}_{1,\alpha;\tilde{\Gamma}_{\rm cd}}\leq \frac{1}{4}$.
Then, for $\tilde{\epsilon}\in(0,\tilde{\epsilon}^{*}_{0})$, $\varrho\in(0,\varrho^{*}_0)$ and $\tau\in(0,\tau_0)$, we can get by \eqref{eq:3.71c} that
\begin{align}\label{eq:3.71d}
\begin{split}
\|\dot{\varphi}^{(\rm n, \tau)}_{\omega_{\rm cd}}\|^{{(-1-\alpha,\tilde{\Sigma}^{(\rm e)}
\setminus \{\mathcal{O}\})}}_{2,\alpha;\tilde{\Omega}^{(\rm e)}}
+\|\dot{z}^{(\rm n, \tau)}_{\omega_{\rm cd}}\|_{1,\alpha;\tilde{\Omega}^{(\rm h)}}
\leq \mathcal{C}\|\delta\omega_{\rm cd}\|^{{(-\alpha,\{\mathcal{P}_{\rm e}\})}}_{1,\alpha;\tilde{\Gamma}_{\rm cd}},
\end{split}
\end{align}
where the constant $\mathcal{C}>0$ depends only on $\tilde{\underline{U}}$, $L$ and $\alpha$.}
Therefore, by the compactness, we can
let $(\dot{\varphi}^{(\rm n,0)}_{\omega_{\rm cd}}, \dot{z}^{(\rm n,0)}_{\omega_{\rm cd}})$ be the limit of the solutions
$(\dot{\varphi}^{(\rm n, \tau)}_{\omega_{\rm cd}}, \dot{z}^{(\rm n, \tau)}_{\omega_{\rm cd}})$ as $\tau \rightarrow 0$. Then, we have
\begin{eqnarray}\label{eq:5.8}
\left\{
\begin{array}{llll}
\sum_{i,j=1,2}\partial_{i}\big(a^{(\rm n)}_{ij, \omega_{\rm cd}}\partial_{j}\dot{\varphi}^{(\rm n,0)}_{\omega_{\rm cd}}\big)
=\sum_{i,j=1,2}\partial_{i}\Big(Da^{(\rm n)}_{ij, \omega_{\rm cd}}\cdot D\dot{\varphi}^{(\rm n,0)}_{\omega_{\rm cd}}
\partial_{j}\delta \varphi^{(\rm n)}_{\omega_{\rm cd}}\Big), &\ \  \mbox{in}\ \ \  \tilde{\Omega}^{(\rm e)}, \\[5pt]
\partial_{1}\dot{z}^{(\rm n,0)}_{-,\omega_{\rm cd}}+ \lambda^{(\rm n-1)}_{+}\partial_{2}\dot{z}^{(\rm n,0)}_{-,\omega_{\rm cd}}
=0, &\ \  \mbox{in}\ \ \ \tilde{\Omega}^{(\rm h)},  \\[5pt]
\partial_{1}\dot{z}^{(\rm n,0)}_{+,\omega_{\rm cd}}+ \lambda^{(\rm n-1)}_{-}\partial_{2}\dot{z}^{(\rm n,0)}_{+,\omega_{\rm cd}}
=0, &\ \  \mbox{in}\ \ \ \tilde{\Omega}^{(\rm h)},  \\[5pt]
\dot{\varphi}^{(\rm n,0)}_{\omega_{\rm cd}}=g^{(\rm n,0)}_{0,\omega_{\rm cd}}(\eta), &\ \
\mbox{on}\ \ \ \tilde{\Gamma}^{(\rm e)}_{\rm in},\\[5pt]
\partial_{1}\dot{\varphi}^{(\rm n,0)}_{\omega_{\rm cd}}=0,  &\ \   \mbox{on}\ \ \ \tilde{\Gamma}^{(\rm e)}_{\rm ex},\\[5pt]
\dot{z}^{(\rm n,0)}_{\omega_{\rm cd}}=0,  &\ \   \mbox{on}\ \ \ \tilde{\Gamma}^{(\rm h)}_{\rm in},\\[5pt]
\dot{\varphi}^{(\rm n,0)}_{\omega_{\rm cd}}=g^{(\rm n,0)}_{+,\omega_{\rm cd}}(\xi), &\ \  \mbox{on}\ \ \ \tilde{\Gamma}_{+},\\[5pt]
\dot{\varphi}^{(\rm n,0)}_{\omega_{\rm cd}}=\int^{\xi}_{0}\delta\omega_{\rm cd}(\nu)d\nu,
&\ \   \mbox{on}\ \ \ \tilde{\Gamma}_{\rm cd},\\[5pt]
\dot{z}^{(\rm n,0)}_{-, \omega_{\rm cd}}-\dot{ z}^{(\rm n,0)}_{+, \omega_{\rm cd}}
=2\beta^{(\rm n-1)}_{\rm cd, 1}\partial_{1}\dot{\varphi}^{(\rm n,0)}_{\omega_{\rm cd}}
+2\beta^{(\rm n-1)}_{\rm cd, 2}\partial_{2}\dot{\varphi}^{(\rm n,0)}_{\omega_{\rm cd}},
&\ \   \mbox{on}\ \ \ \tilde{\Gamma}_{\rm cd},\\[5pt]
\dot{z}^{(\rm n,0)}_{-, \omega_{\rm cd}}+\dot{z}^{(\rm n,0)}_{+, \omega_{\rm cd}}=0, &\ \  \mbox{on}\ \ \ \tilde{\Gamma}_{-},
\end{array}
\right.
\end{eqnarray}
with  $g^{(\rm n, 0)}_{0,\omega_{\rm cd}}(\xi)=g^{(\rm n, 0)}_{+,\omega_{\rm cd}}(m^{(\rm e)})$ and
\begin{eqnarray}\label{eq:5.8a}
\begin{split}
g^{(\rm n, 0)}_{0,\omega_{\rm cd}}(\eta)
&=-\frac{1}{\underline{\beta}_{0,2}}\sum_{j=1,2}\int^{\eta}_{0}
\bigg(\beta^{(\rm n)}_{0,j, \omega_{\rm cd}}-\underline{\beta}_{0,j}\bigg)
\partial_{j}\dot{\varphi}^{(\rm n,0)}_{\omega_{\rm cd}}d\mu\\[5pt]
&\ \ \ -\frac{1}{\underline{\beta}_{0,2}}\sum_{j=1,2}\int^{\eta}_{0}D\beta^{(\rm n)}_{0,j, \omega_{\rm cd}}
D\dot{\varphi}^{(\rm n,0)}_{\omega_{\rm cd}}\partial_{j}\delta \varphi^{(\rm n)}_{\omega_{\rm cd}}d\mu.
\end{split}
\end{eqnarray}
{\color{black}Moreover, by the methods used in Section 4, we have the following estimate
\begin{eqnarray}\label{eq:5.8b}
\begin{split}
\|\dot{\varphi}^{(\rm n,0)}_{\omega_{\rm cd}}\|^{{(-1-\alpha,\tilde{\Sigma}^{(\rm e)}
\setminus \{\mathcal{O}\})}}_{2,\alpha;\tilde{\Omega}^{(\rm e)}}+\|\dot{z}^{(\rm n,0)}_{\omega_{\rm cd}}\|_{1,\alpha;\tilde{\Omega}^{(\rm h)}}
\leq \mathcal{C}\|\delta\omega_{\rm cd}\|^{(-\alpha;\{\mathcal{P}_{\rm e}\})}_{1,\alpha; \tilde{\Gamma}_{\rm cd}},
\end{split}
\end{eqnarray}
where the constant $\mathcal{C}>0$ depends only on $\tilde{\underline{U}}$, $L$ and $\alpha$.}

\par {Now, let us consider the convergence rate of the limits $\dot{\varphi}^{(\rm n, \tau)}_{\omega_{\rm cd}}-\dot{\varphi}^{(\rm n,0)}_{\omega_{\rm cd}}$
	and $\dot{z}^{(\rm n, \tau)}_{\omega_{\rm cd}}-\dot{z}^{(\rm n,0)}_{\omega_{\rm cd}}$ as $\tau\rightarrow0$.} By \eqref{eq:5.7} and \eqref{eq:5.8}, $\dot{\varphi}^{(\rm n, \tau)}_{\omega_{\rm cd}}-\dot{\varphi}^{(\rm n,0)}_{\omega_{\rm cd}}$
and $\dot{z}^{(\rm n, \tau)}_{\omega_{\rm cd}}-\dot{z}^{(\rm n,0)}_{\omega_{\rm cd}}$ satisfy the following problem:
\begin{eqnarray}\label{eq:5.9}
\left\{
\begin{array}{llll}
\sum_{i,j=1,2}\partial_{i}\big(a^{(\rm n)}_{ij, \omega_{\rm cd}}\partial_{j}
(\dot{\varphi}^{(\rm n, \tau)}_{\omega_{\rm cd}}-\dot{\varphi}^{(\rm n,0)}_{\omega_{\rm cd}})\big)\\[5pt]
\  =\sum_{i,j=1,2}\partial_{i}\bigg(\frac{a^{(\rm n)}_{ij, \omega_{\rm cd}}-a^{(\rm n)}_{ij, \omega_{\rm cd}+\tau\delta\omega_{\rm cd}}}{\tau}\partial_{j}\delta \varphi^{(\rm n)}_{\omega_{\rm cd}+\tau\delta\omega_{\rm cd}}
-Da^{(\rm n)}_{ij, \omega_{\rm cd}}\cdot D\dot{\varphi}^{(\rm n,0)}_{\omega_{\rm cd}}
\partial_{j}\delta \varphi^{(\rm n)}_{\omega_{\rm cd}}\bigg),
&\  \mbox{in}\   \tilde{\Omega}^{(\rm e)}, \\[5pt]
\partial_{1}(\dot{z}^{(\rm n, \tau)}_{-,\omega_{\rm cd}}-\dot{z}^{(\rm n,0)}_{-,\omega_{\rm cd}})+
\lambda^{(\rm n-1)}_{+}\partial_{2}(\dot{z}^{(\rm n,\tau)}_{-,\omega_{\rm cd}}-\dot{z}^{(\rm n,0)}_{-,\omega_{\rm cd}})=0,
&\ \mbox{in}\ \tilde{\Omega}^{(\rm h)},  \\[5pt]
\partial_{1}(\dot{z}^{(\rm n, \tau)}_{+,\omega_{\rm cd}}-\dot{z}^{(\rm n,0)}_{+,\omega_{\rm cd}})
+\lambda^{(\rm n-1)}_{-}\partial_{2}(\dot{z}^{(\rm n,\tau)}_{+,\omega_{\rm cd}}-\dot{z}^{(\rm n,0)}_{+,\omega_{\rm cd}})=0,
&\  \mbox{in}\ \tilde{\Omega}^{(\rm h)},  \\[5pt]
\dot{\varphi}^{(\rm n, \tau)}_{\omega_{\rm cd}}-\dot{\varphi}^{(\rm n,0)}_{\omega_{\rm cd}}
=g^{(\rm n,\tau)}_{0,\omega_{\rm cd}}(\eta)-g^{(\rm n,0)}_{0,\omega_{\rm cd}}(\eta),
&\ \mbox{on}\  \tilde{\Gamma}^{(\rm e)}_{\rm in},\\[5pt]
\partial_{1}(\dot{\varphi}^{(\rm n,\tau)}_{\omega_{\rm cd}}-\dot{\varphi}^{(\rm n)}_{\omega_{\rm cd}})=0,
&\   \mbox{on}\ \tilde{\Gamma}^{(\rm e)}_{\rm ex},\\[5pt]
\dot{z}^{(\rm n, \tau)}_{\omega_{\rm cd}}-\dot{z}^{(\rm n,0)}_{\omega_{\rm cd}}=0,  &\  \mbox{on}\ \tilde{\Gamma}^{(\rm h)}_{\rm in},\\[5pt]
\dot{\varphi}^{(\rm n, \tau)}_{\omega_{\rm cd}}-\dot{\varphi}^{(\rm n,0)}_{\omega_{\rm cd}}
=g^{(\rm n,\tau)}_{+,\omega_{\rm cd}}(\xi)-g^{(\rm n,0)}_{+,\omega_{\rm cd}}(\xi), &\  \mbox{on}\ \tilde{\Gamma}_{+},\\[5pt]
\dot{\varphi}^{(\rm n, \tau)}_{\omega_{\rm cd}}-\dot{\varphi}^{(\rm n,0)}_{\omega_{\rm cd}}=0,
&\   \mbox{on}\  \tilde{\Gamma}_{\rm cd},\\[5pt]
\dot{z}^{(\rm n, \tau)}_{-, \omega_{\rm cd}}-\dot{z}^{(\rm n,0)}_{-, \omega_{\rm cd}}
-(\dot{z}^{(\rm n, \tau)}_{+, \omega_{\rm cd}}-\dot{z}^{(\rm n,0)}_{+, \omega_{\rm cd}})\\[5pt]
\qquad \ \ =2\beta^{(\rm n-1)}_{\rm cd, 1}\partial_{1}(\dot{\varphi}^{(\rm n,\tau)}_{\omega_{\rm cd}}
-\dot{\varphi}^{(\rm n,0)}_{\omega_{\rm cd}})
+2\beta^{(\rm n-1)}_{\rm cd, 2}\partial_{2}(\dot{\varphi}^{(\rm n,\tau)}_{\omega_{\rm cd}}-\dot{\varphi}^{(\rm n,0)}_{\omega_{\rm cd}}),
&\   \mbox{on}\  \tilde{\Gamma}_{\rm cd},\\[5pt]
\dot{z}^{(\rm n, \tau)}_{-, \omega_{\rm cd}}-\dot{z}^{(\rm n,0)}_{-, \omega_{\rm cd}}
+\dot{z}^{(\rm n, \tau)}_{+, \omega_{\rm cd}}-\dot{z}^{(\rm n,0)}_{+, \omega_{\rm cd}}=0,
&\ \mbox{on}\ \tilde{\Gamma}_{-},
\end{array}
\right.
\end{eqnarray}
with $g^{(\rm n,\tau)}_{+,\omega_{\rm cd}}(\xi)-g^{(\rm n,0)}_{0,\omega_{\rm cd}}(\xi)
=g^{(\rm n,\tau)}_{0,\omega_{\rm cd}}(m^{(\rm e)})-g^{(\rm n,0)}_{0,\omega_{\rm cd}}(m^{(\rm e)})$  and
\begin{eqnarray}\label{eq:5.9a}
\begin{split}
&g^{(\rm n,\tau)}_{0,\omega_{\rm cd}}(\eta)-g^{(\rm n,0)}_{0,\omega_{\rm cd}}(\eta)\\[5pt]
&=-\frac{1}{\underline{\beta}_{0,2}}\sum_{j=1,2}\int^{\eta}_{0}
\bigg(\beta^{(\rm n)}_{0,j, \omega_{\rm cd}+\tau\delta\omega_{\rm cd}}-\underline{\beta}_{0,j}\bigg)
\Big(\partial_{j}\dot{\varphi}^{(\rm n,\tau)}_{\omega_{\rm cd}}-\partial_{j}\dot{\varphi}^{(\rm n,0)}_{\omega_{\rm cd}}\Big)d\mu\\[5pt]
&\ \ \ -\frac{1}{\underline{\beta}_{0,2}}\sum_{j=1,2}\int^{\eta}_{0}
\bigg(\beta^{(\rm n)}_{0,j, \omega_{\rm cd}+\tau\delta\omega_{\rm cd}}-\beta^{(\rm n)}_{0,j, \omega_{\rm cd}}\bigg)
\partial_{j}\dot{\varphi}^{(\rm n,0)}_{\omega_{\rm cd}}d\mu\\[5pt]
&\ \ \  -\frac{1}{\underline{\beta}_{0,2}}\sum_{j=1,2}\int^{\eta}_{0}
\bigg(\frac{\beta^{(\rm n)}_{0,j, \omega_{\rm cd}+\tau\delta\omega_{\rm cd}}-\beta^{(\rm n)}_{0,j, \omega_{\rm cd}}}{\tau}
-D\beta^{(\rm n)}_{0,j,\omega_{\rm cd}}D\dot{\varphi}^{(\rm n,0)}_{\omega_{\rm cd}}\bigg)
\partial_{j}\delta \varphi^{(\rm n)}_{\omega_{\rm cd}}d\mu,
\end{split}
\end{eqnarray}
satisfying
\begin{eqnarray*}
&&\big\|g^{(\rm n,\tau)}_{0,\omega_{\rm cd}}-g^{(\rm n,0)}_{0,\omega_{\rm cd}}\big\|^{{\color{black}(-1-\alpha,\{\mathcal{O}, \mathcal{Q}_{\rm I}\} )}}_{2,\alpha;\tilde{\Gamma}^{(\rm e)}_{\rm in}}\\[5pt]
&&\leq\frac{1}{\underline{\beta}_{0,2}}\Bigg\|\sum_{j=1,2}\int^{\eta}_{0}
\bigg(\beta^{(\rm n)}_{0,j, \omega_{\rm cd}+\tau\delta\omega_{\rm cd}}-\underline{\beta}_{0,j}\bigg)
\Big(\partial_{j}\dot{\varphi}^{(\rm n,\tau)}_{\omega_{\rm cd}}-\partial_{j}\dot{\varphi}^{(\rm n,0)}_{\omega_{\rm cd}}\Big)d\mu\Bigg\|^{{\color{black}(-1-\alpha,\{\mathcal{O}, \mathcal{Q}_{\rm I}\} )}}_{2,\alpha;\tilde{\Gamma}^{(\rm e)}_{\rm in}}\\[5pt]
&&\ \ \ +\frac{1}{\underline{\beta}_{0,2}}\Bigg\|\sum_{j=1,2}\int^{\eta}_{0}
\bigg(\beta^{(\rm n)}_{0,j, \omega_{\rm cd}+\tau\delta\omega_{\rm cd}}-\beta^{(\rm n)}_{0,j, \omega_{\rm cd}}\bigg)
\partial_{j}\dot{\varphi}^{(\rm n,0)}_{\omega_{\rm cd}}d\mu\Bigg\|^{{\color{black}(-1-\alpha,\{\mathcal{O}, \mathcal{Q}_{\rm I}\} )}}_{2,\alpha;\tilde{\Gamma}^{(\rm e)}_{\rm in}}\\[5pt]
&&\ \ \  +\frac{1}{\underline{\beta}_{0,2}}\Bigg\|\sum_{j=1,2}\int^{\eta}_{0}
\bigg(\frac{\beta^{(\rm n)}_{0,j, \omega_{\rm cd}+\tau\delta\omega_{\rm cd}}-\beta^{(\rm n)}_{0,j, \omega_{\rm cd}}}{\tau}
-D\beta^{(\rm n)}_{0,j,\omega_{\rm cd}}D\dot{\varphi}^{(\rm n,0)}_{\omega_{\rm cd}}\bigg)
\partial_{j}\delta \varphi^{(\rm n)}_{\omega_{\rm cd}}d\mu\Bigg\|^{{\color{black}(-1-\alpha,\{\mathcal{O}, \mathcal{Q}_{\rm I}\} )}}_{2,\alpha;\tilde{\Gamma}^{(\rm e)}_{\rm in}}\\[5pt]
&&\leq \mathcal{C}\big\|\delta\varphi^{(\rm n, \tau)}_{\omega_{\rm cd}+\tau\delta\omega_{\rm cd}}\big\|^{{(-1-\alpha,\tilde{\Sigma}^{(\rm e)}
\setminus \{\mathcal{O}\})}}_{2,\alpha;\tilde{\Omega}^{(\rm e)}}
\big\|\dot{\varphi}^{(\rm n, \tau)}_{\omega_{\rm cd}}-\dot{\varphi}^{(\rm n,0)}_{\omega_{\rm cd}}\big\|^{{(-1-\alpha,\tilde{\Sigma}^{(\rm e)}
\setminus \{\mathcal{O}\})}}_{2,\alpha;\tilde{\Omega}^{(\rm e)}}\\[5pt]
&&\ \ \  +\mathcal{C}\big\|\dot{\varphi}^{(\rm n, \tau)}_{\omega_{\rm cd}}\big\|^{{(-1-\alpha,\tilde{\Sigma}^{(\rm e)}
\setminus \{\mathcal{O}\})}}_{2,\alpha;\tilde{\Omega}^{(\rm e)}}\big\|\dot{\varphi}^{(\rm n, 0)}_{\omega_{\rm cd}}\big\|^{{(-1-\alpha,\tilde{\Sigma}^{(\rm e)}
\setminus \{\mathcal{O}\})}}_{2,\alpha;\tilde{\Omega}^{(\rm e)}}\tau\\[5pt]
&&\ \ \ +\mathcal{C}\bigg(\big\|\dot{\varphi}^{(\rm n, \tau)}_{\omega_{\rm cd}}\big\|^{{(-1-\alpha,\tilde{\Sigma}^{(\rm e)}
\setminus \{\mathcal{O}\})}}_{2,\alpha;\tilde{\Omega}^{(\rm e)}}\bigg)^{2}\big\|\delta\varphi^{(\rm n)}_{\omega_{\rm cd}}\big\|^{{(-1-\alpha,\tilde{\Sigma}^{(\rm e)}
\setminus \{\mathcal{O}\})}}_{2,\alpha;\tilde{\Omega}^{(\rm e)}}\tau,
\end{eqnarray*}
for $\sigma>0$ sufficiently small.
Here the constant $\mathcal{C}>0$ depends only on $\underline{U}^{(\rm e)}$, $L$ and $\alpha$.
In the same way, by choosing $\sigma>0$ sufficiently small, we also have
\begin{eqnarray}\label{eq:5.9c}
\begin{split}
&\big\|g^{(\rm n,\tau)}_{+,\omega_{\rm cd}}-g^{(\rm n,0)}_{+,\omega_{\rm cd}}\big\|^{{\color{black}(-1-\alpha, \{\mathcal{Q}_{\rm I}, \mathcal{Q}_{\rm e} \})}}_{2,\alpha;\tilde{\Gamma}_{+}}\\[5pt]
&\leq \mathcal{C}\big\|\delta\varphi^{(\rm n, \tau)}_{\omega_{\rm cd}+\tau\delta\omega_{\rm cd}}\big\|^{{(-1-\alpha,\tilde{\Sigma}^{(\rm e)}
\setminus \{\mathcal{O}\})}}_{2,\alpha;\tilde{\Omega}^{(\rm e)}}\big\|\dot{\varphi}^{(\rm n, \tau)}_{\omega_{\rm cd}}-\dot{\varphi}^{(\rm n,0)}_{\omega_{\rm cd}}\big\|^{{(-1-\alpha,\tilde{\Sigma}^{(\rm e)}
\setminus \{\mathcal{O}\})}}_{2,\alpha;\tilde{\Omega}^{(\rm e)}}\\[5pt]
&\ \ \ +\mathcal{C}\tau\big\|\dot{\varphi}^{(\rm n, \tau)}_{\omega_{\rm cd}}\big\|^{{(-1-\alpha,\tilde{\Sigma}^{(\rm e)}
\setminus \{\mathcal{O}\})}}_{2,\alpha;\tilde{\Omega}^{(\rm e)}}\big\|\dot{\varphi}^{(\rm n, 0)}_{\omega_{\rm cd}}\big\|^{{(-1-\alpha,\tilde{\Sigma}^{(\rm e)}
\setminus \{\mathcal{O}\})}}_{2,\alpha;\tilde{\Omega}^{(\rm e)}}\\[5pt]
&\ \ \ +\mathcal{C}\bigg(\big\|\dot{\varphi}^{(\rm n, \tau)}_{\omega_{\rm cd}}\big\|^{{(-1-\alpha,\tilde{\Sigma}^{(\rm e)}
\setminus \{\mathcal{O}\})}}_{2,\alpha;\tilde{\Omega}^{(\rm e)}}\bigg)^{2}\big\|\delta\varphi^{(\rm n)}_{\omega_{\rm cd}}\big\|^{{(-1-\alpha,\tilde{\Sigma}^{(\rm e)}
\setminus \{\mathcal{O}\})}}_{2,\alpha;\tilde{\Omega}^{(\rm e)}}\tau.
\end{split}
\end{eqnarray}

Similarly to the proof of Proposition \ref{prop:4.1} in Section 4.1, we have in $\tilde{\Omega}^{(\rm e)}$ that

\begin{align*}
&\big\|\dot{\varphi}^{(\rm n, \tau)}_{\omega_{\rm cd}}-\dot{\varphi}^{(\rm n,0)}_{\omega_{\rm cd}}\big\|^{{(-1-\alpha,\tilde{\Sigma}^{(\rm e)}
\setminus \{\mathcal{O}\})}}_{2,\alpha;\tilde{\Omega}^{(\rm e)}}\\[5pt]
\leq& \mathcal{C}\Bigg( \sum_{i,j=1,2}\bigg\|\frac{a^{(\rm n)}_{ij, \omega_{\rm cd}}-a^{(\rm n)}_{ij, \omega_{\rm cd}+\tau\delta\omega_{\rm cd}}}{\tau}\partial_{j}\delta \varphi^{(\rm n)}_{\omega_{\rm cd}+\tau\delta\omega_{\rm cd}}
-Da^{(\rm n)}_{ij, \omega_{\rm cd}}\cdot D\dot{\varphi}^{(\rm n,0)}_{\omega_{\rm cd}}
\partial_{j}\delta \varphi^{(\rm n)}_{\omega_{\rm cd}}\bigg\|^{{(-\alpha,\tilde{\Sigma}^{(\rm e)}
\setminus \{\mathcal{O}\})}}_{1,\alpha;\tilde{\Omega}^{(\rm e)}}\Bigg)\\[5pt]
&\ \ \ +\mathcal{C}\Bigg(\Big\|g^{(\rm n,\tau)}_{0,\omega_{\rm cd}}-g^{(\rm n,0)}_{0,\omega_{\rm cd}}\Big\|^{{\color{black}(-1-\alpha, \{\mathcal{O}, \mathcal{Q}_{\rm I}\})}}_{2,\alpha;\tilde{\Gamma}^{(\rm e)}_{\rm in}}+\Big\|g^{(\rm n,\tau)}_{+,\omega_{\rm cd}}-g^{(\rm n,0)}_{+,\omega_{\rm cd}}\Big\|^{{\color{black}(-1-\alpha,  \{\mathcal{Q}_{\rm I}, \mathcal{Q}_{\rm e}\})}}_{2,\alpha;\tilde{\Gamma}_{+}}\Bigg)\\[5pt]
\leq& \mathcal{C} \sum_{i,j=1,2}\bigg\|Da^{(\rm n)}_{ij, \omega_{\rm cd}}\cdot D\dot{\varphi}^{(\rm n,0)}_{\omega_{\rm cd}}
\Big(\partial_{j}\delta \varphi^{(\rm n)}_{\omega_{\rm cd}+\tau\delta\omega_{\rm cd}}-\partial_{j}\delta \varphi^{(\rm n)}_{\omega_{\rm cd}}\Big)\bigg\|^{{(-\alpha,\tilde{\Sigma}^{(\rm e)}
\setminus \{\mathcal{O}\})}}_{1,\alpha;\tilde{\Omega}^{(\rm e)}}\\[5pt]
&\quad +\mathcal{C} \sum_{i,j=1,2}\bigg\|\Big(\frac{a^{(\rm n)}_{ij, \omega_{\rm cd}}-a^{(\rm n)}_{ij, \omega_{\rm cd}+\tau\delta\omega_{\rm cd}}}{\tau}-Da^{(\rm n)}_{ij, \omega_{\rm cd}}\cdot D\dot{\varphi}^{(\rm n,0)}_{\omega_{\rm cd}}\Big)\partial_{j}\delta \varphi^{(\rm n)}_{\omega_{\rm cd}+\tau\delta\omega_{\rm cd}}\bigg\|^{{(-\alpha,\tilde{\Sigma}^{(\rm e)}
\setminus \{\mathcal{O}\})}}_{1,\alpha;\tilde{\Omega}^{(\rm e)}}\\[5pt]
&\ \ \ +\mathcal{C}\big\|\delta\varphi^{(\rm n, \tau)}_{\omega_{\rm cd}+\tau\delta\omega_{\rm cd}}\big\|^{{(-1-\alpha,\tilde{\Sigma}^{(\rm e)}
\setminus \{\mathcal{O}\})}}_{2,\alpha;\tilde{\Omega}^{(\rm e)}}\big\|\dot{\varphi}^{(\rm n, \tau)}_{\omega_{\rm cd}}-\dot{\varphi}^{(\rm n,0)}_{\omega_{\rm cd}}\big\|^{(-1-\alpha,\tilde{\Sigma}^{(\rm e)}
\setminus \{\mathcal{O}\} )}_{2,\alpha;\tilde{\Omega}^{(\rm e)}}\\[5pt]
&\ \ \ +\mathcal{C}\big\|\dot{\varphi}^{(\rm n, \tau)}_{\omega_{\rm cd}}\big\|^{(-1-\alpha,\tilde{\Sigma}^{(\rm e)}
\setminus \{\mathcal{O}\})}_{2,\alpha;\tilde{\Omega}^{(\rm e)}}\big\|\dot{\varphi}^{(\rm n, 0)}_{\omega_{\rm cd}}\big\|^{(-1-\alpha,\tilde{\Sigma}^{(\rm e)}
\setminus \{\mathcal{O}\})}_{2,\alpha;\tilde{\Omega}^{(\rm e)}}\tau\\[5pt]
&\ \ \ +\mathcal{C}\bigg(\big\|\dot{\varphi}^{(\rm n, \tau)}_{\omega_{\rm cd}}\big\|^{(-1-\alpha,\tilde{\Sigma}^{(\rm e)}
\setminus \{\mathcal{O}\})}_{2,\alpha;\tilde{\Omega}^{(\rm e)}}\bigg)^{2}\big\|\delta\varphi^{(\rm n)}_{\omega_{\rm cd}}\big\|^{(-1-\alpha,\tilde{\Sigma}^{(\rm e)}
\setminus \{\mathcal{O}\})}_{2,\alpha;\tilde{\Omega}^{(\rm e)}}\tau\\[5pt]
\leq& \mathcal{C}\|\dot{\varphi}^{(\rm n,0)}_{\omega_{\rm cd}}\|^{(-1-\alpha,\tilde{\Sigma}^{(\rm e)}
\setminus \{\mathcal{O}\})}_{2,\alpha;\tilde{\Omega}^{(\rm e)}}
\|\dot{\varphi}^{(\rm n,\tau)}_{\omega_{\rm cd}}\|^{(-1-\alpha,\tilde{\Sigma}^{(\rm e)}
\setminus \{\mathcal{O}\})}_{2,\alpha;\tilde{\Omega}^{(\rm e)}}\tau
+\mathcal{C}\Big(\big\|\dot{\varphi}^{(\rm n, \tau)}_{\omega_{\rm cd}}\big\|^{(-1-\alpha,\tilde{\Sigma}^{(\rm e)}
\setminus \{\mathcal{O}\})}_{2,\alpha; \tilde{\Omega}^{(\rm e)}}\Big)^{2}\tau\\[5pt]
&\ \ \ +\mathcal{C}\big\|\delta \varphi^{(\rm n)}_{\omega_{\rm cd}+\tau\delta\omega_{\rm cd}}\big\|^{(-1-\alpha,\tilde{\Sigma}^{(\rm e)}
\setminus \{\mathcal{O}\})}_{2,\alpha;\tilde{\Omega}^{(\rm e)}}\big\|\dot{\varphi}^{(\rm n, \tau)}_{\omega_{\rm cd}}-\dot{\varphi}^{(\rm n,0)}_{\omega_{\rm cd}}\big\|^{{(-1-\alpha,\tilde{\Sigma}^{(\rm e)}
\setminus \{\mathcal{O}\})}}_{2,\alpha;\tilde{\Omega}^{(\rm e)}}\\[5pt]
&\ \ \ +\mathcal{C}\big\|\dot{\varphi}^{(\rm n, \tau)}_{\omega_{\rm cd}}\big\|^{{(-1-\alpha,\tilde{\Sigma}^{(\rm e)}
\setminus \{\mathcal{O}\})}}_{2,\alpha;\tilde{\Omega}^{(\rm e)}}\big\|\dot{\varphi}^{(\rm n, 0)}_{\omega_{\rm cd}}\big\|^{{(-1-\alpha,\tilde{\Sigma}^{(\rm e)}
\setminus \{\mathcal{O}\})}}_{2,\alpha;\tilde{\Omega}^{(\rm e)}}\tau\\[5pt]
&\ \ \ +\mathcal{C}\big\|\delta\varphi^{(\rm n)}_{\omega_{\rm cd}}\big\|^{{(-1-\alpha,\tilde{\Sigma}^{(\rm e)}
\setminus \{\mathcal{O}\})}}_{2,\alpha;\tilde{\Omega}^{(\rm e)}}\bigg(\big\|\dot{\varphi}^{(\rm n, \tau)}_{\omega_{\rm cd}}\big\|^{{(-1-\alpha,\tilde{\Sigma}^{(\rm e)}
\setminus \{\mathcal{O}\})}}_{2,\alpha;\tilde{\Omega}^{(\rm e)}}\bigg)^{2}\tau,
\end{align*}
where the constant $\mathcal{C}=\mathcal{C}(\underline{\tilde{U}},L, \alpha)>0$  is independent of $\rm n$. 
Proposition \ref{prop:4.1} and   \eqref{eq:5.5} yield  
{\color{black}
\begin{eqnarray*}
\begin{split}
&\quad\ \big\|\dot{\varphi}^{(\rm n, \tau)}_{\omega_{\rm cd}}-\dot{\varphi}^{(\rm n,0)}_{\omega_{\rm cd}}\big\|^{{(-1-\alpha,\tilde{\Sigma}^{(\rm e)}
\setminus \{\mathcal{O}\})}}_{2,\alpha;\tilde{\Omega}^{(\rm e)}}\\[5pt]
&\leq \mathcal{C}\Big(\|\omega_{0}-\underline{\omega}\|_{\mathcal{W}_{0}}+\|\omega_{\rm cd}\|^{{(-\alpha,\{\mathcal{P}_{\rm e}\})}}_{1,\alpha;\tilde{\Gamma}_{\rm cd}}
+\tau\|\delta\omega_{\rm cd}\|^{{(-\alpha,\{\mathcal{P}_{\rm e}\})}}_{1,\alpha;\tilde{\Gamma}_{\rm cd}}\Big)
\big\|\dot{\varphi}^{(\rm n, \tau)}_{\omega_{\rm cd}}-\dot{\varphi}^{(\rm n,0)}_{\omega_{\rm cd}}\big\|^{{(-1-\alpha,\tilde{\Sigma}^{(\rm e)}
\setminus \{\mathcal{O}\})}}_{2,\alpha;\tilde{\Omega}^{(\rm e)}}\\[5pt]
&\ \ \ +\mathcal{C}\Big(\big\|\delta\omega_{\rm cd}\big\|^{{(-\alpha,\{\mathcal{P}_{\rm e}\})}}_{1,\alpha;\tilde{\Gamma}_{\rm cd}}\Big)^{2}\tau\\[5pt]
&\leq \mathcal{C}\Big(\tilde{\epsilon}+\varrho+\tau\|\delta\omega_{\rm cd}\|^{{(-\alpha,\{\mathcal{P}_{\rm e}\})}}_{1,\alpha;\tilde{\Gamma}_{\rm cd}}\Big)
\big\|\dot{\varphi}^{(\rm n, \tau)}_{\omega_{\rm cd}}-\dot{\varphi}^{(\rm n,0)}_{\omega_{\rm cd}}\big\|^{{(-1-\alpha,\tilde{\Sigma}^{(\rm e)}
\setminus \{\mathcal{O}\})}}_{2,\alpha;\tilde{\Omega}^{(\rm e)}}
+\mathcal{C}\Big(\big\|\delta\omega_{\rm cd}\big\|^{{(-\alpha,\{\mathcal{P}_{\rm e}\})}}_{1,\alpha;\tilde{\Gamma}_{\rm cd}}\Big)^{2}\tau.
\end{split}
\end{eqnarray*}
}
Thus we can choose $\tilde{\epsilon}^{*}_{0}>0$, $\varrho^{*}_{0}>0$ and $\tau_{0}>0$ depending only on $\tilde{\underline{U}}$, $L$ and $\alpha$ such that $\mathcal{C}\epsilon^{*}_{0}\leq \frac{1}{4}$, $\mathcal{C}\varrho^{*}_{0}\leq \frac{1}{4}$ and $\mathcal{C}\tau_{0}\|\delta\omega_{\rm cd}\|^{{(-\alpha,\{\mathcal{P}_{\rm e}\})}}_{1,\alpha;\tilde{\Gamma}_{\rm cd}}\leq \frac{1}{4}$, and then for $\tilde{\epsilon}\in(0,\tilde{\epsilon}^{*}_{0})$, $\varrho\in(0,\varrho^{*}_0)$ and $\tau\in(0,\tau_0)$, we obtain 
{\color{black}
\begin{eqnarray}\label{eq:5.10}
\begin{split}
\big\|\dot{\varphi}^{(\rm n, \tau)}_{\omega_{\rm cd}}-\dot{\varphi}^{(\rm n,0)}_{\omega_{\rm cd}}\big\|^{{(-1-\alpha,\tilde{\Sigma}^{(\rm e)}
\setminus \{\mathcal{O}\})}}_{2,\alpha;\tilde{\Omega}^{(\rm e)}}
\leq \mathcal{C}\Big(\big\|\delta\omega_{\rm cd}\big\|^{{(-\alpha,\{\mathcal{P}_{\rm e}\})}}_{1,\alpha;\tilde{\Gamma}_{\rm cd}}\Big)^{2} \tau,
\end{split}
\end{eqnarray}
where $\mathcal{C}>0$ is a constant depending  only on $\underline{\tilde{U}}$, $L$ and $\alpha$.
}
\par With the estimate for $\dot{\varphi}^{(\rm n, \tau)}_{\omega_{\rm cd}}-\dot{\varphi}^{(\rm n,0)}_{\omega_{\rm cd}}$ and together
with the boundary condition on $\tilde{\Gamma}_{\rm cd}$, we can further estimate the solution $\dot{z}^{(\rm n, \tau)}_{\omega_{\rm cd}}- \dot{z}^{(\rm n,0)}_{\omega_{\rm cd}}$ in $\tilde{\Omega}^{(\rm h)}$
as in Section 4.2 to deduce that
\begin{eqnarray}\label{eq:5.11}
\begin{split}
\big\|\dot{z}^{(\rm n, \tau)}_{\omega_{\rm cd}}-\dot{z}^{(\rm n,0)}_{\omega_{\rm cd}}\big\|_{1,\alpha;\tilde{\Omega}^{(\rm e)}}
\leq \mathcal{C}\big\|\dot{\varphi}^{(\rm n, \tau)}_{\omega_{\rm cd}}-\dot{\varphi}^{(\rm n,0)}_{\omega_{\rm cd}}\big\|^{{(-1-\alpha,\tilde{\Sigma}^{(\rm e)}
\setminus \{\mathcal{O}\})}}_{2,\alpha;\tilde{\Omega}^{(\rm e)}}.
\end{split}
\end{eqnarray}
Hence, it follows from \eqref{eq:5.10} and \eqref{eq:5.11} that
{\color{black}
\begin{eqnarray}\label{eq:5.11a}
\begin{split}
&\big\|\dot{\varphi}^{(\rm n, \tau)}_{\omega_{\rm cd}}-\dot{\varphi}^{(\rm n,0)}_{\omega_{\rm cd}}\big\|^{{(-1-\alpha,\tilde{\Sigma}^{(\rm e)}
\setminus \{\mathcal{O}\})}}_{2,\alpha;\tilde{\Omega}^{(\rm e)}}
+\big\|\dot{z}^{(\rm n, \tau)}_{\omega_{\rm cd}}-\dot{z}^{(\rm n,0)}_{\omega_{\rm cd}}\big\|_{1,\alpha;\tilde{\Omega}^{(\rm h)}}
\leq \mathcal{C}\Big(\big\|\delta\omega_{\rm cd}\big\|^{{(-\alpha,\{\mathcal{P}_{\rm e}\})}}_{1,\alpha;\tilde{\Gamma}_{\rm cd}}\Big)^{2}\tau.
\end{split}
\end{eqnarray}
}
Therefore,  $(\dot{\varphi}^{(\rm n, \tau)}_{\omega_{\rm cd}}, \dot{z}^{(\rm n, \tau)}_{\omega_{\rm cd}})
\rightarrow (\dot{\varphi}^{(\rm n,0)}_{\omega_{\rm cd}}, \dot{z}^{(\rm n,0)}_{\omega_{\rm cd}})$ as $\tau\rightarrow 0$ in
$C^{{(-1-\alpha,\tilde{\Sigma}^{(\rm e)}
\setminus \{\mathcal{O}\})}}_{2,\alpha}
(\tilde{\Omega}^{(\rm e)})\times C^{1, \alpha}(\tilde{\Omega}^{(\rm h)})$.

{\color{black}Define
\begin{align}\label{eq:5.12}
\mathscr{F}_{\omega_{\rm cd}}[\omega_{\rm cd}](\delta \omega_{\rm cd})
:=\frac{1}{2}\sec^{2}\Big(\frac{\delta z^{(\rm n)}_{-,\omega_{\rm cd}}+\delta z^{(\rm n)}_{+,\omega_{\rm cd}}}{2}\Big)
\big(\dot{z}^{(\rm n,0)}_{-,\omega_{\rm cd}}+\dot{z}^{(\rm n,0)}_{+,\omega_{\rm cd}}\big)-\delta \omega_{\rm cd},
\end{align}
which is a linear map.
Then, we have
\begin{eqnarray*}
\begin{split}
&\mathscr{F}(\omega_{\rm cd}+\tau\delta\omega_{\rm cd})-\mathscr{F}(\omega_{\rm cd})-\tau\mathscr{F}_{\omega_{\rm cd}}[\omega_{\rm cd}](\delta \omega_{\rm cd})\\[5pt]
&\quad =\tan\Big(\frac{\delta z^{(\rm n)}_{-,\omega_{\rm cd}+\tau \delta \omega_{\rm cd}}
+\delta z^{(\rm n)}_{+,\omega_{\rm cd}+\tau \delta \omega_{\rm cd}}}{2}\Big)-\tan\Big(\frac{\delta z^{(\rm n)}_{-,\omega_{\rm cd}}
+\delta z^{(\rm n)}_{+,\omega_{\rm cd}}}{2}\Big)-\tau\delta \omega_{\rm cd}\\[5pt]
&\ \  \ \ \quad - \frac{1}{2}\tau\sec^{2}\Big(\frac{\delta z^{(\rm n)}_{-,\omega_{\rm cd}}+\delta z^{(\rm n)}_{+,\omega_{\rm cd}}}{2}\Big)
\big(\dot{z}^{(\rm n,0)}_{-,\omega_{\rm cd}}+\dot{z}^{(\rm n,0)}_{+,\omega_{\rm cd}}\big)+ \tau\delta \omega_{\rm cd} \\[5pt]
&\quad =\frac{\tau}{2}\int^{1}_{0}\sec^{2}\Big(\frac{\delta z^{(\rm n)}_{-,\omega_{\rm cd}}+\delta z^{(\rm n)}_{+,\omega_{\rm cd}}
+\varsigma\tau(\dot{z}^{(\rm n, \tau)}_{-,\omega_{\rm cd}}+\dot{z}^{(\rm n, \tau)}_{+,\omega_{\rm cd}})}{2}\Big)d\varsigma
\Big(\dot{z}^{(\rm n, \tau)}_{-,\omega_{\rm cd}}+\dot{z}^{(\rm n, \tau)}_{+,\omega_{\rm cd}}\Big)\\[5pt]
&\ \  \ \ \quad - \frac{\tau}{2}\sec^{2}\Big(\frac{\delta z^{(\rm n)}_{-,\omega_{\rm cd}}+\delta z^{(\rm n)}_{+,\omega_{\rm cd}}}{2}\Big)
\big(\dot{z}^{(\rm n,0)}_{-,\omega_{\rm cd}}+\dot{z}^{(\rm n,0)}_{+,\omega_{\rm cd}}\big).
\end{split}
\end{eqnarray*}
Thus, by \eqref{eq:3.71d} and \eqref{eq:5.11a}, we can further deduce that
\begin{eqnarray*}
\begin{split}
& \big\|\mathscr{F}(\omega_{\rm cd}+\tau\delta\omega_{\rm cd})-\mathscr{F}(\omega_{\rm cd})-\tau\mathscr{F}_{\omega_{\rm cd}}[\omega_{\rm cd}](\delta \omega_{\rm cd})\big\|^{(-\alpha,\{\mathcal{P}_{\rm e}\})}_{1,\alpha; \tilde{\Gamma}_{\rm cd}}\\[5pt]
&\qquad\qquad\qquad\qquad \leq \mathcal{C}\Big(\tau\|\dot{z}^{(\rm n, \tau)}_{\omega_{\rm cd}}\big\|^2_{1,\alpha;\tilde{\Omega}^{(\rm h)}}+\big\|\dot{z}^{(\rm n, \tau)}_{\omega_{\rm cd}}-\dot{z}^{(\rm n,0)}_{\omega_{\rm cd}}\big\|_{1,\alpha;\tilde{\Omega}^{(\rm h)}}\Big)\tau\\[5pt]
&\qquad\qquad\qquad\qquad \leq \mathcal{C}\Big(\big\|\delta\omega_{\rm cd}\big\|^{{(-\alpha,\{\mathcal{P}_{\rm e}\})}}_{1,\alpha;\tilde{\Gamma}_{\rm cd}}\Big)^{2}\tau^2,
\end{split}
\end{eqnarray*}
which implies that
\begin{eqnarray*}
\frac{\big\|\mathscr{F}(\omega_{\rm cd}+\tau\delta\omega_{\rm cd})-\mathscr{F}(\omega_{\rm cd})-\tau\mathscr{F}_{\omega_{\rm cd}}[\omega_{\rm cd}](\delta \omega_{\rm cd})\big\|^{(-\alpha,\{\mathcal{P}_{\rm e}\})}_{1,\alpha; \tilde{\Gamma}_{\rm cd}}}{\tau\big\|\delta \omega_{\rm cd}\big\|^{(-\alpha,\{\mathcal{P}_{\rm e}\})}_{1,\alpha; \tilde{\Gamma}_{\rm cd}}} \rightarrow 0,
\end{eqnarray*}
as $\tau\rightarrow 0$. 
Thus, $\mathscr{F}_{\omega_{\rm cd}}[\omega_{\rm cd}]$ is a Fr\'{e}chet derivative with respect to $\omega_{\rm cd}$ for the functional $\mathscr{F}(\omega_{\rm cd})$.
}

\smallskip

{\bf Step\ 2}:  Continuity of the map $\mathscr{F}$ and $\mathscr{F}_{\omega_{\rm cd}}[\omega_{\rm cd}]$
at the point $(\omega_{\rm cd}, \boldsymbol{\omega}_{0})=(0,\underline{\boldsymbol{\omega}})$.
\smallskip

\par The continuity of map $\mathscr{F}$ at the point $(\omega_{\rm cd}, \boldsymbol{\omega}_{0})=(0,\underline{\boldsymbol{\omega}})$
follows directly from Theorem \ref{thm:4.1}. Now, let us consider the continuity of the map
$\mathscr{F}_{\omega_{\rm cd}}[\omega_{\rm cd}]$.
{\color{black}For any fixed $\delta\omega_{\textrm{cd}}$ with $\|\delta\omega_{\textrm{cd}}\|^{(-\alpha, \{\mathcal{P}_{\rm e}\})}_{1,\alpha;\tilde{\Gamma}_{\rm cd}}<+\infty$,}
we assume   $\omega^{k}_{\textrm{cd}}\rightarrow \omega_{\textrm{cd}}$ in ${ C^{1,\alpha}_{(-\alpha, {\color{black}\{\mathcal{P}_{\rm e}\}})}}(\tilde{ \Gamma}_{\rm cd})$
as $k\rightarrow \infty$, and  will show that as $k\rightarrow \infty$,
{\color{black}
\begin{eqnarray}\label{eq:5.13}
\mathscr{F}_{\omega_{\rm cd}}[\omega^{k}_{\rm cd}](\delta \omega_{\textrm cd})
\rightarrow \mathscr{F}_{\omega_{\rm cd}}[\omega_{\rm cd}](\delta \omega_{\textrm cd}),
\ \ \ \mbox{in}\ \  \ { C^{1,\alpha}_{(-\alpha, \{\mathcal{P}_{\rm e}\})}(\tilde{ \Gamma}_{\rm cd})}.
\end{eqnarray}
}
\par By \eqref{eq:5.8}, the solutions
$(\dot{\varphi}^{(\rm n,0)}_{\omega^{k}_{\textrm cd}},\dot{z}^{(\rm n,0)}_{\omega^{k}_{\rm cd}})$
corresponding to $\omega^{k}_{\textrm cd}$ satisfy
\begin{eqnarray}\label{eq:5.14}
\left\{
\begin{array}{llll}
\sum_{i,j=1,2}\partial_{i}\big(a^{(\rm n)}_{ij, \omega^{k}_{\rm cd}}\partial_{j}\dot{\varphi}^{(\rm n,0)}_{\omega^{k}_{\rm cd}}\big)
=\sum_{i,j=1,2}\partial_{i}\Big(Da^{(\rm n)}_{ij, \omega^{k}_{\rm cd}}\cdot D\dot{\varphi}^{(\rm n,0)}_{\omega^{k}_{\rm cd}}
\partial_{j}\delta \varphi^{(\rm n)}_{\omega^{k}_{\rm cd}}\Big), &\ \  \mbox{in}\ \ \  \tilde{\Omega}^{(\rm e)}, \\[5pt]
\partial_{1}\dot{z}^{(\rm n,0)}_{-,\omega^{k}_{\rm cd}}+ \lambda^{(\rm n-1)}_{+}\partial_{2}\dot{z}^{(\rm n,0)}_{-,\omega^{k}_{\rm cd}}
=0, &\ \  \mbox{in}\ \ \ \tilde{\Omega}^{(\rm h)},  \\[5pt]
\partial_{1}\dot{z}^{(\rm n,0)}_{+,\omega^{k}_{\rm cd}}+ \lambda^{(\rm n-1)}_{-}\partial_{2}\dot{z}^{(\rm n,0)}_{+,\omega^{k}_{\rm cd}}
=0, &\ \  \mbox{in}\ \ \ \tilde{\Omega}^{(\rm h)},  \\[5pt]
\dot{\varphi}^{(\rm n,0)}_{\omega^{k}_{\rm cd}}=g^{(\rm n,0)}_{0, \omega^{k}_{\rm cd}}, &\ \
\mbox{on}\ \ \ \tilde{\Gamma}^{(\rm e)}_{\rm in},\\[5pt]
\partial_{1}\dot{\varphi}^{(\rm n,0)}_{\omega^{k}_{\rm cd}}=0,  &\ \   \mbox{on}\ \ \ \tilde{\Gamma}^{(\rm e)}_{\rm ex},\\[5pt]
\dot{z}^{(\rm n,0)}_{\omega^{k}_{\rm cd}}=0,  &\ \   \mbox{on}\ \ \ \tilde{\Gamma}^{(\rm h)}_{\rm in},\\[5pt]
\dot{\varphi}^{(\rm n,0)}_{\omega^{k}_{\rm cd}}=g^{(\rm n,0)}_{+,\omega^{k}_{\rm cd}}, &\ \  \mbox{on}\ \ \ \tilde{\Gamma}_{+},\\[5pt]
\dot{\varphi}^{(\rm n,0)}_{\omega^{k}_{\rm cd}}=\int^{\nu}_{0}{\color{black}\delta\omega_{\rm cd}(\nu)}d\nu,
&\ \   \mbox{on}\ \ \ \tilde{\Gamma}_{\rm cd},\\[5pt]
\dot{z}^{(\rm n,0)}_{-, \omega^{k}_{\rm cd}}-\dot{z}^{(\rm n,0)}_{+, \omega^{k}_{\rm cd}}
=2\beta^{(\rm n-1)}_{\rm cd, 1}\partial_{1}\dot{\varphi}^{(\rm n,0)}_{\omega^{k}_{\rm cd}}
+2\beta^{(\rm n-1)}_{\rm cd, 2}\partial_{2}\dot{\varphi}^{(\rm n,0)}_{\omega^{k}_{\rm cd}},
&\ \   \mbox{on}\ \ \ \tilde{\Gamma}_{\rm cd},\\[5pt]
\dot{z}^{(\rm n,0)}_{-, \omega^{k}_{\rm cd}}+\dot{z}^{(\rm n,0)}_{+, \omega^{k}_{\rm cd}}=0, &\ \  \mbox{on}\ \ \ \tilde{\Gamma}_{-},
\end{array}
\right.
\end{eqnarray}
where $g^{(\rm n,0)}_{0, \omega^{k}_{\rm cd}}$ is given in \eqref{eq:5.8a} by changing $\omega_{\rm cd}$ to $\omega^{k}_{\rm cd}$
and $g^{(\rm n,0)}_{+, \omega^{k}_{\rm cd}}(\xi)=g^{(\rm n,0)}_{0, \omega^{k}_{\rm cd}}(m^{(\rm e)})$.

\par In the same way as   in {\bf Step 1}, we know that
{\color{black}
\begin{eqnarray*}
(\dot{\varphi}^{(\rm n,0)}_{\omega^{k}_{\rm cd}},
\dot{z}^{(\rm n,0)}_{\omega^{k}_{\rm cd}})\rightarrow (\dot{\varphi}^{(\rm n,0)}_{\omega_{\rm cd}},  \dot{z}^{(\rm n,0)}_{\omega_{\rm cd}}), \qquad \mbox{in} \qquad C^{{(-1-\alpha, \tilde{\Sigma}^{(\rm e)}
\setminus \{\mathcal{O}\})}}_{2,\alpha}(\tilde{\Omega}^{(\rm e)})\times C^{1, \alpha}(\tilde{\Omega}^{(\rm h)}),
\end{eqnarray*}
and
\begin{eqnarray*}
\delta z^{(\rm n)}_{\pm, \omega^{k}_{\rm cd}}\rightarrow \delta z^{(\rm n)}_{\pm, \omega_{\rm cd}}, \qquad \mbox{in} \quad C^{1, \alpha}(\tilde{\Omega}^{(\rm h)}), \quad \mbox{as} \quad k\rightarrow \infty.
\end{eqnarray*}
Then
\begin{eqnarray*}
\begin{split}
&\quad\ \big\|\mathscr{F}_{\omega_{\rm cd}}[\omega^{k}_{\rm cd}](\delta \omega_{\textrm cd})-\mathscr{F}_{\omega_{\rm cd}}[\omega_{\rm cd}](\delta \omega_{\textrm cd})\big\|^{(-\alpha, \{\mathcal{P}_{\rm e}\})}_{1,\alpha; \tilde{\Gamma}_{\rm cd}}\\[5pt]
& \leq \frac{1}{2}\bigg\|\sec^{2}\Big(\frac{\delta z^{(\rm n)}_{-,\omega^{k}_{\rm cd}}+\delta z^{(\rm n)}_{+,\omega^{k}_{\rm cd}}}{2}\Big)
\big(\dot{z}^{(\rm n,0)}_{-,\omega^{k}_{\rm cd}}+\dot{z}^{(\rm n,0)}_{+,\omega^{k}_{\rm cd}}\big)-\sec^{2}\Big(\frac{\delta z^{(\rm n)}_{-,\omega_{\rm cd}}+\delta z^{(\rm n)}_{+,\omega_{\rm cd}}}{2}\Big)
\big(\dot{z}^{(\rm n,0)}_{-,\omega_{\rm cd}}+\dot{z}^{(\rm n,0)}_{+,\omega_{\rm cd}}\big)\bigg\|^{(-\alpha, \{\mathcal{P}_{\rm e}\})}_{1,\alpha; \tilde{\Gamma}_{\rm cd}}\\[5pt]
&\leq \mathcal{C}\Big(\|\delta z^{(\rm n)}_{\omega^{k}_{\rm cd}}-\delta z^{(\rm n)}_{\omega_{\rm cd}}\|_{1,\alpha; \tilde{\Omega}^{(\rm h)}}
+\|\dot{z}^{(\rm n,0)}_{\omega^{k}_{\rm cd}}-\dot{z}^{(\rm n,0)}_{\omega_{\rm cd}}\|_{1,\alpha; \tilde{\Omega}^{(\rm h)}}\Big)\\[5pt]
& \rightarrow 0, \qquad \mbox{as} \qquad k\rightarrow \infty,
\end{split}
\end{eqnarray*}
which implies that \eqref{eq:5.13} holds by \eqref{eq:5.11} and Theorem \ref{thm:4.1}.
}
\medskip

\par {\bf Step\ 3}:  ${\left.\mathscr{F}_{\omega_{\rm cd}}[\omega_{\rm cd}]\right.}\big|_{(\omega_{\rm cd}, \boldsymbol{\omega}_{0})=(0,\underline{\boldsymbol{\omega}})}$ is an isomorphism.
\smallskip

\par We need to show that for any given function $f\in { C^{1,\alpha}_{(-\alpha, \{\mathcal{P}_{\rm e}\})}}(\tilde{ \Gamma}_{\rm cd})$,
there exists a unique $\widehat{\delta \omega}_{\rm cd}\in C^{1,\alpha}_{(-\alpha, \{\mathcal{P}_{\rm e}\})}(\tilde{ \Gamma}_{\rm cd})$ such that
$\Big(\left.\mathscr{F}_{\omega_{\rm cd}}[\omega_{\rm cd}]\right.\big|_{(\omega_{\rm cd}, \boldsymbol{\omega}_{0})=(0,\underline{\boldsymbol{\omega}})}\Big)(\widehat{\delta \omega}_{\rm cd}) =f$, i.e.,
\begin{align}\label{eq:5.15}
\frac{1}{2}\Big(\widehat{\dot{z}^{(\rm n,0)}}_{-,\omega_{\rm cd}}+\widehat{\dot{z}^{(\rm n,0)}}_{+,\omega_{\rm cd}}\Big)
-\widehat{\delta \omega}_{\rm cd}=f,\ \ \ \  \mbox{on}\ \ \  \underline{\tilde{\Gamma}}_{\rm cd}.
\end{align}

\par At the background state, the solution $\widehat{\delta z^{(\rm n,0)}}_{\omega_{\rm cd}}$ satisfies
\begin{eqnarray}\label{eq:5.16}
\left\{
\begin{array}{llll}
e_{1}\partial_{11}\widehat{\dot{\varphi}^{(\rm n,0)}}_{\omega_{\rm cd}}+e_{2}\partial_{22}\widehat{\dot{\varphi}^{(\rm n,0)}}_{\omega_{\rm cd}}=0, &\qquad \  \mbox{in}\ \ \  \underline{\tilde{\Omega}}^{(\rm e)}, \\[5pt]
\partial_{1}\widehat{\dot{z}^{(\rm n,0)}}_{-,\omega_{\rm cd}}+ \underline{\lambda}_{+}\partial_{2}\widehat{\dot{z}^{(\rm n,0)}}_{-,\omega_{\rm cd}}=0, &\qquad \  \mbox{in}\ \ \ \underline{\tilde{\Omega}}^{(\rm h)},  \\[5pt]
\partial_{1}\widehat{\dot{z}^{(\rm n,0)}}_{+,\omega_{\rm cd}}+ \underline{\lambda}_{-}\partial_{2}\widehat{\dot{z}^{(\rm n,0)}}_{+,\omega_{\rm cd}}=0, &\qquad \  \mbox{in}\ \ \ \underline{\tilde{\Omega}}^{(\rm h)},  \\[5pt]
\widehat{\dot{\varphi}^{(\rm n,0)}}_{\omega_{\rm cd}}=0, &\qquad \
\mbox{on}\ \ \ \underline{\tilde{\Gamma}}^{(\rm e)}_{\rm in},\\[5pt]
\partial_{1}\widehat{\dot{\varphi}^{(\rm n,0)}}_{\omega_{\rm cd}}=0,  &\qquad \   \mbox{on}\ \ \ \underline{\tilde{\Gamma}}^{(\rm e)}_{\rm ex},\\[5pt]
\widehat{\dot{z}^{(\rm n,0)}}_{\omega_{\rm cd}}=0,  &\qquad \   \mbox{on}\ \ \ \underline{\tilde{\Gamma}}^{(\rm h)}_{\rm in},\\[5pt]
\widehat{\dot{\varphi}^{(\rm n,0)}}_{\omega_{\rm cd}}=0, &\qquad \  \mbox{on}\ \ \ \underline{\tilde{\Gamma}}_{+},\\[5pt]
\partial_{1}\widehat{\dot{\varphi}^{(\rm n,0)}}_{\omega_{\rm cd}}=\widehat{\delta\omega}_{\rm cd},
&\qquad \   \mbox{on}\ \ \ \underline{\tilde{\Gamma}}_{\rm cd},\\[5pt]
\widehat{\dot{z}^{(\rm n,0)}}_{-, \omega_{\rm cd}}-\widehat{\dot{z}^{(\rm n,0)}}_{+, \omega_{\rm cd}}
=2\underline{\beta}_{\rm cd, 2}\partial_{2}\widehat{\dot{\varphi}^{(\rm n,0)}}_{\omega_{\rm cd}},
&\qquad \   \mbox{on}\ \ \ \underline{\tilde{\Gamma}}_{\rm cd},\\[5pt]
\widehat{\dot{z}^{(\rm n,0)}}_{-, \omega_{\rm cd}}+\widehat{\dot{z}^{(\rm n,0)}}_{+, \omega_{\rm cd}}=0, &\qquad \  \mbox{on}\ \ \ \underline{\tilde{\Gamma}}_{-},
\end{array}
\right.
\end{eqnarray}
where $e_{1}$, $e_{2}$ are given by \eqref{eq:4.8a} and
\begin{eqnarray*}
\begin{split}
\underline{\lambda}_{+}=-\underline{\lambda}_{-}=\frac{\underline{\rho}^{(\rm h)}\underline{u}^{(\rm h)}\underline{c}^{(\rm h)}}
{\sqrt{(\underline{u}^{(\rm h)})^{2}-(\underline{c}^{(\rm h)})^{2}}},\ \ \
\underline{\beta}_{\rm cd, 2}=\frac{(\underline{\rho}^{(\rm e)}\underline{c}^{(\rm e)})^{2}(\underline{u}^{(\rm e)})^{3}
\sqrt{(\underline{u}^{(\rm h)})^{2}-(\underline{c}^{(\rm h)})^{2}}}
{\underline{\rho}^{(\rm h)}\underline{c}^{(\rm h)}(\underline{u}^{(\rm h)})^{2}
\big((\underline{c}^{(\rm e)})^{2}-(\underline{u}^{(\rm e)})^{2}\big)}>0.
\end{split}
\end{eqnarray*}

From \eqref{eq:5.4}, since
\begin{align*}
\frac{\underline{m}^{(\rm h)}}{\underline{\lambda}_{+}}=\frac{\underline{\rho}^{(\rm h)}\underline{u}^{(\rm h)}}{\underline{\lambda}_{+}}
=\sqrt{(\underline{M}^{(\rm h)})^{2}-1}\geq L/2,
\end{align*}
by the characteristic method we know that $\widehat{\dot{z}^{(\rm n,0)}}_{-,\omega_{\rm cd}}=0$
in $\underline{\tilde{\Omega}}^{(\rm h)}$.  By  \eqref{eq:5.15}, the
boundary condition for $\widehat{\dot{\varphi}^{(\rm n,0)}}_{\omega_{\rm cd}}$ on $\underline{\tilde{\Gamma}}_{\rm cd}$ is
\begin{align}\label{eq:5.17}
\partial_{1}\widehat{\dot{\varphi}^{(\rm n,0)}}_{\omega_{\rm cd}}+\underline{\beta}_{\rm cd, 2}
\partial_{2}\widehat{\dot{\varphi}^{(\rm n,0)}}_{\omega_{\rm cd}}
=-f,\ \ \ \ \mbox{on} \ \ \  \  \underline{\tilde{\Gamma}}_{\rm cd}.
\end{align}
Then   solving $\widehat{\delta\omega}_{\rm cd}$ is equivalent to solving the following boundary value problem for
$\widehat{\dot{\varphi}^{(\rm n, 0)}}_{\omega_{\rm cd}}$ in $\underline{\tilde{\Omega}}^{(\rm e)}$ for a given function $f\in C^{1,\alpha}_{(-\alpha, \{\mathcal{P}_{\rm e}\})}(\tilde{\underline \Gamma}_{\rm cd})$:
\begin{eqnarray}\label{eq:5.18}
\left\{
\begin{array}{llll}
e_{1}\partial_{11}\widehat{\dot{\varphi}^{(\rm n, 0)}}_{\omega_{\rm cd}}+e_{2}\partial_{22}\widehat{\dot{\varphi}^{(\rm n, 0)}}_{\omega_{\rm cd}}=0, &\qquad \  \mbox{in}\ \ \  \underline{\tilde{\Omega}}^{(\rm e)}, \\[5pt]
\widehat{\dot{\varphi}^{(\rm n, 0)}}_{\omega_{\rm cd}}=0, &\qquad \ \mbox{on}\ \ \ \underline{\tilde{\Gamma}}^{(\rm e)}_{\rm in},\\[5pt]
\partial_{1}\widehat{\dot{\varphi}^{(\rm n, 0)}}_{\omega_{\rm cd}}=0,
&\qquad \   \mbox{on}\ \ \ \underline{\tilde{\Gamma}}^{(\rm e)}_{\rm ex},\\[5pt]
\partial_{1}\widehat{\dot{\varphi}^{(\rm n, 0)}}_{\omega_{\rm cd}}
+\underline{\beta}_{\rm cd, 2}\partial_{2}\widehat{\dot{\varphi}^{(\rm n, 0)}}_{\omega_{\rm cd}}
=-f,    &\qquad \   \mbox{on}\ \ \ \underline{\tilde{\Gamma}}_{\rm cd}.
\end{array}
\right.
\end{eqnarray}

\par In order to show the existence and uniqueness of problem \eqref{eq:5.18}, we introduce a new function
\begin{eqnarray}\label{eq:5.19}
\widehat{\delta\psi^{(\rm n)}}_{\omega_{\rm cd}}:=\partial_{1}\widehat{\dot{\varphi}^{(\rm n,0)}}_{\omega_{\rm cd}}.
\end{eqnarray}
 Then, $\widehat{\delta\psi^{(\rm n)}}$ satisfies
\begin{eqnarray}\label{eq:5.20}
\left\{
\begin{array}{llll}
e_{1}\partial_{11}\widehat{\delta\psi^{(\rm n)}}_{\omega_{\rm cd}}+e_{2}\partial_{22}\widehat{\delta\psi^{(\rm n)}}_{\omega_{\rm cd}}=0, &\qquad \  \mbox{in}\ \ \  \underline{\tilde{\Omega}}^{(\rm e)}, \\[5pt]
\partial_{1}\widehat{\delta\psi^{(\rm n)}}_{\omega_{\rm cd}}=0, &\qquad \ \mbox{on}\ \ \ \underline{\tilde{\Gamma}}^{(\rm e)}_{\rm in},\\[5pt]
\widehat{\delta\psi^{(\rm n)}}_{\omega_{\rm cd}}=0,
&\qquad \   \mbox{on}\ \ \ \underline{\tilde{\Gamma}}^{(\rm e)}_{\rm ex},\\[5pt]
\partial_{1}\widehat{\delta\psi^{(\rm n)}}_{\omega_{\rm cd}}
+\underline{\beta}_{\rm cd, 2}\partial_{2}\widehat{\delta\psi^{(\rm n)}}_{\omega_{\rm cd}}
=-\partial_{1}f,    &\qquad \   \mbox{on}\ \ \ \underline{\tilde{\Gamma}}_{\rm cd}.
\end{array}
\right.
\end{eqnarray}
  By Theorem 1.1 in \cite{lg2},   \eqref{eq:5.20} admits a unique solution
$\widehat{\delta\psi^{(\rm n)}}_{\omega_{\rm cd}}\in C^{(-\alpha,  {\tilde{\Sigma}^{(\rm e)}
\setminus \{\mathcal{O}\}} )}_{1,\alpha}(\underline{\tilde{\Omega}}^{(\rm e)})$.
Then, by \eqref{eq:5.19} and the boundary condition on $\underline{\tilde{\Gamma}}^{(\rm e)}_{\rm ex}$
for $\widehat{\delta\psi^{(\rm n)}}_{\omega_{\rm cd}}$, we see that
{\begin{eqnarray*}
\widehat{\dot{\varphi}^{(\rm n,0)}}_{\omega_{\rm cd}}(\xi,\eta)=\int^{\xi}_{L}\widehat{\delta\psi^{(\rm n)}}_{\omega_{\rm cd}}(s,\eta)ds.
\end{eqnarray*}}
{By \eqref{eq:5.20}}, it implies   $\widehat{\dot{\varphi}^{(\rm n,0)}}_{\omega_{\rm cd}}\in C^{(-1-\alpha,  {\tilde{\Sigma}^{(\rm e)}
\setminus \{\mathcal{O}\}})}_{2,\alpha}(\underline{\tilde{\Omega}}^{(\rm e)})$.

\medskip
Now, we can apply the implicit function theorem to conclude that, there exists a small constant $\tilde{\epsilon}^{*}_{0}>0$ depending only on $\underline{U}$ and $L$ such that
for $\epsilon\in (0, \tilde{\epsilon}^{*}_{0})$, the equation $\mathscr{F}(\omega_{\rm cd}, \boldsymbol{\omega}_{0})=0$ admits a unique solution
$\omega_{\rm cd}=\omega_{\rm cd}(\boldsymbol{\omega}_{0})$ satisfying
\begin{eqnarray*}
\begin{split}
{\|\omega_{\rm cd}\|^{(-\alpha, \{\mathcal{P}_{\rm e}\})}_{1,\alpha; \tilde{\Gamma}_{\rm cd}}}
\leq \tilde{C}_{\omega}\| \boldsymbol{\omega}_{0}-\underline{\boldsymbol{\omega}}\|_{\mathcal{W}_{0}}\leq \tilde{C}_{\omega}\tilde{\epsilon},
\end{split}
\end{eqnarray*}
where $\tilde{C}_{\omega}$ depends only on $\underline{\boldsymbol{\omega}}$ and $L$.

\end{proof}

\bigskip

\section{
Existence and Uniqueness of Problem $(\mathbf{FP})$}\setcounter{equation}{0}

Based on the arguments in Sections 3-4, we will prove in this section 
that the map $\mathcal{J}$   in \eqref{eq:3.76}   constructed by the iteration algorithm $(\mathbf{FP})_{\rm n}$
is a well-defined contractive map so that there exists a fixed point, which is actually the solution of the fixed boundary value problem $(\mathbf{FP})$. We summarize this in the following lemma.

\begin{lemma}\label{lem:6.1}
{\color{black}For any given $m^{(\rm e)} \in \mathcal{M}_{\sigma}$ with $\sigma>0$ sufficiently small and under the assumptions \eqref{eq:3.65}-\eqref{eq:3.66} in Theorem \ref{thm:3.3},
there exist some constants $\alpha_{0}\in(0,1)$, $\mathcal{C}^{*}_{I,0}>0$ and $\tilde{\epsilon}^{*}_{I,0}>0$ depending only on $\tilde{\underline{U}}$ and $L$, 
such that for each $\alpha\in(0,\alpha_{0})$ and $\tilde{\epsilon}<\epsilon_{I} \in(\mathcal{C}^{*}_{I,0}\tilde{\epsilon}, \tilde{\epsilon}^{*}_{I,0})$ with $\tilde{\epsilon}>0$ sufficiently small,}
the iteration algorithm $(\mathbf{FP})_{\rm n}$ generates a well-defined sequence $\big\{(\varphi^{(\rm n)}, z^{(\rm n)})\big\}^{\infty}_{n=1}$ in $\mathscr{K}_{2{\color{black}\epsilon_{I}}}$.
Moreover, the sequence $\big\{(\varphi^{(\rm n)}, z^{(\rm n)})\big\}^{\infty}_{\rm n=1}$
is convergent in {\color{black}$C^{1,\alpha}\big(\tilde{\Omega}^{(\rm e)}\big)\times C^{0,\alpha}\big(\tilde{\Omega}^{(\rm h)}\big)$}
and its limit is the unique solution of the fixed boundary value problem $(\mathbf{FP})$.
\end{lemma}

\begin{proof}
Firstly, it follows from Theorem \ref{thm:4.1} and Proposition \ref{prop:5.1} that
\begin{eqnarray*}
\begin{split}
&\big\|\delta \varphi^{(\rm n)}\big\|^{{(-1-\alpha, \tilde{\Sigma}^{(\rm e)}
\setminus \{\mathcal{O}\})}}_{2, \alpha; \tilde{\Omega}^{(\rm e)}}
+\|\delta z^{(\rm n)}\|_{1, \alpha; \tilde{\Omega}^{(\rm h)}}\\[5pt]
\leq& C^{*}\bigg({\big\|\delta\tilde{p}^{(\rm e)}_{0}\big\|_{1,\alpha;\tilde{\Gamma}^{(\rm e)}_{\rm in} }}+
\sum_{k=\rm e, \rm h}\big\|\delta\tilde{B}^{(\rm k)}_{0}\big\|_{1,\alpha;\tilde{\Gamma}^{(\rm k)}_{\rm in}}
+\sum_{k=\rm e, \rm h}\big\|\delta\tilde{S}^{(\rm k)}_{0}\big\|_{1,\alpha;\tilde{\Gamma}^{(\rm k)}_{\rm in}}\\[5pt]
& \qquad +\big\|\delta z_{0}\big\|_{1,\alpha; \tilde{\Gamma}^{(\rm h)}_{\rm in}}
+{\big\|\omega_{\rm cd}\big\|^{{\color{black}(-\alpha,\{\mathcal{P}_{\rm e}\})}}_{1,\alpha;\tilde{\Gamma}_{\rm cd}}}+\big\|\tilde{\omega}_{\rm e}
\big\|^{{\color{black}(-1-\alpha, \{\mathcal{P}_{\rm e}, \mathcal{Q}_{\rm e}\})}}_{2,\alpha;\tilde{\Gamma}^{(\rm h)}_{\rm in}}
+\big\|g_{+}-1\big\|_{2,\alpha;\tilde{ \Gamma}_{+}}+\big\|g_{-}+1\big\|_{2,\alpha;  \tilde{\Gamma}_{-}}\bigg)\\[5pt]
\leq& \tilde{C}^{*}{\color{black}\tilde{\epsilon}},
\end{split}
\end{eqnarray*}
where the constant $\tilde{C}^*=\tilde{C}^*(\tilde{\underline{U}}, L)>0$ is independent of $\rm n$ and {\color{black}$\tilde{\epsilon}$}.
By taking {\color{black}$\tilde{\epsilon}$} small enough we can get that $\tilde{C}^{*}{\color{black}\tilde{\epsilon}}\leq 2{\color{black}\epsilon_{I}}$, which in turn
indicates that the iteration algorithm $(\mathbf{FP})_{\rm n}$ is well-defined in
$\mathscr{K}_{2{\color{black}\epsilon_{I}}}$. This means that the map $\mathcal{J}$ given in \eqref{eq:3.80} is well defined  and maps $\mathscr{K}_{2{\color{black}\epsilon_{I}}}$ to itself:
\begin{eqnarray*}
\mathcal{J}:\ \mathscr{K}_{2{\color{black}\epsilon_{I}}} \longmapsto \mathscr{K}_{2{\color{black}\epsilon_{I}}}.
\end{eqnarray*}

\par Next, we will show that $\mathcal{J}$ is a contraction mapping in $C^{1}(\tilde{\Omega}^{(\rm e)})\times C^{0}(\tilde{\Omega}^{(\rm h)})$ which 
admits a fixed point in $C^{2,\alpha}_{{(-1-\alpha, \tilde{\Sigma}^{(\rm e)}
\setminus \{\mathcal{O}\})}}(\tilde{\Omega}^{(\rm e)})\times C^{1}(\tilde{\Omega}^{(\rm h)})$.
Let
$(\varphi^{(\rm n+1)}, z^{(\rm n+1)}):=\mathcal{J}(\varphi^{(\rm n)}, z^{(\rm n)})$ and
\begin{eqnarray}\label{eq:5.1}
\begin{split}
\delta Z^{(\rm n+1)}=\delta z^{(\rm n+1)}-\delta z^{(\rm n)},\ \ \
\delta\Phi^{(\rm n+1)}=\delta\varphi^{(\rm n+1)}-\delta\varphi^{(\rm n)},
\end{split}
\end{eqnarray}
where $\delta z^{(\rm n+1)}, \ \delta z^{(\rm n)}$ and $\delta\varphi^{(\rm n+1)},\ \delta\varphi^{(\rm n)}$
are the solutions to the $(\mathbf{FP})_{\rm n}$, respectively.
 Then, $\delta Z^{(\rm n+1)}$ and $\delta\Phi^{(\rm n+1)}$ satisfy
\begin{eqnarray}\label{eq:6.2-a}
\left\{
\begin{array}{llll}
\sum_{i,j=1,2}\partial_{j}\Big(a^{(\rm n)}_{ij}\partial_{i}\delta\Phi^{(\rm n+1)}\Big)=
\sum_{i,j=1,2}\partial_{j}\Big(\big(a^{(\rm n)}_{ij}-a^{(\rm n+1)}_{ij}\big)\partial_{i}\delta\varphi^{(\rm n+1)}\Big),
&\ \ \  \mbox{in} \ \ \  \tilde{\Omega}^{(\rm e)},   \\[5pt]
\partial_{1}\delta Z^{(\rm n+1)}+\textrm{diag}\big(\lambda^{(\rm n)}_{+},\ \lambda^{(\rm n)}_{-}\big)
\partial_{2}\delta Z^{(\rm n+1)}\\[5pt]
\ \ \qquad\qquad  =-\textrm{diag} \Big(\lambda^{(\rm n)}_{+}-\lambda^{(\rm n-1)}_{+}, \
\lambda^{(\rm n)}_{-}-\lambda^{(\rm n-1)}_{-}\Big)\partial_{2}\delta z^{(\rm n)},
&\ \ \  \mbox{in} \ \ \ \tilde{\Omega}^{(\rm h)},  \\[5pt]
\delta \Phi^{(\rm n+1)}=g^{(\rm n+1)}_{0}(\eta)-g^{(\rm n)}_{0}(\eta), &\ \ \  \mbox{on} \ \ \ \tilde{\Gamma}^{(\rm e)}_{\rm in}, \\[5pt]
\partial_{1}\delta \Phi^{(\rm n+1)}=0,  &\ \ \  \mbox{on} \ \ \ \tilde{\Gamma}^{(\rm e)}_{\rm ex}, \\[5pt]
\delta Z^{(\rm n+1)}=0, &\ \ \  \mbox{on} \ \ \ \tilde{\Gamma}^{(\rm h)}_{\rm in}, \\[5pt]
\delta \Phi^{(\rm n+1)}= g^{(\rm n+1)}_{+}(\xi)-g^{(\rm n)}_{+}(\xi), &\ \ \  \mbox{on} \ \ \ \tilde{\Gamma}_{+},\\[5pt]
\delta Z^{(\rm n+1)}_{-}+\delta Z^{(\rm n+1)}_{+}=2\arctan\partial_{1}(\delta \varphi^{(\rm n)}+\delta \Phi^{(\rm n+1)})-2\arctan\partial_{1}\delta \varphi^{(\rm n)},
&\ \ \  \mbox{on} \ \ \ \tilde{\Gamma}_{\rm cd}, \\[5pt]
\delta Z^{(\rm n+1)}_{-}-\delta Z^{(\rm n+1)}_{+}=2\beta^{(\rm n)}_{\rm cd, 1}\partial_{1}\delta \Phi^{(\rm n+1)}
+2\beta^{(\rm n)}_{\rm cd,2}\partial_{2}\delta \Phi^{(\rm n+1)}\\[5pt]
\ \ \ \qquad\qquad  +2\big(\beta^{(\rm n)}_{\rm cd,1}-\beta^{(\rm n-1)}_{\rm cd, 1}\big)\partial_{1}\delta \varphi^{(\rm n)}
+2\big(\beta^{(\rm n)}_{\rm cd, 2}-\beta^{(\rm n-1)}_{\rm cd, 2}\big)\partial_{2}\delta \varphi^{(\rm n)},
&\ \ \  \mbox{on}\ \ \ \tilde{\Gamma}_{\rm cd}, \\[5pt]
\delta Z^{(\rm n+1)}_{-}+\delta Z^{(\rm n+1)}_{+}=0,   &\ \ \ \mbox{on}\ \ \ \tilde{\Gamma}_{-},
\end{array}
\right.
\end{eqnarray}
where $a^{(\ell)}_{ij}=\int^{1}_{0}\partial_{i}\mathcal{N}_{i}(D\underline{\varphi}+\varsigma D\delta\varphi^{(\ell)})d\varsigma$
for $\ell=\rm n,\ n+1$, and
\begin{eqnarray}\label{eq:6.2a}
\begin{split}
&g^{(\rm n+1)}_{0}-g^{(\rm n)}_{0}
=-\frac{1}{\underline{\beta}_{0,2}}\sum_{j=1,2}\int^{\eta}_{0}\bigg( \Big(\beta^{(\rm n+1)}_{0,j}
-\underline{\beta}_{0,j}\Big)\partial_{j}\delta\Phi^{(\rm n+1)}-\Big(\beta^{(\rm n+1)}_{0,j}
-\beta^{(\rm n)}_{0,j}\Big)\partial_{j}\delta\varphi^{(\rm n)}\bigg)d\mu,\\[5pt]
&g^{(\rm n+1)}_{+}-g^{(\rm n)}_{+}
=-\frac{1}{\underline{\beta}_{0,2}}\sum_{j=1,2}\int^{m^{(\rm e)}}_{0}\bigg( \Big(\beta^{(\rm n+1)}_{0,j}
-\underline{\beta}_{0,j}\Big)\partial_{j}\delta\phi^{(\rm n+1)}-\Big(\beta^{(\rm n+1)}_{0,j}
-\beta^{(\rm n)}_{0,j}\Big)\partial_{j}\delta\varphi^{(\rm n)}\bigg)d\mu.
\end{split}
\end{eqnarray}
Here $\beta^{(\rm n+1)}_{0,j}, \ \beta^{(\rm n)}_{0,j}, j=1,2$ are defined by \eqref{eq:3.71} for $\delta\varphi^{(\rm n)}$ and $\delta\varphi^{(\rm n+1)}$.

Choosing {$\tilde{\epsilon}>0$ sufficiently small} with $|\lambda_{\pm}|/m^{(\rm h)}\geq L/2 $ and then
following the argument in Section 4.2, we know that there exists a constant $C>0$
depending only on $\tilde{\underline{U}}$, $L$ and $\alpha$ such that
\begin{eqnarray}\label{eq:6.3}
\begin{split}
\big\|\delta Z^{(\rm n+1)}_{-}\big\|_{0,\alpha,\tilde{\Omega}^{(\rm h)}}
&
\leq C\bigg(\big\|\big(\lambda^{(\rm n)}_{+}-\lambda^{(\rm n-1)}_{+}\big)\partial_{2}\delta z^{(\rm n)}_{-}\big\|_{0,\alpha;\tilde{\Omega}^{(\rm h)}}
+\big\|\big(\lambda^{(\rm n)}_{-}-\lambda^{(\rm n-1)}_{-}\big)\partial_{2}\delta z^{(\rm n)}_{+}\big\|_{0,\alpha;\tilde{\Omega}^{(\rm h)}}\bigg)\\[5pt]
&\leq C \epsilon_{I}\big\|\delta Z^{(\rm n)}\big\|_{0,\alpha,\tilde{\Omega}^{(\rm h)}}.
\end{split}
\end{eqnarray}
On the other hand, eliminating $\delta Z^{(\rm n+1)}_{+}$ on $\tilde{\Gamma}_{\rm cd}$ gives the boundary condition:
\begin{eqnarray}\label{eq:6.4}
\begin{split}
&\tilde{\beta}^{(\rm n)}_{\rm cd, 1}\partial_{1}\delta \Phi^{(\rm n+1)}
+\tilde{\beta}^{(\rm n)}_{\rm cd,2}\partial_{2}\delta \Phi^{(\rm n+1)}\\[5pt]
=&\big(\beta^{(\rm n)}_{\rm cd,1}-\beta^{(\rm n-1)}_{\rm cd, 1}\big)\partial_{1}\delta \varphi^{(\rm n)}
+\big(\beta^{(\rm n)}_{\rm cd, 2}-\beta^{(\rm n-1)}_{\rm cd, 2}\big)\partial_{2}\delta \varphi^{(\rm n)}+\delta Z^{(\rm n+1)}_{-},
&\ \ \  \mbox{on}\ \ \ \tilde{\Gamma}_{\rm cd},
\end{split}
\end{eqnarray}
where
\begin{eqnarray*}
\begin{split}
\tilde{\beta}^{(\rm n)}_{\rm cd, 1}=\beta^{(\rm n)}_{\rm cd, 1}+\int^{1}_{0}\frac{d\varsigma}{1+(\partial_{1}\delta \varphi^{(\rm n)}+\varsigma\partial_{1}\delta \Phi^{(\rm n+1)})^{2}},
\qquad \ \   \tilde{\beta}^{(\rm n)}_{\rm cd, 2}=\beta^{(\rm n)}_{\rm cd, 2}.
\end{split}
\end{eqnarray*}
For the solution $\delta \Phi^{(\rm n+1)}$ in $\tilde{\Omega}^{(\rm e)}$ with boundary condition \eqref{eq:6.4},   we have
\begin{eqnarray}\label{eq:6.5}
\begin{split}
&\big\|\delta \Phi^{(\rm n+1)}\big\|_{1,\alpha,\tilde{\Omega}^{(\rm e)}}\\[5pt]
\leq& C\Big\|\sum_{i,j=1,2}\partial_{j}\Big(\big(a^{(\rm n)}_{ij}-a^{(\rm n+1)}_{ij}\big)\partial_{i}\delta\varphi^{(\rm n+1)}\Big)\Big\|_{0,\alpha;\tilde{\Omega}^{(\rm e)}}
+\big\|g^{(\rm n+1)}_{0}-g^{(\rm n)}_{0}\big\|_{0,\alpha; \tilde{\Gamma}^{(\rm e)}_{\rm in}}
+\big\|g^{(\rm n+1)}_{+}-g^{(\rm n)}_{+}\big\|_{0,\alpha; \tilde{\Gamma}_{+}}\\[5pt]
&+C\Big\|\big(\beta^{(\rm n)}_{\rm cd,1}-\beta^{(\rm n-1)}_{\rm cd, 1}\big)\partial_{1}\delta \varphi^{(\rm n)}
+\big(\beta^{(\rm n)}_{\rm cd, 2}-\beta^{(\rm n-1)}_{\rm cd, 2}\big)\partial_{2}\delta \varphi^{(\rm n)}+\delta Z^{(\rm n+1)}_{-}\Big\|
_{0,\alpha; \tilde{\Omega}^{(\rm e)}}\\[5pt]
\leq& C\sum_{j=1,2}\bigg\|\int^{\eta}_{0}\bigg( \Big(\beta^{(\rm n+1)}_{0,j}
-\underline{\beta}_{0,j}\Big)\partial_{j}\delta\Phi^{(\rm n+1)}-\Big(\beta^{(\rm n+1)}_{0,j}
-\beta^{(\rm n)}_{0,j}\Big)\partial_{j}\delta\varphi^{(\rm n)}\bigg)d\mu \bigg\|_{0,\alpha; \tilde{\Gamma}^{(\rm e)}_{\rm in}}\\[5pt]
&\ \ \ +C\Big\|\big(\beta^{(\rm n)}_{\rm cd,1}-\beta^{(\rm n-1)}_{\rm cd, 1}\big)\partial_{1}\delta \varphi^{(\rm n)}
+\big(\beta^{(\rm n)}_{\rm cd, 2}-\beta^{(\rm n-1)}_{\rm cd, 2}\big)\partial_{2}\delta \varphi^{(\rm n)}+\delta Z^{(\rm n+1)}_{-}\Big\|
_{0,\alpha; \tilde{\Omega}^{(\rm e)}}\\[5pt]
&\ \ \ +C \big\|\delta \varphi^{(\rm n)}\big\|^{{(-1-\alpha, \tilde{\Sigma}^{(\rm e)}
\setminus \{\mathcal{O}\})}}_{2,\alpha,\tilde{\Omega}^{(\rm e)}}\big\|\delta \Phi^{(\rm n)}\big\|_{1,\alpha,\tilde{\Omega}^{(\rm e)}}+\big\|\delta Z^{(\rm n+1)}_{-}\big\|_{0,\alpha;\tilde{\Gamma}_{\rm cd}}\\[5pt]
\leq& C{\color{black}\big(\tilde{\epsilon}+\epsilon_{I}\big)}\Big(\big\|\delta Z^{(\rm n)}\big\|_{0,\alpha,\tilde{\Omega}^{(\rm h)}}+\big\|\delta \Phi^{(\rm n)}\big\|_{1,\alpha,\tilde{\Omega}^{(\rm e)}}\Big)+C{\color{black}\epsilon_{I}}\big\|\delta \Phi^{(\rm n+1)}\big\|_{1,\alpha,\tilde{\Omega}^{(\rm e)}},
\end{split}
\end{eqnarray}
where we have used the estimate \eqref{eq:6.3}, provided that $\sigma>0$ is sufficiently small.
Then, similarly to Section 4.2, we can solve $\delta Z^{(\rm n+1)}_{+}$ with the boundary condition on $\tilde{\Gamma}_{\rm cd}$ and
the initial data on $\tilde{\Gamma}^{(\rm h)}_{\rm in}$, together with the estimates \eqref{eq:6.3} and \eqref{eq:6.5}. Thus,
\begin{eqnarray}\label{eq:6.6}
\begin{split}
&\big\|\delta Z^{(\rm n+1)}_{+}\big\|_{0,\alpha,\tilde{\Omega}^{(\rm h)}}\\[5pt]
\leq& C\Big\|2\beta^{(\rm n)}_{\rm cd, 1}\partial_{1}\delta \Phi^{(\rm n+1)}
+2\beta^{(\rm n)}_{\rm cd,2}\partial_{2}\delta \Phi^{(\rm n+1)}\Big\|_{0,\alpha;\tilde{\Gamma}_{\rm cd}}
 +C\Big\|2\big(\beta^{(\rm n)}_{\rm cd,1}-\beta^{(\rm n-1)}_{\rm cd, 1}\big)\partial_{1}\delta \varphi^{(\rm n)}\Big\|_{0,\alpha;\tilde{\Gamma}_{\rm cd}}\\[5pt]
&\quad +C\Big\|2\big(\beta^{(\rm n)}_{\rm cd, 2}-\beta^{(\rm n-1)}_{\rm cd, 2}\big)
\partial_{2}\delta \varphi^{(\rm n)}\Big\|_{0,\alpha;\tilde{\Gamma}_{\rm cd}}
+\big\|\delta Z^{(\rm n+1)}_{-}\big\|_{0,\alpha,\tilde{\Gamma}_{\rm cd}}\\[5pt]
\leq&  C{\color{black}\big(\tilde{\epsilon}+\epsilon_{I}\big)}\Big(\big\|\delta Z^{(\rm n)}\big\|_{0,\alpha,\tilde{\Omega}^{(\rm h)}}+\big\|\delta \Phi^{(\rm n)}\big\|_{1,\alpha,\tilde{\Omega}^{(\rm e)}}\Big)+C{\color{black}\epsilon_{I}}\big\|\delta \Phi^{(\rm n+1)}\big\|_{1,\alpha,\tilde{\Omega}^{(\rm e)}}.
\end{split}
\end{eqnarray}

\par Finally, combining the estimates \eqref{eq:6.3}, \eqref{eq:6.5} and \eqref{eq:6.6} together, we get
\begin{eqnarray*}
\begin{split}
&\big\|\delta \Phi^{(\rm n+1)}\big\|_{1,\alpha;\tilde{\Omega}^{(\rm e)}}
+\big\|\delta Z^{(\rm n+1)}\big\|_{0,\alpha,\tilde{\Omega}^{(\rm h)}}\\[5pt]
&\quad \leq  C{\color{black}\big(\tilde{\epsilon}+\epsilon_{I}\big)}\Big(\big\|\delta Z^{(\rm n)}\big\|_{0,\alpha,\tilde{\Omega}^{(\rm h)}}+\big\|\delta \Phi^{(\rm n)}\big\|_{1,\alpha,\tilde{\Omega}^{(\rm e)}}\Big)+C{\color{black}\epsilon_{I}}\big\|\delta \Phi^{(\rm n+1)}\big\|_{1,\alpha,\tilde{\Omega}^{(\rm e)}},
\end{split}
\end{eqnarray*}
where constant $C>0$ depends only on $\tilde{\underline{U}}$, $L$ and $\alpha$.

{Then we can choose constants $\mathcal{C}^{*}_{I,0}>0$, $\epsilon^{*}_{I,0}>0$ and $\alpha_0\in(0,1)$ depending on $\tilde{\underline{U}}$ and $L$,  such that for any $\alpha\in(0,\alpha_0)$
and $\tilde{\epsilon}<\epsilon_I\in (\mathcal{C}^{*}_{I,0} \tilde{\epsilon},\epsilon^{*}_{I})$ with $\tilde{\epsilon}>0$  sufficiently small, the following holds}
\begin{eqnarray}\label{equ:6.7}
\begin{split}
&\big\|\delta \Phi^{(\rm n+1)}\big\|_{1,\alpha;\tilde{\Omega}^{(\rm e)}}
+\big\|\delta Z^{(\rm n+1)}\big\|_{0,\alpha,\tilde{\Omega}^{(\rm h)}}
\leq \frac{1}{2}\Big(\big\|\delta \Phi^{(\rm n)}\big\|_{1,\alpha;\tilde{\Omega}^{(\rm e)}}
+\big\|\delta Z^{(\rm n)}\big\|_{0,\alpha;\tilde{\Omega}^{(\rm h)}}\Big),
\end{split}
\end{eqnarray}
i.e.,
\begin{eqnarray*}
\begin{split}
\big\|\mathcal{J}(\delta \Phi^{(\rm n)}, \delta Z^{(\rm n)})\big\|_{C^{1,\alpha}(\tilde{\Omega}^{(\rm e)})\times C^{0,\alpha}(\tilde{\Omega}^{(\rm h)})}
\leq \frac{1}{2}\big\|(\delta \Phi^{(\rm n)}, \delta Z^{(\rm n)})\big\|_{C^{1,\alpha}(\tilde{\Omega}^{(\rm e)})\times C^{0,\alpha}(\tilde{\Omega}^{(\rm h)})}.
\end{split}
\end{eqnarray*}
This implies that $\mathcal{J}$ is a contraction mapping. Hence, there exists a fixed point 
$(\delta{\varphi},\delta {z} )\in C^{1, \alpha}(\tilde{\Omega}^{(\rm e)})\times C^{0,\alpha}(\tilde{\Omega}^{(\rm h)})$ such that
\begin{eqnarray*}
\begin{split}
(\delta{\varphi}^{(\rm n)},\delta {z}^{(\rm n)} )\longrightarrow (\delta{\varphi},\delta {z} )\quad \ \
 \emph{in} \quad  \ {\color{black}C^{1,\alpha}(\tilde{\Omega}^{(\rm e)})\times C^{0,\alpha}(\tilde{\Omega}^{(\rm h)})},
\end{split}
\end{eqnarray*}
as $\rm {n}\rightarrow \infty$.

On the other hand, by  the estimate \eqref{eq:4.3} and  the compactness of the approximate solutions in  $C^{2,\alpha}_{{(-1-\alpha, \tilde{\Sigma}^{(\rm e)}
\setminus \{\mathcal{O}\})}}(\tilde{\Omega}^{(\rm e)})\times C^{1, \alpha}(\tilde{\Omega}^{(\rm h)})$, there exists a subsequence $\big\{(\delta{\varphi}^{(\rm n_{i})},\delta {z}^{(\rm n_{i})} )\big\}^{\infty}_{i=0}$ such that
\begin{eqnarray*}
\begin{split}
(\delta{\varphi}^{(\rm n_{i})},\delta {z}^{(\rm n_{i})} )\longrightarrow (\delta{\varphi}_{*},\delta {z}_{*} )\quad \ \
 \emph{in} \quad  \ C^{2,\alpha}_{{(-1-\alpha, \tilde{\Sigma}^{(\rm e)}
\setminus \{\mathcal{O}\})}}(\tilde{\Omega}^{(\rm e)})\times C^{1, \alpha}(\tilde{\Omega}^{(\rm h)}),
\end{split}
\end{eqnarray*}
as $i\rightarrow \infty$. By the uniqueness of the limit, one has
\begin{eqnarray*}
\begin{split}
(\delta{\varphi},\delta {z} )=(\delta{\varphi}_{*},\delta {z}_{*} )\in C^{2,\alpha}_{{(-1-\alpha, \tilde{\Sigma}^{(\rm e)}
\setminus \{\mathcal{O}\})}}(\tilde{\Omega}^{(\rm e)})
\times C^{1,\alpha}(\tilde{\Omega}^{(\rm h)}).
\end{split}
\end{eqnarray*}
Define
\begin{eqnarray*}
\varphi:=\delta \varphi+\underline{\varphi},\ \ \  z:=\delta z+\underline{z}.
\end{eqnarray*}
Obviously, $(\varphi, z)\in C^{2,\alpha}_{{(-1-\alpha, \tilde{\Sigma}^{(\rm e)}
\setminus \{\mathcal{O}\})}}(\tilde{\Omega}^{(\rm e)})\times C^{1,\alpha}(\tilde{\Omega}^{(\rm h)})$ is a weak solution of the problem $(\mathbf{FP})$. 
\end{proof}

\begin{proof}[ Proof of Theorem \ref{thm:3.3}]
Now we are ready to prove Theorem \ref{thm:3.3}. The existence of solutions to the problem $(\mathbf{FP})$
can be derived directly from Lemma \ref{lem:6.1}. Next, for the uniqueness, we take two solutions $(\varphi_{1}, z_{1})$ and $(\varphi_{2}, z_{2})$ which   satisfy the estimate \eqref{eq:3.67}.
Define
$$
\delta \Phi:=\delta \varphi_1-\delta \varphi_2,\quad \delta Z:=\delta z_1-\delta z_2.
$$
Then, following the  same method as   in the proof of  the Lemma \ref{lem:6.1}, one can show that
$(\delta \Phi, \delta Z)$ satisfies the estimate \eqref{equ:6.7}, 
where $(\delta \Phi^{(\rm n+1)}, \delta Z^{(\rm n+1)})$ and $(\delta \Phi^{(\rm n)}, \delta Z^{(\rm n)})$
are replaced by $(\delta \Phi, \delta Z)$. This means that $(\delta \Phi, \delta Z)=(0,0)$.
Thus we have $(\varphi_1,z_1)=(\varphi_2,z_{2})$.
\end{proof}

\bigskip

\section{Existence and Uniqueness of the Free Boundary Value Problem $(\mathbf{NP})$}\setcounter{equation}{0}

Up to now, we have proved that for any $m^{(\rm e)}\in \mathcal{M}_{\sigma}$,  the problem $(\mathbf{FP})$ admits a unique solution $\big(\varphi,z\big)$ satisfying the estimate \eqref{eq:3.67}.
Then, we can define a new number $\hat{m}^{(\rm e)}$ by
\begin{eqnarray}\label{eq:7.1}
\begin{split}
\int^{\hat{m}^{(\rm e)}}_{0}\frac{d\tau}{(\rho^{(\rm e)} u^{(\rm e)})(0,\tau)}=g_{+}(0).
\end{split}
\end{eqnarray}
By Theorem \ref{thm:3.3}, we know that $\rho^{(\rm e)} u^{(\rm e)}>0$. Then the unique existence of the number $\hat{m}^{(e)}$ of equation \eqref{eq:7.1} follows from the implicit function theorem, and hence the mapping  $\mathcal{T}$:
$\hat{m}^{(\rm e)}=\mathcal{T}(m^{(\rm e)})$ is well defined on $\mathcal{M}_{\sigma}$ (see \eqref{eq:3.62}). Notice that
\begin{eqnarray*}
\begin{split}
\int^{\underline{m}^{(\rm e)}}_{0}\frac{d\tau}{\underline{\rho}^{(\rm e)} \underline{u}^{(\rm e)}}=1,
\end{split}
\end{eqnarray*}
then
\begin{eqnarray*}
\begin{split}
\int^{\hat{m}^{(\rm e)}}_{0}\frac{d\tau}{(\rho^{(\rm e)} u^{(\rm e)})(0,\tau)}-\int^{\underline{m}^{(\rm e)}}_{0}\frac{d\tau}{\underline{\rho}^{(\rm e)} \underline{u}^{(\rm e)}}=g_{+}(0)-1,
\end{split}
\end{eqnarray*}
i.e.,
\begin{eqnarray*}
\begin{split}
\int^{\hat{m}^{(\rm e)}}_{\underline{m}^{(\rm e)}}\frac{d\tau}{(\rho^{(\rm e)} u^{(\rm e)})(0,\tau)}
=\int^{\underline{m}^{(\rm e)}}_{0}\Big(\frac{1}{(\rho^{(\rm e)} u^{(\rm e)})(0,\tau)}
-\frac{1}{\underline{\rho}^{(\rm e)} \underline{u}^{(\rm e)}}\Big)d\tau
+g_{+}(0)-1.
\end{split}
\end{eqnarray*}
Therefore,
\begin{eqnarray*}
\begin{split}
|\hat{m}^{(\rm e)}-\underline{m}^{(\rm e)}|&\leq \mathcal{C}\Big(\big\|\varphi-\underline{\varphi}\big\|^{{(-1-\alpha, \tilde{\Sigma}^{(\rm e)}
\setminus \{\mathcal{O}\})}}_{2, \alpha; \tilde{\Omega}^{(\rm e)}}
+\|g_{+}-1\|_{2,\alpha; \tilde{\Gamma}_{+}}\Big)
\leq \mathcal{C}{\color{black}\tilde{\epsilon}},
\end{split}
\end{eqnarray*}
where the constant $\mathcal{C}>0$ depends only on $\underline{U}$, $L$ and $\alpha$.
Hence, we can choose {\color{black}$\tilde{\epsilon}>0$} sufficiently small such that {\color{black}$\mathcal{C}\tilde{\epsilon}<\sigma$}.
Then the mapping $\mathcal{T}$ is from $\mathcal{M}_{\sigma}$ to $\mathcal{M}_{\sigma}$.

\par Next, we will show that the mapping $\mathcal{T}$ is contractive. Give two numbers $m^{(\rm e)}_{A},
m^{(\rm e)}_{B}\in \mathcal{M}_{\sigma}$ and the corresponding solutions are
$(\varphi_{A}, z_{A})$ and $(\varphi_{B}, z_{B})$, which are defined respectively in the domains $\tilde{\Omega}^{(\rm e)}_{A}\cup\tilde{\Omega}^{(\rm h)}_{A}$
and $\tilde{\Omega}^{(\rm e)}_{B}\cup\tilde{\Omega}^{(\rm h)}_{B}$. 
Let
\begin{eqnarray}\label{eq:7.2}
\pi_{A}:\ \left\{
\begin{array}{llll}
\breve{\xi}=\xi,\quad \breve{\eta}=\frac{\eta}{m^{(\rm e)}_{A}},
&\  (\xi,\eta) \in \tilde{\Omega}^{(\rm e)}_{A},   \\[5pt]
\breve{\xi}=\xi,\quad \breve{\eta}=\eta,
&\  (\xi,\eta) \in \tilde{\Omega}^{(\rm h)}_{A},  \\[5pt]
\end{array}
\right.
\end{eqnarray}
and
\begin{eqnarray}\label{eq:7.3}
\pi_{B}:\ \left\{
\begin{array}{llll}
\breve{\xi}=\xi,\quad \breve{\eta}=\frac{\eta}{m^{(\rm e)}_{B}},
&\ \ \ \ (\xi,\eta) \in \tilde{\Omega}^{(\rm e)}_{B},   \\[5pt]
\breve{\xi}=\xi,\quad \breve{\eta}=\eta,
&\ \ \ \ (\xi,\eta) \in \tilde{\Omega}^{(\rm h)}_{B}.
\end{array}
\right.
\end{eqnarray}
Then, under the coordinate transformation $\pi_{j}$, the domain $\tilde{\Omega}^{(\rm e)}_{j}\cup\tilde{\Omega}^{(\rm h)}_{j}(j=A,B)$ is transformed into the same domain:
\begin{eqnarray*}
\begin{split}
\breve{\Omega}^{(\rm e)}\cup\breve{\Omega}^{(\rm h)}
=\big\{(\breve{\xi}, \breve{\eta}):0<\breve{\xi}<1,\ 0<\breve{\eta}<1\big\}\cup
\big\{(\breve{\xi}, \breve{\eta}):0<\breve{\xi}<L,\ -m^{(\rm h)}<\breve{\eta}<0\big\}.
\end{split}
\end{eqnarray*}
The boundaries {\color{black}$\tilde{\Gamma}^{(\rm e)}_{{\rm in}, j}$, $\tilde{\Gamma}^{(\rm e)}_{{\rm ex}, j}$, $\tilde{\Gamma}^{(\rm h)}_{{\rm in},j}$}, $\tilde{\Gamma}_{+,j}$, $\tilde{\Gamma}_{{\rm cd},j}$ and $\tilde{\Gamma}_{-,j}$, {\color{black}$(j=A,B)$} become
\begin{eqnarray*}
\begin{split}
&\breve{\Gamma}^{(\rm e)}_{\rm in}=\big\{(\breve{\xi},\breve{\eta}):0<\breve{\eta}<1, \breve{\xi}=0\big\},\
\breve{\Gamma}^{(\rm e)}_{\rm ex}=\big\{(\breve{\xi},\breve{\eta}):0<\breve{\eta}<1, \breve{\xi}=L\big\},\\[5pt]
&\breve{\Gamma}^{(\rm h)}_{\rm in}=\big\{(\breve{\xi},\breve{\eta}):-m^{(\rm h)}<\breve{\eta}<0, \breve{\xi}=0\big\},\
\breve{\Gamma}_{\rm cd}=\big\{(\breve{\xi},\breve{\eta}):0<\breve{\xi}<L, \breve{\eta}=0\big\},
\end{split}
\end{eqnarray*}
and
\begin{eqnarray*}
\begin{split}
\breve{\Gamma}_{+}=\big\{(\breve{\xi},\breve{\eta}):0<\breve{\xi}<L, \breve{\eta}=1\big\},\
\breve{\Gamma}_{-}=\big\{(\breve{\xi},\breve{\eta}):0<\breve{\xi}<L, \breve{\eta}=-m^{(\rm h)}\big\}.
\end{split}
\end{eqnarray*}

\par For $j=A,\ B$, denote $\breve{D}^{(j)}=\big(\breve{\partial}^{(j)}_{1},\  \breve{\partial}^{(j)}_{2}\big)
=\big(\partial_{\breve{\xi}},\ \frac{\partial_{\breve{\eta}}}{m^{(\rm e)}_{j}}\big)$,
\begin{eqnarray}\label{eq:7.4}
\begin{split}
\breve{\varphi}_{j}(\breve{\xi}, \breve{\eta})=\varphi_{j}(\breve{\xi}, m^{(\rm e)}_{j}\breve{\eta}),\ \
\big(\breve{B}^{(\rm e)}_{0,j}, \breve{S}^{(\rm e)}_{0,j}\big)(\breve{\eta})
=\big(\tilde{B}^{(\rm e)}_{0}, \tilde{S}^{(\rm e)}_{0}\big)(\breve{\xi}, m^{(\rm e)}_{j}\breve{\eta}),
\end{split}
\end{eqnarray}
for $(\breve{\xi}, \breve{\eta})\in \breve{\Omega}^{(\rm e)}$, and
\begin{eqnarray}\label{eq:7.5}
\begin{split}
\breve{z}_{j}(\breve{\xi}, \breve{\eta})=z_{j}(\breve{\xi}, \breve{\eta}),\ \
\big(\breve{B}^{(\rm h)}_{0,j}, \breve{S}^{(\rm h)}_{0,j}\big)(\breve{\xi}, \breve{\eta})
=\big(\tilde{B}^{(\rm h)}_{0},\tilde{S}^{(\rm h)}_{0}\big)(\breve{\xi}, \breve{\eta}),
\end{split}
\end{eqnarray}
for $(\breve{\xi}, \breve{\eta})\in \breve{\Omega}^{(\rm h)}$.
Then, $\breve{\varphi}_{j}$ and $\breve{z}_{j}$, $(j=A, B)$ satisfy
\begin{eqnarray}\label{eq:7.6}
\left\{
\begin{array}{llll}
\partial_{\breve{\xi}}\mathcal{N}_{1}\big(\breve{D}^{(j)}\breve{\varphi}_{j}; \breve{B}^{(\rm e)}_{0,j}, \breve{S}^{(\rm e)}_{0,j}\big)
+m^{(\rm e)}_{j}\partial_{\breve{\eta}}\mathcal{N}_{2}\big(\breve{D}^{(j)}\breve{\varphi}_{j};
\breve{B}^{(\rm e)}_{0,j}, \breve{S}^{(\rm e)}_{0,j}\big)=0, &\ \ \  \mbox{in}\ \ \ \breve{\Omega}^{(\rm e)},   \\[5pt]
\partial_{\breve{\xi}}\breve{z}_{j}+\textrm{diag}( \breve{\lambda}_{+,j}, \breve{\lambda}_{-,j})\partial_{\breve{\eta}}\breve{z}_{j}=0,
&\ \ \ \mbox{in} \ \ \  \breve{\Omega}^{(\rm h)},  \\[5pt]
\breve{p}^{(\rm e)}_{j}=\tilde{p}_{0}(m^{(\rm e)}_{j}\breve{\eta}),     &\ \ \
\mbox{on}\ \ \ \breve{\Gamma}^{(\rm e)}_{\rm in},\\[5pt]
\partial_{\breve{\xi}}\breve{\varphi}_{j}=\tilde{\omega}_{\rm e}(m^{(\rm e)}_{j}\breve{\eta}),
&\ \ \  \mbox{on}\ \ \ \breve{\Gamma}^{(\rm e)}_{\rm ex},\\[5pt]
\breve{z}_{j}=z_{0}(\breve{\eta}),   &\ \ \   \mbox{on}\ \ \ \breve{\Gamma}^{(\rm h)}_{\rm in}, \\[5pt]
\partial_{\breve{\xi}}\breve{\varphi}_{j}= g_{+}(\breve{\xi}), &\ \ \  \mbox{on}\ \ \ \breve{\Gamma}_{+},\\[5pt]
\partial_{\breve{\xi}}\breve{\varphi}_{j}=\breve{\omega}^{(\rm h)}_{j}, \  \breve{p}^{(\rm e)}_{j}=\breve{p}^{(\rm h)}_{j},
&\ \ \ \mbox{on}\ \ \ \breve{\Gamma}_{\rm cd},\\[5pt]
\breve{z}_{-,j}+\breve{z}_{+,j}=2\arctan g'_{-}(\breve{\xi}), &\ \ \ \mbox{on} \ \ \ \breve{\Gamma}_{-},
\end{array}
\right.
\end{eqnarray}
where $\breve{\lambda}_{\pm}=\lambda_{\pm}\big(\breve{z}; \breve{B}^{(\rm h)}_{0},\breve{S}^{(\rm h)}_{0}\big)$ and
\begin{eqnarray}\label{eq:7.7}
\begin{split}
\breve{\omega}^{(\rm h)}_{j}=\tilde{\omega}^{(\rm h)}_{j}(\breve{\xi},\breve{\eta}), \quad \breve{p}^{(\rm e)}_{j}=\tilde{p}^{(\rm e)}\big(\breve{D}^{(j)}\breve{\varphi}_{j};\breve{B}^{(\rm e)}_{0,j}, \breve{S}^{(\rm e)}_{0,j}\big),
\quad \breve{p}^{(\rm h)}_{j}=\tilde{p}^{(\rm h)}_{j}(\breve{\xi},\breve{\eta}).
\end{split}
\end{eqnarray}
\par Set $\breve{\varphi}_{AB}=\breve{\varphi}_{A}-\breve{\varphi}_{B}$ and $ \breve{z}_{AB}=\breve{z}_{A}-\breve{z}_{B}$.
Then, $\breve{\varphi}_{AB}$ and $\breve{z}_{AB}$ satisfy the following initial boundary value problem:
\begin{eqnarray}\label{eq:7.8}
\left\{
\begin{array}{llll}
\sum_{i,j=1,2}\breve{\partial}^{(A)}_{i}\big(\breve{a}_{ij}(\breve{\xi},\breve{\eta})
\breve{\partial}^{(A)}_{j}\breve{\varphi}_{AB}\big)=\breve{b}, &\ \ \  \mbox{in}\ \ \ \breve{\Omega}^{(\rm e)},   \\[5pt]
\breve{\partial}^{(A)}_{1}\breve{z}_{AB}+\textrm{diag}(\breve{\lambda}_{+,A}, \breve{\lambda}_{-,A})\breve{\partial}^{(A)}_{2}\breve{z}_{AB}\\[5pt]
\ \ \qquad\qquad =\textrm{diag}\big(\breve{\lambda}_{+,B}-\breve{\lambda}_{+,A}, \breve{\lambda}_{-,B}-\breve{\lambda}_{-,A}\big)\breve{\partial}_{2}\breve{z}^{(A)}_{-,B},
&\ \ \ \mbox{in}\ \ \  \breve{\Omega}^{(\rm h)},  \\[5pt]
\breve{\varphi}_{AB}=\breve{g}_{0, AB}(\breve{\eta}),
&\ \ \  \mbox{on}\ \ \ \breve{\Gamma}^{(\rm e)}_{\rm in},\\[5pt]
\breve{\partial}^{(A)}_{1}\breve{\varphi}_{AB}=\tilde{\omega}_{\rm e}(m^{(\rm e)}_{A}\breve{\eta})
-\tilde{\omega}_{\rm e}(m^{(\rm e)}_{B}\breve{\eta}),
&\ \ \  \mbox{on}\ \ \ \breve{\Gamma}^{(\rm e)}_{\rm ex},\\[5pt]
\breve{z}_{AB}=0,   &\ \ \   \mbox{on}\ \ \ \breve{\Gamma}^{(\rm h)}_{\rm in}, \\[5pt]
\breve{\varphi}_{AB}= \breve{g}_{+, AB}(\breve{\xi}), &\ \ \  \mbox{on}\ \ \ \breve{\Gamma}_{+},\\[5pt]
\breve{z}_{-,AB}+\breve{z}_{+,AB}=2\arctan(\breve{\partial}^{(A)}_{1}\breve{\varphi}_{AB}+\breve{\partial}^{(A)}_{1}\breve{\varphi}_{B})
-2\arctan\breve{\partial}^{(B)}_{1}\breve{\varphi}_{B},
  &\ \ \ \mbox{on}\ \ \ \breve{\Gamma}_{\rm cd},\\[5pt]
\breve{z}_{-,AB}-\breve{z}_{+,AB}=2\breve{\beta}_{\rm cd, 1}\breve{\partial}^{(A)}_{1}\breve{\varphi}_{AB}
+2\breve{\beta}_{\rm cd, 2}\breve{\partial}^{(A)}_{2}\breve{\varphi}_{AB}+2\breve{c}_{\rm cd}(\breve{\xi}),
&\ \ \ \mbox{on}\ \ \ \breve{\Gamma}_{\rm cd},\\[5pt]
\breve{z}_{-,AB}+\breve{z}_{+,AB}=0, &\ \ \ \mbox{on}\ \ \ \breve{\Gamma}_{-},
\end{array}
\right.
\end{eqnarray}
where $\breve{a}_{ij}(\breve{\xi},\breve{\eta})=a_{ij}\big(\breve{D}^{(A)}\breve{\varphi}_{A}, \breve{D}^{(A)}\breve{\varphi}_{B},\breve{\eta}\big),$ and
$$\breve{\beta}_{0, \ell}=
\beta_{\rm 0, \ell}\big(\breve{D}^{(A)}\breve{\varphi}_{A}, \breve{D}^{(A)}\breve{\varphi}_{B},\breve{\eta}\big),  \quad 
 \breve{\beta}_{\rm cd,\ell}=\beta_{\rm cd, \ell} \big(\breve{D}^{(A)}\breve{\varphi}_{A},
\breve{D}^{(A)}\breve{\varphi}_{B}, \breve{\eta}\big)$$ for $\ell=1,2$;
\begin{eqnarray*}
\begin{split}
\breve{b}=&\sum_{\ell=1,2}\breve{\partial}^{(A)}_{\ell}\Big(\mathcal{N}_{\ell}\big(\breve{D}_{B}\delta\breve{\varphi}_{B};\breve{B}^{(\rm e)}_{0,A}, \breve{S}^{(\rm e)}_{0,A} \big)-\mathcal{N}_{\ell}\big(\breve{D}_{A}\delta\breve{\varphi}_{B};\breve{B}^{(\rm e)}_{0,A}, \breve{S}^{(\rm e)}_{0,A} \big)\Big)\\[5pt]
& +\sum_{\ell=1,2}\breve{\partial}^{(B)}_{\ell}\Big(\mathcal{N}_{\ell}\big(\breve{D}_{B}\delta\breve{\varphi}_{B};\breve{B}^{(\rm e)}_{0,B}, \breve{S}^{(\rm e)}_{0,B} \big)-\mathcal{N}_{\ell}\big(\breve{D}_{B}\breve{\varphi}_{B};\breve{B}^{(\rm e)}_{0,A}, \breve{S}^{(\rm e)}_{0,A} \big)\Big)\\[5pt]
& +\big(\breve{\partial}^{(B)}_{2}-\breve{\partial}^{(A)}_{2}\big)\mathcal{N}_{2}\big(\breve{D}_{B}\breve{\varphi}_{B};\breve{B}^{(\rm e)}_{0,A}, \breve{S}^{(\rm e)}_{0,A} \big), 
\end{split}
\end{eqnarray*}
$$\breve{c}_{\rm cd}(\breve{\xi})=\Theta\Big(p^{(\rm e)}\big(\breve{D}_{A}\breve{\varphi}_{B};\breve{B}^{(\rm e)}_{0,A}, \breve{S}^{(\rm e)}_{0,A} \big);\breve{B}^{(\rm h)}_{0},\breve{S}^{(\rm h)}_{0} \Big)
-\Theta\Big(p^{(\rm e)}\big(\breve{D}_{B}\breve{\varphi}_{B};\breve{B}^{(\rm e)}_{0,B}, \breve{S}^{(\rm e)}_{0,B} \big);\breve{B}^{(\rm h)}_{0}, \breve{S}^{(\rm h)}_{0} \Big),$$
\begin{eqnarray*}
\begin{split}
\breve{g}_{0,AB}(\check{\eta})&=\frac{1}{\underline{\beta}_{0,2}}
\int^{\check{\eta}}_{0}\Big(\tilde{p}^{(\rm e)}_{0}(m^{(\rm e)}_{A}\mu)-\tilde{p}^{(\rm e)}_{0}(m^{(\rm e)}_{B}\mu)
-\underline{\tilde{p}}^{(\rm e)}_{0}(m^{(\rm e)}_{A}\mu)+\underline{\tilde{p}}^{(\rm e)}_{0}(m^{(\rm e)}_{B}\mu)\Big)d\mu\\[5pt]
&\ \ \ -\frac{1}{\underline{\beta}_{0,2}}\sum_{j=1,2}\int^{\check{\eta}}_{0}
\big(\check{\beta}^{(A)}_{0,j}-\underline{\beta}_{0,j}\big)\check{\partial}^{(A)}_{j}\check{\varphi}_{AB}d\mu\\[5pt]
&\ \ \ -\frac{1}{\underline{\beta}_{0,2}}\sum_{j=1,2}\int^{\check{\eta}}_{0}
\big(\check{\beta}^{(A)}_{0,j}-\check{\beta}^{(B)}_{0,j}\big)\check{\partial}^{(A)}_{j}\delta\check{\varphi}_{B}d\mu\\[5pt]
&\ \ \ -\frac{1}{\underline{\beta}_{0,2}}\sum_{j=1,2}\int^{\check{\eta}}_{0} \big(\check{\beta}^{(B)}_{0,j}-\underline{\beta}_{0,j}\big)\Big(\check{\partial}^{(A)}_{j}-\check{\partial}^{(A)}_{j}\Big)
\delta\check{\varphi}_{B}d\mu,
\end{split}
\end{eqnarray*}
 and $\breve{g}_{+,AB}(\xi)=\breve{g}_{0,AB}(1)$.
Here $\check{\beta}^{(k)}_{0,j}$, $(k=A,B, j=1,2)$ are given by
\begin{eqnarray*}
\begin{split}
\breve{\beta}^{(k)}_{0, 1}
:=\int^{1}_{0}\partial_{\breve{\partial}^{(k)}_{1}\breve{\varphi}}\tilde{p}^{(\rm e)}\big(\breve{D}^{(k)}\underline{\varphi}+\nu \check{D}^{(k)}\delta\check{\varphi}_{k};\breve{\tilde{B}}^{(\rm e)}_{0,k},\breve{\tilde{S}}^{(\rm e)}_{0,k}\big)d\nu,\\[5pt]
\breve{\beta}^{(k)}_{0, 2}
:=\int^{1}_{0}\partial_{\breve{\partial}^{(k)}_{2}\breve{\varphi}}\tilde{p}^{(\rm e)}\big(\breve{D}^{(k)}\underline{\varphi}+\nu \breve{D}^{(k)}\delta\breve{\varphi}_{k};\breve{\tilde{B}}^{(\rm e)}_{0, k},\breve{\tilde{S}}^{(\rm e)}_{0, k}\big)d\nu,
\end{split}
\end{eqnarray*}
and $\delta\breve{\varphi}_{k}=\breve{\varphi}_{k}-\underline{\varphi}$ for $k=A, B$.

\par By direct computations, we can choose $\sigma_{0}>0$ which depends only on the background state $\tilde{\underline{U}}$
such that for any $\sigma\in (0,\sigma_{0})$, we have
\begin{eqnarray*}
\begin{split}
\|\breve{b}\|_{0,\alpha; \breve{\Omega}^{(\rm e)}}
&\leq \sum_{\ell=1,2}\Big\|\breve{\partial}^{(A)}_{\ell}\Big(\mathcal{N}_{\ell}\big(\breve{D}_{B}\breve{\varphi}_{B};\breve{B}^{(\rm e)}_{0,A}, \breve{S}^{(\rm e)}_{0,A} \big)-\mathcal{N}_{\ell}\big(\breve{D}_{A}\breve{\varphi}_{B};\breve{B}^{(\rm e)}_{0,A}, \breve{S}^{(\rm e)}_{0,A} \big)\Big)\Big\|_{0,\alpha; \breve{\Omega}^{(\rm e)}}\\[5pt]
&\ \ \ +\sum_{\ell=1,2}\Big\|\breve{\partial}^{(B)}_{\ell}\Big(\mathcal{N}_{\ell}\big(\breve{D}_{B}\breve{\varphi}_{B};\breve{B}^{(\rm e)}_{0,B}, \breve{S}^{(\rm e)}_{0,B} \big)-\mathcal{N}_{\ell}\big(\breve{D}_{B}\breve{\varphi}_{B};\breve{B}^{(\rm e)}_{0,A}, \breve{S}^{(\rm e)}_{0,A} \big)\Big)\Big\|_{0,\alpha; \breve{\Omega}^{(\rm e)}}\\[5pt]
&\ \ \   +\Big\|\big(\breve{\partial}^{(B)}_{2}-\breve{\partial}^{(A)}_{2}\big)\mathcal{N}_{2}\big(\breve{D}_{B}\breve{\varphi}_{B};\breve{B}^{(\rm e)}_{0,A}, \breve{S}^{(\rm e)}_{0,A} \big)\Big\|_{0,\alpha; \breve{\Omega}^{(\rm e)}}\\[5pt]
&\leq \mathcal{C}\Big(\|\delta\breve{\varphi}_{B}\|^{{(-1-\alpha, \tilde{\Sigma}^{(\rm e)}
\setminus \{\mathcal{O}\})}}_{2,\alpha;\tilde{\Omega}^{(\rm e)}}
+\big\|(\breve{B}^{(\rm e)}_{0}-\underline{B}^{(\rm e)}, \breve{S}^{(\rm e)}_{0} -\underline{S}^{(\rm e)})\big\|_{1,\alpha;\breve{\Gamma}^{(\rm e)}_{\rm in}}\Big)\big|m^{(\rm e)}_{A}-m^{(\rm e)}_{B}\big|\\[5pt]
&\leq \mathcal{C}{\color{black}\tilde{\epsilon}}\big|m^{(\rm e)}_{A}-m^{(\rm e)}_{B}\big|,
\end{split}
\end{eqnarray*}
\begin{eqnarray*}
\begin{split}
&\|\breve{c}_{\rm cd}(\breve{\xi})\|_{0,\alpha; \breve{\Gamma}_{\rm cd}}\\[5pt]
&\leq 2\bigg\|\Theta\Big(p^{(\rm e)}\big(\breve{D}_{A}\breve{\varphi}_{B};\breve{B}^{(\rm e)}_{0,A}, \breve{S}^{(\rm e)}_{0,A} \big);\breve{B}^{(\rm h)}_{0},\breve{S}^{(\rm h)}_{0} \Big)
-\Theta\Big(p^{(\rm e)}\big(\breve{D}_{B}\breve{\varphi}_{B};\breve{B}^{(\rm e)}_{0,B}, \breve{S}^{(\rm e)}_{0,B} \big);\breve{B}^{(\rm h)}_{0}, \breve{S}^{(\rm h)}_{0} \Big)\bigg\|_{0,\alpha; \breve{\Gamma}_{\rm cd}}\\[5pt]
&\leq \mathcal{C}\Big(\|\delta\breve{\varphi}_{B}\|_{1,\alpha;\breve{\Omega}^{(\rm e)}}
+\big\|(\breve{B}^{(\rm e)}_{0}-\underline{B}^{(\rm e)}, \breve{S}^{(\rm e)}_{0}-\underline{S}^{(\rm e)} )\big\|_{1,\alpha;\breve{\Gamma}^{(\rm e)}_{\rm in}}\Big)\big|m^{(\rm e)}_{A}-m^{(\rm e)}_{B}\big|\\[5pt]
&\leq \mathcal{C}{\color{black}\tilde{\epsilon}}\big|m^{(\rm e)}_{A}-m^{(\rm e)}_{B}\big|,
\end{split}
\end{eqnarray*}
and
\begin{eqnarray*}
\begin{split}
&\|\breve{g}_{0,AB}\|_{1,\alpha; \breve{\Gamma}^{(\rm e)}_{\rm in}}+\|\breve{g}_{+,AB}\|_{1,\alpha; \breve{\Gamma}_{+}}\\[5pt]
&\leq \mathcal{C}\Big\|\tilde{p}^{(\rm e)}_{0}(m^{(\rm e)}_{A}\breve{\eta})-\tilde{p}^{(\rm e)}_{0}(m^{(\rm e)}_{B}\breve{\eta})
-\underline{\tilde{p}}^{(\rm e)}_{0}(m^{(\rm e)}_{A}\breve{\eta})+\underline{\tilde{p}}^{(\rm e)}_{0}(m^{(\rm e)}_{B}\breve{\eta})\Big\|_{0,\alpha; \breve{\Gamma}^{(\rm e)}_{\rm ex}}\\[5pt]
&\ \ \  +\Big\|\big(\check{\beta}^{(A)}_{0,j}-\underline{\beta}_{0,j}\big)\check{\partial}^{(A)}_{j}\check{\varphi}_{AB}\Big\|_{0,\alpha; \breve{\Gamma}^{(\rm e)}_{\rm in}}
+\Big\|\big(\check{\beta}^{(A)}_{0,j}-\check{\beta}^{(B)}_{0,j}\big)\check{\partial}^{(A)}_{j}\delta\check{\varphi}_{B}\Big\|_{0,\alpha; \breve{\Gamma}^{(\rm e)}_{\rm in}}\\[5pt]
&\ \ \  +\Big\|\big(\check{\beta}^{(B)}_{0,j}-\underline{\beta}_{0,j}\big)\Big(\check{\partial}^{(A)}_{j}-\check{\partial}^{(A)}_{j}\Big)
\delta\check{\varphi}_{B}\Big\|_{0,\alpha; \breve{\Gamma}^{(\rm e)}_{\rm in}}\\[5pt]
&\leq \mathcal{C}\Big(\|\delta\breve{\varphi}_{B}\|_{1,\alpha;\breve{\Omega}^{(\rm e)}}
+\|\breve{p}^{(\rm e)}_{0}-\underline{p}^{(\rm e)}\|_{1,\alpha;{\color{black}\breve{\Gamma}^{(\rm e)}_{\rm in}}}\Big)\big|m^{(\rm e)}_{A}-m^{(\rm e)}_{B}\big|\\[5pt]
&\leq \mathcal{C}{\color{black}\tilde{\epsilon}}\big|m^{(\rm e)}_{A}-m^{(\rm e)}_{B}\big|.
\end{split}
\end{eqnarray*}
Here, the constant $\mathcal{C}>0$ depends only on the background state $\underline{U}$ and $\alpha$.
Now from  the similar argument in the proof of Theorem \ref{thm:3.3}, due to \eqref{eq:3.66}, if {\color{black}$\tilde{\epsilon}>0$} is sufficiently small, then in $\breve{\Omega}^{(\rm h)}$ there exists a constant $\mathcal{C}>0$ depending only on the background state $\tilde{\underline{U}}$, $L$ and $\alpha$ such that
\begin{eqnarray*}
\begin{split}
\big\|\breve{z}_{-,AB}\big\|_{0,\alpha;\breve{\Omega}^{(\rm h)}}
\leq \mathcal{C}\big\|\partial_{\breve{\eta}}\breve{z}_{-,B}\big\|_{0,\alpha; \breve{\Omega}^{(\rm h)}}
\big\|\breve{z}_{AB}\big\|_{0,\alpha; \breve{\Omega}^{(\rm h)}}\leq \mathcal{C}\tilde{\epsilon}
\big\|\breve{z}_{AB}\big\|_{0,\alpha; \breve{\Omega}^{(\rm h)}}.
\end{split}
\end{eqnarray*}

Next, we can solve $\breve{\varphi}_{AB}$ in $\breve{\Omega}^{(\rm e)}$ with the following boundary condition on $\breve{\Gamma}_{\rm cd}$:
\begin{eqnarray*}
\begin{split}
\tilde{\breve{\beta}}_{\rm cd, 1}\breve{\partial}^{(A)}_{\breve{\xi}}\breve{\varphi}_{AB}
+\tilde{\breve{\beta}}_{\rm cd, 2}\breve{\partial}^{(A)}_{\breve{\eta}}\breve{\varphi}_{AB}=\tilde{\breve{c}}_{\rm cd}(\breve{\xi}),
\end{split}
\end{eqnarray*}
where
\begin{eqnarray*}
\begin{split}
\tilde{\breve{\beta}}_{\rm cd, 1}=\breve{\beta}_{\rm cd, 1}
+\int^{1}_{0}\frac{d\varsigma}{1+(\varsigma\breve{\partial}^{(A)}_{\xi}\breve{\varphi}_{AB}
+\breve{\partial}^{(A)}_{\xi}\breve{\varphi}_{B})^{2}},\quad
\tilde{\breve{\beta}}_{\rm cd, 2}=\breve{\beta}_{\rm cd, 2},
\end{split}
\end{eqnarray*}
and
$\tilde{\breve{c}}_{\rm cd}(\breve{\xi})=-\breve{c}_{\rm cd}(\breve{\xi})+\arctan\breve{\partial}^{(B)}_{\xi}\breve{\varphi}_{B}
-\arctan\breve{\partial}^{(A)}_{\xi}\breve{\varphi}_{B}+\breve{z}_{-,AB}$.

Thus, by choosing a constant $\alpha_{0}\in(0,1)$ depending only on $\tilde{\underline{U}}$ and $L$ such that for $\alpha\in(0,\alpha_0)$, we have
\begin{eqnarray*}
\begin{split}
&{\big\|\breve{\varphi}_{AB}\big\|_{1,\alpha;\breve{\Omega}^{(\rm e)}}}\\[5pt]
\leq& \mathcal{C}\bigg(\|\breve{b}\|_{0,\alpha; \breve{\Omega}^{(\rm e)}}+\|\breve{g}_{0,AB}\|_{1,\alpha; {\color{black}\breve{\Gamma}^{(\rm e)}_{\rm in}}}+\|\breve{g}_{+,AB}\|_{1,\alpha; \breve{\Gamma}_{+}}
+\big\|\tilde{\omega}_{\rm e}(m^{(\rm e)}_{A}\breve{\eta})
-\tilde{\omega}_{\rm e}(m^{(\rm e)}_{B}\breve{\eta})\big\|_{0,\alpha;\breve{\Gamma}^{(\rm e)}_{\rm ex}}
+\|\tilde{\breve{c}}_{\rm cd}\|_{0,\alpha; \breve{\Gamma}_{\rm cd}}\bigg)\\[5pt]
\leq& \mathcal{C}{\color{black}\tilde{\epsilon}}\big|m^{(\rm e)}_{A}-m^{(\rm e)}_{B}\big|+\mathcal{C}{\color{black}\tilde{\epsilon}}\Big(\big\|\breve{z}_{AB}\big\|_{0,\alpha; \breve{\Omega}^{(\rm h)}}+\big\|\breve{\varphi}_{AB}\big\|_{1,\alpha;\breve{\Omega}^{(\rm e)}}\Big).
\end{split}
\end{eqnarray*}

Finally, we further estimate $\breve{z}_{+,AB}$ by using the boundary condition on $\breve{\Gamma}_{\rm cd}$ and the initial condition
on {\color{black}$\breve{\Gamma}^{(\rm h)}_{\rm in}$} to establish
\begin{eqnarray*}
\begin{split}
&\big\|\breve{z}_{+,AB}\big\|_{0,\alpha;\breve{\Omega}^{(\rm h)}}\\[5pt]
\leq& \mathcal{C}\bigg(\big\|2\breve{\beta}_{\rm cd, 1}\breve{\partial}^{(A)}_{\breve{\xi}}\breve{\varphi}_{AB}
+2\breve{\beta}_{\rm cd, 2}\breve{\partial}^{(A)}_{\breve{\eta}}\breve{\varphi}_{AB}+2\breve{c}_{\rm cd}(\breve{\xi})\big\|_{0,\alpha; \breve{\Gamma}_{\rm cd}}+\|\breve{z}_{-,AB}\|_{0,\alpha; \breve{\Gamma}_{\rm cd}}\bigg)+\mathcal{C}{\color{black}\tilde{\epsilon}}\big\|\breve{z}_{AB}\big\|_{0,\alpha; \breve{\Omega}^{(\rm h)}}\\[5pt]
\leq& \mathcal{C}\big\|\breve{\varphi}_{AB}\big\|_{1,\alpha;\breve{\Omega}^{(\rm e)}}
+\mathcal{C}\big\|\breve{c}_{\rm cd}(\breve{\xi})\big\|_{0,\alpha;\breve{\Omega}^{(\rm e)}}
+\mathcal{C}{\color{black}\tilde{\epsilon}}\big\|\breve{z}_{AB}\big\|_{0,\alpha; \breve{\Omega}^{(\rm h)}}\\[5pt]
\leq& \mathcal{C}{\color{black}\tilde{\epsilon}}\big|m^{(\rm e)}_{A}-m^{(\rm e)}_{B}\big|+\mathcal{C}{\color{black}\tilde{\epsilon}}\Big(\big\|\breve{z}_{AB}\big\|_{0,\alpha; \breve{\Omega}^{(\rm h)}}+\big\|\breve{\varphi}_{AB}\big\|_{1,\alpha;\breve{\Omega}^{(\rm e)}}\Big).
\end{split}
\end{eqnarray*}
%
Hence,
\begin{eqnarray*}
\begin{split}
\big\|\breve{\varphi}_{AB}\big\|_{1,\alpha;\breve{\Omega}^{(\rm e)}}
+\big\|\breve{z}_{AB}\big\|_{0,\alpha;\breve{\Omega}^{(\rm h)}}
\leq \mathcal{C}{\color{black}\tilde{\epsilon}}\Big(\|\breve{z}_{AB}\|_{0,\alpha; \breve{\Omega}^{(\rm h)}}
+\big\|\breve{\varphi}_{AB}\big\|_{1,\alpha;\breve{\Omega}^{(\rm e)}}+|m^{(\rm e)}_{A}-m^{(\rm e)}_{B}|\Big),
\end{split}
\end{eqnarray*}
where constant $\mathcal{C}>0$ depends only on $\underline{U}$, $\alpha$ and $L$. Thus, for  {\color{black}$\tilde{\epsilon}>0$} sufficiently small, it follows that
\begin{eqnarray}\label{eq:7.10}
\begin{split}
\big\|\breve{\varphi}_{AB}\big\|_{1,\alpha;\breve{\Omega}^{(\rm e)}}
+\big\|\breve{z}_{AB}\big\|_{0,\alpha;\breve{\Omega}^{(\rm h)}}
\leq \mathcal{C}{\color{black}\tilde{\epsilon}}|m^{(\rm e)}_{A}-m^{(\rm e)}_{B}|.
\end{split}
\end{eqnarray}
Therefore, by \eqref{eq:7.1}, we have that
\begin{eqnarray}\label{eq:7.11}
\begin{split}
\int^{m^{(\rm e)}_{B}}_{m^{(\rm e)}_{A}}\frac{d\tau}{\rho^{(\rm e)}_{B}u^{(\rm e)}_{B}}
=\int^{m^{(\rm e)}_{A}}_{0}\Big(\frac{1}{\rho^{(\rm e)}_{A}u^{(\rm e)}_{A}}-\frac{1}{\rho^{(\rm e)}_{B}u^{(\rm e)}_{B}}\Big)d\tau.
\end{split}
\end{eqnarray}
Then it follows from \eqref{eq:7.10} that
\begin{eqnarray*}
\begin{split}
\big|m^{(\rm e)}_{B}-m^{(\rm e)}_{A}\big|\leq \mathcal{C}\big\|\breve{\varphi}_{AB}\big\|_{1,\alpha;{\color{black}\breve{\Omega}^{(\rm e)}}}
\leq \mathcal{C}{\color{black}\tilde{\epsilon}}\big|m^{(\rm e)}_{A}-m^{(\rm e)}_{B}\big|.
\end{split}
\end{eqnarray*}
Let {\color{black}$\tilde{\epsilon}>0$} be sufficiently small such that {\color{black}$\mathcal{C}\tilde{\epsilon}\leq \frac{1}{2}$}. Then
\begin{eqnarray*}
\begin{split}
\big|m^{(\rm e)}_{B}-m^{(\rm e)}_{A}\big|\leq \frac{1}{2}\big|m^{(\rm e)}_{A}-m^{(\rm e)}_{B}\big|,
\end{split}
\end{eqnarray*}
which implies that the mapping $\mathcal{T}$ is contractive.

Based on the above argument, we can prove Theorem \ref{thm:3.2} easily.
\begin{proof}[Proof of Theorem \ref{thm:3.2}]
The existence and uniqueness of the solution to $(\mathbf{NP})$ follow from 
the Theorem \ref{thm:3.3} and the fact that the mapping $\mathcal{T}$ is contractive.
\end{proof}

\bigskip

\section*{Acknowledgements}
F. Huang  was partially supported by National Center for Mathematics and Interdisciplinary Sciences, AMSS, CAS and NSFC Grant No. 11371349 and 11688101.
J. Kuang was supported in part by the NSFC Project 11801549, NSFC Project 11971024 and the Start-Up Research Grant from Wuhan Institute of Physics and Mathematics, Chinese Academy of Sciences Project No. Y8S001104.
D. Wang was supported in part by NSF under grant  DMS-1907519.
W. Xiang was supported in part by the Research Grants Council of the HKSAR, China (Project No. CityU 11304817, Project No. CityU 11303518,   Project No. CityU 11304820, and Project No. CityU 11300021).

%
%
%
%
%
%
%
%

\end{document}